\documentclass[10pt,twoside]{article}
\usepackage{mathpazo}
\usepackage{geometry}

\usepackage{fullpage}
\usepackage{caption}
\usepackage{subcaption}
\usepackage{enumerate}

\input{commands}
\begin{document}

\begin{center}

{\bf{\LARGE{Hyperparameter tuning via trajectory predictions: Stochastic prox-linear methods in matrix sensing}}}

\vspace*{.2in}

{\large{
\begin{tabular}{ccc}
Mengqi Lou$^{\star}$, Kabir Aladin Verchand$^{\star,\ddagger}$, Ashwin Pananjady$^{\star, \dagger}$
\end{tabular}
}}
\vspace*{.2in}

\begin{tabular}{c}
Schools of $^\star$Industrial and Systems Engineering and
$^\dagger$Electrical and Computer Engineering, \\
Georgia Institute of Technology\\
$^\ddagger$Statistical Laboratory, University of Cambridge
\end{tabular}

\vspace*{.2in}

\today

\vspace*{.2in}

\begin{abstract}
Motivated by the desire to understand stochastic algorithms for nonconvex optimization that are robust to their hyperparameter choices, we analyze a mini-batched prox-linear iterative algorithm for the problem of recovering an unknown rank-1 matrix from rank-$1$ Gaussian measurements corrupted by noise. We derive a deterministic recursion that predicts the error of this method and show, using a non-asymptotic framework, that this prediction is accurate for any batch-size and a large range of step-sizes. In particular, our analysis reveals that this method, though stochastic, converges linearly from a local initialization with a fixed step-size to a statistical error floor.
Our analysis also exposes how the batch-size, step-size, and noise level affect the (linear) convergence rate and the eventual statistical estimation error, and we demonstrate how to use our deterministic predictions to perform hyperparameter tuning (e.g. step-size and batch-size selection) without ever running the method. On a technical level, our analysis is enabled in part by showing that the fluctuations of the empirical iterates around our deterministic predictions scale with the error of the previous iterate. 
\end{abstract}
\end{center}

\section{Introduction}\label{sec:intro}
We consider estimating a rank one matrix $\bcoefX_{\star} \bcoefZ_{\star}^{\top} \in \mathbb{R}^{d \times d}$ from online, i.i.d. observations $(y_i, \bx_i, \bz_i)$ drawn according to the statistical model
\begin{align} \label{eq:model}
y_i =  \langle \bx_i, \bcoefX_{\star} \rangle \cdot \langle \bz_i, \bcoefZ_{\star} \rangle + \epsilon_i.
\end{align}
Here $\bx_i \in \mathbb{R}^d$ and $\bz_i \in \mathbb{R}^d$ are sensing vectors, typically drawn i.i.d. from some distribution, and $\epsilon_i$ denotes zero-mean noise in the measurements.  This model finds applications in diverse areas of science and engineering, including astronomy, medical imaging, and communications~\citep{jefferies1993restoration,wang1998blind,campisi2017blind}.  For instance, it forms an example of the \emph{blind deconvolution} problem in statistical signal processing (see, e.g.,~\cite{recht2010guaranteed,ahmed2013blind} and the references therein for several applications of this problem).

We are interested in the model-fitting problem, and the natural least squares population objective $\widebar{L}: \mathbb{R}^d \times \mathbb{R}^d \rightarrow \mathbb{R}$ (corresponding to the scaled negative log-likelihood of our observations under Gaussian noise) can be written as
\begin{align} \label{eq:pop-loss}
\widebar{L}(\bcoefX, \bcoefZ) = 
\EE\big\{ \big( y - \langle \bx,\bcoefX \rangle \cdot \langle \bz, \bcoefZ \rangle \big)^{2} \big\}, 
\end{align}
where the conditional distribution of $y$ given $\bx, \bz$ is as specified by the model~\eqref{eq:model}.  Note that $\widebar{L}$ is a jointly nonconvex function in the parameters $(\bcoefX, \bcoefZ)$. With the goal of minimizing the population loss $\widebar{L}$, we consider online algorithms which operate on a mini-batch of size $m$ with $1\leq m\leq d$ for which we draw a \textit{fresh}\footnote{Although our indexing of observations does not reflect this, each mini-batch is drawn independently of all other observations.} set of observations $\{y_{i}, \bx_{i}, \bz_{i}\}_{i=1}^{m}$ at each iteration and form the averaged loss
\begin{align}\label{eq:global-loss-scalar-form}
	L_{m}(\bcoefX, \bcoefZ) = \frac{1}{m} \sum_{i=1}^{m} \big( y_{i} - \langle \bx_i,\bcoefX \rangle \cdot \langle \bz_i, \bcoefZ \rangle \big)^{2}.
\end{align}

Our particular focus is on an online, stochastic composite optimization method, which is a member of the aProx family~\citep[Eq. (4)]{asi2019stochastic}.  To elucidate the connection, we write the loss function $L_{m}$ in Eq.~\eqref{eq:global-loss-scalar-form} as a composition of two functions
\begin{align} \label{eq:global-loss}
L_m(\bcoefX, \bcoefZ) = \frac{1}{m} \big\| F_m(\bcoefX, \bcoefZ) \big\|_{2}^{2}, \qquad \text{ where } \qquad F_m(\bcoefX, \bcoefZ) = \by - (\bX \bcoefX) \odot (\bZ \bcoefZ),
\end{align}
where $\odot$ denotes the Hadamard product, and we collect responses into a vector $\by = [y_1 \;\vert\; \cdots \;\vert\; y_m]^{\top}$ as well as the sensing vectors into data matrices $\bX = [\bx_1\; \vert\; \cdots\;\vert\;\bx_m]^{\top} \in \mathbb{R}^{m \times d}$ and $\bZ = [\bz_1\;\vert\;\cdots\;\vert\; \bz_m]^{\top} \in \mathbb{R}^{m \times d}$.  Then the prox-linear update for each iteration  $t=0,1,\dots, T-1$ is given by
\begin{align} \label{eq:prox-linear-updates}
	\begin{bmatrix} \bcoefX_{t+1} \\ \bcoefZ_{t+1} \end{bmatrix} &= \argmin_{\bcoefX,\bcoefZ \in \mathbb{R}^{d}}\; \frac{1}{m} \bigg \| F_m(\bcoefX_t,\bcoefZ_t) + \nabla F_{m}(\bcoefX_t,\bcoefZ_t) \left[\begin{array}{c} \bcoefX - \bcoefX_t \\ \bcoefZ - \bcoefZ_t \end{array}\right] \bigg \|_{2}^{2} + \lambda \cdot \bigg\| \left[\begin{array}{c} \bcoefX - \bcoefX_t \\ \bcoefZ - \bcoefZ_t \end{array}\right] \bigg \|_{2}^{2},
\end{align}
where $\nabla F_{m} \in \mathbb{R}^{m\times 2d}$ denotes the Jacobian of $F_{m}$. Here, $\lambda>0$ is a hyperparameter that can be interpreted as an inverse step-size. Note that the prox-linear update in~\eqref{eq:prox-linear-updates} encompasses complex methods, e.g., it is equivalent to the Gauss--Newton method when $\lambda = 0$.

To set the stage for our analysis to follow, we write the update in~\eqref{eq:prox-linear-updates} in closed form. To this end, we denote $G_i = \bx_i^{\top} \bcoefX_{t}, \GZ_i = \bz_i^{\top} \bcoefZ_t$ and $\ba_{i}^{\top} = \big[\GZ_i \bx_{i}^{\top} \;\; G_i \bz_{i}^{\top}\big]$ for each $1\leq i\leq m$; define the pair of diagonal matrices $\bW = \diag(\bX \bcoefX_t)$, $\widetilde{\bW} = \diag(\bZ \bcoefZ_{t})$ and collect the vectors $\ba_{i}$ into a concatenated data matrix $\bA = [\ba_1 \; \vert \; \ba_2 \; \vert \; \ldots \; \vert \; \ba_m]^{\top} = \bigl[\widetilde{\bW}\bX \; \vert \; \bW \bZ\bigr] \in \mathbb{R}^{m \times 2d}$.
Armed with this notation, we explicitly solve the quadratic program defining the prox-linear iterates in~\eqref{eq:prox-linear-updates} to obtain
\begin{align}\label{eq:closed-form-update}
	\left[\begin{array}{c} \bcoefX_{t+1} \\ \bcoefZ_{t+1} \end{array}\right]  = \bigl(\bA^{\top} \bA + \lambda m \bI\bigr)^{-1} \cdot \Bigl(\bA^{\top} \bigl(\by + \diag(\bW \bWtil)\bigr) + \lambda m \cdot \left[\begin{array}{c} \bcoefX_{t} \\ \bcoefZ_{t} \end{array}\right]\Bigr)=: \prox\big([ \bcoefX_t \;\vert\; \bcoefZ_t]\big),
\end{align}
where we have defined a (random) update function $\prox: \mathbb{R}^{2d} \rightarrow \mathbb{R}^{2d}$ for convenience.
Note that $\bW, \bWtil$ are diagonal matrices, and consequently $\diag(\bW \bWtil)$ is a vector in $\mathbb{R}^m$.  

\subsection{Motivation and main contributions} \label{sec:motivation}
Our goal is to sharply characterize how the problem parameters including batch-size $m$, inverse step-size $\lambda$, and noise level $\sigma$ affect the convergence behavior of the prox-linear update in Eq.~\eqref{eq:closed-form-update}.  We begin by situating our results in the context of a motivating numerical experiment. 

\paragraph{Motivation \#1: Fine-grained convergence phenomena.}  Set the dimension to $d = 200$ and noise level to $\sigma = 0.01$.  We initialize the algorithm locally around the ground truth parameters, setting both $\bcoefX_{0}$ and $\bcoefZ_{0}$ such that $\|\bcoefX_{0} - \bcoefX_{\star}\|_{2}^{2} = \|\bcoefZ_{0} - \bcoefZ_{\star}\|_{2}^{2} = 0.02$.  Subsequently, for a fixed batch-size $m$, inverse step-size $\lambda$ and number of iterations $T$, we run the stochastic prox-linear method in~\eqref{eq:prox-linear-updates}.  At each iteration, we sample independent data $(\bx_i,\bz_i)_{i=1}^{m} \overset{\mathsf{i.i.d.}}{\sim} \mathsf{N}(\boldsymbol{0},\bI_d)$ and noise $(\epsilon_i)_{i=1}^{m} \overset{\mathsf{i.i.d.}}{\sim} \mathsf{N}(0,\sigma^{2})$ to form the responses $(y_i)_{i=1}^{m}$ according to the model~\eqref{eq:model}.  

We isolate the effect of batch-size and inverse step-size by considering two experiments.  In the first, we vary the batch-size $m = 8, 16, 32$.  For each batch-size, we set $\lambda = 100$ for each iteration $t \leq 1500$ and subsequently increase the inverse step-size as $\lambda = 100 + t$ for $t > 1500$.  Increasing the inverse step-size in this manner corresponds to a linearly decaying step-size.  We then plot, in Figure~\ref{subfig:intro-a}, the estimation error\footnote{The metric $\Err_t$ is equivalent to the Frobenius norm error $\|\bcoefX_{t}\bcoefZ_{t}^{\top} - \bcoefX_{\star} \bcoefZ_{\star}^{\top}\|_{F}^{2}$ up to a multiplicative, universal constant factor.} $\mathsf{Err}_t$  over iterations.  
In our second experiment, we fix the batch-size to $m=32$ and set three initial inverse step-sizes $\lambda_{0} = 1,10,100$. For each initial inverse step-size $\lambda_{0}$, we then fix the inverse step-size $\lambda_t = \lambda_{0} + t \cdot \mathbbm{1}\{t > 1500\}$.  That is, as before, we decay the step-size linearly after a constant number of iterations.  We plot the estimation error $\Err_{t}$ over iterations in Figure~\ref{subfig:intro-b}. 

In both of the experiments, the method appears to display three distinct phases in its iterations.  In the first phase, the error decays linearly to a point of stagnation.  Crucially, both the linear rate of convergence as well as error floor reached after this rapid decay appear to depend on both the batch-size as well as the step-size.  In the second phase, the iterates stagnate until the step-size is decreased.  Finally, in the third phase, the error decreases at a sub-linear rate when the step-size is decreased.

While several recent results~\citep[see, e.g., ][]{davis2019stochastic,chadha2022accelerated,asi2020minibatch} have characterized the third phase of sub-linear convergence, several questions remain regarding the first two phases.  For instance, the only linear rates of convergence for stochastic prox-linear iterates that we are aware of require either (i.) noiseless observations in which $\sigma = 0$~\citep{asi2019importance,chadha2022accelerated} or (ii.) sharp growth~\citep[Eq. (2.4)]{davis2023stochastic}.  We note that neither does the loss in~\eqref{eq:global-loss} enjoy sharp growth around its minimizers, nor is $\sigma = 0$ in the experiments shown in Figure~\ref{fig:intro}.  

Motivated by these observations, we focus on the initial phase and characterize the dependence of both the linear rate of convergence as well as the initial noise floor on the batch-size $m$, the initial inverse step-size $\lambda$, and the level of noise $\sigma$.

\paragraph{Motivation \#2: Efficient hyperparameter tuning.} 
The experiment results in Figure~\ref{fig:intro} demonstrates the trade-offs for the choice of batch-size $m$ and inverse step-size $\lambda^{-1}$. For example, for larger $m$, the method enjoys a faster convergence rate but also incurs a larger computational cost at each iteration; on the other hand, for smaller $\lambda$, the method enjoys a faster convergence rate but stagnates at a larger error floor. In order to achieve both computational efficiency and estimation accuracy, a practitioner needs to jointly tune the batch-size $m$ and inverse step-size $\lambda^{-1}$. One approach is via \emph{hyperparameter tuning} which consists of selecting a sequence of combinations of $m$ and $\lambda$, subsequently running the update~\eqref{eq:closed-form-update} for each choice of $m$ and $\lambda$, and selecting the $m$ and $\lambda$ whose empirical performance is the best. The aforementioned approach can be computationally expensive if we want to test many choices of $m$ and $\lambda$ for a high-dimensional problem.  Motivated by this, we develop a low-dimensional, deterministic trajectory prediction that can be run efficiently and without using any data.  The produced trajectory predictions can then be used to tune the hyperparameters in an offline fashion.

\begin{figure}
	\centering
	\begin{subfigure}[b]{0.48\textwidth}
		\centering
		\includegraphics[scale=0.5]{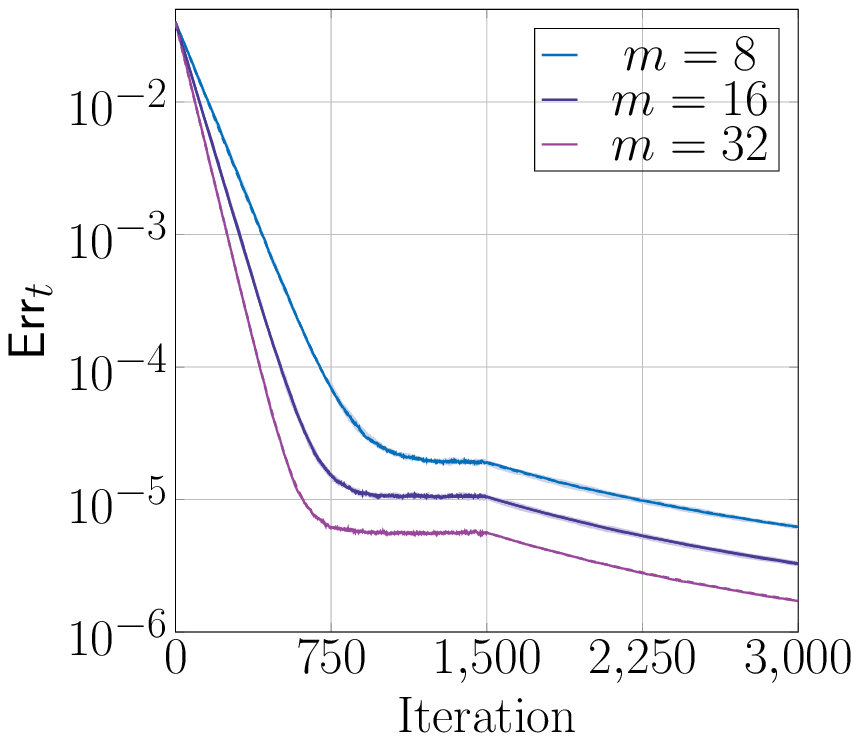}
		\caption{$\sigma = 10^{-2},\; d = 200.$}    
		\label{subfig:intro-a}
	\end{subfigure}
	\hfill
	\begin{subfigure}[b]{0.48\textwidth}  
		\centering 
		\includegraphics[scale=0.5]{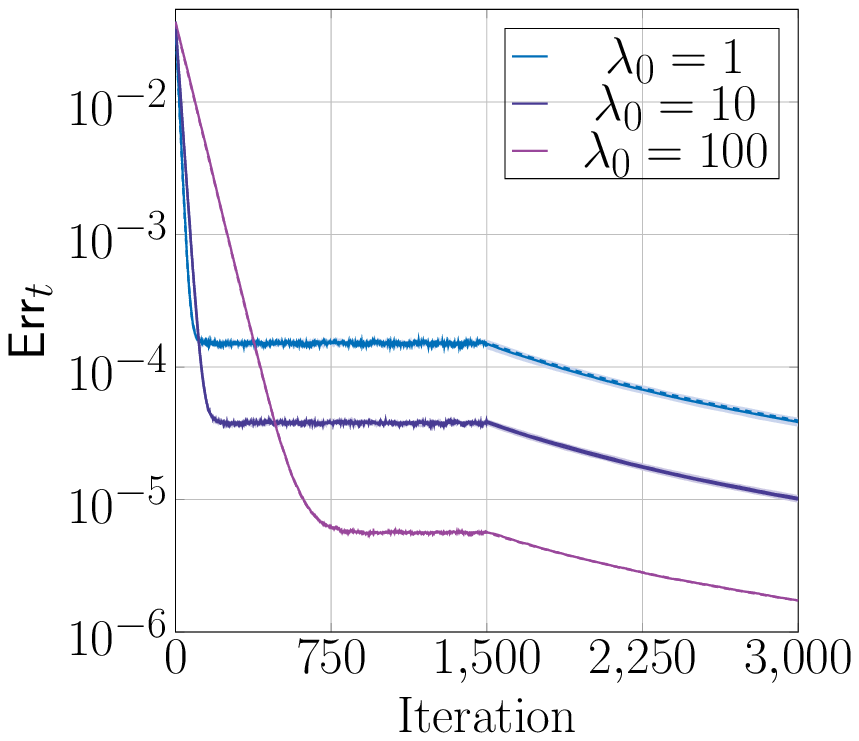}
		\caption{$\sigma = 10^{-2},\;d=200.$}
		\label{subfig:intro-b}
	\end{subfigure}
	\caption{Panel (a) demonstrates the convergence behavior of different batch-sizes $m = 8,16,32$. Panel (b) demonstrates the convergence behavior of different initial inverse step-sizes $\lambda_{0} = 1,10,100$. Each experiment consists of $30$ independent trials and shaded envelopes denote the interquartile range over the $30$ independent trials. Solid lines denote the median of $\Err_{t}$ over the independent trials and dash lines (barely visible) denote the deterministic predicted error $\Err_{t}^{\mathsf{seq}}$ (see Section~\ref{sec:experimental-results} for its definition).} 
	\label{fig:intro}
\end{figure}

\paragraph{Main Contributions.}
In order to describe our main contributions, we first require some definitions and assumptions. Recall that our goal is to estimate the ground-truth pair $(\bcoefX_{\star}, \bcoefZ_{\star})$. To assess the convergence of the prox-linear iterations to this pair, we consider a four-dimensional state $\alpha: \mathbb{R}^d \rightarrow \mathbb{R}, \beta: \mathbb{R}^d \rightarrow \mathbb{R}, \widetilde{\alpha}: \mathbb{R}^d \rightarrow \mathbb{R},$ and $\widetilde{\beta}: \mathbb{R}^d \rightarrow \mathbb{R}$, whose components we define as
\begin{align}\label{eq:generic-state-evolution}
	\parcompX(\bcoefX) := \langle \bcoefX, \bcoefX_{\star}\rangle, \quad \perpcompX(\bcoefX) := \| \bP_{\bcoefX_{\star}}^{\perp} \bcoefX \|_2, \quad \parcompZ(\bcoefZ) := \langle \bcoefZ, \bcoefZ_{\star}\rangle, \quad \perpcompZ(\bcoefZ) := \| \bP_{\bcoefZ_{\star}}^{\perp} \bcoefZ \|_2. 
\end{align}
Using this, we define our state at time $t$ as the quadruple consisting of components
\begin{align} \label{eq:comps}
	\parcompX_{t} = \parcompX(\bcoefX_{t}), \quad  \perpcompX_{t} = \perpcompX(\bcoefX_{t}), \quad
	\parcompZ_{t} = \parcompZ(\bcoefZ_{t}), \quad  \perpcompZ_{t} = \perpcompZ(\bcoefZ_{t}).
\end{align}
Throughout, we make the following assumptions.
\begin{assumption}\label{assumption-unit-norm}
	The coefficient vectors $\bcoefX_{\star},\bcoefZ_{\star}$ satisfy $\|\bcoefX_{\star}\|_{2} = \| \bcoefZ_{\star} \|_{2} = 1$.
\end{assumption}
\begin{assumption}\label{assumption}
	The sensing vectors are drawn as $\{\bx_i,\bz_i\}_{i=1}^{m} \overset{\mathsf{i.i.d.}}{\sim} \mathsf{N}(0,\bI_{d}) \otimes \mathsf{N}(0,\bI_{d})$ and the noise as $\{\epsilon_i\}_{i=1}^{m} \overset{\mathsf{i.i.d.}}{\sim} \mathsf{N}(0,\sigma^2)$, independently of the sensing vectors.
\end{assumption}

Equipped with this notation, we turn now to descriptions of our main results. In short, they characterize how the stochastic prox-linear iterates behave as a function of the batch-size $m$, the inverse step-size $\lambda$, and the noise level $\sigma$.
\begin{enumerate}
	\item \textbf{Sharp, deterministic predictions which adapt to problem error.} Consider running one-step of the prox-linear update specified by the function $\prox$~\eqref{eq:closed-form-update} starting from a pair $(\bcoefX_{\sharp}, \bcoefZ_{\sharp})$ and let $[\bcoefX_{+}^{\top} \; \vert \; \bcoefZ_{+}^{\top}]^{\top} = \prox([\bcoefX_{\sharp} \; \vert\; \bcoefZ_{\sharp}])$. 
	For \emph{all} minibatch-sizes $1\leq m\leq d$ and a large range of step-size $\lambda \gtrsim (1+\sigma)d/m$,
	we derive explicit deterministic predictions $(\parcompdetX_{+},\perpcompdetX_{+},\parcompdetZ_{+},\perpcompdetZ_{+})$ (see Section~\ref{sec:expilicit-formulas-prediction} for their precise forms) that closely track their empirical counterparts.  Furthermore, in Theorem~\ref{thm:one-step-prediction} to follow, we show that with high probability,
	\[
	\max\Big\{ \big|\parcompX(\bcoefX_{+}) - \parcompdetX_{+} \big|,\big|\perpcompX(\bcoefX_{+}) - \perpcompdetX_{+} \big|, \big|\parcompZ(\bcoefZ_{+}) - \parcompdetZ_{+}\big|, \big|\perpcompZ(\bcoefZ_{+}) - \perpcompdetZ_{+}\big|\Big\} \lesssim \frac{\|\bcoefX_{\sharp} \bcoefZ_{\sharp}^{\top} - \bcoefX_{\star} \bcoefZ_{\star}^{\top}\|_{F}+\sigma}{\lambda \sqrt{m}},
	\]
	where $\lesssim$ hides polylogarithmic factors in $d$.  Note that this guarantee is fully non-asymptotic, and
	%---in contrast to previous work~\citep{chandrasekher2022alternating,chandrasekher2023sharp}---
	provides bounds on the deviation which scale with the estimation error $\|\bcoefX_{\sharp} \bcoefZ_{\sharp}^{\top} - \bcoefX_{\star} \bcoefZ_{\star}^{\top}\|_{F}$.  This, in turn, enables a transparent convergence analysis of the iterations for all noise levels $\sigma\geq 0$. See Section~\ref{sec:one-step-prediction} for a detailed discussion.
	
	\item \textbf{Fine-grained convergence analysis.} We use our deterministic predictions to execute an iterate-by-iterate analysis of the stochastic prox-linear algorithm from a local initialization. This analysis reveals several fine-grained properties of the convergence behavior. In particular, for the step-size choice $\lambda^{-1} \asymp m/(d(1+\sigma^2))$ and batch-size $m \gtrsim \mathsf{polylog}(d)$, we show that it takes
	\[
	\tau = \Theta\Big( \frac{d(1+\sigma^2)}{m} \cdot \log\Big(\frac{1}{\sigma^{2}} \Big) \Big)\quad \text{iterations to guarantee} \quad \|\bcoefX_{\tau} \bcoefZ_{\tau}^{\top} - \bcoefX_{\star} \bcoefZ_{\star}^{\top}\|_{F}^{2} \lesssim \sigma^2.
	\] 
	This reveals a linear speed-up in the batch-size $m$ for \emph{all} noise levels $\sigma \geq 0$. As a consequence, the total sample complexity for reaching estimation error $\sigma^{2}$ is $O(d(1+\sigma^2)\log(1/\sigma^2))$.  Moreover, for other step-size choices $\lambda^{-1} \lesssim m/(d(1+\sigma^2))$, it takes 
	\[
	\tau = \Theta\Big(\lambda \cdot \log\Big( \frac{\lambda m}{d\sigma^2} \Big) \Big) \quad \text{iterations to guarantee} \quad
	\|\bcoefX_{\tau} \bcoefZ_{\tau}^{\top} - \bcoefX_{\star} \bcoefZ_{\star}^{\top}\|_{F}^{2} \lesssim \frac{\sigma^2d}{\lambda m},
	\]
which in turn quantifies the dependence of the convergence behavior on the step-size $\lambda^{-1}$. In particular, by decreasing the step-size $\lambda^{-1}$, the iteration complexity increases while the eventual estimation error decreases. We emphasize that our guarantees on iteration complexity are sharp in the sense that our bounds provide both upper and lower bounds on the rate of convergence. See Theorem~\ref{thm:local-sharp-convergence-results} for a precise statement and  Section~\ref{sec:sharp-local-convergence-result} for a detailed discussion.  
\end{enumerate}

\subsection{Related work}
We situate our results within the broader literature on stochastic aProx methods as well as learning dynamics, beginning with stochastic aProx methods.  

\paragraph{Stochastic aProx methods} Several recent works have focused on convergence guarantees for aProx methods.~\cite{duchi2018stochastic} proved that these methods converge to first-order stationary points for weakly convex functions.~\cite{asi2019importance} showed these methods exhibit robustness to problem families and algorithmic parameters.
~\cite{davis2019stochastic} established the sublinear convergence rate $O(T^{-1/4})$ for weakly convex functions, where $T$ is the number of iterations.  In turn,~\cite{asi2019stochastic} identified the class of interpolation problems---for which each component function in the empirical\footnote{Some of these papers define a stronger interpolation condition that depends on the population loss~\eqref{eq:pop-loss}, but in both the empirical and population definitions, the interpolation condition is only satisfied in our problem when $\sigma = 0$.} loss~\eqref{eq:global-loss-scalar-form} shares a minimizer---as a benign class under which stochastic prox-linear methods enjoy linear convergence under expected strong growth conditions on the stochastic objective.  Note that in our setting, any interpolation problem is noiseless, with $\sigma = 0$. \cite{asi2020minibatch} and~\cite{chadha2022accelerated} extended these methods to accommodate mini-batches of size $m$ and established the sublinear convergence rate $O((Tm)^{-1/2})$ for convex functions, as well as a linear convergence rate for interpolation problems under a so-called $\gamma$-growth condition~\cite[Ass. 4]{chadha2022accelerated}.~\cite{davis2023stochastic} proved that these methods with geometric step decay enjoy local linear rate of convergence for non-convex problems that satisfy the sharp growth condition.  We also mention that---under the assumption that the component functions are strongly convex and smooth---\cite{vaswani2022towards} established an initial linear convergence rate of SGD (not the prox-linear method) under a step-size schedule which is both noise adaptive and problem parameter adaptive.
The aforementioned results are all geometric in nature and are not comparable with our own, which are based on a statistical model that does not always satisfy their conditions. On the one hand, our results are proved for a canonical statistical model that possesses its own problem-specific, geometric structure. On the other hand, they are much more fine-grained and sharp, and expose several phenomena in addition to those mentioned above. We defer a detailed comparison to the discussion following Theorem~\ref{thm:local-sharp-convergence-results} in Section~\ref{sec:sharp-local-convergence-result}. 

\paragraph{Learning dynamics and trajectory analyses} 
Another line of relevant literature focuses on deriving learning dynamics for iterative algorithms. In particular, one line of work characterized the dynamics of SGD for least squares problems~\citep{paquette2021sgd,paquette2021dynamics,balasubramanian2023high} and for non-convex problems~\citep{collins2023hitting,arous2021online,arous2023high}. Although these previous works provide powerful machinery, they do not straightforwardly apply to complex algorithms such as the prox-linear method. To see this more clearly, note that the SGD update with step-size $\gamma$ for our loss $L_{m}(\bcoefX,\bcoefZ)$~\eqref{eq:global-loss} can be written as
\[
\begin{bmatrix} \bcoefX_{t+1} \\ \bcoefZ_{t+1} \end{bmatrix} = \begin{bmatrix} \bcoefX_{t} \\ \bcoefZ_{t} \end{bmatrix} - \gamma \cdot \nabla L_{m} (\bcoefX_{t},\bcoefZ_{t}).
\]
Since the SGD update provides a simple relationship between the next iterate $(\bcoefX_{t+1},\bcoefZ_{t+1})$ and the previous iterate $(\bcoefX_{t},\bcoefZ_{t})$, one can directly apply a Taylor expansion to relate the statistics of the two iterates, which forms an important step in the works~\cite{collins2023hitting,arous2023high}. However, the prox-linear update~\eqref{eq:closed-form-update} is more complicated than SGD, e.g., there is a complicated random matrix inverse $(\bA^{\top} \bA + \lambda m \bI)^{-1}$, whence it becomes difficult to apply the Taylor expansion trick to derive learning dynamics for the prox-linear update. Another line of work characterized the asymptotic behavior of ``generalized" first-order methods~\citep{celentano2020estimation,celentano2021high} using techniques from the literature of approximate message passing or AMP~\citep{donoho2009message,bayati2011lasso} in the asymptotic regime where $m,d \rightarrow +\infty$ and $m/d \rightarrow \delta \in (0,+\infty)$.  These analyses offer the distinct advantage that the method need not be online, and samples may be re-used.  However, they do not apply directly to higher-order methods like the method considered in this paper.  Moreover, we are interested in the non-asymptotic regime in which $1 \leq m \leq d$ and $d$ is finite.  Our non-asymptotic analysis allows us to bypass the so-called ``extensive" batch-size assumption in which the batch-size $m$ scales linearly with the dimension $d$~\citep[see, e.g.,][]{gerbelot2022rigorous}.  The most related recent work is the sequence of papers~\citep{chandrasekher2022alternating,chandrasekher2023sharp} in which the authors derived deterministic predictions beyond first-order methods. However, these predictions require the batch-size $m\geq d$ and thus cannot be directly applied to analyze mini-batched algorithms. Additionally, these predictions are only accurate up to fluctuations of order $m^{-1/2}$, which means that the predictions become meaningless when the 
error of interest falls below the level $m^{-1/2}$. This becomes especially problematic for analyzing iterative algorithms for low noise problems ($\sigma \downarrow 0$) and for small batch-sizes $m$.

\subsection{Notation and organization}

We let $[d]$ denote the set of natural numbers less than or equal to $d$, let $\mathbbm{1}\{\cdot\}$ denote the
indicator function and let $\mathcal{S}^{d-1} = \{\bv \in \mathbb{R}^{d} \;|\; \|\bv\|_{2} = 1\}$. For two sequences of non-negative reals $\{f_n\}_{n
	\geq 1}$ and $\{g_n \}_{n \geq 1}$, we use $f_n \lesssim g_n$ to indicate that
there is a universal positive constant $C$ such that $f_n \leq C g_n$ for all
$n \geq 1$. The relation $f_n \gtrsim g_n$ indicates that $g_n \lesssim f_n$,
and we say that $f_n \asymp g_n$ if both $f_n \lesssim g_n$ and $f_n \gtrsim
g_n$ hold simultaneously. We also use standard order notation $f_n = \order
(g_n)$ to indicate that $f_n \lesssim g_n$ and $f_n = \ordertil(g_n)$ to
indicate that $f_n \lesssim
g_n \log^c n$, for a universal constant $c>0$. We say that $f_n = \Omega(g_n)$ (resp. $f_n = \widetilde{\Omega}(g_n)$) if $g_n = \order(f_n)$ (resp. $g_n = \ordertil(f_n)$). The notation $f_n = o(g_n)$ is
used when $\lim_{n \to \infty} f_n / g_n = 0$, and $f_n =
\omega(g_n)$ when $g_n = o(f_n)$. Throughout, we use $c, C$ to denote universal
positive constants, and their values may change from line to line. We denote by $\mathsf{N}(\bm{\mu}, \bSig)$ a normal distribution with mean $\bm{\mu}$ and covariance matrix $\bSig$. We say that $X \overset{(d)}{=} Y$ for two random variables $X$ and $Y$ that are equal in distribution.

The remainder of the paper is organized as follows. We provide our main results in Section~\ref{sec:main-results}. In Section~\ref{sec:one-step-prediction}, we provide our one-step deterministic predictions for the prox-linear update (see Theorem~\ref{thm:one-step-prediction}); in Section~\ref{sec:sharp-local-convergence-result}, we use our one-step updates to prove a sharp linear convergence result for the prox-linear update from a local initialization (see Theorem~\ref{thm:local-sharp-convergence-results}); and in Section~\ref{sec:experimental-results}, we present numerical experiments which demonstrate how our deterministic predictions can be used to tune the hyperparameters without running the prox-linear update. In Section~\ref{sec:overview-techniques}, we provide a heuristic calculation to demonstrate the techniques for deriving the deterministic predictions and provide the explicit formulas of the deterministic predictions. We prove Theorem~\ref{thm:one-step-prediction} and Theorem~\ref{thm:local-sharp-convergence-results} in Sections~\ref{thm:one-step-prediction-main-proof} and~\ref{sec:main-proof-convergence-result}, respectively. Proofs of technical lemmas are postponed to the appendices.
\section{Main Results}\label{sec:main-results}
In this section, we provide our main results on the stochastic prox-linear method in~\eqref{eq:closed-form-update}.  In Section~\ref{sec:one-step-prediction}, we provide a non-asymptotic, one-step guarantee.  Then, in Section~\ref{sec:sharp-local-convergence-result}, we leverage this one-step guarantee to prove a two-sided convergence guarantee for the prox-linear method.

\subsection{Deterministic one-step predictions}\label{sec:one-step-prediction}
Consider a pair $(\bcoefX_{\sharp}, \bcoefZ_{\sharp})$ and let $(\parcompX_{\sharp}, \perpcompX_{\sharp}, \parcompZ_{\sharp}, \perpcompZ_{\sharp})$ denote its corresponding state as defined in Eq.~\eqref{eq:generic-state-evolution}.  The deterministic predictions---starting from the pair $(\bcoefX_{\sharp}, \bcoefZ_{\sharp})$---are specified as 
\begin{subequations}\label{alpha-beta-det-prediction}
	\begin{align}
	\hspace{-0.4cm}
		\parcompdetX_{+} &= \alphaXmap_{m,d,\sigma,\lambda}(\parcompX_{\sharp},\perpcompX_{\sharp},\parcompZ_{\sharp},\perpcompZ_{\sharp}), \quad(\perpcompdetX_{+})^{2} = \big( \iotaXmap_{m,d,\sigma,\lambda}(\parcompX_{\sharp},\perpcompX_{\sharp},\parcompZ_{\sharp},\perpcompZ_{\sharp}) \big)^{2} + \etaXmap_{m,d,\sigma,\lambda}(\parcompX_{\sharp},\perpcompX_{\sharp},\parcompZ_{\sharp},\perpcompZ_{\sharp}),
		 \\
	\hspace{-0.4cm}
		\parcompdetZ_{+} &=  \alphaZmap_{m,d,\sigma,\lambda}(\parcompX_{\sharp},\perpcompX_{\sharp},\parcompZ_{\sharp},\perpcompZ_{\sharp}), \quad (\perpcompdetZ_{+})^{2} = \big( \iotaZmap_{m,d,\sigma,\lambda}(\parcompX_{\sharp},\perpcompX_{\sharp},\parcompZ_{\sharp},\perpcompZ_{\sharp}) \big)^{2} + \etaZmap_{m,d,\sigma,\lambda}(\parcompX_{\sharp},\perpcompX_{\sharp},\parcompZ_{\sharp},\perpcompZ_{\sharp}),
	\end{align}
\end{subequations}
where $\alphaXmap_{m,d,\sigma,\lambda},\;\alphaZmap_{m,d,\sigma,\lambda},\;\iotaXmap_{m,d,\sigma,\lambda},\;\iotaZmap_{m,d,\sigma,\lambda},\;\etaXmap_{m,d,\sigma,\lambda},\;\etaZmap_{m,d,\sigma,\lambda}:\mathbb{R}^{4} \rightarrow \mathbb{R}$ are functions of the states $(\parcompX_{\sharp},\perpcompX_{\sharp},\parcompZ_{\sharp},\perpcompZ_{\sharp})$, which are also parameterized by the tuple of problem-specific parameters $(m, d, \sigma, \lambda)$. Their explicit expressions are provided in Section~\ref{sec:expilicit-formulas-prediction}.  

We state our one-step guarantee in terms of the quantity $\Err_\sharp$, defined as 
\begin{align}\label{eq-definition-error}
	\Err_{\sharp} = (\parcompX_{\sharp}\parcompZ_{\sharp}-1)^{2} + \perpcompX_{\sharp}^{2} + \perpcompZ_{\sharp}^{2}.
\end{align}
We note (see Lemma~\ref{fixed-point-equations-unique-solution}) that if $\perpcompX_{\sharp},\perpcompZ_{\sharp} \leq 0.1$ and $0.3 \leq \|\bcoefX_\sharp\|_{2}, \|\bcoefZ_\sharp\|_{2} \leq 1.7$, then 
\begin{align}\label{equivalent-Errt-frobenius-distance}
	\frac{1}{5} \cdot \|\bcoefX_{\sharp} \bcoefZ_{\sharp}^{\top} - \bcoefX_{\star} \bcoefZ_{\star}^{\top} \|_{F}^{2} \leq \Err_{\sharp} \leq 12.5 \cdot \|\bcoefX_{\sharp} \bcoefZ_{\sharp}^{\top} - \bcoefX_{\star} \bcoefZ_{\star}^{\top} \|_{F}^{2},
\end{align}
so that the reader should think of $\Err_{\sharp}$ as equivalent to $\|\bcoefX_{\sharp} \bcoefZ_{\sharp}^{\top} - \bcoefX_{\star} \bcoefZ_{\star}^{\top} \|_{F}^{2}$ up to a universal constant.  We are now poised to state our main result for this section, which shows that the tuple of random variables $(\alpha(\bcoefX_{+}), \beta(\bcoefX_{+}), \alpha(\bcoefZ_{+}), \beta(\bcoefZ_{+}))$ is closely tracked by the deterministic predictions in Eq.~\eqref{alpha-beta-det-prediction}.
\begin{theorem}\label{thm:one-step-prediction}
	Suppose data are drawn from the model~\eqref{eq:model} and let Assumptions~\ref{assumption-unit-norm} and~\ref{assumption} hold.  Let $\bcoefX_{\sharp}, \bcoefZ_{\sharp} \in \mathbb{R}^d$ satisfy $K_1 \leq \| \bcoefX_\sharp \|_2, \| \bcoefZ_\sharp \|_2 \leq K_2$ for a pair of universal, positive constants $K_1 \leq K_2$.  Consider the updates specified by the function $\prox$~\eqref{eq:closed-form-update} and let $[\bcoefX_{+} \; \vert \; \bcoefZ_{+}]^{\top} = \prox\bigl([\bcoefX_{\sharp} \; \vert \; \bcoefZ_{\sharp}]\bigr)$.  Use the shorthand $(\parcompX_{+}, \perpcompX_{+}, \parcompX_{+}, \perpcompZ_{+}) = (\alpha(\bcoefX_{+}), \beta(\bcoefX_{+}), \alpha(\bcoefZ_{+}), \beta(\bcoefZ_{+}))$~\eqref{eq:generic-state-evolution} and let the deterministic predictions $(\parcompX_{+}^{\det}, \perpcompX_{+}^{\det}, \parcompZ_{+}^{\det}, \perpcompZ_{+}^{\det})$ be as in Eq.~\eqref{alpha-beta-det-prediction}.  Then, there exist positive constants $d_{0},C_{1},C_{2}$, depending only on $K_1$ and $K_2$, such that for
	\[
	1\leq m\leq d, \qquad \qquad \lambda \geq C_{1}\frac{(1+\sigma)d}{m}, \qquad \qquad   \text{ and } \qquad \qquad d\geq d_{0},
	\]
	the following hold with probability at least $1 - d^{-20}$.
	\begin{subequations}\label{ineq:one-step-deviation}
		\begin{itemize}
			\item[(a)] The parallel components satisfy
			\begin{align}
				\big|\parcompX_{+} - \parcompdetX_{+} \big| \vee \big|\parcompZ_{+} - \parcompdetZ_{+}\big| \leq C_{2}\max\bigg\{ \frac{ \sqrt{\Err_{\sharp}} + \sigma}{\lambda} \cdot \frac{\log^{6}(d)}{\sqrt{m}} ,\; d^{-30}\bigg\}.
			\end{align}
			\item[(b)] The perpendicular components satisfy
			\begin{align}
				\big| \perpcompX_{+}^{2} - (\perpcompdetX_{+})^{2} \big| \vee \big| \perpcompZ_{+}^{2} - (\perpcompdetZ_{+})^{2}\big| &\leq C_{2}  \max\bigg\{  \frac{\Err_{\sharp} + \sigma^{2} }{\lambda}  \cdot \frac{\log^{6}(d)}{\sqrt{m}} ,\; d^{-30}\bigg\}.
			\end{align}
		\end{itemize}
	\end{subequations}
\end{theorem}

We provide the proof of Theorem~\ref{thm:one-step-prediction} in Section~\ref{thm:one-step-prediction-main-proof}. Note that we have made no attempt to optimize the log factors in either deviation component, and these can likely be made smaller. Some discussion of the result appears below; we defer commentary on the proof technique to Section~\ref{sec:techniques-one-step-prediction} in which a detailed overview is provided.  

First, we note that our result holds for \emph{any} batch-size $m \in \{1, 2, \ldots, d\}$.  This bridges the gap between deterministic predictions for single sample and constant batch-size methods which rely on controlling stochastic processes~\citep{tan2023online,arous2021online} or equivalent differential equations~\citep{paquette2021sgd} and the extensive batch-size setting in which $m \asymp d$~\citep{gerbelot2022rigorous}.   While these works all study the dynamics of SGD or related first-order methods, it is instructive to make a qualitative comparison even though our result holds for the (significantly more complex) prox-linear iterations.

To illustrate, consider two cases. Recalling that our inverse step-size $\lambda$ may scale as $\lambda \asymp d/m$, consider a batch-size of $m=1$ to obtain
\[
\big|\parcompX_{+} - \parcompdetX_{+} \big| \vee \big|\parcompZ_{+} - \parcompdetZ_{+}\big| \leq C_{2}\max\bigg\{ \frac{ \sqrt{\Err_{\sharp}} + \sigma}{(1 + \sigma) d} \cdot\log^{6}(d) ,\; d^{-30}\bigg\},
\]
and similarly for the perpendicular component $\perpcompX$.  By contrast, if we consider a large batch-size of $m \asymp d$, our bounds read as
\[
\big|\parcompX_{+} - \parcompdetX_{+} \big| \vee \big|\parcompZ_{+} - \parcompdetZ_{+}\big| \leq C_{2}\max\bigg\{ \frac{ \sqrt{\Err_{\sharp}} + \sigma}{1 + \sigma} \cdot \frac{\log^{6}(d)}{\sqrt{d}} ,\; d^{-30}\bigg\}.
\]
That is, for constant batch-size, the fluctuations scale as $\asymp 1/d$, whereas for batch-sizes of commensurate order with the dimension, our fluctuations are larger and scale as $\asymp 1/\sqrt{d}$.  As we will see in the sequel, this distinction in behavior of the fluctuation reflects the relative time-scales on which these bounds are required to hold to ensure convergence.

Second, recall that $\Err_{\sharp} \lesssim \|\bcoefX_{\sharp} \bcoefZ_{\sharp}^{\top} - \bcoefX_{\star}\bcoefZ_{\star}^{\top}\|_{F}^{2}$~\eqref{equivalent-Errt-frobenius-distance} to aid discussion. Theorem~\ref{thm:one-step-prediction} shows that the deviation of the components $\alpha_{+}$ and $\beta_{+}$ around their deterministic predictions scales with both the estimation error $\|\bcoefX_{\sharp} \bcoefZ_{\sharp}^{\top} - \bcoefX_{\star}\bcoefZ_{\star}^{\top}\|_{F}$ as well as the noise level $\sigma$.  Let us now compare this guarantee explicitly with the previous results~\citet[Theorem 1]{chandrasekher2023sharp} and~\citet[Theorem 1]{chandrasekher2022alternating}, setting $\sigma = 0$
and $m \asymp d$ to transparently facilitate the comparison.  The guarantees in~\citet[Theorem 1]{chandrasekher2023sharp} and~\citet[Theorem 1]{chandrasekher2022alternating} scale as $1/d^{1/4}$ and $1/\sqrt{d}$, respectively, so that the fluctuation bounds dominate the error achieved by the algorithm as soon as the error is small enough.  By contrast, 
our guarantee scales as $\sqrt{\Err_{\sharp}/d}$, which is always of lower order as compared to $\sqrt{\Err_{\sharp}}$.  This ensures that as we run the algorithm, the fluctuations of the empirical states around their deterministic counterparts will decrease at the same rate as the estimation error of the method, which in turn facilitates our convergence analysis of the prox-linear update in the low noise regime.  

Finally, we note that the deterministic predictions in Eq.~\eqref{alpha-beta-det-prediction} can be computed efficiently since we only require solutions to two-dimensional fixed point equations and some scalar calculations (see Section~\ref{sec:expilicit-formulas-prediction} for the explicit formulas).  Later, in Section~\ref{sec:experimental-results}, we use these one-step deterministic predictions to obtain a full \emph{trajectory} prediction.

\subsection{Convergence result} \label{sec:sharp-local-convergence-result}
In order to state our convergence guarantee, we define the empirical error as
\begin{align}\label{def:empirical-error-iterations}
	\Err_{t} = (\parcompX_{t} \parcompZ_t -1)^{2} + \perpcompX_{t}^{2} + \perpcompZ_{t}^{2}, \quad \text{ for all } \; 0\leq t\leq T.
\end{align}
We also require an assumption on the initialization. 
\begin{assumption}[Initialization] \label{asm:initialization}
	The initialization $\bcoefX_{0}, \bcoefZ_{0}$ satisfies both
	\begin{align}\label{local-convergence-conditions}
		\Err_{0} \leq K_0 \qquad \text{ and } \qquad 0.5 \leq \|\bcoefX_{0}\|_{2},\|\bcoefZ_{0}\|_{2} \leq 1.5, 
	\end{align}
    for $K_0$ a small enough, universal, positive constant.
\end{assumption}
Equipped with this assumption, we state our main convergence result.
\begin{theorem}\label{thm:local-sharp-convergence-results}
	Suppose Assumptions~\ref{assumption-unit-norm}--\ref{asm:initialization} hold and consider the observation model~\eqref{eq:model} and mini-batched prox-linear updates in~\eqref{eq:closed-form-update} run for $T$ iterations. There exists a tuple of universal, positive constants $(d_0,c_1,c_2,C,C_1,C_2)$ such that the following statement holds with probability at least $1-(T+1)d^{-18}$:  If the inverse step-size $\lambda$, the batch-size $m$, dimension $d$, and number of iterations $T$ satisfy
	\begin{subequations}\label{local-convergence-conditions-all}
		\begin{align}\label{step-batch-conditions}
			\lambda \geq \frac{C(1+\sigma^{2})d}{m}&,   &&C(1+\sigma^{4})\log^{12}(d) \leq m \leq d,\\
			\label{iteration-dimension-condition}
			d \geq d_{0}&, && 0 \leq T \leq \frac{\lambda}{c_{1}} \log \Big( \min \Big(\frac{\lambda m}{\sigma^{2}d} , \frac{\sqrt{m}}{\sigma^{2}\polylog} \Big) \Big),
		\end{align}
	\end{subequations}
	then we have for all $0 \leq t\leq T$,
	\begin{align}\label{ineq-local-sharp-convergence}
		\Big(1-\frac{c_{2}}{\lambda} \Big) \cdot \Err_{t}  +   \frac{C_{2}\sigma^{2}d}{\lambda^{2}m} - \frac{C_{1}\polylog \sigma^{2}}{\lambda \sqrt{m}} \leq \Err_{t+1} \leq \Big(1-\frac{c_{1}}{\lambda} \Big) \cdot \Err_{t} + \frac{C_{1}\sigma^{2}d}{\lambda^{2}m} + \frac{C_{1}\polylog\sigma^{2}}{\lambda\sqrt{m}}.
	\end{align}
\end{theorem}

As before, we have not optimized constant or polylogarithmic factors, which can likely be improved. To aid discussion, recall that inequality~\eqref{equivalent-Errt-frobenius-distance} holds for all iterations, so that the reader should think that $\Err_{t}$ is equivalent (up to a constant) to the error $\|\bcoefX_{t} \bcoefZ_{t}^{\top} - \bcoefX_{\star} \bcoefZ_{\star}^{\top}\|_{F}^{2}$. A few remarks follow.

First, we explain the conditions required by Theorem~\ref{thm:local-sharp-convergence-results}. Note that condition~\eqref{local-convergence-conditions} is equivalent to a local initialization requirement, i.e., the initial estimation error $\Err_{0}$ is smaller than a universal constant and the initial estimators $\bcoefX_{0},\bcoefZ_{0}$ are of constant norm. Such an assumption is typically required for non-convex problems to guarantee linear convergence~\citep[see, e.g., ][as an example]{chi2019nonconvex}. Eq.~\eqref{step-batch-conditions} specifies a condition on the inverse step-size $\lambda$ and batch-size $m$.  Note that while our one-step predictions (Theorem~\ref{thm:one-step-prediction}) hold for any batch-size $1 \leq m \leq d$, our convergence guarantees in Theorem~\ref{thm:one-step-prediction}
require batch-sizes which scale polylogarithmically in $d$, and thus Theorem~\ref{thm:one-step-prediction} does not cover the purely stochastic case $m=1$. Condition~\eqref{iteration-dimension-condition} specifies an upper bound of the number of iterations $T$, which is large enough to guarantee that the iterates converge to a noise floor. 

Second, we simplify the convergence guarantee for the noiseless case. By setting $\sigma = 0$ and $\lambda = Cd/m$, Eq.~\eqref{ineq-local-sharp-convergence} simplifies to 
\begin{align} \label{eq:interpolation-cor}
	\Big( 1 - \frac{c_2m}{Cd} \Big) \cdot \Err_{t} \leq \Err_{t+1} \leq \Big( 1 - \frac{c_1m}{Cd} \Big) \cdot \Err_{t}.
\end{align}
This establishes a linear convergence guarantee and shows that the error decreases at a faster rate as the batch-size $m$ increases. The upper bound in Eq.~\eqref{eq:interpolation-cor} recovers the guarantees of~\cite{asi2019stochastic,asi2020minibatch,chadha2022accelerated} for the interpolation version of our problem, while also providing a matching lower bound on the convergence rate. 

Third, we turn to a general noise level $\sigma>0$.  Theorem~\ref{thm:local-sharp-convergence-results} implies that the prox-linear update adapts to problem difficulty in terms of the rate at which it converges~\citep{chandrasekher2022alternating,agarwal2012fast} and enjoys a linear speed-up in batch-size~\citep{asi2020minibatch,chadha2022accelerated}. In particular, by setting $\lambda = Cd(1+\sigma^{2})/m$ and $(1+\sigma^{4})\log^{12}(d) \ll m \leq d$, inequality~\eqref{ineq-local-sharp-convergence} implies
\begin{align}\label{ineq-local-sharp-convergence-standard-regime}
	\Big( 1-\frac{c_{2}m}{C(1+\sigma^{2})d} \Big) \cdot \Err_{t} + \frac{C_2 \sigma^{2} d}{2\lambda^{2} m} \leq \Err_{t+1} \leq \Big( 1-\frac{c_{1}m}{C(1+\sigma^{2})d} \Big) \cdot \Err_{t} + \frac{2C_1 \sigma^{2} d}{\lambda^{2} m}.
\end{align}
Consequently, starting from a local initialization satisfying inequality~\eqref{local-convergence-conditions} and running the prox-linear update for 
\begin{align}\label{interation-complexity-standard-regime}
	\tau = \Theta\Big( \frac{(1+\sigma^{2})d}{m} \log\Big( \frac{\Err_{0}}{\sigma^{2}}\Big) \Big) \quad \text{iterations}, \quad \text{we obtain} \quad \Err_{\tau} \lesssim \sigma^{2}.
\end{align}
The sandwich relations~\eqref{ineq-local-sharp-convergence-standard-regime} and~\eqref{interation-complexity-standard-regime} establish sharp upper and lower bounds both on the convergence rate and iteration complexity of the prox-linear update.  Note the explicit dependence on the problem parameters, including noise-dependent convergence rates.  We illustrate this difference in convergence behavior in Figure~\ref{fig:standard} where we consider $m = 4,8,16,32$ and $\lambda = (1+\sigma^2)d/m$, noting the monotone relationship (in $m$) in the speed of convergence. 

Fourth, Theorem~\ref{thm:local-sharp-convergence-results} reveals the robustness of choices of $\lambda$~\citep{asi2019importance} and how it affects the convergence behavior. In particular, as long as we choose $\lambda \geq C(1+\sigma^{2})d/m$, inequality~\eqref{ineq-local-sharp-convergence} holds, so that starting by satisfying inequality~\eqref{local-convergence-conditions} and running the prox-linear update for 
\begin{align}\label{interation-complexity-all-lambda-regime}
\hspace{-1cm}
	\tau = \Theta\bigg( \lambda \log \bigg( \Err_{0} \cdot \min \bigg\{ \frac{\lambda m}{\sigma^{2}d} , \frac{\sqrt{m}}{\sigma^{2}\polylog} \bigg\} \bigg) \bigg) \; \text{iterations, we have } \Err_{\tau} \lesssim \frac{\sigma^{2}d}{\lambda m} + \frac{\polylog \sigma^{2}}{\sqrt{m}}.
\end{align}
Inequality~\eqref{interation-complexity-all-lambda-regime} reveals a tension between the noise floor reached and the speed of convergence governed by the inverse step-size $\lambda$.   
In particular, as $\lambda$ increases (i.e., step-size decreases), the iteration complexity increases while the eventual error $\Err_{\tau}$ decreases. We illustrate this difference in convergence behavior in Figure~\ref{fig:lambda} where we consider $\lambda = 1,10,100,200$ and $m=32$, noting the monotone relationship (in $\lambda$) in the speed of convergence and also in the eventual error floor.

Finally, let us concretely situate our contribution in relation to recent works studying aProx methods~\citep[Eq. (4)]{asi2019stochastic}. As previously mentioned, for interpolation problems (corresponding to $\sigma=0$ in our case),~\citet[Proposition 1]{chadha2022accelerated} and~\citet[Theorem 2]{asi2020minibatch} prove that aProx methods enjoy linear convergence and enjoy a linear speedup in the batch-size $m$.   We note that these results hold in a more general setting, whereas ours are specific to the problem at hand.  Specializing to prox-linear methods for rank-one matrix sensing, we complement and improve their guarantees along two axes. First, our result shows there is a linear speed up for all noise levels $\sigma \geq 0$. Second, we prove a method-specific matching lower bound on the iteration complexity, showing that the speed-up by increasing batch-size can be no better than linear. In a complementary line of work,~\citet[Theorem 4.1]{davis2019stochastic} shows that the population loss function value of the average of prox-linear iterates converges to its minimal value at a sublinear rate $O(T^{-1/2})$. Our result improves their guarantee in the sense that the iterates exhibit linear convergence to a noise-dominated neighbor of the ground truth $(\bcoefX_{\star},\bcoefZ_{\star})$, and characterizes the size of this neighborhood.~\citet[Proposition 5]{asi2019stochastic} also obtained a similar linear convergence guarantee for strongly-convex stochastic loss functions while our stochastic loss $L_{m}$ is non-convex.~\cite{davis2023stochastic} proved that aProx methods with geometric step decay enjoy a local linear rate of convergence for non-convex problems that satisfy the sharp growth condition, but our population loss $\widebar{L}$ does not satisfy the sharp growth condition in~\citet[Eq. (2.4)]{davis2023stochastic}.

We provide the proofs of Theorem~\ref{thm:local-sharp-convergence-results} and the consequent inequalities~\eqref{ineq-local-sharp-convergence-standard-regime},~\eqref{interation-complexity-standard-regime} and~\eqref{interation-complexity-all-lambda-regime} in Section~\ref{sec:main-proof-convergence-result}. 
Our proof technique proceeds in two conceptually simple but analytically involved steps. First, we apply the one-step updates from Theorem~\ref{thm:one-step-prediction} to reduce the complexity from studying a high-dimensional, random iteration to studying a four-dimensional, deterministic recursion. Second, we show that convergence properties suggested by the deterministic updates~\eqref{alpha-beta-det-prediction} are not affected by the fluctuations of the empirical states around their predictions when the batch-size $m$ is large enough.

\subsection{Trajectory predictions and application to hyperparameter tuning}\label{sec:experimental-results}
In Theorem~\ref{thm:one-step-prediction}, we derived one-step, deterministic predictions of the update and in turn used these predictions to obtain concrete convergence guarantees in Theorem~\ref{thm:local-sharp-convergence-results}.  In this section, we demonstrate how to obtain a full trajectory prediction and use this to tune the pair of hyperparameters $(m, \lambda)$.  

\subsubsection{Obtaining a trajectory prediction}
Consider an initialization $(\bcoefX_{0}, \bcoefZ_{0}) \in \mathbb{R}^{d} \times \mathbb{R}^{d}$ and use this to construct an initial state $(\parcompX_{0},\perpcompX_{0},\parcompZ_{0},\perpcompZ_{0})$.  Then, given the tuple of parameters $(m,d,\sigma,\lambda)$, we recall the explicit formulae of the prediction functions $\alphaXmap_{m,d,\sigma,\lambda},\alphaZmap_{m,d,\sigma,\lambda},\iotaXmap_{m,d,\sigma,\lambda},\iotaZmap_{m,d,\sigma,\lambda},\etaXmap_{m,d,\sigma,\lambda},\etaZmap_{m,d,\sigma,\lambda}:\mathbb{R}^{4} \rightarrow \mathbb{R}$ in Section~\ref{sec:expilicit-formulas-prediction} and use these to define a deterministic map $\detstatemap_{m,d,\sigma,\lambda}:\mathbb{R}^{4} \rightarrow \mathbb{R}^{4}$, defined as
\begin{align} \label{eq:deterministic-map}
	\detstatemap_{m,d,\sigma,\lambda}\big(\parcompX,\perpcompX,\parcompZ,\perpcompZ\big) &= \big(\parcompdetX,\perpcompdetX,\parcompdetZ,\perpcompdetZ \big).
\end{align}
Above, we have additionally used the quantities 
\begin{align*}
 \parcompdetX = \alphaXmap_{m,d,\sigma,\lambda}\big(\parcompX,\perpcompX,\parcompZ,\perpcompZ\big),\qquad \perpcompdetX &= \Big( \iotaXmap_{m,d,\sigma,\lambda}\big(\parcompX,\perpcompX,\parcompZ,\perpcompZ\big)^{2} + \etaXmap_{m,d,\sigma,\lambda}\big(\parcompX,\perpcompX,\parcompZ,\perpcompZ\big) \Big)^{-1/2}, 
	\\ \parcompdetZ = \alphaZmap_{m,d,\sigma,\lambda}\big(\parcompX,\perpcompX,\parcompZ,\perpcompZ\big),\qquad \perpcompdetZ &= \Big( \iotaZmap_{m,d,\sigma,\lambda}\big(\parcompX,\perpcompX,\parcompZ,\perpcompZ\big)^{2} + \etaZmap_{m,d,\sigma,\lambda}\big(\parcompX,\perpcompX,\parcompZ,\perpcompZ\big) \Big)^{-1/2}.
\end{align*}
Equipped with the map $\mathcal{T}$ in~\eqref{eq:deterministic-map}, we generate a deterministic prediction of the state at iterate $t$ by iterating this map $t$ times.  That is, we generate a deterministic sequence $\big\{\parcompseqX_{t},\perpcompseqX_{t},\parcompseqZ_{t},\perpcompseqZ_{t}\big\}_{t=0}^{T}$ with $\big(\parcompseqX_{0},\perpcompseqX_{0},\parcompseqZ_{0},\perpcompseqZ_{0}\big) = (\parcompX_{0},\perpcompX_{0},\parcompZ_{0},\perpcompZ_{0})$ as
\begin{align} \label{eq:trajectory-prediction}
\big(\parcompseqX_{t},\perpcompseqX_{t},\parcompseqZ_{t},\perpcompseqZ_{t}\big) = \detstatemap_{m,d,\sigma,\lambda}\big(\parcompseqX_{t-1},\perpcompseqX_{t-1},\parcompseqZ_{t-1},\perpcompseqZ_{t-1}\big) = \detstatemap_{m,d,\sigma,\lambda}^{t}\big(\parcompseqX_{0},\perpcompseqX_{0},\parcompseqZ_{0},\perpcompseqZ_{0}\big).
\end{align}
We then define the predicted error sequence $\{\Err_{t}^{\mathsf{seq}}\}_{t = 1}^{T}$ as
\begin{align} \label{eq:error-prediction}
\Err_{t}^{\mathsf{seq}} = \big( \parcompseqX_{t} \parcompseqZ_{t} - 1\big)^{2} + \big(\perpcompseqX_{t} \big)^{2} + \big( \perpcompseqZ_{t} \big)^{2}.
\end{align}
Note that the sequence $\{ \Err_{t}^{\mathsf{seq}} \}$ is the deterministic analog of the empirical error sequence $\{ \Err_{t} \}$ defined in Eq.~\eqref{def:empirical-error-iterations}.  We now use these trajectory predictions to perform hyperparameter tuning. 

\subsubsection{Hyperparameter tuning}
The trajectory prediction in~\eqref{eq:trajectory-prediction} and its associated error~\eqref{eq:error-prediction} require knowledge of the tuple of parameters $(m, d, \sigma, \lambda)$.  Given this tuple of parameters, the key advantage of our trajectory prediction lies in its efficiency: a single trajectory over $T = 1000$ iterations takes less than $0.25$ seconds to generate for all the illustrated values of dimension. Of this quadruple of parameters, the user clearly has exact knowledge of the dimension $d$, and an estimate of the noise standard deviation $\sigma$ may be previously available or estimated via bootstrapping. We would like to use the predictions to tune the batch-size $m$ and inverse step-size $\lambda$ in order to ensure rapid convergence to a reasonable error floor---we now consider two families of experiments that isolate the effects of these parameters. 
\begin{figure}[!hbtp]
	\centering
	\begin{subfigure}[b]{0.48\textwidth}
		\centering
		\includegraphics[scale=0.5]{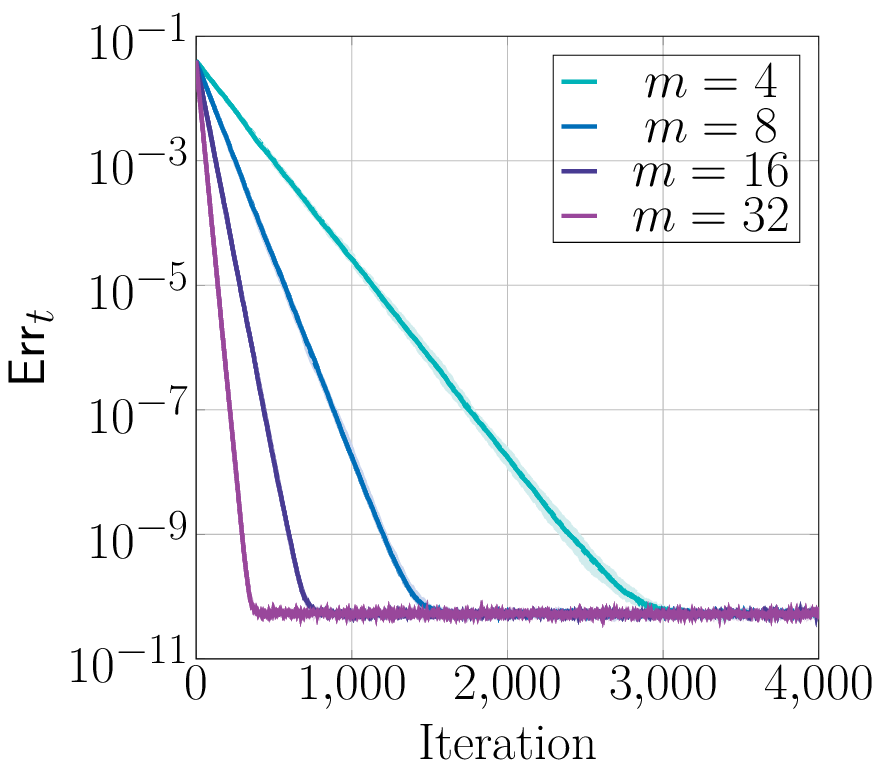}
		\caption{$\sigma = 10^{-5},\; d = 200.$}    
		\label{subfig:standard-a}
	\end{subfigure}
	\hfill
	\begin{subfigure}[b]{0.48\textwidth}  
		\centering 
		\includegraphics[scale=0.5]{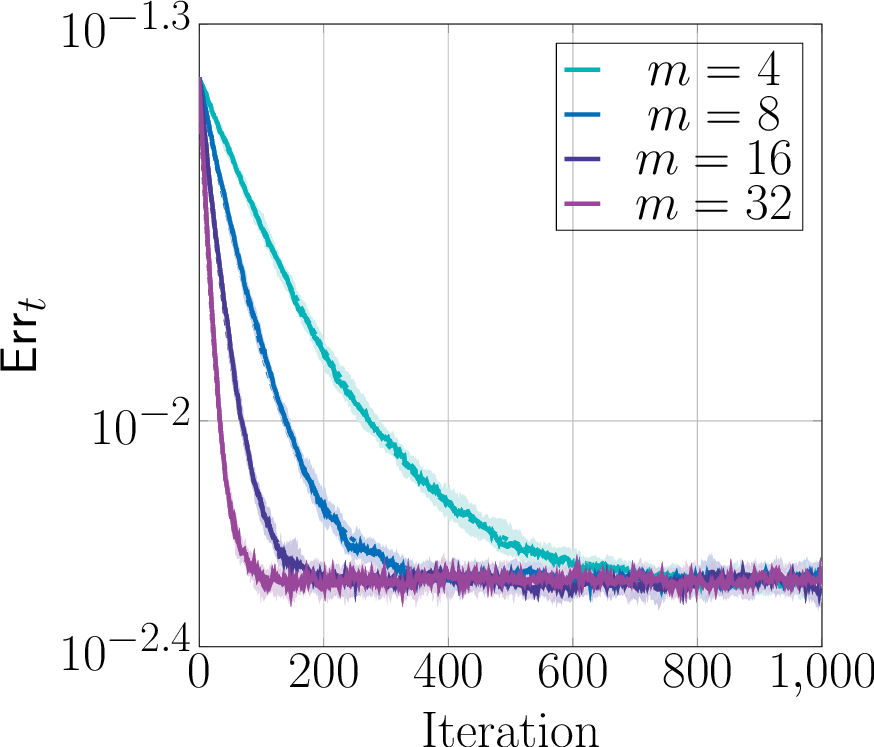}
		\caption{$\sigma = 0.1,\;d=200.$}
		\label{subfig:standard-b}
	\end{subfigure}
	\caption{Low noise (panel (a)) and high noise (panel (b)) behavior of the prox-linear method for batch-sizes $m=4,8,16,32$ and inverse step-size choice $\lambda = (1+\sigma^{2})d/m$. Each experiment starts from an initialization satisfying $\parcompX_{0} = \parcompZ_{0} = 0.99$ and $\|\bcoefX_{0}\|_{2} = \|\bcoefZ_{0}\|_{2} = 1$ and runs to convergence. 
		In panel (a), each experiment consists of $10$ independent trials and the shaded envelopes ($m=4$) denote the range over the $10$ trials. In panel (b), each experiment consists of $30$ independent trials and the shaded envelopes denote the interquartile range over the $30$ trials. Solid lines denote the median of the empirical error $\Err_{t}$~\eqref{def:empirical-error-iterations} over the independent trials and dashed lines (barely visible) denote the predicted error $\Err_{t}^{\mathsf{seq}}$~\eqref{eq:error-prediction}.} 
	\label{fig:standard}
\end{figure}

First, we consider a fixed inverse step-size choice $\lambda(m) = (1 + \sigma^2)d/m$ and study the effect of varying the batch-size.  To this end, recall from Eq.~\eqref{interation-complexity-standard-regime} that this choice of step-size implies an iteration complexity of $\widetilde{\Theta}((1+\sigma^2)d/m)$ to reach a parameter-independent error floor of size $\widetilde{O}(\sigma^2)$.  That is, the iteration complexity monotonically decreases as the batch-size increases and there is no tradeoff with the eventual error floor which is reached.  Equipped with a tight characterization of the per iteration complexity of the method, the user can then compute the iteration complexity for each batch-size $m$ from our trajectory prediction and use the combination of this information to inform their batch-size selection.  We illustrate this in Figure~\ref{fig:standard}, noting that the user may elect to take a batch-size of $m=16$ upon noticing the diminishing gains in iteration complexity purely from observing the deterministic trajectory. As is clear from the figure, the deterministic and empirical trajectories are nearly indistinguishable, so this design choice is also a good one (with high probability) for the random iterates that one would obtain by running the algorithm. 

Next, we turn to fixing the batch-size $m$ and varying the inverse step-size $\lambda$.  Recall from Theorem~\ref{thm:local-sharp-convergence-results} that, unlike the previous case, there exists a tension---induced by the inverse step-size $\lambda$---between the rate of convergence and the size of the noise-dominated neighborhood to which the iterates linearly converge.  In particular, larger $\lambda$ yields slower convergence to a smaller neighborhood, whereas smaller $\lambda$ yields faster convergence albeit to a larger neighborhood.  Faced with this, the user may---without ever running the data---plot several deterministic trajectories, varying $\lambda$ in each, to determine their desired trade-off.  In Figure~\ref{fig:lambda}, we vary $\lambda \in \{1, 10, 100, 200\}$ for a fixed batch-size of $m = 32$ and in dimension $d = 200$.  %In order to illustrate the effectiveness of our predicted trajectory, 
We plot both deterministic trajectories and $30$ independent trials of empirical trajectories, noting that the difference is barely visible.  Zooming into Figure~\ref{fig:lambda}(a), the user may elect to pick $\lambda = 10$ as for larger $\lambda$, the iteration complexity increases by an order of magnitude, but the user may be satisfied with the eventual error on the order of $10^{-10}$.  By contrast, in Figure~\ref{fig:lambda}(b), the user may elect to pick $\lambda = 200$ as the difference in iteration complexity is much smaller, but the difference between $10^{-2}$ and $10^{-3}$ in eventual error may be non-negligible. Both of these design choices can be made purely from observing the deterministic trajectory, but as guaranteed by our theory, they will be suitable for the random iterations with high probability.

\begin{figure}[!hbtp]
	\centering
	\begin{subfigure}[b]{0.48\textwidth}
		\centering
		\includegraphics[scale=0.48]{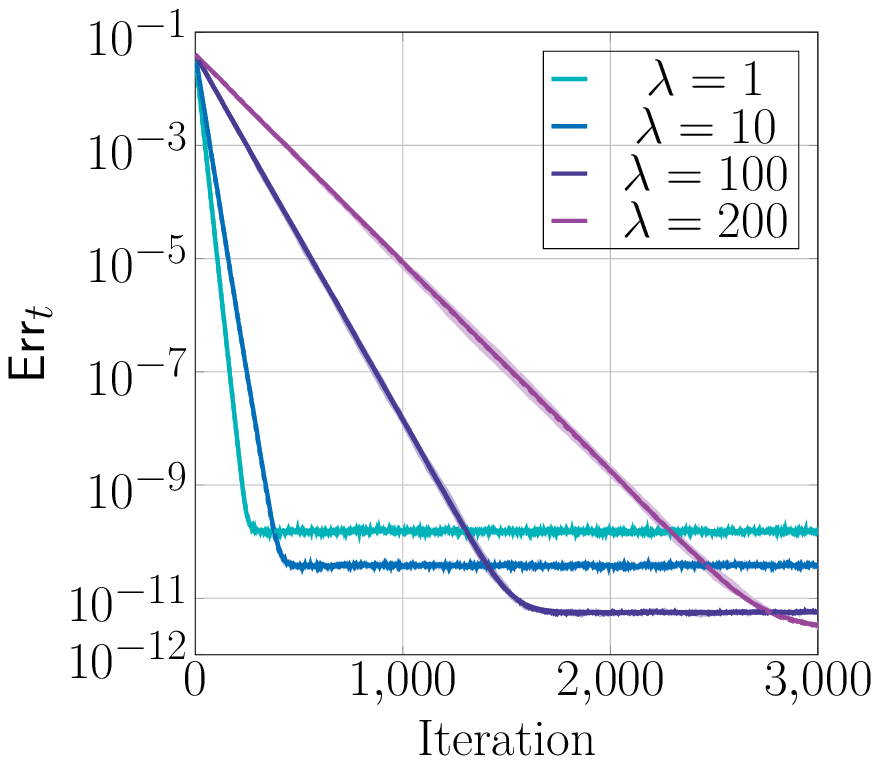}
		\caption{$\sigma = 10^{-5},\; d = 200,\; m = 32$.}    
		\label{subfig:lambda-a}
	\end{subfigure}
	\hfill
	\begin{subfigure}[b]{0.48\textwidth}  
		\centering 
		\includegraphics[scale=0.48]{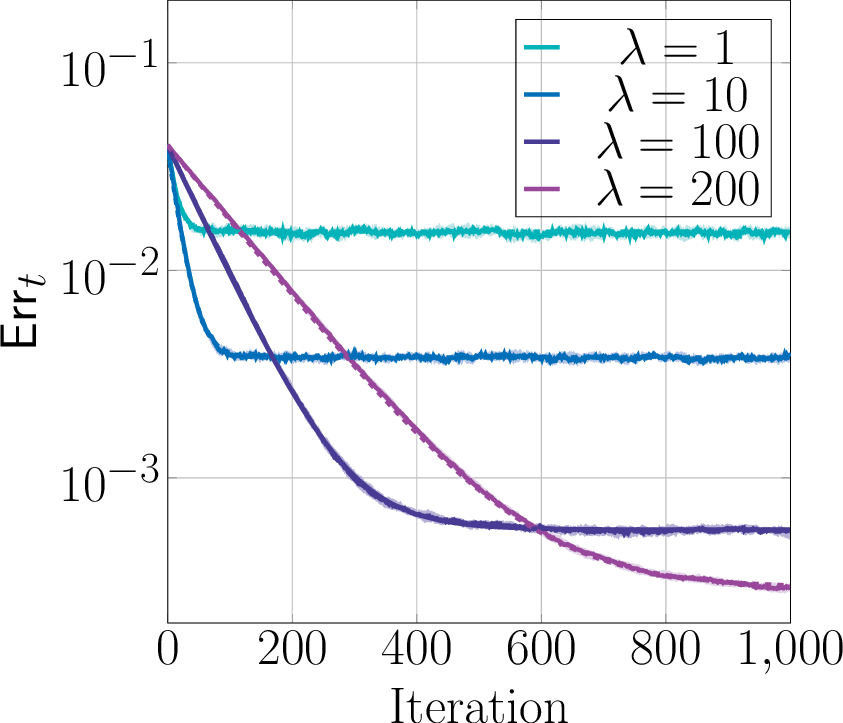}
		\caption{$\sigma = 0.1,\; d=200,\; m = 32$.}
		\label{subfig:lambda-b}
	\end{subfigure}
	\caption{Low noise (panel (a)) and high noise (panel (b)) behavior of the prox-linear method for batch-size $m=32$ and inverse step-size selections $\lambda = 1, 10, 100, 200$. Each experiment starts from an initialization satisfying $\parcompX_{0} = \parcompZ_{0} = 0.99$ and $\|\bcoefX_{0}\|_{2} = \|\bcoefZ_{0}\|_{2} = 1$ and runs to convergence. In panel (a), each experiment consists of $10$ independent trials and shaded envelopes ($\lambda = 200$) denote the range over the $10$ trials. In panel (b), each experiment consists of $30$ independent trials and shaded envelopes denote the interquartile range over the $30$ trials. Solid lines denote the median of $\Err_{t}$~\eqref{def:empirical-error-iterations} over the independent trials and dashed lines (barely visible) denote the predicted error $\Err_{t}^{\mathsf{seq}}$~\eqref{eq:error-prediction}.} 
	\label{fig:lambda}
\end{figure}
\section{A heuristic derivation and explicit formulas}\label{sec:overview-techniques}
In this section, we first present a heuristic derivation of the deterministic prediction of the parallel component. Then we give the deferred explicit formulas of the prediction functions. 

\subsection{A heuristic derivation of the prediction for the parallel component} \label{sec:techniques-one-step-prediction}

In this section, we give a high-level overview of the techniques that we use to derive the deterministic predictions, focusing on the parallel components $\parcompX_{+}$ and $\parcompZ_{+}$. Throughout, we use the shorthand
\begin{align}\label{eq:shorthand-length-L-Ltil}
L_\sharp^2 = \parcompX_\sharp^2+\perpcompX_\sharp^2 \quad \text{and} \quad \LZ_{\sharp}^{2} = \parcompZ_\sharp^2+\perpcompZ_\sharp^2.
\end{align}
\paragraph{Step 1: Decomposition.}
Since our updates are online and the data is Gaussian, the current iterate $\bcoefX_{+}$ is correlated with the past iterates only through the subspace $\mathsf{span}(\bcoefX_{\sharp}, \bcoefX_{\star})$ (and similarly for $\bcoefZ_{+}$).  We thus decompose the pair of parallel components $(\parcompX_{+},\parcompZ_{+})$ on this subspace using Gram--Schmidt orthonormalization.  In particular, consider the four unit vectors
\begin{align}\label{definition-u1-u2-v1-v2}
	\bu_{1} = \frac{\bcoefX_{\sharp}}{\|\bcoefX_{\sharp}\|_{2}},\quad \bu_{2} = \frac{\bP_{\bcoefX_{\sharp}}^{\perp} \bcoefX_{\star}}{\|\bP_{\bcoefX_{\sharp}}^{\perp} \bcoefX_{\star}\|_{2}}, \quad 
	\bv_{1} = \frac{\bcoefZ_{\sharp}}{\|\bcoefZ_{\sharp}\|_{2}} \quad \text{and} \quad \bv_{2} = \frac{\bP_{\bcoefZ_{\sharp}}^{\perp} \bcoefZ_{\star}}{\|\bP_{\bcoefZ_{\sharp}}^{\perp} \bcoefZ_{\star}\|_{2}},
\end{align}
where $\bP_{\bcoefX_{\sharp}}^{\perp}$ denotes the projection matrix onto the orthogonal complement of one-dimensional subspace spanned by $\bcoefX_{\sharp}$ (and similarly for $\bP_{\bcoefZ_{\sharp}}^{\perp}$). We then decompose $\parcompX_{+}$ and $\parcompZ_{+}$ on the two-dimensional subspace spanned by $\bu_1$ and $\bu_2$ as
\begin{align}\label{eq:alpha-component-decomposition-intro}
	\parcompX_{+} = \frac{\parcompX_{\sharp}}{L_{\sharp}} \cdot \langle \bcoefX_{+}, \bu_1 \rangle + \frac{\perpcompX_\sharp}{L_\sharp} \cdot \langle \bcoefX_{+}, \bu_2 \rangle \quad \text{and} \quad \parcompZ_{+} = \frac{\parcompZ_{\sharp}}{\LZ_{\sharp}} \cdot \langle \bcoefZ_{+}, \bv_1 \rangle + \frac{\perpcompZ_\sharp}{\LZ_\sharp} \cdot \langle \bcoefZ_{+}, \bv_2 \rangle.
\end{align}
Then, with $\theta(\bu_i) := \langle \bu_i, \bcoefX_{+} \rangle$ (and similarly for $\thetatil(\bv_i)$), we define deterministic counterparts $\thetaXdet_1$, $\thetaXdet_2$ so that $\parcompdetX_{+} =  (\parcompX_{\sharp}/L_{\sharp}) \cdot \thetaXdet_1 + (\perpcompX_{\sharp}/L_{\sharp}) \cdot \thetaXdet_2$ (see Section~\ref{sec:concentration-proof} for explicit expressions).  Thus, in order to bound the error $\lvert \parcompX_{+} - \parcompdetX_{+} \lvert$, it suffices to bound the decomposition
\[
\bigl \lvert \theta(\bu_1) - \thetaXdet_1 \bigr \rvert \leq \underbrace{\bigl \lvert \theta(\bu_1) - \EE[\theta(\bu_1)] \bigr \rvert}_{\text{Stochastic error}} + \underbrace{\bigl \lvert \EE[\theta(\bu_1)] - \thetaXdet_1 \bigr \rvert}_{\text{Bias}},
\]
and similarly for $\theta(\bu_2), \thetatil(\bv_1), \thetatil(\bv_2)$.  

\paragraph{Step 2a: Controlling the stochastic error via a leave-one-\emph{sample} out argument}
We view the random variable $\theta(\bu_1)$ as a function of the data $\{\bx_k,\bz_k,\epsilon_k\}_{k=1}^{m}$. Indeed, define the function $f_{\bu_1}: \mathbb{R}^{(2d+1) \times m} \rightarrow \mathbb{R}$ as $f_{\bu_1} = \langle \bcoefX_{+},\bu_1\rangle$. It suffices to show that $f_{\bu_1}$ concentrates around its expectation. We then study the bounded difference property of $f_{\bu_1}$ using the leave-one-sample out technique and then employ Warnke's typical bounded difference inequality~\citep[Theorem 2]{warnke2016method} to facilitate our analysis. See Section~\ref{sec:stochastic-error-parallel-component} for details.

\paragraph{Step 2b: Controlling the bias via a leave-one-\emph{direction} out argument}
In order to control the bias, we employ a leave-one-out technique that is similar in spirit to that of~\citet{el2013robust}. 

Let $\bu_1$ and $\bu_2$ be defined as in Step 1 and let $\{\bu_1, \bu_2, \ldots, \bu_d\} \in \mathbb{R}^d \times \mathbb{R}^d \times \ldots \times \mathbb{R}^d$ denote an orthonormal completion of a basis of $\mathbb{R}^d$.  Define $\{\bv_1, \bv_2, \ldots, \bv_d\} \in \mathbb{R}^d \times \mathbb{R}^d \times \ldots \times \mathbb{R}^d$ similarly.  We will analyze the effect of computing the iterates in~\eqref{eq:closed-form-update} upon leaving out one of $\bu_1, \bu_2$ (as well as $\bv_1, \bv_2$).  To this end, we define orthogonal leave-one-out matrices $\bO_{\bu_i},\bO_{\bv_i} \in \mathbb{R}^{d \times (d-1)}$ whose columns consist of $\{\bu_{j}\}_{j\neq i}$ and $\{\bv_{j}\}_{j\neq i}$ respectively.  Let $\bW = \diag(\bX\bcoefX_{\sharp})$ and $\bWtil =  \diag(\bZ\bcoefZ_{\sharp})$ and define the leave one out data $\bA_{\bu_i,\bv_i} = \left[ \bWtil \bX\bO_{\bu_i} \; \bW \bZ \bO_{\bv_i} \right]$, noting that the Gaussianity of $\bX$ implies that $\bX \bO_{\bu_i}$ is independent of $\bX\bu_i$ (and similarly for $\bZ\bO_{\bv_i}, \bZ \bv_i$).  The update in~\eqref{eq:closed-form-update} can then be written in terms of the ``projection" matrix\footnote{Let us emphasize that this is only a projection matrix when $\lambda = 0$, otherwise it can be understood as the corresponding quantity in ridge-regularized least squares regression.} 
\begin{align}\label{definition-of-Puv-Auv}
	\bP_{\bu_i,\bv_i} := \bI - \bA_{\bu_i,\bv_i} ( \bA_{\bu_i,\bv_i}^{\top} \bA_{\bu_i,\bv_i} + \lambda m \bI)^{-1} \bA_{\bu_i,\bv_i}^{\top}.
\end{align}
Equipped with these preliminary notions, we may utilize the KKT conditions which specify the updates in~\eqref{eq:closed-form-update} (see Section~\ref{sec:proof-bias-term-parallel-component} for details) to obtain the following approximations
\begin{align}\label{eq:approx-components-introduction}
		\hspace{-0.1cm}\EE[\theta(\bu_1)] \approx L_\sharp + \EE\Biggl\{\frac{L_\sharp\, \Bigl(\frac{\parcompX_\sharp\parcompZ_\sharp}{L_\sharp^2} - \widetilde{L}_{\sharp}^2\Bigr) M_1}{\lambda L_\sharp^2\widetilde{L}_{\sharp}^2 + M_1(L_\sharp^2 + \widetilde{L}_{\sharp}^2)}\Biggr\}\quad  \text{ and } \quad
	\EE[\theta(\bu_2)] \approx \frac{\parcompZ_{\sharp}}{\widetilde{L}_{\sharp}^{2}}  \frac{\perpcompX_{\sharp}}{L_{\sharp}} \cdot \EE\Biggl\{\frac{ M_{11} }{ M_{11}+\lambda}\Biggr\},  
\end{align}
where the random variables $M_{11}$ and $M_1$ are defined as 
\begin{align*}
	&M_{11} = \frac{\langle \bX\bu_2, \bWtil \bP_{\bu_2,\bv_2} \bWtil \bX\bu_2\rangle}{m} \qquad  \text{ and } \qquad  M_{1} = \frac{\langle \diag(\bW\bWtil), \bP_{\bu_1,\bv_1} \diag(\bW\bWtil)\rangle}{m} .
\end{align*}
With the closed-form expression in Eq.~\eqref{eq:approx-components-introduction}, it suffices to study the concentration of the random variables $M_{11}$ and $M_{1}$.

\paragraph{Step 3: Concentration of the constituent components $M_{11}$ and $M_1$.}  We now derive the typical values around which $M_{11}$ and $M_1$ concentrate. Substituting $M_{11}$ and $M_1$ in Eq.~\eqref{eq:approx-components-introduction} by their typical values yields the formulas of $\thetaXdet_1$ and $\thetaXdet_{2}$ and subsequently yields the formula of $\parcompdetX_{+}$ as $\parcompdetX_{+} =  (\parcompX_{\sharp}/L_{\sharp}) \cdot \thetaXdet_1 + (\perpcompX_{\sharp}/L_{\sharp}) \cdot \thetaXdet_2$.

Recall that by our construction of $\bu_2$ and $\bv_2$, it holds $\langle\bu_2,\bcoefX_{\sharp}\rangle = \langle\bv_2,\bcoefZ_{\sharp}\rangle = 0$, whence since $\bX$ and $\bZ$ are Gaussian, the tuples of random variables  $\{\bX\bu_2$, $\bZ \bv_2\}$ and $\{\bP_{\bu_{2},\bv_{2}}, \bW, \bWtil\}$ are independent of each other.  We thus deduce the following sequence of approximations (see Section~\ref{subsec:proof-concentration-M11-M22-M12-M1-M2} for details)
\begin{align}\label{eq:M11-approximation}
M_{11} 	\overset{\1}{\approx} \frac{1}{m}\trace \big(\bWtil \bP_{\bu_2,\bv_2} \bWtil\big) = \frac{1}{m} \sum_{i=1}^{m} \bWtil(i,i)^{2}\bP_{\bu_2,\bv_2}(i,i) 
\overset{\2}{=} \frac{1}{m} \sum_{i=1}^{m} \frac{(\bz_{i}^{\top} \bcoefZ_{\sharp})^{2} }{1+\ba_i^{\top} (\sum_{j\neq i} \ba_{j} \ba_{j}^{\top} + \lambda m \bI)^{-1} \ba_{i} },
\end{align}
where the approximation $\1$ follows from the Hanson--Wright inequality~\citep[Theorem 6.2.1]{vershynin2018high} and the equivalence $\2$ follows upon applying the Sherman--Morrison rank one update formula to compute $\bP_{\bu_2,\bv_2}(i,i)$, additionally using the notation $\bWtil(i,i) = \bz_i^{\top} \bcoefZ_{\sharp}$ and $\ba_j^{\top} = \begin{bmatrix} \bz_j^{\top}\bcoefZ_\sharp \cdot \bx_{j}^{\top}\bO_{\bu_2} & \bx_j^{\top}\bcoefX_\sharp \cdot \bz_j^{\top} \bO_{\bu_2}  \end{bmatrix} \in \mathbb{R}^{2(d-1)}$ is the $j$th row of $ \bA_{\bu_2,\bv_2}$. Continuing, we express the quantity $\sum_{j\neq i} \ba_{j} \ba_{j}^{\top} + \lambda m \bI$ as a block matrix, writing
\begin{subequations}\label{eq:block-matrix-formula}
\begin{align}
&\sum_{j\neq i} \ba_{j} \ba_{j}^{\top} + \lambda m \bI = \begin{bmatrix} \bB & \bC\\ \bC^{\top} & \bD \end{bmatrix}, \quad \text{where} \\
\bB = \sum_{j\neq i}^{m} (\bz_{j}^{\top} \bcoefZ_{\sharp})^{2} \bO_{\bu_2}^{\top} \bx_{j} (\bO_{\bu_2}^{\top} &\bx_{j})^{\top} + \lambda m \bI, \qquad \bC = \sum_{j\neq i} (\bx_{j}^{\top} \bcoefX_{\sharp} )(\bz_{j}^{\top} \bcoefZ_{\sharp}) \bO_{\bu_2}^{\top} \bx_{j} (\bO_{\bv_2}^{\top}\bz_{j})^{\top},\\
& \bD =  \sum_{j\neq i} (\bx_{j}^{\top} \bcoefX_{\sharp})^{2} \bO_{\bv_2}^{\top}\bz_{j} (\bO_{\bv_2}^{\top}\bz_{j})^{\top} + \lambda m \bI.
\end{align}
Thus, by block matrix inversion,
\begin{align} \label{eq:block-matrix-proof-sketch}
\hspace{-0.5cm}
	\bigg(\sum_{j\neq i} \ba_{j} \ba_{j}^{\top} + \lambda m \bI \bigg)^{-1} = \left[ \begin{array}{cc} (\bB - \bC \bD^{-1} \bC^{\top})^{-1} &  -(\bB - \bC \bD^{-1} \bC^{\top})^{-1}\bC \bD^{-1} \\ -(\bD - \bC^{\top} \bB^{-1} \bC)^{-1}\bC^{\top} \bB^{-1} & (\bD - \bC^{\top} \bB^{-1} \bC)^{-1}\end{array} \right].
\end{align}
\end{subequations}
Since, by construction, $\ba_i$ is independent of each of the matrices $\bB, \bC$, and $\bD$, we deduce the following sequence of approximations (see Section~\ref{subsec:proof-concentration-M11-M22-M12-M1-M2} for details)
\begin{align*}%\label{eq:reduce-to-trace-inverse}
	\ba_i^{\top} \biggl(\sum_{j\neq i} \ba_{j} \ba_{j}^{\top} + \lambda m \bI \biggr)^{-1} \ba_{i} &\overset{\1}{\approx} (\bz_{i}^{\top} \bcoefZ_{\sharp})^{2} \cdot \bx_{i}^{\top} (\bB - \bC \bD^{-1} \bC^{\top})^{-1} \bx_{i} + (\bx_{i}^{\top} \bcoefX_{\sharp})^{2} \cdot \bz_{i}^{\top}  (\bD - \bC^{\top} \bB^{-1} \bC)^{-1} \bz_{i} %\nonumber
	\\
	&\overset{\2}{\approx} (\bz_{i}^{\top} \bcoefZ_{\sharp})^{2}  \cdot \trace\big( (\bB - \bC \bD^{-1} \bC^{\top})^{-1} \big) +  (\bx_{i}^{\top} \bcoefX_{\sharp})^{2} \cdot  \trace\big( (\bD - \bC^{\top} \bB^{-1} \bC)^{-1}\big),%\nonumber,
\end{align*}
where step $\1$ follows as the quadratic forms involving each of the off-diagonal blocks is zero-mean and step $\2$ follows upon applying the Hanson--Wright inequality.  The most technical step, then, is to approximately compute the inverse trace quantities in the preceding display.  Indeed, in Section~\ref{concentration-trace-inverse}, we apply the method of typical bounded differences~\citep{warnke2016method} to deduce that, with probability at least $1 - Ce^{-cdt^2}$, 
\begin{align}\label{introduction-concentartion-tarce-inverse}
	\big| \trace\big( ( \bB - \bC \bD^{-1} \bC^{\top})^{-1} \big) - r_{1}^{-1} \big|\leq t \quad \text{and} \quad \big| \trace\big( ( \bD - \bC^{\top} \bB^{-1} \bC )^{-1} \big) - r_{2}^{-1} \big| \leq t,
\end{align}
where $r_1, r_2$ are scalars which satisfy a system of two equations, stated explicitly in Section~\ref{sec:expilicit-formulas-prediction} to follow.  We show in Lemma~\ref{fixed-point-equations-unique-solution} that as long as $\| \bcoefX_{\sharp} \|_2, \| \bcoefZ_{\sharp} \|_2 \leq \lambda$, we have the sandwich relation $\frac{\lambda m}{d} \leq r_1, r_2 \leq 2\frac{\lambda m}{d}$,
so that the quantities $r_1$ and $r_2$ scale linearly in $\frac{\lambda m}{d}$.

Now, combining the approximations~\eqref{eq:M11-approximation} and~\eqref{introduction-concentartion-tarce-inverse}, and applying Bernstein's inequality~\citep[Theorem 2.8.1]{vershynin2018high} we obtain
\[
	M_{11} \approx \frac{1}{m} \sum_{i=1}^{m} \frac{(\bz_{i}^{\top} \bcoefZ_{\sharp})^{2} }{1 + (\bz_{i}^{\top} \bcoefZ_{\sharp})^{2} \cdot r_{1}^{-1} + (\bx_{i}^{\top} \bcoefX_{\sharp})^{2} \cdot r_{2}^{-1} } \approx \EE\bigg\{ \frac{r_1r_2 G_2^2}{r_1r_2 + r_2 G_2^2 + r_1 G_1^2 } \bigg\},
\]
where in the last step we have let $G_1 \sim \mathsf{N}(0, L_\sharp^2)$ and $G_2 \sim \mathsf{N}(0, \LZ_\sharp^2)$ be two independent random variables.  Proceeding similarly yields the approximation (see Section~\ref{subsec:proof-concentration-M11-M22-M12-M1-M2} for details)
\[
 	M_1 \approx \EE\bigg\{ \frac{r_1r_2 G_1^2 G_2^2}{r_1r_2 + r_2 G_2^2 + r_1 G_1^2 } \bigg\}.
\]
Now, substituting the approximations of $M_{11}$ and $M_1$ into Eq.~\eqref{eq:approx-components-introduction} yields the formulas of $\thetaXdet_1$ and $\thetaXdet_2$.

\subsection{Explicit formulas of the deterministic predictions}\label{sec:expilicit-formulas-prediction}
Armed with the intuition of the previous section, we now provide explicit expressions for the functions $\alphaXmap_{m,d,\sigma,\lambda},\;\alphaZmap_{m,d,\sigma,\lambda},\;\iotaXmap_{m,d,\sigma,\lambda},\;\iotaZmap_{m,d,\sigma,\lambda},\;\etaXmap_{m,d,\sigma,\lambda},\;\etaZmap_{m,d,\sigma,\lambda}:\mathbb{R}^{4} \rightarrow \mathbb{R}$ used to define our deterministic predictions.  

For inputs $(\parcompX,\perpcompX,\parcompZ,\perpcompZ)$, we use the shorthand $L = \sqrt{\parcompX^{2}+\perpcompX^{2}}$ and $\LZ = \sqrt{\parcompZ^{2} + \perpcompZ^{2}}$ and let $G_1,G_2$ be two independent random variables such that $G_1 \sim \mathsf{N}(0,L^2)$ and $G_2 \sim \mathsf{N}(0,\LZ^2)$. Recall that $\oversamp = m/d$.  The scalars $r_1, r_2 \geq 0$ considered in the previous section are defined to be the unique\footnote{See Lemma~\ref{fixed-point-equations-unique-solution} for a proof of uniqueness} solution to the following fixed point equations
\begin{align}\label{eq:fixed-point}
\lambda + \EE\bigg\{ \frac{r_1r_2 G_2^2}{r_1r_2 + r_1G_1^2 + r_2G_2^2}\bigg\} = \frac{r_1}{\oversamp} \qquad \text{ and } \qquad 
\lambda + \EE\bigg\{ \frac{r_1r_2 G_1^2}{r_1r_2 + r_1G_1^2 + r_2G_2^2}\bigg\} = \frac{r_2}{\oversamp}.
\end{align}
We further define the deterministic quantities $(V,V_1,V_2)$ as following
\begin{align}\label{eq:V-V1-V2}
\hspace{-1cm}
V = \EE\bigg\{ \frac{r_1r_2G_1^2G_2^2}{r_1r_2 + r_1G_1^2 + r_2G_2^2} \bigg\}, \;V_1 = \EE\bigg\{ \frac{r_1r_2G_2^2}{r_1r_2 + r_1G_1^2 + r_2G_2^2} \bigg\},\;V_2 = \EE\bigg\{ \frac{r_1r_2G_1^2}{r_1r_2 + r_1G_1^2 + r_2G_2^2} \bigg\}.
\end{align}
With these definitions in hand, we define
\begin{subequations}\label{et_updates_eq}
\begin{align}
	\alphaXmap_{m,d,\sigma,\lambda}(\parcompX,\perpcompX,\parcompZ,\perpcompZ) &= \frac{V\big(\frac{\parcompX \parcompZ}{L^2} + L^2\big) + \lambda L^2\LZ^2}{V(L^2+\LZ^2) + \lambda L^2\LZ^2} \cdot \parcompX + \frac{V_1\perpcompX^2}{L^2\LZ^2(V_1 + \lambda)} \cdot \parcompZ \\
	\alphaZmap_{m,d,\sigma,\lambda}(\parcompX,\perpcompX,\parcompZ,\perpcompZ) &= \frac{V\big(\frac{\parcompX \parcompZ}{\LZ^2} + \LZ^2\big) + \lambda L^2\LZ^2}{V(L^2+\LZ^2) + \lambda L^2\LZ^2} \cdot \parcompZ + \frac{V_2\perpcompZ^2}{L^2\LZ^2(V_2 + \lambda)} \cdot \parcompX \\
	\iotaXmap_{m,d,\sigma,\lambda}(\parcompX,\perpcompX,\parcompZ,\perpcompZ) & = \frac{V\big( \frac{\parcompX \parcompZ}{L^2} + L^2\big) + \lambda L^2\LZ^2}{V(L^2+\LZ^2) + \lambda L^2\LZ^2} \cdot \perpcompX - \frac{\parcompX \parcompZ}{L^2\LZ^2} \cdot \frac{V_1}{V_1 + \lambda} \cdot \perpcompX \\
	\iotaZmap_{m,d,\sigma,\lambda}(\parcompX,\perpcompX,\parcompZ,\perpcompZ) & = \frac{V\big(\frac{\parcompX \parcompZ}{\LZ^2} + \LZ^2 \big) + \lambda L^2\LZ^2}{V(L^2+\LZ^2) + \lambda L^2\LZ^2} \cdot \perpcompZ - \frac{\parcompX \parcompZ}{L^2\LZ^2} \cdot \frac{V_2}{V_2 + \lambda} \cdot \perpcompZ.
\end{align}
\end{subequations}
Continuing, we define two deterministic quantities $V_3$ and $V_4$ as 
\begin{align}
&V_{3} = \Big( \sigma^2 + \frac{ \perpcompX^{2}\perpcompZ^{2}}{L^{2}\LZ^{2}} \Big) \EE\bigg\{\frac{r_2^2G_2^2}{(r_1r_2 + r_1G_1^2 + r_2G_2^2)^2}\bigg\} + 
\frac{\lambda^{2} \Big( \frac{\parcompX\parcompZ}{L^{2}\LZ^{2}} - 1 \Big)^{2}}{\big(\lambda + V(L^{-2}+\LZ^{-2}) \big)^{2}}  \EE\bigg\{\frac{r_2^2G_1^2G_2^4}{(r_1r_2 + r_1G_1^2 + r_2G_2^2)^2}\bigg\} \nonumber 
\\ \label{definition-V3}
& + \frac{(\lambda\parcompZ\perpcompX)^{2}}{(\lambda+V_1)^{2}\LZ^4L^2} \EE\bigg\{\frac{r_2^2G_2^4}{(r_1r_2 + r_1G_1^2 + r_2G_2^2)^2}\bigg\} + \frac{(\lambda\parcompX \perpcompZ)^{2}}{(\lambda+V_2)^{2}L^4\LZ^2} \EE\bigg\{\frac{r_2^2G_{1}^{2}G_2^2}{(r_1r_2 + r_1G_1^2 + r_2G_2^2)^2}\bigg\}.\\
% \hspace{-1cm}
&V_{4} = \Big( \sigma^2 + \frac{\perpcompX^{2}\perpcompZ^{2}}{L^{2}\LZ^{2} }\Big) \EE\bigg\{\frac{r_1^2G_1^2}{(r_1r_2 + r_1G_1^2 + r_2G_2^2)^2}\bigg\} + 
\frac{\lambda^{2} \cdot \Big( \frac{ \parcompX\parcompZ}{L^{2}\LZ^{2}} - 1 \Big)^{2}}{\big(\lambda + V(L^{-2}+\LZ^{-2}) \big)^{2}} \EE\bigg\{\frac{r_1^2G_1^4G_2^2}{(r_1r_2 + r_1G_1^2 + r_2G_2^2)^2}\bigg\} \nonumber\\ 
\label{definition-V4}
& + \frac{(\lambda\parcompX\perpcompZ)^{2}}{(\lambda+V_2)^{2}\LZ^2L^3} \EE\bigg\{\frac{r_1^2G_1^4}{(r_1r_2 + r_1G_1^2 + r_2G_2^2)^2}\bigg\} + 
\frac{(\lambda\parcompZ\perpcompX)^{2}}{(\lambda+V_1)^{2}L^2\LZ^4} \EE\bigg\{\frac{r_1^2G_{1}^{2}G_2^2}{(r_1r_2 + r_1G_1^2 + r_2G_2^2)^2}\bigg\}.
\end{align}
Now, let $\etaX,\etaZ \geq 0$ be the solution of the following fixed point equations
\begin{subequations}\label{fixed-point-eq-eta-updates}
\begin{align}
\label{det_updates_eta1}
\hspace{-0.5cm}
\etaX^2 = \frac{(d-2)m}{d^{2}} \bigg(\etaX^{2} \cdot \EE\bigg\{\frac{  r_2^2G_2^4}{(r_1r_2 + r_1G_1^2 + r_2G_2^2)^2}\bigg\} + 
\etaZ^{2} \cdot \EE\bigg\{ \frac{ r_2^2G_1^2G_2^2}{(r_1r_2 + r_1G_1^2 + r_2G_2^2)^2}\bigg\} + V_{3} \bigg),\\
\label{det_updates_eta2}
\hspace{-0.5cm}
\etaZ^2 = \frac{(d-2)m}{d^{2}} \bigg(\etaZ^{2} \cdot \EE\bigg\{\frac{ r_1^2G_1^4}{(r_1r_2 + r_1G_1^2 + r_2G_2^2)^2}\bigg\} + \etaX^{2} \cdot \EE\bigg\{\frac{ r_1^2G_1^2G_2^2}{(r_1r_2 + r_1G_1^2 + r_2G_2^2)^2}\bigg\} + V_{4} \bigg).
\end{align}
\end{subequations}
Now, we define 
\begin{align}\label{eta_updates_eq}
	\etaXmap_{m,d,\sigma,\lambda}(\parcompX,\perpcompX,\parcompZ,\perpcompZ) = \etaX^2 \quad \text{and} \quad 
	\etaZmap_{m,d,\sigma,\lambda}(\parcompX,\perpcompX,\parcompZ,\perpcompZ) = \etaZ^2.
\end{align}
We show (see Lemma~\ref{fixed-point-equations-unique-solution}) that Eq.~\eqref{det_updates_eta1} and Eq.~\eqref{det_updates_eta2} have a unique non-negative solution, so that the functions $\etaXmap_{m,d,\sigma,\lambda}$ and $\etaZmap_{m,d,\sigma,\lambda}$ are well defined.

\section{Proof of Theorem~\ref{thm:one-step-prediction}: One-step predictions}\label{thm:one-step-prediction-main-proof}
We prove the parallel component and the perpendicular component separately.
\subsection{Parallel component: Proof of Theorem~\ref{thm:one-step-prediction}(a)}\label{sec:concentration-proof}
We proceed by executing the steps outlined in Section~\ref{sec:techniques-one-step-prediction}.  In particular, recall the tuple of unit vectors $(\bu_1, \bu_2, \bv_1, \bv_2)$ from Eq.~\eqref{definition-u1-u2-v1-v2} and decompose $\parcompX_+$ and $\parcompZ_+$ as in Eq.~\eqref{eq:alpha-component-decomposition-intro} and recall that 
\[
\parcompX_{+} = \frac{\parcompX_{\sharp}}{L_{\sharp}} \cdot \theta(\bu_1) +  \frac{\perpcompX_\sharp}{L_\sharp} \cdot \theta(\bu_2) \qquad \text{ and } \qquad \parcompZ_{+} = \frac{\parcompZ_{\sharp}}{\LZ_{\sharp}} \cdot \thetatil(\bv_1) + \frac{\perpcompZ_\sharp}{\LZ_\sharp} \cdot \thetaZ(\bv_2),
\]
where we defined $\theta(\bu_i) = \langle \bu_i, \bcoefX_{+} \rangle$ and $\theta(\bv_i) = \langle \bv_i, \bcoefZ_{+} \rangle$. 
Consequently, it suffices to understand the concentration properties of the random variables $\theta(\bu_1), \theta(\bu_2), \thetaZ(\bv_1)$, and $\thetaZ(\bv_2)$ around their typical values
\begin{subequations}
	\begin{align}\label{definition-theta12-det}
		&\thetaXdet_{1} = L_{\sharp} + L_{\sharp} \cdot \frac{\big(\frac{\parcompX_{\sharp}\parcompZ_{\sharp}}{L_{\sharp}^2} - \widetilde{L}_{\sharp}^2\big) \cdot V  }{\lambda L_\sharp^2\widetilde{L}_{\sharp}^2 + V(L_\sharp^2 + \widetilde{L}_{\sharp}^2)}, &&\thetaXdet_{2} = \frac{\parcompZ_{\sharp}\perpcompX_{\sharp}}{\widetilde{L}_{\sharp}^{2}L_{\sharp} } \cdot \frac{V_{1}}{V_{1} +\lambda },\\
		\label{definition-tilde-theta12-det}
		& \thetaZdet_{1} = \LZ_{\sharp} + \LZ_{\sharp} \cdot \frac{\big( \frac{\parcompX_{\sharp}\parcompZ_{\sharp}}{\LZ_{\sharp}^2} - L_{\sharp}^2 \big) \cdot V  }{\lambda L_\sharp^2\widetilde{L}_{\sharp}^2 + V(L_\sharp^2 + \widetilde{L}_{\sharp}^2)},
		&&\thetaZdet_{2} = \frac{\parcompX_{\sharp}\perpcompZ_{\sharp}}{L_{\sharp}^{2}\LZ_{\sharp} } \cdot \frac{V_{2}}{V_{2} +\lambda },
	\end{align}
\end{subequations}
respectively, where the tuple of scalars $(V,V_{1},V_{2})$ as defined in Eq.~\eqref{eq:V-V1-V2}.  The deterministic prediction can then be written as $\parcompdetX_+ = \frac{\parcompX_{\sharp}}{L_{\sharp}}\cdot \thetaXdet_{1} + \frac{\perpcompX_{\sharp}}{L_{\sharp}} \cdot \thetaXdet_{2}$.  Thus, by the triangle inequality, 
\begin{subequations}\label{decomposition-alpha-concentration}
	\begin{align}
		|\parcompX_{+} - \parcompdetX_{+}| &\leq \frac{\parcompX_{\sharp}}{L_{\sharp}}\cdot\big( |\theta(\bu_{1}) - \thetaXdet_{1} |\big) +
		\frac{\perpcompX_{\sharp}}{L_{\sharp}}\cdot \big( |\theta(\bu_{2}) - \thetaXdet_{2}| \big).
	\end{align}
	Applying the triangle inequality once more yields the decomposition:
	\begin{align}\label{eq:decomp-parallel-component-step1}
		\big|\thetaX(\bu_i) - \thetaXdet_i \big| \leq \underbrace{\big| \thetaX(\bu_i) - \EE\{ \thetaX(\bu_i) \} \big|}_{\text{Stochastic error}} + 
		\underbrace{\big| \EE\{ \thetaX(\bu_i) \} - \thetaXdet_{i} \big|}_{\text{Bias}},\quad  \text{for } i =1,2.
	\end{align}
\end{subequations}
We claim the following pair of bounds
\begin{subequations}\label{ineq:bias-stochastic-error-terms-parallel}
	\begin{align}
	\label{bound-u-expectation-u}
	\bigl \lvert \theta(\bu_1) - \EE\{ \theta(\bu_1) \} \bigr \rvert \; \vee \; \bigl \lvert \theta(\bu_2) - \EE\{ \theta(\bu_2) \} \bigr \rvert &\lesssim \max\bigg\{ \frac{\log^{3}(d)(\sqrt{\Err_{\sharp}} + \sigma )}{\lambda \sqrt{m}} , d^{-50}\bigg\} \\
	\label{upper-bound-expectation-u-det}
	\bigl \lvert \EE\{\theta(\bu_{1})\} - \thetaXdet_{1} \bigr \rvert \vee \bigl \lvert  \EE\{\theta(\bu_{2})\} - \thetaXdet_{2} \bigr \rvert &\lesssim \frac{\sqrt{\Err_{\sharp}}}{\lambda} \cdot \bigg( \frac{\log^{6}(d)}{\sqrt{d}} + \frac{\log^{2.5}(m)}{\sqrt{m}} \bigg),
	\end{align}
\end{subequations}
where the bound~\eqref{bound-u-expectation-u} holds with probability at least $1 - 2d^{-110}$.  We provide the proof of inequalities~\eqref{bound-u-expectation-u} and~\eqref{upper-bound-expectation-u-det} in Sections~\ref{sec:stochastic-error-parallel-component} and~\ref{sec:proof-bias-term-parallel-component}, respectively.  Noting that $\parcompX_{\sharp}/L_{\sharp},\perpcompX_{\sharp}/L_{\sharp} \leq 1$, and substituting the above bounds into the pair of inequalities~\eqref{decomposition-alpha-concentration} yields the result.  It remains to establish the bounds~\eqref{ineq:bias-stochastic-error-terms-parallel}.

\subsubsection{Bounding the stochastic error~\eqref{bound-u-expectation-u}}\label{sec:stochastic-error-parallel-component}
This section is devoted to establishing the following stronger bound, which holds with probability at least $1 - d^{-110}$ for any $\bu \in \mathbb{S}^{2d-1}$:
\begin{align} \label{ineq:stochastic-error-concentration}
\bigg| \Big \langle \bu, \left[\begin{array}{c} \bcoefX_{+}\\ \bcoefZ_{+} \\ \end{array}\right] \Big \rangle - \EE\bigg\{  \Big \langle \bu, \left[\begin{array}{c} \bcoefX_{+}\\ \bcoefZ_{+} \\ \end{array}\right] \Big \rangle \bigg\} \bigg| \leq \max\bigg\{ \frac{C_3\log^{3}(d)(\sqrt{\Err_{\sharp}} + \sigma )}{\lambda \sqrt{m}} , d^{-50}\bigg\},
\end{align}
since the inequality~\eqref{bound-u-expectation-u} then follows upon letting $\bu = [\bu_1^{\top} \; \vert \; \boldsymbol{0}^{\top}]^{\top}$ and $\bu = [\bu_2^{\top} \; \vert \; \boldsymbol{0}^{\top}]^{\top}$.  We turn now to the proof of the inequality~\eqref{ineq:stochastic-error-concentration}. 

Note that we require fine-grained control on the deviations since the guarantee~\eqref{ineq:stochastic-error-concentration} depends on $\Err_{\sharp}$.  In order to facilitate these fine-grained deviation bounds, we employ Warnke's typical bounded differences inequality~\citep[Theorem 2]{warnke2016method}.  In particular, consider the function $f: \mathbb{R}^{(2d +1) \times m} \rightarrow \mathbb{R}$, defined as
 \begin{align}\label{definition-f-component-innerproduct}
 	f_{\bu}\big(\{\bx_{k},\bz_{k},\epsilon_{k}\}_{k=1}^{m} \big) := \big \langle \bu, [\bcoefX_{+}^{\top} \; \vert \; \bcoefZ_{+}^{\top}]^{\top} \big \rangle, 
 \end{align}
and note that it suffices to show that $f_{\bu}\big(\{\bx_{k},\bz_{k},\epsilon_{k}\}_{k=1}^{m} \big)$ concentrates around its expectation with deviations which are $O\big((\sqrt{\Err_{\sharp}} + \sigma)/(\lambda \sqrt{m})\big)$.  In order to apply the typical bounded differences inequality, we must construct a regularity set $\mathcal{S}$ on which (i.) the data $\{\bx_k, \bz_k, \epsilon_k\}_{k=1}^{m}$ lies with high probability and (ii.) the function $f_{\bu}$ enjoys the bounded differences property.  

Towards constructing the regularity set $\mathcal{S}$, we define the pair of estimators 
\begin{subequations}
	\begin{align}
		\label{estimators-leave-i-sample-out}
		\left[\begin{array}{c} \bcoefX^{-(i)} \\ \bcoefZ^{-(i)} \\ \end{array}\right] &= \argmin_{\bcoefX,\bcoefZ \in \mathbb{R}^{d}} \sum_{k \neq i}\big(y_{k} + G_{k}\widetilde{G}_{k} - \widetilde{G}_{k} \bx_{k}^{\top}\bcoefX - G_{k} \bz_{k}^{\top}\bcoefZ \big)^{2} + \lambda m \bigg \| \begin{bmatrix} \bcoefX - \bcoefX_{\sharp} \\ \bcoefZ - \bcoefZ_{\sharp} \end{bmatrix} \bigg\|_2^{2}, \\
		\label{estimators-leave-ij-sample-out}
		\left[\begin{array}{c} \bcoefX^{-(i,j)} \\ \bcoefZ^{-(i,j)} \\ \end{array}\right] &= \argmin_{\bcoefX,\bcoefZ \in \mathbb{R}^{d}}
		\sum_{k \neq i,j}\big(y_{k} + G_{k}\widetilde{G}_{k} - \widetilde{G}_{k} \bx_{k}^{\top}\bcoefX - G_{k} \bz_{k}^{\top}\bcoefZ \big)^{2} + \lambda m \bigg \| \begin{bmatrix} \bcoefX - \bcoefX_{\sharp} \\ \bcoefZ - \bcoefZ_{\sharp} \end{bmatrix} \bigg\|_2^{2},
	\end{align}
\end{subequations}
where we form $[\bcoefX^{-(i)} \, \vert \, \bcoefZ^{-(i)}]$ by leaving out the sample $(\bx_{i},\bz_{i},y_{i})$ and similarly, we form $[\bcoefX^{-(i,j)}\, \vert\, \bcoefZ^{-(i,j)}]$ by leaving out the two samples $\{\bx_{k},\bz_{k},y_{k}\}_{k=i,j}$.  We recall (see the discussion following Eq.~\eqref{eq:closed-form-update}) the shorthand $G_i = \bx_{i}^{\top} \bcoefX_{\sharp}$, $\GZ_i = \bz_i^{\top} \bcoefZ_{\sharp}$ and $\ba_i^{\top} = [\GZ_i \bx_i^{\top} \; \vert \; G_i \bz_i^{\top}]$. We further use the shorthand $\bcoefX^{\perp} = \bP_{\bcoefX_{\star}}^{\perp}\bcoefX_{\sharp}/ \perpcompX_{\sharp}$, $\bcoefZ^{\perp} = \bP_{\bcoefZ_{\star}}^{\perp}\bcoefZ_{\sharp}/ \perpcompZ_{\sharp}$ to describe the normalized orthogonal components of $\bcoefX_{\sharp}$ and $\bcoefZ_{\sharp}$, respectively, as well as the matrix $\bSig = \sum_{k=1}^{m} \ba_{k}\ba_{k}^{\top} + \lambda m \bI$ and its leave-one-out and leave-two-out analogues 
\[
\bSig_{i} = \sum_{k \neq i} \ba_{k}\ba_{k}^{\top} + \lambda m \bI,\qquad \text{ and } \qquad \bSig_{i,j} =  \sum_{k \neq i,j} \ba_{k}\ba_{k}^{\top} + \lambda m \bI, \text{ for } i,j \in [m].
\]
The regularity set $\mathcal{S} \subseteq \mathbb{R}^{(2d+1) \times m}$ is then defined as the subset stable upon leaving one or two samples out.  In particular, let $\mathcal{S}_1 \subseteq \mathbb{R}^{(2d+1) \times m}$ denote the event that---for a fixed unit vector $\bu$---the vectors $\ba_i$ and $\ba_j$ are approximately orthogonal to the leave one out matrices $\bSig_{i}, \bSig_{i,j}$:
	\begin{align*}
		\mathcal{S}_1 := \Big\{ &\{\bx_{i},\bz_{i},\epsilon_{i}\}_{i=1}^{m} \in \mathbb{R}^{(2d+1) \times m}: \forall\;i,j \in [m],\quad  | \bu^{\top} \bSig_{i}^{-1} \ba_{i} |,| \bu^{\top} \bSig_{i,j}^{-1} \ba_{j} | \leq \frac{C\log d}{\lambda m} \Bigr\}.
	\end{align*}
We then let the set $\mathcal{S}_2 \subseteq \mathbb{R}^{(2d+1) \times m}$ denote the event that the noise is bounded above by $\sqrt{\log{d}}$, the data $\bx_i, \bz_i$ are bounded in norm by $\sqrt{d}$ and are approximately orthogonal to $\bcoefX_{\star}, \bcoefX^{\perp}$ and $\bcoefZ_{\star}, \bcoefZ^{\perp}$, respectively:
\begin{align*}
	\mathcal{S}_2 := \Big\{ &\{\bx_{i},\bz_{i},\epsilon_{i}\}_{i=1}^{m} \in \mathbb{R}^{(2d+1) \times m}:\; \forall\;i \in [m],\quad |\bx_{i}^{\top} \bcoefX_{\star}|, |\bx_{i}^{\top} \bcoefX^{\perp}|, |\bz_{i}^{\top} \bcoefZ_{\star}|, |\bz_{i}^{\top} \bcoefZ^{\perp}| \leq C\sqrt{\log d}, \nonumber\\
	& \qquad \qquad \qquad \qquad \qquad \qquad \qquad \qquad\|\bx_{i}\|_{2}, \|\bz_{i}\|_{2} \leq C\sqrt{d},  \qquad \text{ and } \qquad \epsilon_{i} \leq C \sigma \sqrt{\log d} \Bigr\}.
\end{align*}
Finally, we require that $\bx_i, \bz_i$ are approximately orthogonal to the differences $\bcoefX_{\sharp} - \bcoefX^{-(i)}$ and $\bcoefZ_{\sharp} - \bcoefZ^{-(i)}$(and similarly for the leave two out quantities):
\begin{align*}
		\mathcal{S}_3 := \bigg\{ \{\bx_{i},\bz_{i},&\epsilon_{i}\}_{i=1}^{m} \in \mathbb{R}^{(2d+1) \times m}: \; \frac{|\bx_{i}^{\top} (\bcoefX_{\sharp} - \bcoefX^{-(i)})|}{\| \bcoefX_{\sharp} - \bcoefX^{-(i)} \|_{2}} \leq C \sqrt{\log d}, \quad \frac{|\bx_{i}^{\top} (\bcoefX_{\sharp} - \bcoefX^{-(i,j)})|}{\| \bcoefX_{\sharp} - \bcoefX^{-(i,j)} \|_{2}} \leq C \sqrt{\log d}, \nonumber \\
		& \qquad \frac{|\bz_{i}^{\top} (\bcoefZ_{\sharp} - \bcoefZ^{-(i)})|}{\| \bcoefZ_{\sharp} - \bcoefZ^{-(i)} \|_{2}} \leq C \sqrt{\log d}, \quad \frac{|\bz_{i}^{\top} (\bcoefZ_{\sharp} - \bcoefZ^{-(i,j)})|}{\| \bcoefZ_{\sharp} - \bcoefZ^{-(i,j)} \|_{2}} \leq C \sqrt{\log d}, \qquad \forall\;i,j \in [m]  \bigg\},
\end{align*}
where in each of the above three sets, $C$ is a large, positive universal constant.  The regularity set $\mathcal{S}$ is then defined to be the intersection of the above three sets:
\begin{align} \label{definition-regularity-set-component-expectation}
	\mathcal{S} = \mathcal{S}_1 \cap \mathcal{S}_2 \cap \mathcal{S}_3.
\end{align}
Let us now outline the properties that the regularity set $\mathcal{S}$ enjoys. First, define the Hamming metric $\rho:\mathbb{R}^{(2d+1)\times m} \times \mathbb{R}^{(2d+1)\times m} \rightarrow \mathbb{Z}_{\geq 0}$ as
\begin{align}\label{eq:hamming-metric-function}
	\rho\big(\{\bx_{k},\bz_{k},\epsilon_{k}\}_{k=1}^{n}, \{\bx'_{k},\bz'_{k},\epsilon'_{k}\}_{k=1}^{n} \big) = \sum_{k=1}^{m} \mathbbm{1}\big\{ (\bx_{k},\bz_{k},\epsilon_{k}) \neq (\bx'_{k},\bz'_{k},\epsilon'_{k}) \big\},
\end{align}
and let $\bM := \{\bx_{i},\bz_{i},\epsilon_{i}\}_{i=1}^{m}$ and $\bM' := \{\bx'_{i},\bz'_{i},\epsilon'_{i}\}_{i=1}^{m}$.  The following lemma demonstrates that the function $f_{\bu}$ enjoys the bounded difference property on the regularity set $\mathcal{S}$; we provide its proof in Section~\ref{proof-auxiliary-lemma1-component-expectation-concentration}.
\begin{lemma}\label{auxiliary-lemma1-component-expectation-concentration}
	Consider $f_{\bu}$ in~\eqref{definition-f-component-innerproduct} and the regularity set $\mathcal{S}$ in~\eqref{definition-regularity-set-component-expectation}. The following hold.
	\begin{itemize}
		\item[(a)] There exists a universal, positive constant $C'$ such that if $\lambda m \geq C' d(1+\sigma)$, then $|f_{\bu}(\bM)| \leq d,\; \forall\; \bM \in \mathcal{S}$.
		\item[(b)] There exists a universal, positive constant $C_{0}$ such that for all $\bM, \bM' \in \mathcal{S}$ which satisfy $\rho(\bM, \bM') \leq 2$, 
		\begin{align*}
			&| f_{\bu}(\bM) - f_{\bu}(\bM') | \leq \Delta \qquad \text{ where }  \qquad \Delta = \frac{C_{0}\log^{3}(d)(\sqrt{\Err_{\sharp}} + \sigma)}{\lambda m}.
		\end{align*} 
	\end{itemize}
\end{lemma}
Having shown that $f_{\bu}$ enjoys the bounded differences property on $\mathcal{S}$, we extend it to a function $f_{\bu}^{\downarrow}$, which enjoys the bounded differences property when \emph{only one} of its arguments lies in $\mathcal{S}$ via truncation.  In particular, let $D = d$ and note that Lemma~\ref{auxiliary-lemma1-component-expectation-concentration}(a) implies $D \geq \sup_{\bM \in \mathcal{S}} |f_{\bu}(\bM)|$ as long as $\lambda m\geq C'd(1+\sigma)$.  We define the extension $f_{\bu}^{\downarrow}:\mathbb{R}^{(2d+1)\times m} \rightarrow \mathbb{R}$ as 
\begin{align}\label{definition-f-down-u-component}
	f_{\bu}^{\downarrow}(\bM)  = \inf_{\bM' \in \mathcal{S}} \Big\{ f_{\bu}\big(\bM'\big) + \Delta \cdot \rho\big(\bM,\bM'\big) + 2D \cdot \mathbbm{1}\big\{ \rho\big(\bM,\bM'\big) > 1 \big\} \Big\},
\end{align}
where $\Delta$ is as in Lemma~\ref{auxiliary-lemma1-component-expectation-concentration}(b). Consequently, by Lemmas~\ref{auxiliary-lemma1-component-expectation-concentration} and~\ref{lemma:bounded-difference-truncate}, we deduce for all $\bM \in \mathcal{S}$ and $\bM' \in \mathbb{R}^{(2d + 1) \times m}$ which satisfy $\rho(\bM, \bM') \leq 1$, both of the following properties hold
\begin{align}\label{properties-truncate-f-u-component}
	f_{\bu}^{\downarrow}(\bM) = f_{\bu}(\bM), \qquad \text{ and } \qquad 
	|f_{\bu}^{\downarrow}(\bM) - f_{\bu}^{\downarrow}(\bM')| \leq  \Delta.
\end{align}
The following lemma establishes the desideratum that the data lies in $\mathcal{S}$ with high probability and demonstrates that concentration of the truncated function $f_{\bu}^{\downarrow}$ suffices to establish the concentration of the original function $f_{\bu}$.  We provide its proof in Section~\ref{proof-auxiliary-lemma2-component-expectation-concentration}.
\begin{lemma}\label{auxiliary-lemma2-component-expectation-concentration}
	Suppose $\{\bx_{i},\bz_{i}\}_{i=1}^{m} \overset{\mathsf{i.i.d.}}{\sim} \mathsf{N}(\boldsymbol{0},\bI_{d})$, $\{\epsilon_{i}\}_{i=1}^{m} \overset{\mathsf{i.i.d.}}{\sim} \mathsf{N}(0,\sigma^{2})$. There exists universal constants $C,d_{0}$ such that for $d\geq d_{0}, 1\leq m \leq d$, $\lambda m \geq C(1+\sigma)d$ and $t\geq d^{-80}$, we have
	\begin{align*}
		(a.)\;  \Pr\big\{ \{\bx_{i},\bz_{i},\epsilon_{i}\}_{i=1}^{m} &\notin \mathcal{S} \big\} \leq d^{-180}, \qquad (b.)\; \sup_{\{\bx_{i},\bz_{i},\epsilon_{i}\}_{i=1}^{m} \in \mathbb{R}^{(2d+1)\times m}} \big|f_{\bu}^{\downarrow}(\{\bx_{i},\bz_{i},\epsilon_{i}\}_{i=1}^{m}) \big| \lesssim d, \\ 
		\text{ and } \qquad (c.)\; &\Pr\big\{ |f_{\bu} - \EE\{f_{\bu}\}| \geq t \big\} \leq \Pr\big\{ |f_{\bu}^{\downarrow} - \EE\{f_{\bu}^{\downarrow}\}| \geq t/2 \big\} + d^{-180}.
	\end{align*}
\end{lemma}

It remains to control the quantity $\Pr\{|f_{\bu}^{\downarrow} - \EE\{f_{\bu}^{\downarrow}\}| \geq t/2 \}$.  Combining inequality~\eqref{properties-truncate-f-u-component} and Lemma~\ref{auxiliary-lemma2-component-expectation-concentration}(b.) yields (for $\rho(\bM,\bM') \leq 1$)
\[
\big| f_{\bu}^{\downarrow}(\bM) - f_{\bu}^{\downarrow}(\bM') \big| \leq 
\left\{ \begin{array}{c} \Delta, \text{ if } \bM \in \mathcal{S},\\ Cd, \text{otherwise}. \\ \end{array}\right.
\]
Then, we set $c_{k} = \Delta, d_{k} = Cd$ and $\gamma_{k} \in (0,1]$ for each $k \in [n]$ and apply~\citet[Theorem 2]{warnke2016method} to deduce  that for all $t>0$
\[
\Pr\Big\{ |f_{\bu}^{\downarrow}(\bM) - \EE\{ f_{\bu}^{\downarrow}(\bM) \} | \geq t \Big\} \leq 2\exp\Big\{ -\frac{t^{2}}{4\sum_{k=1}^{m}(c_{k}^{2} + \gamma_{k}^{2}d_{k}^{2}) }\Big\} + \sum_{k=1}^{m} \gamma_{k}^{-1} \cdot \Pr\{\bM \notin \mathcal{S}\}.
\]
Setting $\gamma_{k} = d^{-60}$ for each $k\in[m]$ and applying Lemma~\ref{auxiliary-lemma2-component-expectation-concentration}(a.), we obtain that with probability at least $1-d^{-115}$,
\begin{align*}
	|f_{\bu}^{\downarrow}(\bM) - \EE\{ f_{\bu}^{\downarrow}(\bM) \} | &\leq C\sqrt{\log(d)} \max\Big\{ \Big(\sum_{k=1}^{m}c_{k}^{2} \Big)^{1/2}, \Big( \sum_{k=1}^{m}\gamma_{k}^{2}d_{k}^{2} \Big)^{1/2}  \Big\} \\& \leq C\sqrt{\log(d)} \max \big\{ \sqrt{m}\Delta, d^{-52}  \big\}  \leq \max\Big\{ \frac{C'\log^{3.5}(d)(\sqrt{\Err_{\sharp}} + \sigma )}{\lambda \sqrt{m}} , d^{-50}\Big\}.
\end{align*}
Finally, applying Lemma~\ref{auxiliary-lemma2-component-expectation-concentration}(c.), we obtain that with probability at least $1-d^{-110}$,
\[
|f_{\bu} - \EE\{f_{\bu}\}| \leq \max\Big\{ \frac{C'\log^{3.5}(d)(\sqrt{\Err_{\sharp}} + \sigma )}{\lambda \sqrt{m}} , d^{-50}\Big\},
\]
which concludes the proof of inequality~\eqref{ineq:stochastic-error-concentration}.

\subsubsection{Bounding the bias~\eqref{upper-bound-expectation-u-det}}\label{sec:proof-bias-term-parallel-component}
Our bounds on the bias term follow step 2b of the strategy outlined in Section~\ref{sec:techniques-one-step-prediction}.  We begin with some notation.  For any directions $\bu, \bv \in \mathbb{S}^{d-1}$, we let $\theta(\bu) = \langle \bu, \bcoefX_{+}\rangle$ and $\widetilde{\theta}(\bv) = \langle \bv, \bcoefZ_{+}\rangle$.  Additionally, we let $\bO_{\bu} \in \mathbb{R}^{d \times (d-1)}$ and $\bO_{\bv} \in \mathbb{R}^{d \times (d-1)}$ consist of columns which form orthonormal bases of the $d-1$ dimensional subspaces orthogonal to $\bu$ and $\bv$, respectively.  Recalling that $\bW = \diag(\bX\bcoefX_{\sharp})$ and $\widetilde{\bW} = \diag(\bZ \bcoefZ_{\sharp})$, we then let $\bA_{\bu, \bv} = [ \widetilde{\bW} \bX \bO_{\bu} \; \vert \; \bW \bZ \bO_{\bv}]$ and $\bP_{\bu, \bv} = \bI - \bA_{\bu,\bv} \big(\bA_{\bu,\bv}^{\top}\bA_{\bu,\bv} + \lambda m \bI\big)^{-1}\bA_{\bu,\bv}^{\top}$ denote leave-one-out variants of the data $\bA$ and the ``projection" matrix $\bP$ which do not depend on the directions $\bu$ and $\bv$.  The following lemma, whose proof we defer to Section~\ref{proof-lemma-leave-one-direction-out}, characterizes $\theta(\bu)$ and $\widetilde{\theta}(\bv)$ as solutions to a linear system. 
\begin{lemma}\label{lemma-leave-one-out}
	Let $\bcoefX_{+}, \bcoefZ_{+}, \bcoefX_{\sharp}, \bcoefZ_{\sharp}$ be as in Theorem~\ref{thm:one-step-prediction} and consider $\bu, \bv \in \mathbb{S}^{d-1}$. The following holds.
	\begin{align}\label{eq-leave-one-out}
		&\bigg(
		\frac{1}{m}  \begin{bmatrix} (\widetilde{\bW}\bX \bu)^{\top} \\ (\bW \bZ \bv)^{\top} \end{bmatrix} \bP_{\bu,\bv} \bigl[ \widetilde{\bW} \bX \bu \; \vert \; \bW \bZ \bv \bigr]  + \lambda \bI \bigg)
	\begin{bmatrix}
		\theta(\bu)\\ \widetilde{\theta}(\bv)
	\end{bmatrix} =  \lambda 
		\begin{bmatrix}
			\langle \bcoefX_{\sharp} , \bu \rangle \\
			\langle \bcoefZ_{\sharp} , \bv \rangle
		\end{bmatrix}  \\
	& + \frac{1}{m}  \begin{bmatrix}
			(\widetilde{\bW}\bX \bu)^{\top}\\
			(\bW\bZ \bv)^{\top}
		\end{bmatrix} \bP_{\bu,\bv} \bigl(\by + \diag(\bW\widetilde{\bW})\bigr)  - 
		\lambda \begin{bmatrix}
			(\widetilde{\bW}\bX \bu)^{\top}\\
			(\bW\bZ \bv)^{\top}
		\end{bmatrix} \bA_{\bu,\bv} \big(\bA_{\bu,\bv}^{\top} \bA_{\bu,\bv} + \lambda m\bI \big)^{-1} 
		\begin{bmatrix}
			\bO_{\bu}^{\top} \bcoefX_{\sharp}  \\
			\bO_{\bv}^{\top} \bcoefZ_{\sharp}
		\end{bmatrix}. \nonumber
	\end{align} 
\end{lemma}
\noindent We next state explicit expressions for the components $\theta(\bu_1), \widetilde{\theta}(\bv_1), \theta(\bu_2),$ and $\widetilde{\theta}(\bv_2)$.  The proofs of the validity of these expressions rely on Lemma~\ref{lemma-leave-one-out}, and we provide them at the end of the section.  The expressions for $\theta(\bu_1)$ and $\widetilde{\theta}(\bv_1)$ are stated in terms of the random variables $M_1$ and $M_2$, defined as 
\begin{align}\label{def:M1M2}
	M_1 &= \frac{1}{m}\diag(\bW\bWtil)^{\top}\bP_{\bu_1,\bv_1}\diag(\bW\bWtil),\nonumber\\ 
	M_2 &= \frac{1}{m}\diag(\bW\bWtil)^{\top}\bP_{\bu_1,\bv_1}\Big( \frac{\parcompX_\sharp\perpcompZ_\sharp}{L_\sharp^{2}\widetilde{L}_{\sharp}} 
	\bW\bZ\bv_2 + \frac{\parcompZ_\sharp\perpcompX_\sharp}{\widetilde{L}_{\sharp}^2L_\sharp}\bWtil\bX\bu_2 + \frac{\perpcompX_\sharp\perpcompZ_\sharp}{L_\sharp\widetilde{L}_{\sharp}} \bX\bu_2 \odot \bZ\bv_2  + \boldsymbol{\epsilon} \Big), 
\end{align}
We then have
\begin{subequations}
	\label{eq:theta-u1-char}
	\begin{align}
		\theta(\bu_1) &= L_\sharp + \frac{L_\sharp \big(\frac{\parcompX_\sharp\parcompZ_\sharp}{L_\sharp^2} - \widetilde{L}_{\sharp}^2 \big) M_1 + L_\sharp\widetilde{L}_{\sharp}^2M_2}{\lambda L_\sharp^2\widetilde{L}_{\sharp}^2 + M_1(L_\sharp^2 + \widetilde{L}_{\sharp}^2)}, \qquad \text{ and }\\
	\widetilde{\theta}(\bv_1) &= \widetilde{L}_{\sharp} + \frac{\widetilde{L}_{\sharp} \big( \frac{\parcompX_\sharp\parcompZ_\sharp}{\widetilde{L}_{\sharp}^2} - L_\sharp^2 \big)M_1 + L_\sharp^2\widetilde{L}_{\sharp}M_2}{\lambda L_\sharp^2\widetilde{L}_{\sharp}^2 + M_1(L_\sharp^2 + \widetilde{L}_{\sharp}^2)}.
	\end{align}
\end{subequations}
We will later also require the random variable $M_3 = \frac{1}{m} \diag(\bW\bWtil)^{\top} \bP_{\bu_{1},\bv_{1}} \bX\bu_{2}\odot \bZ\bv_{2}$.
The expressions for $\theta(\bu_2)$ and $\widetilde{\theta}(\bv_2)$, further rely on the random variables $M_{11}, M_{12}, M_{22}$, as well as $E_1, E_2$ and $F_1, F_2$, which we define presently.
\begin{align} \label{def:M11M12M22}
	M_{11} &= \frac{1}{m}(\bWtil\bX\bu_2)^{\top}\bP_{\bu_2,\bv_2}\bWtil\bX\bu_2, \qquad M_{12} = \frac{1}{m}(\bWtil\bX\bu_2)^{\top}\bP_{\bu_2,\bv_2}\bW\bZ\bv_2,\nonumber\\
	&\qquad \text{ and } \qquad M_{22} = \frac{1}{m}(\bW\bZ\bv_2)^{\top}\bP_{\bu_2,\bv_2}\bW\bZ\bv_2.
\end{align}
The key properties of the random variables $M_1, M_2, M_3, M_{11}, M_{12}$, and $M_{22}$ are summarized in the following lemma, whose proof we defer to Section~\ref{subsec:proof-concentration-M11-M22-M12-M1-M2}.
\begin{lemma}\label{concentration-of-M11-M22-M12-M1-M2}
	Recall the tuple of deterministic quantities $(V,V_1,V_2)$ in Eq.~\eqref{eq:V-V1-V2} and the tuple of random variables $(M_{11},M_{12},M_{22},M_{1}, M_{3})$ in Eq.~\eqref{def:M1M2} and Eq.~\eqref{def:M11M12M22}, and consider the assumptions in Theorem~\ref{thm:one-step-prediction}. Then, the following four properties hold.
	\begin{align*}	
		&(a.)\;\; \EE\big\{|M_{11}-V_{1}|\big\}\;\vee\; \EE\big\{|M_{22}-V_{2}|\big\} \lesssim \frac{\log^{1.5}(d)}{\sqrt{d}} + \frac{1}{\sqrt{m}},\qquad (b.)\;\; \EE\big\{M_{12}^{2}\big\} \lesssim \frac{1}{m},\\
		& (c.)\;\; \EE\big\{|M_{1} - V|\big\} \lesssim \frac{\log^{6}(d)}{\sqrt{d}} + \frac{\log^{2.5}(m)}{\sqrt{m}} \qquad \text{ and } \qquad (d.)\;\; \EE\big\{|M_{3}|\big\} \lesssim \frac{\log^{6}(d)}{\sqrt{d}} + \frac{1}{\sqrt{m}}.
	\end{align*}
\end{lemma}
\noindent The random variables $E_1$ and $E_2$ depend on the noise $\bepsilon$ and are defined as
\begin{align} %\label{def:E1E2}
	E_1 &= 
	\frac{1}{m} (\bWtil\bX\bu_2)^{\top} \bP_{\bu_2,\bv_2} \bigg( \Big(1+ \frac{\parcompX_{\sharp}}{L_{\sharp}^{2}}  \frac{\parcompZ_{\sharp}}{\widetilde{L}_{\sharp}^{2}} \Big) \cdot \diag(\bW\bWtil) + \frac{\parcompX_{\sharp}}{L_{\sharp}^{2}}  \frac{\perpcompZ_{\sharp}}{\widetilde{L}_{\sharp}} \cdot \bW\bZ\bv_2 + \frac{\perpcompX_{\sharp} \perpcompZ_{\sharp}}{L_{\sharp}\widetilde{L}_{\sharp}} \cdot \bX\bu_2 \odot \bZ\bv_2 + \boldsymbol{\epsilon} \bigg), \nonumber \\
	E_2 &= 
	\frac{1}{m} (\bW\bZ\bv_2)^{\top} \bP_{\bu_2,\bv_2} \bigg( \Big(1+ \frac{\parcompX_{\sharp}}{L_{\sharp}^{2}} \frac{\parcompZ_{\sharp}}{\widetilde{L}_{\sharp}^{2}} \Big) \cdot \diag(\bW\bWtil) + \frac{\parcompZ_{\sharp}}{\widetilde{L}_{\sharp}^{2}}  \frac{\perpcompX_{\sharp}}{L_{\sharp}} \cdot \bWtil\bX\bu_2 + \frac{\perpcompX_{\sharp} \perpcompZ_{\sharp}}{L_{\sharp}\widetilde{L}_{\sharp}} \cdot \bX\bu_2 \odot \bZ\bv_2 + \boldsymbol{\epsilon} \bigg), \nonumber.
\end{align}
Finally, we define the random variables $F_1$ and $F_2$ as 
\begin{align} %\label{def:F1F2}
	F_1 &= \lambda(\bWtil\bX \bu_2)^{\top}\bA_{\bu_2,\bv_2} \big(\bA_{\bu_2,\bv_2}^{\top} \bA_{\bu_2,\bv_2} + \lambda m\bI \big)^{-1} 
	\left[\begin{array}{c}
		\bO_{\bu_2}^{\top}\bcoefX_{\sharp}  \\
		\bO_{\bv_2}^{\top}\bcoefZ_{\sharp}
	\end{array}\right] \qquad \text{ and } \nonumber \\
	F_2 &= \lambda(\bW\bZ \bv_2)^{\top}\bA_{\bu_2,\bv_2} \big(\bA_{\bu_2,\bv_2}^{\top} \bA_{\bu_2,\bv_2} + \lambda m\bI \big)^{-1} 
	\left[\begin{array}{c}
		\bO_{\bu_2}^{\top}\bcoefX_{\sharp}  \\
		\bO_{\bv_2}^{\top}\bcoefZ_{\sharp}
	\end{array}\right].
\end{align}
Equipped with these definitions, we have the following expressions for $\theta(\bu_2)$ and $\widetilde{\theta}(\bv_2)$
\begin{subequations}\label{eq:theta-u2-char}
\begin{align}
\theta(\bu_2) &= \frac{(M_{22}+\lambda)\big(\frac{\parcompZ_{\sharp}}{\widetilde{L}_{\sharp}^{2}} \frac{\perpcompX_{\sharp}}{L_{\sharp}} \cdot M_{11} + E_{1}-F_{1} \big) - M_{12}\big(\frac{\parcompX_{\sharp}}{L_{\sharp}^{2}}  \frac{\perpcompZ_{\sharp}}{\widetilde{L}_{\sharp}} \cdot M_{22} + E_{2} - F_{2} \big)}{(M_{11}+\lambda)(M_{22}+\lambda) - M_{12}^{2}}, \\
\widetilde{\theta}(\bv_2) &= \frac{(M_{11}+\lambda)\big(\frac{\parcompX_{\sharp}}{L_{\sharp}^{2}}  \frac{\perpcompZ_{\sharp}}{\widetilde{L}_{\sharp}} \cdot M_{22} + E_{2} - F_{2} \big) - M_{12}\big(\frac{\parcompZ_{\sharp}}{\widetilde{L}_{\sharp}^{2}} \frac{\perpcompX_{\sharp}}{L_{\sharp}} \cdot M_{11} + E_{1}-F_{1} \big) }{(M_{11}+\lambda)(M_{22}+\lambda) - M_{12}^{2}}.
\end{align}
\end{subequations}
Taking these expressions for granted, we claim the following expressions characterize $\EE\{\theta(\bu_1)\}$ and $\EE\{\theta(\bu_2)\}$.  We defer its proof to the end of the section.
\begin{align}	\label{claim-alpha-iota-concentration}
	\EE\{\theta(\bu_{1})\} &= L_{\sharp} + \EE\Bigg\{ \frac{L_\sharp(\frac{\parcompX_\sharp\parcompZ_\sharp}{L_\sharp^2} - \widetilde{L}_{\sharp}^2) M_1 + \widetilde{L}_{\sharp}\perpcompX_{\sharp} \perpcompZ_{\sharp} M_{3} }{\lambda L_\sharp^2\widetilde{L}_{\sharp}^2 + M_1(L_\sharp^2 + \widetilde{L}_{\sharp}^2)} \Bigg\}, \qquad \text{ and } \nonumber\\ \EE\{\theta(\bu_{2})\} &= \EE\Bigg\{ \frac{\frac{\parcompZ_{\sharp}}{\widetilde{L}_{\sharp}^{2}}  \frac{\perpcompX_{\sharp}}{L_{\sharp}} M_{11} - \frac{\parcompZ_{\sharp}}{\widetilde{L}_{\sharp}^{2}}  \frac{\perpcompX_{\sharp}}{L_{\sharp}} \frac{M_{12}^{2}}{M_{22}+\lambda} }{(M_{11}+\lambda) - \frac{M_{12}^{2}}{M_{22}+\lambda}}  \Bigg\}.
\end{align}
We turn now to bounding $\lvert \EE\{\theta(\bu_1)\} - \theta_1^{\mathsf{det}}\rvert$.  To this end, we combine the claim~\eqref{claim-alpha-iota-concentration} with the explicit expression $\thetaXdet_{1}$ in~\eqref{definition-theta12-det} and apply the triangle inequality to obtain the decomposition
\begin{align*}
	\big| \EE\{\theta(\bu_{1})\} - \thetaXdet_{1} \big| &\leq \big|L_{\sharp}\big(\frac{\parcompX_\sharp\parcompZ_\sharp}{L_\sharp^2} - \widetilde{L}_{\sharp}^2 \big) \big| \cdot T_1 + \widetilde{L}_{\sharp}\perpcompX_{\sharp} \perpcompZ_{\sharp} \cdot T_2,
\end{align*}
where $T_1 = \Big| \EE\Big\{ \frac{M_1 }{\lambda L_\sharp^2\widetilde{L}_{\sharp}^2 + M_1(L_\sharp^2 + \widetilde{L}_{\sharp}^2)} \Big\} -\frac{V}{\lambda L_\sharp^2\widetilde{L}_{\sharp}^2 + V(L_\sharp^2 + \widetilde{L}_{\sharp}^2)}  \Big|$ and $T_2 = \Big| \EE\Big\{ \frac{ M_{3} }{\lambda L_\sharp^2\widetilde{L}_{\sharp}^2 + M_1(L_\sharp^2 + \widetilde{L}_{\sharp}^2)} \Big\}  \Big| $.  Note that $M_{1},M_{11},M_{22}\geq 0$ since $\bP_{\bu_{i},\bv_{i}} \succeq \boldsymbol{0}$ for $i=1,2$.  Moreover, by definition, $V,V_{1}\geq 0$~\eqref{eq:V-V1-V2}. Consequently, applying the triangle inequality yields the pair of bounds $T_{1} \leq \frac{2}{\lambda L_{\sharp}^{2} \LZ_{\sharp}^{2}} \cdot \EE\big\{ |M_{1} - V|\big\}$ and $T_{2} \leq \frac{1}{\lambda L_{\sharp}^{2} \LZ_{\sharp}^{2}} \cdot \EE\big\{|M_{3}|\big\}$.  Then, applying Lemma~\ref{concentration-of-M11-M22-M12-M1-M2}(c.),(d.) yields the pair of inequalities
\begin{align*} %\label{ineq:T1-T2-intermediate-bound}
	T_1 \lesssim \frac{1}{\lambda L_{\sharp}^{2} \LZ_{\sharp}^{2}} \bigg(\frac{\log^{6}(d)}{\sqrt{d}} + \frac{\log^{2.5}(m)}{\sqrt{m}} \bigg),\qquad \text{ and } \qquad T_{2} \lesssim \frac{1}{\lambda L_{\sharp}^{2} \LZ_{\sharp}^{2}} \bigg( \frac{\log^{6}(d)}{\sqrt{d}} + \frac{1}{\sqrt{m}}\bigg).
\end{align*}
Combining the previous two displays and using the assumption that $L_{\sharp}, \widetilde{L}_{\sharp} \asymp 1$ yields the inequality
\begin{align} \label{bound-theta1-bias}
\big| \EE\{\theta(\bu_{1})\} - \thetaXdet_{1} \big| \lesssim  \frac{\big|\parcompX_{\sharp}\parcompZ_{\sharp} - L_{\sharp}^{2} \LZ_{\sharp}^{2} \big| + \perpcompX_{\sharp} + \perpcompZ_{\sharp}}{\lambda} \cdot \bigg( \frac{\log^{6}(d)}{\sqrt{d}} + \frac{\log^{2.5}(m)}{\sqrt{m}} \bigg).
\end{align}
We next turn to bounding $\big| \EE\{\theta(\bu_{2})\} - \thetaXdet_{2} \big| $.  Proceeding similarly, we decompose $\big| \EE\{\theta(\bu_{2})\} - \thetaXdet_{2} \big| \leq \frac{\parcompZ_{\sharp}}{\widetilde{L}_{\sharp}^{2}} \frac{\perpcompX_{\sharp}}{L_{\sharp}} \cdot (T_3 + T_4)$, where $T_3 =  \Big| \EE\Big\{ \frac{  M_{11} }{(M_{11}+\lambda) - \frac{ M_{12}^{2} }{M_{22}+\lambda} }  \Big\} -  \frac{V_{1} }{V_{1}+\lambda}  \Big|$ and $T_4 = \Big| \EE\Big\{ \frac{\frac{M_{12}^{2}}{M_{22}+\lambda} }{(M_{11}+\lambda) - \frac{ M_{12}^{2} }{M_{22}+\lambda} }  \Big\} \Big \vert$.  Towards bounding $T_3$, we note that 
\begin{align*}
	M_{12}^{2} &= \frac{1}{m^{2}}\Big( (\bWtil\bX\bu_{2})^{\top} \bP_{\bu_{2},\bv_{2}} \bW \bZ\bv_{2} \Big)^{2} 
	= \frac{1}{n^{2}} \Big \langle \bP_{\bu_{2},\bv_{2}}^{\frac{1}{2}} \bWtil\bX\bu_{2} ,\; \bP_{\bu_{2},\bv_{2}}^{\frac{1}{2}} \bW\bZ\bv_{2}  \Big \rangle^{2} 
	\\& \leq \frac{1}{m^{2}} \big\| \bP_{\bu_{2},\bv_{2}}^{\frac{1}{2}} \bWtil\bX\bu_{2} \big\|_{2}^{2} \cdot 
	\big\| \bP_{\bu_{2},\bv_{2}}^{\frac{1}{2}} \bW\bZ\bv_{2}  \big\|_{2}^{2} = M_{11} \cdot M_{22},
\end{align*}
where the inequality follows from the Cauchy--Schwarz inequality. We thus deduce the inequalities $M_{11} \geq M_{12}^{2}/M_{22} \geq M_{12}^{2}/(\lambda+M_{22})$. Consequently,
\begin{align}\label{ineq:T3-T4-intermediate-bound}
	\hspace{-0.5cm}
	T_{3} = \bigg| \EE\bigg\{ \frac{\lambda (M_{11}-V_{1}) + \frac{V_{1}M_{12}^{2}}{M_{22}+\lambda}}{\Big(\lambda + M_{11} - \frac{M_{12}^{2}}{M_{22} + \lambda} \Big)(\lambda + V_1)}\bigg\} \bigg| \leq \frac{\EE\big\{|M_{11}-V_{1}|\big\}}{\lambda} + \frac{\EE\big\{M_{12}^{2}\big\}}{\lambda^{2}},
\end{align}
and $T_4 \leq \EE\{M_{12}^2\}/\lambda^2$.  Combining with Lemma~\ref{concentration-of-M11-M22-M12-M1-M2}(a.),(b.) yields the bound $T_{3}+T_{4} \lesssim \frac{1}{\lambda} \Big(\frac{\log^{6}(d)}{\sqrt{d}} + \frac{1}{\sqrt{m}} \Big)$, which in turn implies the bound
\begin{align}\label{bound-theta2-bias}
 \big| \EE\big\{\theta(\bu_{2})\big\} - \thetaXdet_{2} \big| \lesssim \frac{\big|\parcompX_{\sharp}\parcompZ_{\sharp} - L_{\sharp}^{2} \LZ_{\sharp}^{2} \big| + \perpcompX_{\sharp} + \perpcompZ_{\sharp}}{\lambda} \cdot \bigg( \frac{\log^{6}(d)}{\sqrt{d}} + \frac{\log^{2.5}(m)}{\sqrt{m}} \bigg).
\end{align}
The desired inequality~\eqref{upper-bound-expectation-u-det} follows directly by noting
\begin{align*}
	\big|\parcompX_{\sharp}\parcompZ_{\sharp} - L_{\sharp}^{2} \LZ_{\sharp}^{2} \big| + \perpcompX_{\sharp} + \perpcompZ_{\sharp} &= \big|\parcompX_{\sharp}\parcompZ_{\sharp}(1-\parcompX_{\sharp}\parcompZ_{\sharp}) - \parcompX_{\sharp}^{2} \perpcompZ_{\sharp}^{2} - \parcompZ_{\sharp}^{2}\perpcompX_{\sharp}^{2} - \perpcompX_{\sharp}^{2}\perpcompZ_{\sharp}^{2} \big| + \perpcompX_{\sharp} + \perpcompZ_{\sharp}
	\\&\lesssim |1-\parcompX_{\sharp}\parcompZ_{\sharp}| + \perpcompX_{\sharp} + \perpcompZ_{\sharp} \lesssim \sqrt{ \Err_{\sharp} },
\end{align*}
and combining with the inequalities~\eqref{bound-theta1-bias} and~\eqref{bound-theta2-bias}.  It remains to prove the characterizations~\eqref{eq:theta-u1-char},~\eqref{eq:theta-u2-char}, and~\eqref{claim-alpha-iota-concentration}.

\paragraph{Proof of the characterizations~\eqref{eq:theta-u1-char}.}
By definition of $\bu_1$ and $\bv_1$~\eqref{eq:shorthand-length-L-Ltil},
we obtain that $\bO_{\bu_1}^{\top} \bcoefX_{\sharp} = \bO_{\bv_1}^{\top}\bcoefZ_{\sharp} = \boldsymbol{0}$, $\langle \bcoefX_\sharp,\bu_1\rangle = \|\bcoefX_\sharp\|_{2} = L_\sharp$ and $\langle \bcoefZ_\sharp,\bv_1\rangle = \|\bcoefZ_\sharp\|_{2} = \LZ_\sharp$. Further, since $\diag(\bX\bu_{1}) = \bW/L_\sharp$ and $\diag(\bZ\bv_1)= \bWtil / \LZ_\sharp$, we note
\begin{align*}
(\bWtil \bX\bu_1)^{\top} \bP_{\bu_1,\bv_1} &\bWtil \bX\bu_1 = \frac{M_1}{L_{\sharp}^{2}},\qquad (\bW \bZ\bv_1)^{\top} \bP_{\bu_1,\bv_1} \bW \bZ\bv_1 = \frac{M_1}{\LZ_{\sharp}^{2}},\\
& \text{ and } (\bWtil \bX\bu_1)^{\top} \bP_{\bu_1,\bv_1} \bW \bZ\bv_1 = \frac{M_1}{L_\sharp \LZ_\sharp}.
\end{align*}
Consequently, Eq.~\eqref{eq-leave-one-out} becomes
\begin{align}\label{eq1-proof-lemma-eq-alpha-iota-component}
	\hspace{-0.5cm}
	\left[\begin{array}{cc}
		\frac{M_1}{L_\sharp^2} + \lambda & \frac{M_1}{L_\sharp\LZ_\sharp} \\
		\frac{M_1}{L_\sharp\LZ_\sharp} & \frac{M_1}{\LZ_\sharp^2} + \lambda
	\end{array}\right] \left[\begin{array}{c} \theta(\bu_1) \\ \thetatil(\bv_1) \end{array}\right] = \left[\begin{array}{c} \lambda L_\sharp \\ \lambda \LZ_\sharp \end{array}\right] + \frac{1}{m} \cdot \left[\begin{array}{c}
		(\widetilde{\bW}\bX \bu_1)^{\top}\\
		(\bW\bZ \bv_1)^{\top}
	\end{array}\right] \bP_{\bu_1,\bv_1} (\by + \diag(\bW\widetilde{\bW})).
\end{align}
Continuing, recall the definition of $\bu_2,\bv_2$ in Eq.~\eqref{definition-u1-u2-v1-v2} and note that
\begin{align}\label{Eq-decomposition-of-by}
	\by = \bX\bcoefX_{\star} \odot \bZ\bcoefZ_{\star} + \bepsilon = \Big( \frac{\parcompX_{\sharp}}{L_\sharp} \bX\bu_1 + \frac{\perpcompX_\sharp}{L_\sharp} \bX\bu_2\Big) \odot \Big( \frac{\parcompZ_{\sharp}}{\LZ_\sharp} \bZ\bv_1 + \frac{\perpcompZ_\sharp}{\LZ_\sharp} \bZ\bv_2\Big)  + \bepsilon.
\end{align}
Consequently, using $\diag(\bX\bu_1) = \bW/L_\sharp$ and $\diag(\bZ\bv_1) = \bWtil/\LZ_\sharp$, we deduce that
\begin{align*}
	\frac{1}{m}(\widetilde{\bW}\bX \bu_1)^{\top} \bP_{\bu_1,\bv_1} (\by + \diag(\bW\widetilde{\bW})) &= \frac{1}{L_\sharp}\Big(1 + \frac{\parcompX_{\sharp}\parcompZ_\sharp}{L_\sharp^2\LZ_\sharp^2} \Big) M_1 + \frac{1}{L_\sharp}M_2 \quad \text{and}\\
	\frac{1}{m}(\bW \bZ \bv_1)^{\top} \bP_{\bu_1,\bv_1} (\by + \diag(\bW\widetilde{\bW})) &= \frac{1}{\LZ_\sharp}\Big(1 + \frac{\parcompX_{\sharp}\parcompZ_\sharp}{L_\sharp^2\LZ_\sharp^2} \Big) M_1 + \frac{1}{\LZ_\sharp}M_2.
\end{align*}
 Substituting the equation in the previous display into Eq.~\eqref{eq1-proof-lemma-eq-alpha-iota-component} yields
\[
\left[\begin{array}{cc}
	\frac{M_1}{L_\sharp^2} + \lambda & \frac{M_1}{L_\sharp\LZ_\sharp} \\
	\frac{M_1}{L_\sharp\LZ_\sharp} & \frac{M_1}{\LZ_\sharp^2} + \lambda
\end{array}\right] \left[\begin{array}{c} \theta(\bu_1) \\ \thetatil(\bv_1) \end{array}\right] = \left[\begin{array}{c} \lambda L_\sharp + \frac{1}{L_\sharp}\Big( 1+ \frac{\parcompX_{\sharp}\parcompZ_\sharp}{L_\sharp^2\LZ_\sharp^2} \Big) M_1 + \frac{1}{L_\sharp}M_2 \\ \lambda \LZ_\sharp + \frac{1}{\LZ_\sharp}\Big(1 + \frac{\parcompX_{\sharp}\parcompZ_\sharp}{L_\sharp^2\LZ_\sharp^2} \Big) M_1 + \frac{1}{\LZ_\sharp}M_2 \end{array}\right].
\]
Note that the equation in the display above admits a unique solution as $M_1\geq 0$. Solving the equation yields
\begin{align*}
	\theta(\bu_1) &= \frac{(\lambda+\frac{M_1}{\LZ_\sharp^2})\Big(\lambda L_\sharp + \frac{1}{L_\sharp}\Big( 1+ \frac{\parcompX_{\sharp}\parcompZ_\sharp}{L_\sharp^2\LZ_\sharp^2} \Big) M_1 + \frac{1}{L_\sharp}M_2 \Big) - \frac{M_1}{L_\sharp\LZ_\sharp} \Big(\lambda \LZ_\sharp + \frac{1}{\LZ_\sharp}\Big(1 + \frac{\parcompX_{\sharp}\parcompZ_\sharp}{L_\sharp^2\LZ_\sharp^2} \Big) M_1 + \frac{1}{\LZ_\sharp}M_2 \Big) }{(\lambda+\frac{M_1}{L_\sharp^2})(\lambda+\frac{M_1}{\LZ_\sharp^2}) - \frac{M_1^2}{L_\sharp^2 \LZ_\sharp^2}}
	\\ &= \frac{\lambda L_\sharp + \frac{1}{L_\sharp}\Big( 1+ \frac{\parcompX_{\sharp}\parcompZ_\sharp}{L_\sharp^2\LZ_\sharp^2} \Big) M_1 + \frac{1}{L_\sharp}M_2 + \frac{M_1L_\sharp}{\LZ_\sharp^2} - \frac{M_1}{L_\sharp} }{\lambda + M_1(L_\sharp^{-2} + \LZ_\sharp^{-2})}
	\\&= L_\sharp + \frac{\frac{1}{L_\sharp}\Big(  \frac{\parcompX_{\sharp}\parcompZ_\sharp}{L_\sharp^2\LZ_\sharp^2} -1 \Big) M_1 + \frac{1}{L_\sharp}M_2}{\lambda + M_1(L_\sharp^{-2} + \LZ_\sharp^{-2})} = L_\sharp + \frac{L_\sharp\Big(  \frac{\parcompX_{\sharp}\parcompZ_\sharp}{L_\sharp^2} -\LZ_\sharp^2 \Big) M_1 + L_\sharp \LZ_\sharp^2 M_2}{\lambda L_\sharp^2 \LZ_\sharp^2 + M_1(L_\sharp^{2} + \LZ_\sharp^{2})}.
\end{align*}
This proves the equation for $\theta(\bu_1)$. Proceeding similarly proves the equation for $\widetilde{\theta}(\bv_1)$. 

\paragraph{Proof of the characterizations~\eqref{eq:theta-u2-char}.}
By definition of $(\bu_2,\bv_2)$~\eqref{definition-u1-u2-v1-v2}, $\bcoefX_{\sharp}^{\top} \bu_2 = \bcoefZ_{\sharp}^{\top} \bv_2 = 0$. Now, setting $\bu = \bu_2$ and $\bv = \bv_2$
in Eq.~\eqref{eq-leave-one-out} yields
\begin{align*}
	\left[\begin{array}{cc}
		M_{11} + \lambda & M_{12} \\
		M_{12} & M_{22} + \lambda
	\end{array}\right] \left[\begin{array}{c} \theta(\bu_2) \\ \thetatil(\bv_2) \end{array}\right] = \frac{1}{m} \cdot \left[\begin{array}{c}
		(\widetilde{\bW}\bX \bu_2)^{\top}\\
		(\bW\bZ \bv_2)^{\top}
	\end{array}\right] \bP_{\bu_2,\bv_2} (\by + \diag(\bW\widetilde{\bW})) - \left[\begin{array}{c} F_1 \\ F_2 \end{array}\right].
\end{align*}
Then, using the decomposition of $\by$~\eqref{Eq-decomposition-of-by}, we obtain
\begin{align*}
	&\frac{1}{m} (\widetilde{\bW}\bX \bu_2)^{\top} \bP_{\bu_2,\bv_2} (\by + \diag(\bW\widetilde{\bW}))
	= \frac{\parcompZ_\sharp\perpcompX_\sharp}{\LZ_\sharp^2 L_\sharp}  \frac{1}{m} (\widetilde{\bW}\bX \bu_2)^{\top} \bP_{\bu_2,\bv_2} (\widetilde{\bW}\bX \bu_2) +
	\\ &\frac{1}{m} (\widetilde{\bW}\bX \bu_2)^{\top} \bP_{\bu_2,\bv_2} \Big( (1+\frac{\parcompX_\sharp\parcompZ_\sharp}{L_\sharp^2 \LZ_\sharp^2}) \diag(\bW\widetilde{\bW}) + \frac{\parcompX_\sharp \perpcompZ_\sharp}{L_\sharp^2 \LZ_\sharp}\bW \bZ \bv_{2} + \frac{\perpcompX_\sharp\perpcompZ_\sharp}{L_\sharp\LZ_\sharp} \bX\bu_2 \odot \bZ\bv_2+ \bepsilon \Big) 
	= \frac{\parcompZ_\sharp\perpcompX_\sharp}{\LZ_\sharp^2 L_\sharp} M_{11} + E_1,
\end{align*}
where the last step follows by definition of $M_{11}$ and $E_{1}$. Similarly, we obtain
\[
\frac{1}{m} (\bW\bZ \bv_2)^{\top} \bP_{\bu_2,\bv_2} (\by + \diag(\bW\widetilde{\bW})) = \frac{\parcompX_\sharp\perpcompZ_\sharp}{L_\sharp^2 \LZ_\sharp} M_{22} + E_2.
\]
Putting the three pieces together yields
\begin{align*}
	\left[\begin{array}{cc}
		M_{11} + \lambda & M_{12} \\
		M_{12} & M_{22} + \lambda
	\end{array}\right] \left[\begin{array}{c} \theta(\bu_2) \\ \thetatil(\bv_2) \end{array}\right] = \left[\begin{array}{c} \frac{\parcompZ_\sharp\perpcompX_\sharp}{\LZ_\sharp^2 L_\sharp} M_{11} + E_1 - F_1 \\ \frac{\parcompX_\sharp\perpcompZ_\sharp}{L_\sharp^2 \LZ_\sharp} M_{22} + E_2 - F_2 \end{array}\right].
\end{align*}
Note that the equation in the display above admits a unique solution since $(M_{11} + \lambda)(M_{22} + \lambda) - M_{12}^{2} > M_{11}M_{22} - M_{12}^{2} \geq 0$. Solving the $2\times 2$ equation in the display yields the desired result.

\paragraph{Proof of the relation~\eqref{claim-alpha-iota-concentration}.} Starting with $\EE\{\theta(\bu_{1})\}$, it suffices to show that
\begin{align*}
	\EE\bigg\{ \frac{\widetilde{L}_{\sharp}\perpcompX_{\sharp}\perpcompZ_{\sharp} M_{3} }{\lambda L_{\sharp}^{2} \LZ_{\sharp}^{2} + M_{1}(L_{\sharp}^{2} + \LZ_{\sharp}^{2})}\bigg\} = \EE\bigg\{ \frac{L_{\sharp}\LZ_{\sharp}^{2} M_{2} }{\lambda L_{\sharp}^{2} \LZ_{\sharp}^{2} + M_{1}(L_{\sharp}^{2} + \LZ_{\sharp}^{2})}\bigg\},
\end{align*}
which upon expanding in terms of $M_2$ and $M_3$, further reduces to proving that $\EE\{R\} = 0$, where 
\begin{align*}
	R :=  \frac{ \frac{L_{\sharp} \LZ_{\sharp}^{2}}{m} \cdot \diag(\bW\bWtil)^{\top}\bP_{\bu_1,\bv_1}\Big( \frac{\parcompX_\sharp\perpcompZ_\sharp}{L_\sharp^{2}\LZ_{\sharp}} 
			\bW\bZ\bv_2 + \frac{\parcompZ_\sharp \perpcompX_\sharp}{\LZ_{\sharp}^2L_\sharp}\bWtil\bX\bu_2+ \boldsymbol{\epsilon} \Big) }{\lambda L_{\sharp}^{2} \LZ_{\sharp}^{2} + M_{1}(L_{\sharp}^{2} + \LZ_{\sharp}^{2})}.
\end{align*}
By our construction of $(\bu_1,\bv_1)$~\eqref{definition-u1-u2-v1-v2} and using the Gaussianity of $\bX,\bZ$, we deduce that the tuples of random variables $(\bW,\bWtil,\boldsymbol{\epsilon})$ and $(\bX\bO_{\bu_{1}},\bZ\bO_{\bv_{1}})$ are independent of each other. Moreover, note that conditionally on $(\bX\bO_{\bu_{1}},\bZ\bO_{\bv_{1}})$, $M_{1}$ is an even function of $(\bW,\bWtil,\boldsymbol{\epsilon})$, whence the denominator of $R$ is an even function of $(\bW,\bWtil,\boldsymbol{\epsilon})$. We also note that conditionally on $(\bX\bO_{\bu_{1}},\bZ\bO_{\bv_{1}})$, the numerator of $R$ is an odd function of $(\bW,\bWtil,\boldsymbol{\epsilon})$. Consequently, we obtain that $R \;|\;(\bX\bO_{\bu_{1}},\bZ\bO_{\bv_{1}})$ is an odd function of $(\bW,\bWtil,\boldsymbol{\epsilon})$. Since the distribution of $(\bW,\bWtil,\boldsymbol{\epsilon})$ is symmetric around zero, we obtain that
\[
\EE\{R\} = \EE \big\{ \EE\big\{ R \;|\; (\bX\bO_{\bu_{1}},\bZ\bO_{\bv_{1}}) \big\} \big\} = 0.
\]
In order to establish $\EE\{\theta(\bu_{2})\}$, it suffices to show that
\begin{align}\label{eq-proof-claim-alpha-iota-concentration}
	\EE\Bigg\{ \frac{E_{1}-F_{1} - \frac{M_{12}(\frac{\parcompX_{\sharp}}{L_{\sharp}^{2}} \frac{\perpcompZ_{\sharp}}{\LZ_{\sharp}} \cdot M_{22} + E_{2} - F_{2})}{M_{22}+\lambda}}{M_{11}+\lambda - \frac{M_{12}^{2}}{M_{22} + \lambda}} \Bigg\} = 
	\EE\Bigg\{ \frac{ - \frac{\parcompZ_{\sharp}\perpcompX_{\sharp} }{\LZ_{\sharp}^{2}L_{\sharp}} \cdot  \frac{M_{12}^{2} }{M_{22}+\lambda}}{M_{11} + \lambda - \frac{M_{12}^{2}}{M_{22}+\lambda} } \Bigg\}.
\end{align}
By definition of $E_2$ and $M_{12}$, $E_2 = R' +  \frac{\parcompZ_{\sharp}}{\LZ_{\sharp}^{2}} \frac{\perpcompX_{\sharp}}{L_{\sharp}}\cdot M_{12}$, where
\[
R' = \frac{1}{m} (\bW\bZ\bv_2)^{\top} \bP_{\bu_2,\bv_2} \bigg( \big(1+ \frac{\parcompX_{\sharp}}{L_{\sharp}^{2}} \cdot \frac{\parcompZ_{\sharp}}{\LZ_{\sharp}^{2}} \big) \cdot \diag(\bW\bWtil) + \frac{\perpcompX_{\sharp} \perpcompZ_{\sharp}}{L_{\sharp}\LZ_{\sharp}} \cdot \bX\bu_2 \odot \bZ\bv_2 + \boldsymbol{\epsilon} \bigg).
\]
Substituting this expression for $E_2$ into Eq.~\eqref{eq-proof-claim-alpha-iota-concentration} implies the equivalence of Eq.~\eqref{eq-proof-claim-alpha-iota-concentration} and
\[
\EE\Bigg\{ \frac{E_{1}-F_{1} - \frac{M_{12}\big(\frac{\parcompX_{\sharp}}{L_{\sharp}^{2}} \frac{\perpcompZ_{\sharp}}{\LZ_{\sharp}} \cdot M_{22} + R' - F_{2} \big)}{M_{22}+\lambda}}{M_{11}+\lambda - \frac{M_{12}^{2}}{M_{22} + \lambda}} \Bigg\} = 0.
\]
By our construction, $\langle \bcoefX_{\sharp}, \bu_2\rangle = \langle \bcoefZ_{\sharp},\bv_2\rangle = 0$, whence since $\bX$ and $\bZ$ are Gaussian, the tuples of random variables $(\bW,\bWtil,\bP_{\bu_{2},\bv_{2}},\bA_{\bu_{2},\bv_{2}})$, $(\bX\bu_{2})$ and $(\bZ\bv_{2})$ are independent of each other. Thus, conditionally on $(\bW,\bWtil,\bP_{\bu_{2},\bv_{2}},\bA_{\bu_{2},\bv_{2}})$, $E_{1}$ is a summation of several random variables each of which is either an odd function of $\bX\bu_{2}$ or an odd function of $\bZ\bv_{2}$, and $F_{1}$ is an odd function of $\bX\bu_{2}$. Further note that, conditionally on the tuple $(\bW,\bWtil,\bP_{\bu_{2},\bv_{2}},\bA_{\bu_{2},\bv_{2}})$, $M_{11},M_{22},M_{12}^{2}$ are even functions of the inputs $\bX\bu_{2}$ and $\bZ\bv_{2}$. Thus, conditioning on $(\bW,\bWtil,\bP_{\bu_{2},\bv_{2}},\bA_{\bu_{2},\bv_{2}})$ and taking expectation over $\bX\bu_{2}$ and $\bZ\bv_{2}$, we obtain that
\[
\EE\Bigg\{ \frac{E_{1}-F_{1}}{M_{11}+\lambda - \frac{M_{12}^{2}}{M_{22} + \lambda}} \Bigg\} = \EE\Bigg\{  \EE\Bigg\{ \frac{E_{1}-F_{1}}{M_{11}+\lambda - \frac{M_{12}^{2}}{M_{22} + \lambda}} \;\bigg|\; (\bW,\bWtil,\bP_{\bu_{2},\bv_{2}},\bA_{\bu_{2},\bv_{2}}) \Bigg\} \Bigg\} = 0.
\]
Similarly, condition on $(\bW,\bWtil,\bP_{\bu_{2},\bv_{2}},\bA_{\bu_{2},\bv_{2}})$, $M_{12}\cdot M_{22}$, $M_{12}\cdot R'$ and $M_{12}\cdot F_{2}$ are odd functions of input $(\bX\bu_{2},\bZ\bv_{2})$. Consequently, we obtain that
\[
\EE\Bigg\{ \frac{\frac{M_{12}(\frac{\parcompX_{t}}{L_{t}^{2}} \cdot \frac{\perpcompZ_{t}}{\widetilde{L}_{t}} \cdot M_{22} + R' - F_{2})}{M_{22}+\lambda}}{M_{11}+\lambda - \frac{M_{12}^{2}}{M_{22} + \lambda}}	 \;\bigg|\; (\bW,\bWtil,\bP_{\bu_{2},\bv_{2}}, \bA_{\bu_{2},\bv_{2}}) \Bigg\} = 0.
\]
Putting the pieces together yields the desired result. \qed

\subsection{Orthogonal component: Proof of Theorem~\ref{thm:one-step-prediction}(b)}\label{sec:concentration-eta-component}
We will focus on bounding $\big|(\perpcompX_{+})^{2} - (\perpcompdetX_{+})^{2}\big|$ since the proof for bounding $\big| (\perpcompZ_{+})^{2} - (\perpcompdetZ_{+})^{2} \big|$ is identical.  We begin by decomposing the perpendicular component into an intermediate component---contained in $\mathsf{span}\{\bcoefX_{\star}, \bcoefX_{\sharp}\}$---and a fully orthogonal component, setting
\begin{align} \label{eq:eta-iota-def}
	\iotaX_{+} &= \big\langle \bcoefX_{+}, \frac{\bP_{\bcoefX_{\star}}^{\perp} \bcoefX_{\sharp} }{\big\|\bP_{\bcoefX_{\star}}^{\perp} \bcoefX_{\sharp}\big\|_{2}} \big\rangle, \qquad &\etaX_{+}^{2} = \big \| \bP_{\mathsf{span}\{\bcoefX_{\star}, \bcoefX_{\sharp}\}}^{\perp} \bcoefX_{+} \big\|_{2}^{2}\nonumber\\
	\iotaZ_{+} &= \langle \bcoefZ_{+}, \frac{\bP_{\bcoefZ_{\star}}^{\perp} \bcoefZ_{\sharp} }{\big\|\bP_{\bcoefZ_{\star}}^{\perp} \bcoefZ_{\sharp}\big\|_{2}} \rangle, \qquad &\etaZ_{+}^{2} = \big \| \bP_{\mathsf{span}\{\bcoefZ_{\star},\bcoefZ_{\sharp}\}}^{\perp} \bcoefZ_{+} \big\|_{2}^{2}.
\end{align}
We then use the functions $\iotaXmap_{m,d,\sigma,\lambda},\; \iotaZmap_{m,d,\sigma,\lambda},\; \etaXmap_{m,d,\sigma,\lambda},\; \etaZmap_{m,d,\sigma,\lambda}$ in Eq.~\eqref{et_updates_eq} and Eq.~\eqref{eta_updates_eq} to define the corresponding deterministic predictions as
\begin{align}
	\iotadetX_{+} &= \iotaXmap_{m,d,\sigma,\lambda}\big( \parcompX_\sharp,\perpcompX_\sharp,\parcompZ_\sharp,\perpcompZ_\sharp \big),  &(\etadetX_{+})^{2} = \etaXmap_{m,d,\sigma,\lambda}(\parcompX_\sharp,\perpcompX_\sharp,\parcompZ_\sharp,\perpcompZ_\sharp) \nonumber\\
	\iotadetZ_{+} &= \iotaZmap_{m,d,\sigma,\lambda}\big( \parcompX_\sharp,\perpcompX_\sharp,\parcompZ_\sharp,\perpcompZ_\sharp \big),  \qquad &(\etadetZ_{+})^{2} = \etaZmap_{m,d,\sigma,\lambda}(\parcompX_\sharp,\perpcompX_\sharp,\parcompZ_\sharp,\perpcompZ_\sharp), \label{eta-det-predictions-in-proof}
\end{align}
Thus, by the definition of $\perpcompX_{+}$~\eqref{eq:generic-state-evolution} and $\perpcompdetX_{+}$~\eqref{alpha-beta-det-prediction}, we see that both $(\perpcompdetX_{+})^{2} = (\iotadetX_{+})^{2} + (\etadetX_{+})^{2}$ as well as $\perpcompX_{+}^{2} = \|\bP_{\bcoefX_{\star}}^{\perp} \bcoefX_{+} \|_{2}^{2} =  \iotaX_{+}^{2} + \etaX_{+}^{2}$.  Applying the triangle inequality yields
\begin{align}\label{beta-decompose-iota-eta}
	\big| \perpcompX_{+}^{2} - (\perpcompdetX_{+})^{2} \big| \leq \big|  \iotaX_{+}^{2} - (\iotadetX_{+})^{2}  \big|
	+ \big| \etaX_{+}^{2} - (\etadetX_{+})^{2} \big|.
\end{align}
We claim the following bounds on the error of the intermediate prediction and the error of the fully orthogonal prediction, respectively.
\begin{subequations}
	\begin{align}
		\big|  \iotaX_{+}^{2} - (\iotadetX_{+})^{2}  \big|  \lesssim \max\bigg\{ \frac{(\Err_{\sharp} + \sigma^{2})\log^{6}(d)}{\lambda \sqrt{m}}, d^{-30}\bigg\}, \qquad \text{ with probability } \geq 1 - d^{-25}, \label{ineq:intermediate-orthogonal}\\
		\big| \etaX_{+}^{2} - (\etadetX_{+})^{2} \big| \lesssim \max\bigg\{ \frac{(\Err_{\sharp} + \sigma^{2})\log^{6}(d)}{\lambda \sqrt{m}}, d^{-30}\bigg\}, \qquad \text{ with probability } \geq 1 - d^{-22}. \label{ineq:fully-orthogonal}
	\end{align}
\end{subequations}
The result follows upon combining the previous three inequalities.  It remains to prove inequalities~\eqref{ineq:intermediate-orthogonal} and~\eqref{ineq:fully-orthogonal}.

\paragraph{Bounding the error of the intermediate prediction: Proof of inequality~\eqref{ineq:intermediate-orthogonal}.}
The proof of this component largely follows the strategy of Section~\ref{sec:concentration-proof}.  Recall the pair of unit vectors $(\bu_1,\bu_2)$ in Eq.~\eqref{definition-u1-u2-v1-v2} and note that $\frac{\bP_{\bcoefX_{\star}}^{\perp} \bcoefX_{\sharp}}{\big\| \bP_{\bcoefX_{\star}}^{\perp} \bcoefX_{\sharp} \big\|_{2}} = \frac{\perpcompX_{\sharp}}{L_{\sharp}} \cdot \bu_1 - \frac{\parcompX_{\sharp}}{L_{\sharp}} \cdot \bu_2$.  Consequently, $\iotaX_{+} =  \frac{\perpcompX_{\sharp}}{L_{\sharp}} \cdot \theta(\bu_1) - \frac{\parcompX_{\sharp}}{L_{\sharp}} \cdot \theta(\bu_2)$.  Then, using the definition of $\thetaXdet_{1}$ and $\thetaXdet_{2}$ in Eq.~\eqref{definition-theta12-det}, write the corresponding deterministic prediction as $\iotadetX_{+} = \frac{\perpcompX_{\sharp}}{L_{\sharp}} \cdot \thetaXdet_{1} - \frac{\parcompX_{\sharp}}{L_{\sharp}} \cdot \thetaXdet_{2}$.  Thus, by the triangle inequality,
\[
\big| \iotaX_{+} - \iotadetX_{+} \big| \leq \frac{\perpcompX_{\sharp}}{L_{\sharp}} \cdot \big| \theta(\bu_1) - \thetaXdet_{1} \big| + \frac{\parcompX_{\sharp}}{L_{\sharp}} \cdot \big| \theta(\bu_2) - \thetaXdet_{2} \big|.
\] 
Noting that $\parcompX_{\sharp}/L_{\sharp}, \perpcompX_{\sharp}/L_{\sharp} \leq 1$ and applying the pair of inequalities~\eqref{ineq:bias-stochastic-error-terms-parallel} yields 
\begin{align} \label{ineq:iota-intermediate-deviation}
|\iotaX_{+} - \iotadetX_{+}| \lesssim \max\Big\{ \frac{(\sqrt{\Err_{\sharp}}+\sigma)\log^{6}(d)}{\lambda\sqrt{m}},d^{-30} \Big\}, \qquad \text{ with probability } \geq 1-d^{-25}.
\end{align}
The inequality~\eqref{ineq:intermediate-orthogonal} then follows upon bounding $|\iotaX_{+} + \iotadetX_{+}|$.  To this end, we note  that by the definitions of $\iotadetX_{+}$ and $\iotaXmap_{m,d,\sigma,\lambda}$~\eqref{et_updates_eq}, we obtain 
\begin{align*}
	|\iotadetX_{+}| = \big|\iotaXmap_{m,d,\sigma,\lambda}(\parcompX_\sharp,\perpcompX_\sharp,\parcompZ_\sharp,\perpcompZ_\sharp) \big| \leq \frac{V\big(\frac{\parcompX_\sharp \parcompZ_\sharp}{L_\sharp^2} + L_\sharp^2 \big) + \lambda L_\sharp^2\LZ_\sharp^2}{V(L_\sharp^2+\LZ_\sharp^2) + \lambda L_\sharp^2\LZ_\sharp^2} \cdot \perpcompX_\sharp  +  \frac{\parcompX_\sharp \parcompZ_\sharp}{L_\sharp^2\LZ_\sharp^2} \frac{V_1}{V_1 + \lambda} \cdot \perpcompX_\sharp \lesssim \perpcompX_\sharp,
\end{align*}
where in the last step we use $L_\sharp,\LZ_\sharp \asymp 1$, $V_{1}\geq 0$ and $V \lesssim 1$~\eqref{eq:V-V1-V2}.  We complete the proof upon bounding $\iota_{+}$, for which we consider two cases.

\medskip
\noindent\underline{Case 1: $(\sqrt{\Err_{\sharp}}+\sigma)\log^{6}(d)/(\lambda\sqrt{m}) \geq d^{-30}$.}  Working on the event in which the deviation bound~\eqref{ineq:iota-intermediate-deviation} holds, we apply the triangle inequality to obtain
\[
|\iotaX_{+}| \leq |\iotadetX_{+}| + |\iotaX_{+} - \iotadetX_{+}| \lesssim \perpcompX_{\sharp} + \frac{(\sqrt{\Err_{\sharp}}+\sigma)\log^{6}(d)}{\lambda\sqrt{m}} \lesssim \sqrt{\Err_\sharp} + \sigma,
\]
where in the last step we used the bounds $\perpcompX_{\sharp} \leq \sqrt{\Err_{\sharp}}$ as well as $\log^{6}(d)/(\lambda \sqrt{m}) \lesssim 1$, where the latter holds by assumption, since $m\leq d$ and $\lambda m\geq d$.  Putting the pieces together yields
\begin{align*}
\big|  \iotaX_{+}^{2} - (\iotadetX_{+})^{2}  \big| \leq (|\iotaX_{+}| + |\iotadetX_{+}|) \cdot |\iotaX_{+} - \iotadetX_{+}| &\lesssim \frac{(\sqrt{\Err_{\sharp}}+\sigma)^{2}\log^{6}(d)}{\lambda \sqrt{m}}\\
& \lesssim \max\Big\{ \frac{(\Err_{\sharp} + \sigma^{2})\log^{6}(d)}{\lambda \sqrt{m}}, d^{-30}\Big\},
\end{align*}
which concludes the first case. 

\medskip
\noindent \underline{Case 2: $(\sqrt{\Err_{\sharp}}+\sigma)\log^{6}(d)/(\lambda\sqrt{m}) \leq d^{-30}$.} In this case, $|\iotaX_{+}| \lesssim \perpcompX_{\sharp} + d^{-30}$.  Thus, 
\begin{align*}
\big|  \iotaX_{+}^{2} - (\iotadetX_{+})^{2}  \big| \leq (|\iotaX_{+}| + |\iotadetX_{+}|) \cdot |\iotaX_{+} - \iotadetX_{+}| \lesssim (\perpcompX_\sharp + d^{-30})d^{-30} &\lesssim d^{-30}\\
& \lesssim \max\Big\{ \frac{(\Err_{\sharp} + \sigma^{2})\log^{6}(d)}{\lambda \sqrt{m}}, d^{-30}\Big\},
\end{align*}
which concludes the second case.   Combining the two cases yields the desired result. 

\paragraph{Bounding the error of the fully orthogonal prediction: Proof of inequality~\eqref{ineq:fully-orthogonal}.}
As in the proof of Theorem~\ref{thm:one-step-prediction}(a), we decompose the error into the sum of a bias term and a fluctuation term and control each separately:
\begin{align}
	\big| \etaX_{+}^{2} - (\etadetX_{+})^{2} \big| \leq 
	\underbrace{ \big| \etaX_{+}^{2} - \EE\{\etaX_{+}^{2}\} \big|}_{\text{Stochastic error}} + 
	\underbrace{ \big| \EE\{\etaX_{+}^{2}\} - (\etadetX_{+})^{2} \big| }_{\text{Bias}}.
\end{align}
The following pair of lemmas controls each of the two errors, respectively.

\begin{lemma}[Stochastic error of the fully orthogonal component]\label{lem:stochastic-error-orthogonal-component}
	Consider the assumptions of Theorem~\ref{thm:one-step-prediction} and let $\etaX_{+}$ and $\etaX_{+}^{\det}$ be as in~\eqref{eq:eta-iota-def} and~\eqref{eta-det-predictions-in-proof}.  There exists a positive constant $C_3$, depending only on $K_1, K_2$ such that with probability at least $1 - d^{-22}$, 
	\begin{align}
		\label{ineq:stochastic-error-perpendicular-component}
		\big| \etaX_{+}^{2} - \EE\{\etaX_{+}^{2}\} \big| \;\vee\; \big| \etaZ_{+}^{2} - \EE\{\etaZ_{+}^{2}\} \big| \leq  C_3 \max\bigg\{ \frac{\log^{5.5}(d)\big(\Err_{\sharp} + \sigma^{2} \big)}{\lambda \sqrt{m}} , d^{-30}\bigg\}.
	\end{align}
\end{lemma}
The proof of Lemma~\ref{lem:stochastic-error-orthogonal-component} follows the strategy applied in Section~\ref{sec:stochastic-error-parallel-component}, so we defer it to Section~\ref{sec:proof-stochastic-error-orthogonal-component}.  

\begin{lemma}[Bias of the fully orthogonal component] \label{lem:bias-orthogonal-component}
		Consider the assumptions of Theorem~\ref{thm:one-step-prediction} and let $\etaX_{+}$ and $\etaX_{+}^{\det}$ be as in~\eqref{eq:eta-iota-def} and~\eqref{eta-det-predictions-in-proof}.  There exists a positive constant $C_4$, depending only on $K_1, K_2$ such that
		\begin{align}
			\label{ineq:bias-perpendicular-component}
			\big| \EE\{\etaX_{+}^{2}\} - (\etadetX_{+})^{2} \big|\; \vee \; \big| \EE\{\etaZ_{+}^{2}\} - (\etadetZ_{+})^{2} \big|  \leq  C_4 \max\bigg\{ \frac{\log^{6}(d)\big(\Err_{\sharp} + \sigma^{2} \big)}{\lambda \sqrt{m}}, d^{-30} \bigg\}.
		\end{align}
\end{lemma}
We provide the proof of Lemma~\ref{lem:bias-orthogonal-component} in Section~\ref{sec:proof-bias-orthogonal-component}.  Note that the desired bound follows immediately upon applying Lemmas~\ref{lem:stochastic-error-orthogonal-component} and~\ref{lem:bias-orthogonal-component}.  \qed

\subsubsection{Proof of Lemma~\ref{lem:bias-orthogonal-component}}\label{sec:proof-bias-orthogonal-component}
Let $\{\bu_{i}\}_{i=3}^{d}$ and $\{\bv_{i}\}_{i=3}^{d}$ be an orthonormal basis of the complementary subspaces orthogonal to $\mathsf{span}\{\bcoefX_{\star},\bcoefX_{\sharp}\}$ and $\mathsf{span}\{\bcoefZ_{\star},\bcoefZ_{\sharp}\}$ respectively. Recall $G_i = \bx_i^{\top}\bcoefX_{\sharp},\widetilde{G}_i = \bz_i^{\top}\bcoefZ_{\sharp}$ and $\bO_{\bu_i}\in \mathbb{R}^{d \times (d-1)}$ whose columns are $\{\bu_j\}_{j\neq i}$ (similar define $\bO_{\bv_i}$). For each $3\leq i \leq d$, let 
\begin{align}\label{defintion-estimator-u3-v3}
	\left[\begin{array}{c}  \bcoefX_{\bu_i,\bv_i} \\ \bcoefZ_{\bu_i,\bv_i} \\ \end{array}\right] &= \argmin_{\bcoefX,\bcoefZ \in \mathbb{R}^{d-1}} \sum_{k=1}^{m}\big(y_{k}+G_{k}\widetilde{G}_{k} - \widetilde{G}_{k} \cdot \bx_{i}^{\top} \bO_{\bu_i} \bcoefX - G_{k} \cdot \bz_{k}^{\top} \bO_{\bv_i} \bcoefZ \big)^{2} + \lambda m \cdot \bigg\|
	\begin{bmatrix} \bcoefX - \bO_{\bu_i}^{\top}\bcoefX_{\sharp} \\ \bcoefZ - \bO_{\bv_i}^{\top}\bcoefZ_{\sharp} \end{bmatrix} \bigg\|_{2}^{2} \nonumber
	\\&= \Big( \bA_{\bu_i,\bv_i}^{\top}\bA_{\bu_i,\bv_i} + \lambda m \bI \Big)^{-1} \Big( \bA_{\bu_i,\bv_i}^{\top}\big(\by + \diag(\bW \widetilde{\bW})\big) + \lambda m \left[\begin{array}{c} \bO_{\bu_i}^{\top}\bcoefX_{\sharp}  \\ \bO_{\bv_i}^{\top}\bcoefZ_{\sharp} \\ \end{array}\right] \Big).
\end{align}
Applying Lemma~\ref{lemma-leave-one-out} yields the closed-form expression for $3\leq i\leq d$,
\begin{align*}
	\left[\begin{array}{cc} M_{11}+\lambda & M_{12} \\ M_{12} & M_{22}+\lambda \\ \end{array}\right] 
	\left[\begin{array}{c}  \langle \bcoefX_{+},\; \bu_i \rangle \\ \langle \bcoefZ_{+},\; \bv_i \rangle \\ \end{array}\right]
	= \frac{1}{m} \left[\begin{array}{c}  (\widetilde{\bW} \bX\bu_i)^{\top} \bR_{\bu_i,\bv_i} \\ (\bW \bZ\bv_i)^{\top} \bR_{\bu_i,\bv_i}\\ \end{array}\right],\quad \text{where}
\end{align*}
\begin{align}\label{definition-Ruv-M11-etc}
	&\bR_{\bu_i,\bv_i} = \by + \diag(\bW\widetilde{\bW}) - \bA_{\bu_i,\bv_i} \left[\begin{array}{c}  \bcoefX_{\bu_i,\bv_i} \\ \bcoefZ_{\bu_i,\bv_i} \\ \end{array}\right],\;M_{11} = \frac{(\widetilde{\bW} \bX\bu_i)^{\top} \bP_{\bu_i,\bv_i} (\widetilde{\bW} \bX\bu_i)}{m},\nonumber\\
	&M_{12} = \frac{(\widetilde{\bW} \bX\bu_i)^{\top} \bP_{\bu_i,\bv_i} (\bW \bZ\bv_i)}{m} \quad \text{and} \quad
	M_{22} = \frac{(\bW \bZ\bv_i)^{\top} \bP_{\bu_i,\bv_i} (\bW \bZ\bv_i)}{m}.
\end{align}
Solving the equation in the display above yields
\begin{align}\label{eq-closed-form-u3-component}	
	\langle \bcoefX_{+}, \bu_i \rangle = \frac{ \frac{1}{m} (\widetilde{\bW}\bX \bu_i )^{\top} \bR_{\bu_i,\bv_i} - \frac{M_{12}}{M_{22}+\lambda} \frac{1}{m} (\bW\bZ\bv_i)^{\top} \bR_{\bu_i,\bv_i} }{ M_{11}+\lambda - \frac{M_{12}^{2}}{M_{22}+\lambda} }.
\end{align}
Note that since $\{\langle \bcoefX_{+},\; \bu_{i} \rangle\}_{i=3}^{d}$ are identically distributed, we obtain
\begin{align}\label{eq:eta-expectation-(d-2)-sum}
	\EE\{ \etaX_{+}^{2} \} = \EE\Big\{ \sum_{k=3}^{d} \langle \bcoefX_{+},\; \bu_{k} \rangle^{2} \Big\} = (d-2) \cdot \EE\Big\{ \langle \bcoefX_{+},\; \bu_{3} \rangle^{2} \Big\}.
\end{align}
For the remainder of the proof, we omit the index and let $\bu = \bu_3$ and $\bv = \bv_3$.  We require the following sequence of lemmas to characterize $\EE\{ \langle \bcoefX_{+},\; \bu \rangle^{2}\}$.  The first, whose proof we provide in Section~\ref{sec:proof-ineq-step1-eta-component-expectation}, relates it to the expectation of $\|\widetilde{\bW} \bR_{\bu,\bv}\|_{2}^{2}$.

\begin{lemma}\label{lem:ineq-step1-eta-component-expectation}
	Consider the assumptions of Theorem~\ref{thm:one-step-prediction}.  Let $\bu = \bu_3$ denote a unit norm vector orthogonal to $\bcoefX_{\star}$ and $\bcoefX_{\sharp}$, $\bv$ denote a unit norm vector orthogonal to $\bcoefZ_{\star}$ and $\bcoefZ_{\sharp}$, and let $\bR_{\bu, \bv}$ be as in~\eqref{definition-Ruv-M11-etc}.  There exists a positive constant $C_5$, depending only on $K_1$ and $K_2$, such that the following holds. 
	\begin{align}\label{ineq-step1-eta-component-expectation}
		\Big| \EE\{\langle \bcoefX_{+}, \bu \rangle^{2}\} - \frac{\EE\big\{ \|\widetilde{\bW} \bR_{\bu,\bv}\|_{2}^{2} \big\}}{m^{2}(\lambda + V_{1})^{2}} \Big| \leq C_5 \cdot \Bigl[\frac{\sigma^{2} + \Err_{\sharp}}{\lambda m} \cdot  \frac{\log^{3}(d)}{\lambda \sqrt{m}}\Bigr].
	\end{align}
\end{lemma}
Expanding, we see that $\EE\big\{\|\widetilde{\bW} \bR_{\bu,\bv}\|_{2}^{2} \big\} = m\EE\{ \widetilde{G}_i^2 R_{i}^2\}$, where
\[
R_i =  y_{i}+G_{i}\widetilde{G}_{i} - \widetilde{G}_{i} \bx_{i}^{\top} \bO_{\bu} \bcoefX_{\bu,\bv} - G_{i} \bz_{i}^{\top} \bO_{\bv} \bcoefZ_{\bu,\bv}.
\]
Note that the estimator $(\bcoefX_{\bu,\bv},\bcoefZ_{\bu,\bv})$ depends on the data $(\bx_{i},\bz_{i},y_{i})$, which precludes direct computation of $\EE\{ \widetilde{G}_i^2 R_{i}^2\}$.  We thus consider the following estimator obtained by leaving out the data $\bx_{i},\bz_{i},y_{i}$
\begin{align}\label{defintion-estimator-u3-v3-i-sample-out}
	\left[\begin{array}{c} \bcoefX_{\bu,\bv}^{-(i)}\\ \bcoefZ_{\bu,\bv}^{-(i)}\\
	\end{array}\right] = \argmin_{\bcoefX,\bcoefZ \in \mathbb{R}^{d-1}} &\sum_{j\neq i}\big(y_{j}+G_{j}\widetilde{G}_{j} - \widetilde{G}_{j}\bx_{j}^{\top} \bO_{\bu} \bcoefX - G_{j}\bz_{j}^{\top} \bO_{\bv} \bcoefZ \big)^{2} + \lambda m \cdot \bigg\| \begin{bmatrix} \bcoefX - \bO_{\bu}^{\top}\bcoefX_{\sharp} \\
		\bcoefZ - \bO_{\bv}^{\top}\bcoefZ_{\sharp} \end{bmatrix} \bigg\|_{2}^{2}.
\end{align}
We emphasize that the estimator in the previous display is obtained by leaving both a direction $(\bu, \bv)$ out as well as a sample $(\bx_i, \bz_i, y_i)$ out.  
Let $\ba_{i}^{\top} = [\widetilde{G}_{i}\bx_{i}^{\top} \bO_{\bu} \; G_{i}\bz_{i}^{\top}\bO_{\bv}]$ and $\bSig_{i} = \sum_{j\neq i}\ba_{j}\ba_{j}^{\top} + \lambda m \bI$ for each $i\in [m]$. We claim that
\begin{align}\label{claim-step2-eta-component}
	R_{i} = \frac{\tau_{i}}{1 + \ba_{i}^{\top} \bSig_{i}^{-1} \ba_{i}}, \qquad \text{ where } \qquad \tau_{i} = y_{i}+G_{i}\widetilde{G}_{i} - \widetilde{G}_{i}\bx_{i}^{\top}\bO_{\bu} \bcoefX_{\bu,\bv}^{-(i)} -G_{i}\bz_{i}^{\top}\bO_{\bv} \bcoefZ_{\bu,\bv}^{-(i)},
\end{align}
deferring its proof to the end.
The next lemma relates the expectation of $\|\widetilde{\bW} \bR_{\bu,\bv}\|_{2}^{2}$ to the scalars $r_1$ and $r_2$, which are solutions to the fixed point equations in~\eqref{eq:fixed-point} and the quantity $\tau_i$ defined in the previous display.  We provide its proof in Section~\ref{sec:proof-ineq3-step2-eta-component}.
\begin{lemma}\label{lem:ineq3-step2-eta-component}
	Consider the assumptions of Lemma~\ref{lem:ineq-step1-eta-component-expectation}.  There exists a positive constant $C_6 > 0$, depending only on $K_1$ and $K_2$, such that
	\begin{align}\label{ineq3-step2-eta-component}
		\bigg| \frac{\EE\big\{ \big\| \widetilde{\bW}\bR_{\bu,\bv} \big\|_{2}^{2} \big\}}{m^{2}(\lambda + V_{1})^{2} } - \frac{1}{m(\lambda + V_1)^2} \EE\bigg\{ \frac{\widetilde{G}_{i}^{2} \tau_{i}^{2}}{\big(1+\widetilde{G}_{i}^{2}r_{1}^{-1} + G_{i}^{2}r_{2}^{-1} \big)^{2}} \bigg\}\bigg| \leq C_6 \cdot \big(  \sigma^{2} + \Err_{\sharp} \big) \frac{\log^{1.5}(d)}{\sqrt{d} m\lambda^{2}}.
	\end{align}
\end{lemma}

The final lemma removes the dependence on $\tau_i$~\eqref{claim-step2-eta-component} from the previous lemma.  We provide its proof in Section~\ref{sec:proof-ineq-step3-eta-component}.
\begin{lemma}\label{lem:ineq-step3-eta-component}
	Consider the assumptions of Lemma~\ref{lem:bias-orthogonal-component} and let $V_3$ be as in~\eqref{definition-V3}.  Let
	\[
	\psi(r_1, r_2, V_3) = V_{3}r_{1}^{2} + \EE\{\etaX_{+}^{2}\} \cdot  \EE\bigg\{ \frac{ r_{1}^{2}r_{2}^{2} \widetilde{G}_{i}^{4} }{(r_{1}r_{2} + r_{1}G_{i}^{2} + r_{2}\widetilde{G}_{i}^{2})^{2}}\bigg\} + \EE\{\etaZ_{+}^{2}\} \cdot \EE\bigg\{ \frac{r_{1}^{2}r_{2}^{2} G_{i}^{2}\widetilde{G}_{i}^{2} }{(r_{1}r_{2} + r_{1}G_{i}^{2} + r_{2}\widetilde{G}_{i}^{2})^{2}}\bigg\}.
	\]
	Then, there exists a positive constant $C_7$, which depends only on $K_1$ and $K_2$, such that 
\begin{align}\label{ineq-step3-eta-component}
	\bigg| \EE\bigg\{ \frac{\widetilde{G}_{i}^{2} \tau_{i}^{2}}{\big(1+\widetilde{G}_{i}^{2}r_{1}^{-1} + G_{i}^{2}r_{2}^{-1} \big)^{2}} \bigg\} - \psi(r_1, r_2, V_3) \bigg| \leq C_7 \cdot \max\bigg\{ \frac{C( \sigma^{2} + \Err_{\sharp})}{\lambda} \frac{\log^{6}(d)}{\sqrt{m}}, d^{-50} \bigg\}.
\end{align} 
\end{lemma}
\noindent Applying the triangle inequality and Lemmas~\ref{lem:ineq-step1-eta-component-expectation},~\ref{lem:ineq3-step2-eta-component}, and~\ref{lem:ineq-step3-eta-component} in sequence yields the inequality
\[
\bigg| \EE \{ \langle \bcoefX_{+}, \bu \rangle^{2} \} - \frac{1}{m(\lambda + V_{1})^{2}} \psi(r_1, r_2, V_3) \bigg| \lesssim
\frac{1}{\lambda m} \max\bigg\{ \frac{( \sigma^{2} + \Err_{\sharp})}{\lambda} \frac{\log^{6}(d)}{\sqrt{m}}, d^{-50} \bigg\},
\]
where $\psi(r_1, r_2, V_3)$ is as in Lemma~\ref{lem:ineq-step3-eta-component}.  Now, applying the relation~\eqref{eq:eta-expectation-(d-2)-sum} in conjunction with the relation $\lambda + V_{1} = r_{1}d/m$, which holds by definition~\eqref{eq:fixed-point}, we deduce the inequality
\begin{align*}
	\bigg| \EE\{\eta_{+}^{2}\} -  \frac{(d-2)m}{d^{2}} \frac{\psi(r_1, r_2, V_3)}{r_1^2} \bigg| \lesssim \max\bigg\{ \frac{( \sigma^{2} + \Err_{\sharp})}{\lambda} \frac{\log^{6}(d)}{\sqrt{m}}, d^{-50} \bigg\}.
\end{align*}
Arguing in a parallel fashion establishes the analogous bound on $\EE\{\widetilde{\eta}_+^2\}$.  Re-arranging then yields the pair of upper bounds
\begin{subequations}\label{eq:expectation-eta-fixed-point-equation}
	\begin{align}
	\Big| c_1 \cdot \EE\{\etaX_{+}^{2}\} - c_2 \cdot \EE\{\etaZ_{+}^{2}\} - \frac{(d-2)m}{d^{2}} \cdot V_3 \Big| &\leq \Delta \\
	 \Big| c_3 \cdot \EE\{\etaZ_{+}^{2}\} - c_4 \cdot  \EE\{\etaX_{+}^{2}\} - \frac{(d-2)m}{d^{2}} \cdot V_4 \Big| &\leq \Delta,
\end{align}
\end{subequations}
where $\Delta = C_{2}\max\bigg\{ \frac{\log^{6}(d)\big(\Err_{\sharp} + \sigma^{2} \big)}{\lambda \sqrt{m}}, d^{-50} \bigg\}$ and
\begin{align*}
	&c_{1} = 1 - \frac{(d-2)m}{d^{2}} \EE\bigg\{\frac{r_2^2G_2^4}{(r_1r_2 + r_1G_1^2 + r_2G_2^2)^2}\bigg\},\;
	c_{2} = \frac{(d-2)m}{d^{2}} \EE\bigg\{\frac{r_2^2G_1^2G_2^2}{(r_1r_2 + r_1G_1^2 + r_2G_2^2)^2}\bigg\},\\
	&c_{3} = 1 - \frac{(d-2)m}{d^{2}} \EE\bigg\{\frac{r_1^2G_1^4}{(r_1r_2 + r_1G_1^2 + r_2G_2^2)^2}\bigg\},\;
	c_{4} = \frac{(d-2)m}{d^{2}} \EE\bigg\{\frac{r_1^2G_1^2G_2^2}{(r_1r_2 + r_1G_1^2 + r_2G_2^2)^2}\bigg\}.
\end{align*}
Moreover, by definitions of $\etadetX_{+}$ and $\etadetZ_{+}$ in~\eqref{fixed-point-eq-eta-updates} and~\eqref{eta-det-predictions-in-proof}, we have the following relations
\[
c_1 \cdot (\etadetX_{+})^2 - c_{2}\cdot (\etadetZ_{+})^{2} = \frac{(d-2)m}{d^{2}} \cdot V_3,\qquad \text{ and } \qquad 
c_3 \cdot (\etadetZ_{+})^2 - c_{4}\cdot (\etadetX_{+})^{2} = \frac{(d-2)m}{d^{2}} \cdot V_4.
\] 
Combing the relations in the previous display with the pair of inequalities~\eqref{eq:expectation-eta-fixed-point-equation} yields
\begin{align*}
	\Big| c_{1} \EE\{\etaX_{+}^{2}\} - c_{2}\EE\{\etaZ_{+}^{2}\} - \Big( c_{1} (\etadetX_{+})^{2} - c_{2}(\etadetZ_{+})^{2} \Big) \Big| &\leq \Delta\\
	\Big| c_{3} \EE\{\etaZ_{+}^{2}\} - c_{4}\EE\{\etaX_{+}^{2}\} - \Big( c_{3} (\etadetZ_{+})^{2} - c_{4}(\etadetX_{+})^{2} \Big) \Big| &\leq \Delta.
\end{align*}
Then, upon re-arranging terms and applying the triangle inequality, we find that
\begin{align*}
	&|c_{1}| \cdot  \Big| \EE\{\etaX_{+}^{2}\} - (\etadetX_{+})^{2} \Big| \leq |c_{2}| \cdot \Big| \EE\{\etaZ_{+}^{2}\} - (\etadetZ_{+})^{2} \Big| + \Delta \qquad \text{ and }\\
	&|c_{3}| \cdot  \Big| \EE\{\etaZ_{+}^{2}\} - (\etadetZ_{+})^{2} \Big| \leq |c_{4}| \cdot \Big| \EE\{\etaX_{+}^{2}\} - (\etadetX_{+})^{2} \Big| + \Delta.
\end{align*}
Combining the two inequalities in the previous display yields
\[
(|c_{1}| - |c_{2}| |c_{4}|/|c_{3}|) \cdot \Big| \EE\{\etaX_{t+1}^{2}\} - (\etadetX_{t+1})^{2} \Big| \leq (1+|c_{2}|/|c_{3}|) \cdot \Delta.
\]
Note that
\[
c_{1} \geq 1 - \frac{(d-2)m}{d^{2}} \frac{\EE\{G_{2}^{4}\}}{r_{1}^{2}} \overset{\1}{\geq} 1 - \frac{(d-2)m}{d^{2}} \frac{\EE\{G_{2}^{4}\}}{(\lambda m/d)^{2}} \overset{\2}{\geq} 0.5,
\]
where step $\1$ follows from $r_{1} \geq \lambda m/d$ (see item 1 in Lemma~\ref{fixed-point-equations-unique-solution}), and step $\2$ follows as $\lambda m \geq Cd$ for some constant $C$ large enough and $(d-2)m\EE\{G_{2}^{4}\}/d^{2} \lesssim 1$. Similarly, we see that $c_{3}\geq 0.5$ and $|c_{2}|,|c_{4}| \leq 0.1$. Putting together the pieces yields
\[
\Big| \EE\{\etaX_{+}^{2}\} - (\etadetX_{+})^{2} \Big| \leq 10 \Delta \lesssim \max\bigg\{ \frac{\log^{6}(d)\big(\Err_{\sharp} + \sigma^{2} \big)}{\lambda \sqrt{m}}, d^{-50} \bigg\},
\]
as desired.  It remains to establish the claim~\eqref{claim-step2-eta-component}.

\paragraph{Proof of the claim~\eqref{claim-step2-eta-component}.} Let $\ba_{j}^{\top} = [\widetilde{G}_{j} \bx_{j}^{\top}\bO_{\bu} \; G_{j}\bz_{j}^{\top}\bO_{\bv} ]$ for each $j\in[n]$. KKT condition of $(\bcoefX_{\bu,\bv},\bcoefZ_{\bu,\bv})$~\eqref{defintion-estimator-u3-v3} and $(\bcoefX_{\bu,\bv}^{-(i)}, \bcoefZ_{\bu,\bv}^{-(i)})$~\eqref{defintion-estimator-u3-v3-i-sample-out} yields
\begin{subequations}
	\begin{align}
		\sum_{j\in [n]}\Big(y_{j} + G_{j}\widetilde{G}_{j} - \ba_{j}^{\top} \left[\begin{array}{c} \bcoefX_{\bu,\bv}\\ \bcoefZ_{\bu,\bv} \\ \end{array}\right] \Big) \ba_{j} &= \lambda m \left[\begin{array}{c} \bcoefX_{\bu,\bv} - \bO_{\bu}^{\top} \bcoefX_{\sharp} \\ \bcoefZ_{\bu,\bv} - \bO_{\bv}^{\top} \bcoefZ_{\sharp} \\ \end{array}\right] \label{KKT-condition-estimator-u3-v3}
		\\ \sum_{j \neq i} \Big(y_{j} + G_{j}\widetilde{G}_{j} - \ba_{j}^{\top} \left[\begin{array}{c} \bcoefX_{\bu,\bv}^{-(i)}\\ \bcoefZ_{\bu,\bv}^{-(i)} \\ \end{array}\right] \Big) \ba_{j} &= \lambda m \left[\begin{array}{c} \bcoefX_{\bu,\bv}^{-(i)} - \bO_{\bu}^{\top} \bcoefX_{\sharp} \\ \bcoefZ_{\bu,\bv}^{-(i)} - \bO_{\bv}^{\top} \bcoefZ_{\sharp} \\ \end{array}\right].\label{KKT-condition-estimator-u3-v3-isample-out}
	\end{align}
\end{subequations}
Subtracting one inequality by the other inequality in the display above yields
\[
\left[\begin{array}{c} \bcoefX_{\bu,\bv} \\ \bcoefZ_{\bu,\bv} \\ \end{array}\right] - 
\left[\begin{array}{c} \bcoefX_{\bu,\bv}^{-(i)} \\ \bcoefZ_{\bu,\bv}^{-(i)} \\ \end{array}\right] = \Big( \sum_{j \neq i} \ba_{j} \ba_{j}^{\top} + \lambda m \bI \Big)^{-1} R_{i}\ba_{i} = \bSig_{i}^{-1}R_{i}\ba_{i}.
\]
Note that, by definition of $R_{i}$ and $\tau_{i}$,
$
\tau_{i} - R_{i} = \ba_{i}^{\top} \bigg( \left[\begin{array}{c} \bcoefX_{\bu,\bv} \\ \bcoefZ_{\bu,\bv} \\ \end{array}\right] - 
\left[\begin{array}{c} \bcoefX_{\bu,\bv}^{-(i)} \\ \bcoefZ_{\bu,\bv}^{-(i)} \\ \end{array}\right] \bigg).
$
Putting the two pieces together yields $\tau_{i} - R_{i} = \ba_{i}^{\top} \bSig_{i}^{-1}\ba_{i} \cdot R_{i}
$. Arranging the terms yields the desired result. \qed
\section{Proof of Theorem~\ref{thm:local-sharp-convergence-results}: Convergence guarantees}\label{sec:main-proof-convergence-result}
Our convergence result reposes on the following lemma, which provides a two-sided, one-step improvement bound.  We provide its proof---which relies on properties of the deterministic predictions---in Section~\ref{proof-lemma-one-step-empirical-convergence}.
\begin{lemma}\label{lemma:one-step-empirical-convergence} Consider $(\bcoefX_{\sharp}, \bcoefZ_{\sharp}) \in \mathbb{R}^d$ and let $\Err_{\sharp}, L_{\sharp}$, and $\widetilde{L}_{\sharp}$ be as in~\eqref{eq-definition-error} and~\eqref{eq:shorthand-length-L-Ltil}.  There exists a tuple of universal, positive constants $(c,c_{1},c_{2},C,C_{1},C_{2})$ such that if
	\[
	\Err_{t} \leq c,\qquad \lambda m \geq C(1+\sigma),\qquad 0.2 \leq L_{t}, \LZ_{t} \leq 3 \qquad \text{ and } \qquad m\geq C\log^{12}(d),
	\]
	then with probability at least $1 - d^{-18}$, both of the following hold.
	\begin{subequations} 
		\begin{align}\label{sharp-bounds-empirical-error}
			\hspace{-1cm}
			\Big( 1- \frac{c_{2}}{\lambda}\Big)\cdot \Err_{t} + \frac{C_{2}\sigma^{2}d}{\lambda^{2}m} - \frac{C_{1}\polylog\sigma^{2}}{\lambda \sqrt{m}} &\leq \Err_{t+1} \nonumber\\
			&\leq \Big( 1- \frac{c_{1}}{\lambda}\Big)\cdot \Err_{t} + C_{1} \Big( \frac{\sigma^{2}d}{\lambda^{2}m} + \frac{\polylog\sigma^{2}}{\lambda\sqrt{m}} \Big),
		\end{align}
		and
		\begin{align}
			|\parcompX_{t+1} - \parcompX_{t}|,\;|\parcompZ_{t+1} - \parcompZ_{t}| &\leq \frac{C_{1}\sqrt{\Err_{t}} }{\lambda} + \frac{C_{1}\polylog \sigma}{\lambda\sqrt{m}} \label{bound-alpha-t+1-t-epsilon}.
		\end{align}
\end{subequations}
\end{lemma}
\noindent Equipped with this lemma, we prove Theorem~\ref{thm:local-sharp-convergence-results} by induction on $T$. Let $T_{\star} = \frac{\lambda}{c_{1}} \log \Big(\min\Big\{ \frac{\lambda m}{\sigma^{2}d}, \frac{\sqrt{m}}{\polylog \sigma^{2}} \Big\}\Big)$. 
We first state our induction hypothesis for $0 \leq T \leq T_{\star}$.
\paragraph{Induction Hypothesis:} With probability at least $1-(T+1)d^{-18}$, the iterates $(\bcoefX_{t}, \bcoefZ_{t})_{t=0}^{T}$ defined in~\eqref{eq:prox-linear-updates} satisfy the sandwich relation~\eqref{sharp-bounds-empirical-error} as well as the bound~\eqref{bound-alpha-t+1-t-epsilon}.  Further, $\Err_{t} \leq c$, $0.3 \leq L_t ,\LZ_{t} \leq 1.7$ for all $0 \leq t\leq T$, where $c$ is the same constant as in Lemma~\ref{lemma:one-step-empirical-convergence}.

\paragraph{Base case: $T=0$.} Note $\Err_{0} \leq K_{0} \leq c$ as long as $K_{0} \leq c$ and $0.5 \leq L_{0},\LZ_{0} \leq 1.5$ by assumption. Then applying Lemma~\ref{lemma:one-step-empirical-convergence} we obtain that with probability at least $1-d^{-18}$, both inequalities~\eqref{sharp-bounds-empirical-error} and inequality~\eqref{bound-alpha-t+1-t-epsilon} hold, which concludes the base case.

\paragraph{Induction step:} Suppose the induction hypothesis holds for some $T \leq T_{\star}-1$.  That is,  inequalities~\eqref{sharp-bounds-empirical-error} and~\eqref{bound-alpha-t+1-t-epsilon} holds for all $0\leq t\leq T$ with probability at least $1-(T+1)d^{-18}$. We need to show the induction hypothesis holds at iterate $T+1$.   Applying inequality~\eqref{sharp-bounds-empirical-error} recursively for all $0\leq t\leq T$ yields
\begin{align}\label{ineq1-convergence-proof-error}
	\Err_{t+1} \leq \rho_{1}^{t+1} \Err_{0} + \sum_{k=0}^{t} \rho_{1}^{k} \cdot C_{1} \Big( \frac{\sigma^{2}d}{\lambda^{2}m} + \frac{\polylog\sigma^{2}}{\lambda\sqrt{m}} \Big) \leq \rho_{1}^{t+1} \Err_{0} + \frac{C_{1}}{c_{1}} \Big(\frac{\sigma^{2}d}{\lambda m} + \frac{\polylog\sigma^{2}}{\sqrt{m}}  \Big),
\end{align}
 where $\rho_{1} = 1-\frac{c_{1}}{\lambda}$ and in the last step we used the bound $\sum_{k=0}^{t} \rho_{1}^{k} \leq \lambda/c_{1}$.
 Consequently, 
\begin{align}\label{ineq2-convergence-proof-error}
	\Err_{T+1} \leq \Err_{0} + \frac{C_{1}}{c_{1}} \Big(\frac{\sigma^{2}d}{\lambda m} + \frac{\polylog\sigma^{2}}{\sqrt{m}}  \Big) \leq c,
\end{align}
where the final step follows since $\Err_{0} \leq K_{0}$, $\lambda m \geq C(1+\sigma^{2})d$ and $m\geq C(1+\sigma^{4})\log^{12}(d)$ and by taking $K_{0}$ small enough and $C$ large enough. Continuing, applying inequality~\eqref{bound-alpha-t+1-t-epsilon} recursively for all $0\leq t\leq T$ yields
\begin{align*}
	|\parcompX_{T+1} - \parcompX_{0}| \leq \sum_{t=0}^{T} |\parcompX_{t+1} - \parcompX_{t}| \leq  \frac{C_{1}}{\lambda} \sum_{t=0}^{T} \sqrt{\Err_{t}} + \frac{C_{1}(T+1)\polylog \sigma}{\lambda\sqrt{m}}.
\end{align*}
Then, using inequality~\eqref{ineq1-convergence-proof-error} to upper boound $\Err_{t}$ yields
\begin{align*}
	|\parcompX_{T+1} - \parcompX_{0}| &\leq \frac{C_{1}}{\lambda} \sum_{t=0}^{T} \rho_{1}^{t/2}\sqrt{K_{0}} + \frac{C_{1}}{\lambda} \frac{\sqrt{C_{1}}}{\sqrt{c_{1}}} \Big(\frac{\sigma^{2}d}{\lambda m} + \frac{\polylog\sigma^{2}}{\sqrt{m}} \Big)^{1/2}(T+1) + \frac{C_{1}(T+1)\polylog \sigma}{\lambda\sqrt{m}} \\&\leq
	\frac{2C_{1}}{c_{1}} \sqrt{K_{0}} + \frac{C_{1}^{1.5}}{\lambda \sqrt{ c_{1}}} \Big( \frac{\sigma^{2}d}{\lambda m}  + \frac{\polylog\sigma^{2}}{\sqrt{m}} \Big)^{1/2}T_{\star} + \frac{C_{1}T_{\star}\polylog \sigma}{\lambda\sqrt{m}},
\end{align*}
where in the last step we use $T+1 \leq T_{\star}$ and $\sum_{t=0}^{T} \rho_{1}^{t/2} \leq \sum_{t=0}^{T} \Big(1-\frac{c_{1}}{2\lambda}\Big)^{t} \leq \frac{2\lambda}{c_{1}}$.  Thus, using the numeric inequality $\sqrt{a+b} \leq \sqrt{a} +\sqrt{b}$ for all $a,b>0$ and substituting the definition of $T_{\star}$, we obtain the inequality
\begin{align*}
	\frac{C_{1}^{1.5}}{\lambda \sqrt{ c_{1}}} \Big( \frac{\sigma^{2}d}{\lambda m}  + \frac{\polylog\sigma^{2}}{\sqrt{m}} \Big)^{1/2}T_{\star}  \leq  \frac{C_{1}^{1.5}}{  c_{1}^{1.5}} \bigg( \sqrt{\frac{\sigma^{2}d}{\lambda m}} \log\Big( \frac{\lambda m}{\sigma^{2}d} \Big) + \sqrt{\frac{\polylog \sigma^{2}}{\sqrt{m}}} \log\Big( \frac{\sqrt{m}}{\sigma^{2} \polylog} \Big)  \bigg).
\end{align*}
Continuing, since $\lambda m \geq C(1+\sigma^2)d$ and $m\geq C(1+\sigma^{4})\log^{12}(d)$ then
\[
	\frac{\sigma^{2}d}{\lambda m} \leq \frac{1}{C},\qquad  \frac{\polylog \sigma^{2}}{\sqrt{m}} \leq \frac{1}{\sqrt{C}}, \qquad \text{ and } \qquad \sqrt{x} \log(1/x) \rightarrow 0 \text{ as } x\rightarrow 0.
\]
Thus, if we let $C$ be a large enough constant, we obtain the estimate
\[
	\sqrt{\frac{\sigma^{2}d}{\lambda m}} \log\Big( \frac{\lambda m}{\sigma^{2}d} \Big) + \sqrt{\frac{\polylog \sigma^{2}}{\sqrt{m}}} \log\Big( \frac{\sqrt{m}}{\sigma^{2} \polylog} \Big)   \leq 0.01\frac{c_{1}^{1.5}}{  C_{1}^{1.5}}.
\]
Similarly, 
\[
	\frac{C_{1}T_{\star}\polylog \sigma}{\lambda\sqrt{m}} \leq \frac{C_{1}}{c_{1}} \frac{\polylog\sigma}{\sqrt{m}} \log \Big( \frac{\sqrt{m}}{\polylog\sigma^{2}} \Big) \leq 0.01.
\]
Putting all the pieces together yields
\[
	|\parcompX_{T+1} - \parcompX_{0}| \leq \frac{2C_{1}}{c_{1}} \sqrt{K_{0}} + 0.02 \leq 0.03,
\]
where the last step we take $K_{0}$ small enough. By assumption, we have $\parcompX_{0} \leq L_{0} \leq 1.5$ and $ \perpcompX_{0} \leq \Err_{0} = K_{0} \leq 0.01$ so that $\parcompX_{0} \geq L_{0} - \perpcompX_{0} \geq 0.5-0.01 \geq 0.4$. Combining this bound with the inequality in the display above, we deduce that $\parcompX_{T+1} \leq 1.5 + 0.03 \leq 1.6$ and $\parcompX_{T+1} \geq 0.4 - 0.03 \geq 0.3$. Further, by inequality~\eqref{ineq2-convergence-proof-error}, we obtain that $\perpcompX_{T+1} \leq \sqrt{\Err_{T+1}} \leq \sqrt{c} \leq 0.1$. Putting the pieces together yields $L_{T+1} \geq \parcompX_{T+1} \geq 0.3$ and $L_{T+1} \leq \parcompX_{T+1} + \perpcompX_{T+1} \leq 1.7$.  Proceeding similarly yields $0.3 \leq \LZ_{T+1} \leq 1.7$.  Summarizing, we have shown that at iteration $T+1$, both of the following hold
\[
	\Err_{T+1} \leq c \qquad \text{ and } \qquad 0.3 \leq L_{T+1}, \LZ_{T+1} \leq 1.7.
\]
Then applying Lemma~\ref{lemma:one-step-empirical-convergence}, we obtain that inequality~\eqref{sharp-bounds-empirical-error} and inequality~\eqref{bound-alpha-t+1-t-epsilon} holds for $t = T+1$ with probability at least $1-d^{-18}$. Combining with the induction hypothesis, we obtain that for $0 \leq t\leq T+1$, inequality~\eqref{sharp-bounds-empirical-error} and inequality~\eqref{bound-alpha-t+1-t-epsilon} holds with probability at least $1-(T+2)d^{-18}$, which establishes the inductive step. 

Theorem~\ref{thm:local-sharp-convergence-results} directly follows from inequality~\eqref{sharp-bounds-empirical-error}. This concludes the proof. \qed

\medskip
\noindent We next justify several consequences of Theorem~\ref{thm:local-sharp-convergence-results}, i.e., inequalities~\eqref{ineq-local-sharp-convergence-standard-regime},~\eqref{interation-complexity-standard-regime} and~\eqref{interation-complexity-all-lambda-regime}. 
\paragraph{Proof of inequality~\eqref{ineq-local-sharp-convergence-standard-regime}:} Note that by setting $\lambda = C(1+\sigma^2)d/m$ and $m \gg \log^{12}(d)(1+\sigma^{4})$, we obtain that
\[
	\frac{\log^{6}(d)\sigma^{2}}{\lambda\sqrt{m} } \ll \frac{\sigma^{2}}{\lambda(1+\sigma^{2})} = \frac{C\sigma^{2}d}{\lambda^{2}m}.
\]
It means that in inequality~\eqref{ineq-local-sharp-convergence}, the term $\frac{\log^{6}(d)\sigma^{2}}{\lambda\sqrt{m} }$ is a small order term compared with the term $\frac{\sigma^{2}d}{\lambda^{2}m}$. Consequently, the desired result follows immediately from inequality~\eqref{ineq-local-sharp-convergence} by letting in $\lambda  = C(1+\sigma^{2})d/m$. \qed

\paragraph{Proof of inequality~\eqref{interation-complexity-standard-regime}:} To reduce the notational burden, we use the shorthand $\rho_1 = 1 - \frac{c_1 m}{C(1+\sigma^{2})d}$, $ \rho_2 = 1 - \frac{c_2 m}{C(1+\sigma^{2})d}$, and $T_{\star} = \frac{\lambda}{c_1} \log\Big( \min\Big\{\frac{\lambda m}{\sigma^{2}d}, \frac{\sqrt{m}}{\sigma^{2}\log^{6}(d)} \Big\} \Big) = \frac{\lambda}{c_1} \log\Big(\frac{C_{0}(1+\sigma^{2})}{\sigma^{2}}\Big)$.
Applying inequality~\eqref{ineq-local-sharp-convergence-standard-regime} recursively yields, for all $1 \leq t\leq T_{\star}$,
\begin{align*}
	 \rho_{2}^{t} \cdot \Err_{0} + \sum_{k=0}^{t-1} \rho_{2}^{k} \cdot \frac{C_2\sigma^{2}}{2\lambda^2 m} \leq \Err_{t} \leq \rho_{1}^{t} \cdot \Err_{0} + \sum_{k=0}^{t-1} \rho_{1}^{k} \cdot \frac{2C_1\sigma^{2}}{\lambda^2 m}.
\end{align*}
Consequently, for $t \geq \frac{\lambda}{c_{1}}\log(\Err_{0}/\sigma^{2}) = \frac{C(1+\sigma^{2})d}{c_1m}\log(\Err_{0}/\sigma^{2})$, we obtain
\begin{align}\label{iteration-complexity-upper-bound-standard-regime}
	\Err_{t} \leq \exp\Big\{ -\frac{c_1 m}{C_0(1+\sigma^{2}d)}  \cdot t \Big\} \cdot \Err_{0} + \frac{2C_1\sigma^{2}}{c_1 \lambda m} \lesssim \sigma^{2}
\end{align}
where in the first step, we used the bounds
\begin{align*}
\rho_{1}^{t} &= \Big( 1 - \frac{c_1 m}{C(1+\sigma^{2})d} \Big)^{t} \leq \exp\Big\{ -\frac{c_1 m}{C(1+\sigma^{2})d}  \cdot t \Big\}, \quad \text{ and } \quad \sum_{k = 0}^{t-1} \rho_{1}^{k} \leq \sum_{k = 0}^{+\infty} \rho_{1}^{k} = \frac{C(1+\sigma^{2})d}{c_1 m} = \frac{\lambda}{c_1},
\end{align*}
and in the second step, we use the setting $\lambda = C(1+\sigma^2)d/m$. Continuing, we obtain that for $t \leq 0.5 \cdot \frac{\lambda}{c_{2}}\log(\Err_{0}/\sigma^{2}) = \frac{0.5C(1+\sigma^{2})d}{c_2m}\log(\Err_{0}/\sigma^{2})$,
\begin{align}\label{iteration-complexity-lower-bound-standard-regime}
	\Err_{t} \geq \rho_{2}^{t} \cdot \Err_{0} + \sum_{k=0}^{t-1} \rho_{2}^{k} \cdot \frac{C_2\sigma^{2}}{2\lambda^2 m} \geq \rho_{2}^{t} \cdot \Err_{0} \overset{\1}{\geq} \exp\Big\{ - 2 \frac{c_2 m}{C(1+\sigma^{2})d} \cdot t \Big\} \geq \sigma^{2},
\end{align}
where in step $\1$ we used the numerical inequality $1-x \geq \exp\{-2x\}$ for $0 \leq x \leq 1/2$ so that
\[
	\rho_{2}^{t} = \Big( 1 -  \frac{c_2 m}{C(1+\sigma^{2})d} \Big)^{t} \geq \exp\Big\{ -  2\frac{c_2 m}{C(1+\sigma^{2})d} \cdot t \Big\}.
\]
Now, putting together inequalities~\eqref{iteration-complexity-upper-bound-standard-regime},~\eqref{iteration-complexity-lower-bound-standard-regime} together, we see that it takes
\[
	\tau = \Theta\Big( \lambda \log\Big(\frac{\Err_{0}}{\sigma^{2}}\Big) \Big) = \Theta\Big( \frac{d(1+\sigma^2)}{m} \log\Big( \frac{\Err_{0}}{\sigma^{2}} \Big) \Big) \quad \text{iterations to guarantee} \quad \Err_{\tau} \lesssim \sigma^{2}.
\]
\qed

\paragraph{Proof of inequality~\eqref{interation-complexity-all-lambda-regime}:} We first prove the upper bound of the iteration complexity. Applying inequality~\eqref{ineq-local-sharp-convergence} recursively yields
\begin{align*}
	\Err_t &\leq \Big(1-\frac{c_1}{\lambda} \Big)^{t} \cdot \Err_{0} + \sum_{k=0}^{t-1} \Big(1-\frac{c_1}{\lambda} \Big)^{k} \cdot \Big( \frac{C_1 \sigma^2 d}{\lambda^2 m} + \frac{C_1\log^{6}(d)\sigma^2}{\lambda \sqrt{m}} \Big)
	\\& \lesssim e^{-\frac{c_1 t}{\lambda} } \cdot \Err_{0} + \frac{d\sigma^{2}}{\lambda m} + \frac{\log^{6}(d) \sigma^{2}}{\sqrt{m}},
\end{align*}
where in the last step we used $\sum_{k=0}^{t-1} (1-c_1/\lambda)^{k} \lesssim \lambda$. Consequently, we obtain 
\begin{align}\label{upper-bound-iteration-complexity-general}
	\Err_{t} \lesssim \frac{d\sigma^{2}}{\lambda m} + \frac{\log^{6}(d) \sigma^{2}}{\sqrt{m}} \quad \text{for } t\geq \frac{\lambda}{c_1} \log\Big( \Err_{0} \cdot \min\Big\{ \frac{\lambda m}{d\sigma^{2}}, \frac{\sqrt{m}}{ \log^{6}(d) \sigma^{2}} \Big\} \Big).
\end{align}
We then prove the lower bound of iteration complexity. Applying inequality~\eqref{ineq-local-sharp-convergence} recursively yields
\begin{align*}
	 \Err_t &\geq \Big(1-\frac{c_2}{\lambda} \Big)^{t} \cdot \Err_{0} + \sum_{k=0}^{t-1} \Big(1-\frac{c_2}{\lambda} \Big)^{k} \cdot \Big( \frac{C_2 \sigma^2 d}{\lambda^2 m} - \frac{C_1\log^{6}(d)\sigma^2}{\lambda \sqrt{m}} \Big)
	\\& \geq e^{- \frac{2c_2 t}{\lambda} } \cdot \Err_{0}  -  \frac{C_1 \log^{6}(d) \sigma^{2}}{c_2 \sqrt{m}},
\end{align*}
where in the last step we used the numerical inequality $1-x\geq \exp\{-2x\}$ for $0\leq x\leq 1/2$ so that
$\Big(1-\frac{c_2}{\lambda} \Big)^{t} \geq e^{ -\frac{2c_2 t }{\lambda}}$, and $\sum_{k=0}^{t-1} \Big(1-\frac{c_2}{\lambda} \Big)^{k} \leq \frac{\lambda}{c_2}$.
Consequently, we obtain 
\[
	\Err_{t} \geq  \frac{C_1}{2c_2} \Big( \frac{\sigma^{2}d}{\lambda m} + \frac{\sigma^{2} \log^{6}(d)}{\sqrt{m}} \Big), \quad \text{for } t \leq \frac{0.5\lambda}{c_{2}} \bigg(  \log\bigg( \Err_{0} \cdot \min\bigg\{ \frac{\lambda m}{d\sigma^{2}}, \frac{\sqrt{m}}{ \log^{6}(d) \sigma^{2}} \bigg\} \bigg) - \log\Big( \frac{2C_1}{ c_2} \Big) \bigg).
\]
Combining the inequality in the display above with inequality~\eqref{upper-bound-iteration-complexity-general} together yields the desired result. \qed

\subsection{Proof of Lemma~\ref{lemma:one-step-empirical-convergence}}\label{proof-lemma-one-step-empirical-convergence}
Our proof relies crucially on properties of the deterministic predictions, summarized in the following lemma, whose proof we provide in Section~\ref{proof-lemma-one-step-deterministic-convergence}.
\begin{lemma}\label{lemma:one-step-deterministic-convergence}
Let the functions $\alphaXmap_{m,d,\sigma,\lambda},\alphaZmap_{m,d,\sigma,\lambda},\iotaXmap_{m,d,\sigma,\lambda},\iotaZmap_{m,d,\sigma,\lambda},\etaXmap_{m,d,\sigma,\lambda}, \etaZmap_{m,d,\sigma,\lambda}:\mathbb{R}^{4} \rightarrow \mathbb{R}$ be as defined in Section~\ref{sec:expilicit-formulas-prediction}. Given $\parcompX,\parcompZ,\perpcompX,\perpcompZ \in \mathbb{R}$, let
\begin{align*}
	\parcompdetX &= \alphaXmap_{m,d,\sigma,\lambda}(\parcompX,\perpcompX,\parcompZ,\perpcompZ), 
	&&(\perpcompdetX)^{2} =  \big(\iotaXmap_{m,d,\sigma,\lambda}(\parcompX,\perpcompX,\parcompZ,\perpcompZ) \big)^{2} + \etaXmap_{m,d,\sigma,\lambda}(\parcompX,\perpcompX,\parcompZ,\perpcompZ),
	\\ \parcompdetZ &= \alphaZmap_{m,d,\sigma,\lambda}(\parcompX,\perpcompX,\parcompZ,\perpcompZ),
	&&(\perpcompdetZ)^{2} =  \big(\iotaZmap_{m,d,\sigma,\lambda}(\parcompX,\perpcompX,\parcompZ,\perpcompZ) \big)^{2} + \etaZmap_{m,d,\sigma,\lambda}(\parcompX,\perpcompX,\parcompZ,\perpcompZ),
	\\ \Err &= (\parcompX\parcompZ-1)^{2} + \perpcompX^{2} + \perpcompZ^{2},  &&\Errdet = (\parcompdetX \parcompdetZ -1)^{2} + (\perpcompdetX)^{2} + (\perpcompdetZ)^{2}.
\end{align*}
There exists a tuple of universal, positive constants $(c,c_1,c_2,C_1,C_2)$ such that if $\Err \leq c$ and $0.2 \leq (\parcompX^{2} + \perpcompX^{2})^{\frac{1}{2}}, (\parcompZ^{2} + \perpcompZ^{2})^{\frac{1}{2}} \leq 3$ then the following hold.
\begin{subequations}
\begin{align}\label{sharp-bounds-det-error}
	\Big( 1 - \frac{c_{2}}{\lambda} \Big) \cdot \Err + \frac{C_{2}\sigma^{2}d}{\lambda^{2}m} \leq \Errdet &\leq \Big( 1 - \frac{c_{1}}{\lambda} \Big) \cdot \Err + \frac{C_{1}\sigma^{2}d}{\lambda^{2}m},
	\\ (\parcompdetX\parcompdetZ-1)^{2} + (\iotadetX)^{2} + (\iotadetZ)^{2} &\leq C_{1} \Err, \qquad |\parcompdetX|,|\parcompdetZ| \leq 4
	\label{upper-bound-det-error-alpha-iota}
	\\ \text{and} \quad |\parcompdetX - \parcompX|, |\parcompdetZ - \parcompZ|  &\leq  \frac{C_{1}\sqrt{\Err}}{\lambda}. \label{alpha-change-after-det-map}
\end{align} 
\end{subequations}
\end{lemma}

\noindent We now make use of Lemma~\ref{lemma:one-step-deterministic-convergence} to prove Lemma~\ref{lemma:one-step-empirical-convergence}.  We now prove inequalities~\eqref{sharp-bounds-empirical-error} and~\eqref{bound-alpha-t+1-t-epsilon} in turn, starting with inequality~\eqref{sharp-bounds-empirical-error}.
\paragraph{Proof of inequality~\eqref{sharp-bounds-empirical-error}:} We start with bounding $(\parcompX_{t+1}\parcompZ_{t+1}-1)^{2}$. Recall the deterministic predictions $\parcompdetX_{t+1} = \alphaXmap_{m,d,\sigma,\lambda}(\parcompX_{t},\perpcompX_{t},\parcompZ_{t},\perpcompZ_{t})$ and $\parcompdetZ_{t+1} = \alphaZmap_{m,d,\sigma,\lambda}(\parcompX_{t},\perpcompX_{t},\parcompZ_{t},\perpcompZ_{t})$. Applying Theorem~\ref{thm:one-step-prediction}, we obtain that with probability at least $1-d^{-20}$
\begin{align}\label{ineq:deviation-alpha'-alphadet}
	| \parcompX_{t+1} - \parcompdetX_{t+1}|, | \parcompZ_{t+1} - \parcompdetZ_{t+1}| \leq C'\polylog \big( \sqrt{\Err_{t}} + \sigma \big) /(\lambda\sqrt{m}),
\end{align}
where $C'$ is some universal constant. Using inequality~\eqref{upper-bound-det-error-alpha-iota} in Lemma~\ref{lemma:one-step-deterministic-convergence} yields $|\parcompdetX_{t+1}|,|\parcompdetZ_{t+1}| \leq 4$. Putting the two pieces together yields the pair of upper bounds
\begin{align}\label{ineq:|alpha|-constant-bound}
	|\parcompX_{t+1}| \leq |\parcompdetX_{t+1}| + C'\polylog \big(\sqrt{\Err_{t}} +\sigma \big) /(\lambda\sqrt{m}) \overset{\1}{\leq} 5 \qquad \text{ and } \qquad |\parcompZ_{t+1}| \leq 5,
\end{align}
where step $\1$ follows since $m \geq C\log^{12}(d)$, $\lambda \geq C(1+\sigma)d/m$ and $\Err_{t} \leq c$. Then, we apply the triangle inequality to obtain
\begin{align*}
	\big| |\parcompX_{t+1}\parcompZ_{t+1}-1| - |\parcompdetX_{t+1}\parcompdetZ_{t+1}-1| \big|  
	& \leq |\parcompZ_{t+1}|\cdot |\parcompX_{t+1} - \parcompdetX_{t+1}| +  |\parcompdetX_{t+1}|\cdot |\parcompZ_{t+1} - \parcompdetZ_{t+1}|
	\\& \leq  10C'\polylog \big(\sqrt{\Err_{t}} + \sigma \big)/ (\lambda\sqrt{m}),
\end{align*}
where the last step follows upon invoking inequalities~\eqref{ineq:deviation-alpha'-alphadet} and~\eqref{ineq:|alpha|-constant-bound}. Using inequality~\eqref{upper-bound-det-error-alpha-iota} in Lemma~\ref{lemma:one-step-deterministic-convergence} yields $|\parcompdetX_{t+1}\parcompdetZ_{t+1}-1| \lesssim \sqrt{\Err_{t}}$. Consequently, we obtain 
\begin{align}\label{ineq:alpha'-Errt-bound}
	|\parcompX_{t+1}\parcompZ_{t+1}'-1| \leq |\parcompdetX_{t+1}\parcompdetZ_{t+1}-1|  + 10C'\polylog \big( \sqrt{\Err_{t}} + \sigma \big)/(\lambda\sqrt{m}) \lesssim \sqrt{\Err_{t}} + \sigma,
\end{align}
where the last inequality follows as $m \geq C\log^{12}(d)$ and $\lambda \geq C(1+\sigma)d/m \geq C$. Putting all the pieces together yields 
\begin{align}\label{ineq:empirical-det-error-alpha'}	
|(\parcompX_{t+1}&\parcompZ_{t+1}-1)^{2} - (\parcompdetX_{t+1}\parcompdetZ_{t+1}-1)^{2}| \leq \Big( |\parcompX_{t+1}\parcompZ_{t+1}-1| + |\parcompdetX_{t+1}\parcompdetZ_{t+1}-1|\Big)  \cdot |\parcompX_{t+1}\parcompZ_{t+1} - \parcompdetX_{t+1}\parcompdetZ_{t+1} | \nonumber \\ &\lesssim \big(\sqrt{\Err_{t}} + \sigma\big) \cdot \polylog \big(\sqrt{\Err_{t}}+\sigma\big)/(\lambda\sqrt{m}) \lesssim \polylog (\Err_{t} + \sigma^{2})/(\lambda\sqrt{m}).
\end{align}
We then turn to bound $(\perpcompX_{t+1})^{2} + (\perpcompZ_{t+1})^{2}$. We let 
\begin{align*}
(\perpcompdetX_{t+1})^{2} &= \big(\iotaXmap_{m,d,\sigma,\lambda}(\parcompX_t,\perpcompX_t,\parcompZ_t,\perpcompZ_t)\big)^{2} + \etaXmap_{m,d,\sigma,\lambda}(\parcompX_t,\perpcompX_t,\parcompZ_t,\perpcompZ_t)^{2},\\
(\perpcompdetZ_{t+1})^{2} &= \big(\iotaZmap_{m,d,\sigma,\lambda}(\parcompX_t,\perpcompX_t,\parcompZ_t,\perpcompZ_t)\big)^{2} + \etaZmap_{m,d,\sigma,\lambda}(\parcompX_t,\perpcompX_t,\parcompZ_t,\perpcompZ_t)^{2}.
\end{align*}
Applying Theorem~\ref{thm:one-step-prediction}, we obtain that with probability at least $1-d^{-20}$, 
\begin{align}\label{ineq:empirical-det-error-beta}
	\big| (\perpcompdetX_{t+1})^{2} + (\perpcompZ_{t+1})^{2} - \big((\perpcompdetX_{t+1})^{2} + (\perpcompdetZ_{t+1})^{2} \big) \big| \lesssim \polylog(\Err_{t} + \sigma^{2})/(\lambda \sqrt{m}).
\end{align}
Let $\Errdet_{t+1} = (\parcompdetX_{t+1}\parcompdetZ_{t+1}-1)^{2} + (\perpcompdetX_{t+1})^{2} + (\perpcompdetZ_{t+1})^{2}$. Combining the inequalities~\eqref{ineq:empirical-det-error-alpha'} and~\eqref{ineq:empirical-det-error-beta} yields
\begin{align}\label{ineq:Err-Errdet-bound}
	| \Err_{t+1} - \Errdet_{t+1} | \leq C'\polylog(\Err_{t} + \sigma^{2})/(\lambda \sqrt{m}), 
\end{align}
where $C'$ is a universal constant. Consequently, we obtain that 
\begin{align*}
	\Err_{t+1} \leq \Errdet_{t+1} + \frac{C'\polylog(\Err_{t} + \sigma^{2})}{\lambda \sqrt{m}} 
	 &\leq \Big(1-\frac{c_{1}}{\lambda} \Big) \Err_{t} + \frac{C_{1}\sigma^{2}d}{\lambda^{2}m} + \frac{C'\polylog(\Err_{t} + \sigma^{2})}{\lambda \sqrt{m}} \\&\leq \Big(1-\frac{c_{1}}{2\lambda} \Big) \Err_{t} + \frac{C_{1}\sigma^{2}d}{\lambda^{2}m} + \frac{C'\polylog \sigma^{2}}{\lambda \sqrt{m}},
\end{align*}
where the second inequality follows by invoking~\eqref{sharp-bounds-det-error} to upper bound $\Errdet_{t+1}$ and the last step follows as $m\geq C\log^{12}(d)$, so that by letting $C$ large enough,
\[
	\frac{C'\polylog \Err_{t} }{\lambda \sqrt{m}} \leq \frac{C'\Err_{t}}{\lambda \sqrt{C}} \leq \frac{c'_{1}\Err_{t}}{2\lambda}.
\]
This proves the desired upper bound of $\Err_{t+1}$ in inquality~\eqref{sharp-bounds-empirical-error} by noting $c_{1},C_{1}$ and $C'$ are universal, positive constants. Turning to the lower bound of $\Err_{t+1}$, using inequalities~\eqref{sharp-bounds-det-error},~\eqref{ineq:Err-Errdet-bound} and letting $m\geq C\log^{12}(d)$ for $C$ large enough yields
\begin{align*}
	\Err_{t+1} \geq \Big(1-\frac{2c_{2}}{\lambda} \Big) \Err_{t} + \frac{C_{2}\sigma^{2}d}{\lambda^{2}m} - \frac{C'\polylog \sigma^{2}}{\lambda \sqrt{m}}.
\end{align*}
This proves the desired lower bound of $\Err_{t+1}$ in inquality~\eqref{sharp-bounds-empirical-error} upon adjusting constants.

\paragraph{Proof of inequality~\eqref{bound-alpha-t+1-t-epsilon}:} Applying inequality~\eqref{alpha-change-after-det-map}, we obtain that
\[
	|\parcompdetX_{t+1} - \parcompX_{t}| \; \vee \; |\parcompdetZ_{t+1} - \parcompZ_{t}| \leq C_{1}'\sqrt{\Err_{t}}/\lambda. 
\]
Combining the inequality in the display above with inequality~\eqref{ineq:deviation-alpha'-alphadet}, we obtain
\begin{align*}
	|\parcompX_{t+1} - \parcompX_{t}| \leq |\parcompX_{t+1} - \parcompdetX_{t+1}| + |\parcompdetX_{t+1} - \parcompX_{t}| 
	&\leq 
	 \frac{C'\polylog \big( \sqrt{\Err_{t}} + \sigma \big) }{\lambda\sqrt{m}} + \frac{C_{1}'\sqrt{\Err_{t}}}{\lambda}\\
	&\leq
	 \frac{C_{1} \polylog \sigma }{\lambda\sqrt{m}} +\frac{C_{1}\sqrt{\Err_{t}}}{\lambda},
\end{align*}
where the last step follows from the assumption $m\geq C\log^{12}(d)$. This proves the desired inequality~\eqref{bound-alpha-t+1-t-epsilon}. Proceeding in identical fashion yields the bound on $|\parcompZ_{t+1} - \parcompZ_{t}|$, which completes the proof. \qed

\section{Discussion}
In this paper, we derived deterministic trajectory predictions for the stochastic prox-linear method when applied to rank one matrix sensing with Gaussian data.  We then used these predictions to derive concrete, two-sided convergence guarantees elucidating the trade-off between batch-size and step-size as well as the sensitivity of the convergence guarantees to the inherent noise in the problem.  Several interesting open questions remain and we detail a few here.  

While our one-step guarantees in Theorem~\ref{thm:one-step-prediction} hold for all batch-sizes $1 \leq m \leq d$, our convergence guarantees in Theorem~\ref{thm:local-sharp-convergence-results} require batch-sizes which scale poly-logarithmically in the dimension $d$.  This deficiency stems from bounds on the deviation around the deterministic predictions which incur a $\text{polylog}(d)$ factor.  Removing these logarithmic factors in Theorem~\ref{thm:one-step-prediction} would immediately extend our convergence guarantees to constant batch-size settings.  In a similar spirit, while our deterministic trajectory predictions---which we illustrate in Section~\ref{sec:experimental-results}---demonstrate excellent adherence to the empirical trajectory, our results only show that the empirical trajectory and deterministic trajectory enjoy the same convergence rate.  It would be interesting to show exact convergence of the empirical trajectory to the deterministic trajectory, for instance via a so-called envelope guarantee~\citep[Theorem 3(b)]{chandrasekher2023sharp}. 

More broadly, our guarantees required the mini-batches to be obtained in an online fashion.  While this incurs a nearly optimal sample complexity of $O(d (1 + \sigma^2) \log(1/\sigma^2))$ to reach error $\sigma^2$, it does not allow for samples to be re-used.  An interesting open question in this direction would be to understand the effect of sample re-use and develop trajectory predictions in this setting.  

\subsection*{Acknowledgments} This work was supported in part by the NSF under grants CCF-2107455 and DMS-2210734, by research awards/gifts from Adobe, Amazon, and Mathworks, and by European Research Council Advanced Grant 101019498. KAV would like to thank Aaron Mishkin and Karan Chadha for helpful conversations and useful pointers to the literature.  

\bibliographystyle{abbrvnat}
\bibliography{refs-intro,refs-blind-deconvolution}

\appendix
\section{Deferred proofs for the parallel component}
This appendix is organized as follows.  Sections~\ref{proof-auxiliary-lemma1-component-expectation-concentration} and~\ref{proof-auxiliary-lemma2-component-expectation-concentration} are dedicated to the proofs of Lemmas~\ref{auxiliary-lemma1-component-expectation-concentration} and~\ref{auxiliary-lemma2-component-expectation-concentration}, respectively, each of which is required to bound the stochastic error of the parallel component.  Then, in Sections~\ref{proof-lemma-leave-one-direction-out} and~\ref{subsec:proof-concentration-M11-M22-M12-M1-M2}, we provide the proofs of Lemmas~\ref{lemma-leave-one-out} and~\ref{concentration-of-M11-M22-M12-M1-M2}, respectively, each of which is required to bound the bias of the parallel component. 

\subsection{Proof of Lemma~\ref{auxiliary-lemma1-component-expectation-concentration}}\label{proof-auxiliary-lemma1-component-expectation-concentration} We prove each part in turn, starting with part (a).

\paragraph{Proof of Lemma~\ref{auxiliary-lemma1-component-expectation-concentration}(a):} By definition of $f_{\bu}$~\eqref{definition-f-component-innerproduct} and the assumption $\|\bu\|_{2}=1$, we bound $| f_{\bu} |$ as  $\big|f_{\bu}\big| \leq \big\|[\bcoefX_{+}^{\top}\;\vert\; \bcoefX_{+}^{\top}] \big\|_{2}$. By the definition of $(\bcoefX_{+},\bcoefZ_{+})$~\eqref{eq:closed-form-update} and $\bSig$, we consider the equivalent characterization
\begin{align}\label{eq:closed-form-mu-nu-plus}
	\begin{bmatrix} \bcoefX_{+} \\ \bcoefZ_{+} \\ \end{bmatrix} = \bSig^{-1} \bigg( \sum_{k = 1}^{n}(y_{k}+G_{k}\widetilde{G}_{k})\cdot \ba_{k} + \lambda m \begin{bmatrix} \bcoefX_{\sharp} \\ \bcoefZ_{\sharp} \\ \end{bmatrix} \bigg).
\end{align}
Applying the triangle inequality and using the three bounds $\|\bSig^{-1} \|_{\mathrm{op}} \leq (\lambda m)^{-1}$, $\big \lvert y_{k}+G_k\widetilde{G}_{k} \big \rvert  \leq C(\sigma + 1)\log(d)$, and $\|\bx_{k}\|_{2},\|\bz_{k}\|_{2} \lesssim \sqrt{d}$,
we obtain the bound
\begin{align*}
	\bigg \| \begin{bmatrix} \bcoefX_{+}\\ \bcoefZ_{+} \\ \end{bmatrix} \bigg\|_{2} &\leq \frac{C(\sigma + 1)\log^{2}(d) \sqrt{d} m}{\lambda m} + L_{\sharp} + \LZ_{\sharp} \leq d,
\end{align*}
where in the last step we invoked the assumptions $\lambda m \geq C_{0}d(1+\sigma)$ and $m\leq d$.  We turn now to the proof of part (b).

\paragraph{Proof of Lemma~\ref{auxiliary-lemma1-component-expectation-concentration}(b):}
Since $\rho\big(\{\bx_{k},\bz_{k},\epsilon_{k}\}_{k=1}^{n},\{\bx'_{k},\bz'_{k},\epsilon'_{k}\}_{k=1}^{n} \big) \leq 2$, without loss of generality, we suppose $(\bx_{k},\bz_{k},\epsilon_{k}) = (\bx'_{k},\bz'_{k},\epsilon'_{k})$ for all $k \in [m] \setminus \{i,j\}$. We begin by claiming the following pair of structural relations
\begin{subequations}\label{claim1-auxiliary-lemma-1}
\begin{align}
	\label{claim1-auxiliary-lemma-1a}
	\begin{bmatrix} \bcoefX_{+} \\ \bcoefZ_{+} \\ \end{bmatrix} &= \begin{bmatrix} \bcoefX^{-(i)} \\ \bcoefZ^{-(i)} \\ \end{bmatrix} + \frac{y_{i}+G_{i}\widetilde{G}_{i}-\widetilde{G}_{i}\bx_{i}^{\top}\bcoefX^{-(i)} - G_{i}\bz_{i}^{\top}\bcoefZ^{-(i)}}{1+\ba_{i}^{\top}\bSig_{i}^{-1} \ba_{i}} \cdot \bSig_{i}^{-1}\ba_{i},\\
	\label{claim1-auxiliary-lemma-1b}
	\begin{bmatrix} \bcoefX^{-(i)} \\ \bcoefZ^{-(i)} \\ \end{bmatrix} &= \begin{bmatrix} \bcoefX^{-(i,j)} \\ \bcoefZ^{-(i,j)} \\ \end{bmatrix} + \frac{y_{j}+G_{j}\widetilde{G}_{j}-\widetilde{G}_{j}\bx_{j}^{\top}\bcoefX^{-(i,j)} - G_{j}\bz_{j}^{\top}\bcoefZ^{-(i,j)}}{1+\ba_{j}^{\top}\bSig_{i,j}^{-1} \ba_{j}} \cdot \bSig_{i,j}^{-1}\ba_{j},
\end{align}
\end{subequations}
deferring their proofs to the end. Taking these relations as given, and noting that $\bSig_{i}^{-1},\bSig_{i,j}^{-1}\succeq \boldsymbol{0}$, we deduce the inequality
\begin{align*}
 	\bigg| f_{\bu}\big( \{\bx_{k},\bz_{k},\epsilon_{k}\}_{k=1}^{m} \big) - \Big \langle \bu, \begin{bmatrix} \bcoefX^{-(i,j)} \\ \bcoefZ^{-(i,j)} \end{bmatrix} \Big \rangle \bigg| &\leq \big \lvert y_{i}+G_{i}\widetilde{G}_{i}-\widetilde{G}_{i}\bx_{i}^{\top}\bcoefX^{-(i)} - G_{i}\bz_{i}^{\top}\bcoefZ^{-(i)} \big \rvert  \cdot \big \lvert \bu^{\top}\bSig_{i}^{-1}\ba_{i} \big \rvert \\& + \big \lvert y_{j}+G_{j}\widetilde{G}_{j}-\widetilde{G}_{j}\bx_{j}^{\top}\bcoefX^{-(i,j)} - G_{j}\bz_{j}^{\top}\bcoefZ^{-(i,j)} \big \rvert \cdot \big \lvert \bu^{\top} \bSig_{i,j}^{-1}\ba_{j} \big\rvert.
\end{align*}
Recalling that $y_i = \bx_i^{\top}\bcoefX_{\star} \cdot \bz_i^{\top}\bcoefZ_{\star} + \epsilon_i$ and applying triangle inequality yields
\begin{align}\label{ineq-auxiliary-lemma1-proof-u-component}
	\big \lvert y_{i}+G_{i}\widetilde{G}_{i}-\widetilde{G}_{i}\bx_{i}^{\top}\bcoefX^{-(i)} - & G_{i}\bz_{i}^{\top}\bcoefZ^{-(i)} \big \rvert \nonumber \leq
	\lvert \epsilon_{i} \rvert + \big \lvert \bx_{i}^{\top}\bcoefX_{\star} \cdot \bz_{i}^{\top}\bcoefZ_{\star} - \bx_{i}^{\top}\bcoefX_{\sharp} \cdot \bz_{i}^{\top}\bcoefZ_{\sharp} \big \rvert \\ 
	&\qquad \qquad + \big \lvert \bx_{i}^{\top}(\bcoefX_{\sharp} - \bcoefX^{-(i)}) \cdot \bz_{i}^{\top}\bcoefZ_{\sharp}  \big\rvert  + \big \lvert \bx_{i}^{\top} \bcoefX_{\sharp} \cdot \bz_{i}^{\top}(\bcoefZ_{\sharp} - \bcoefZ^{-(i)} ) \big \rvert.
\end{align}
We then decompose $\bcoefX_{\sharp}$ into two orthogonal directions $\bcoefX_{\star},\bcoefX^{\perp}$(similarly for $\bcoefZ_{\sharp}$) to bound
\begin{align}\label{ineq1-auxiliary-lemma1-proof-u-component}
	&\big \lvert \bx_{i}^{\top}\bcoefX_{\star}  \bz_{i}^{\top}\bcoefZ_{\star} - \bx_{i}^{\top}\bcoefX_{\sharp} \bz_{i}^{\top}\bcoefZ_{\sharp} \big\rvert \nonumber \\ &= \big\lvert (\parcompX_{\sharp}\parcompZ_{\sharp}-1)\bx_{i}^{\top}\bcoefX_{\star} \cdot \bz_{i}^{\top}\bcoefZ_{\star} - \perpcompX_{\sharp}\parcompZ_{\sharp} \bx_{i}^{\top}\bcoefX^{\perp} \bz_{i}^{\top}\bcoefZ_{\star} - \perpcompZ_{\sharp}\parcompX_{\sharp} \bx_{i}^{\top}\bcoefX_{\star} \bz_{i}^{\top}\bcoefZ^{\perp} - \perpcompX_{\sharp}\perpcompZ_{\sharp} \bx_{i}^{\top}\bcoefX^{\perp} \bz_{i}^{\top}\bcoefZ^{\perp} \big \rvert \nonumber \\&\lesssim \big(\big \lvert \parcompX_{\sharp}\parcompZ_{\sharp}-1 \big \rvert + \perpcompX_{\sharp} + \perpcompZ_{\sharp} \big) \cdot \log(d),
\end{align}
where in the last step we used $\{\bx_i,\bz_i,\epsilon_i\}_{i=1}^{m} \in \mathcal{S}$ and $L_{\sharp},\LZ_{\sharp} \lesssim 1$.  Proceeding similarly yields the pair of upper bounds
\begin{align*}
	\big\lvert \bx_{i}^{\top}(\bcoefX_{\sharp} - \bcoefX^{-(i)}) \cdot \bz_{i}^{\top}\bcoefZ_{\sharp}  \big \rvert &\lesssim \log(d) \big\|\bcoefX_{\sharp} - \bcoefX^{-(i)} \big\|_{2}, \qquad \text{ and }\\
	 \big\lvert \bx_{i}^{\top} \bcoefX_{\sharp} \cdot \bz_{i}^{\top}(\bcoefZ_{\sharp} - \bcoefZ^{-(i)} ) \big\rvert &\lesssim \log(d) \big\|\bcoefZ_{\sharp} - \bcoefZ^{-(i)} \big\|_{2}.
\end{align*}
Then, using the definition of $(\bcoefX^{-(i)},\bcoefZ^{-(i)})$ as minimizers of~\eqref{estimators-leave-i-sample-out}, we obtain the inequality
\begin{align*}
	&\sum_{k \neq i}\big(y_{k} + G_{k}\widetilde{G}_{k} - \widetilde{G}_{k} \bx_{k}^{\top}\bcoefX^{-(i)} - G_{k} \bz_{k}^{\top}\bcoefZ^{-(i)} \big)^{2} + \lambda m\cdot \big( \big\|\bcoefX^{-(i)} - \bcoefX_{\sharp} \big\|_{2}^{2} + \big\|\bcoefZ^{-(i)} - \bcoefZ_{\sharp} \big\|_{2}^{2} \big)
	\\&\leq \sum_{k \neq i}\big(y_{k} + G_{k}\widetilde{G}_{k} - \widetilde{G}_{k} \bx_{k}^{\top}\bcoefX_{\sharp} - G_{k} \bz_{k}^{\top}\bcoefZ_{\sharp} \big)^{2} = \sum_{k \neq i}\big(y_{k} - \bx_{k}^{\top}\bcoefX_{\sharp} \cdot \bz_{k}^{\top}\bcoefZ_{\sharp} \big)^{2}.
\end{align*}
Re-arranging yields
\begin{align*} %\label{ineq1.5-auxiliary-lemma1-proof-u-component}
% \hspace{-1cm}
	\big\|\bcoefX^{-(i)} - \bcoefX_{\sharp} \big\|_{2}^{2} + \big\|\bcoefZ^{-(i)} - \bcoefZ_{\sharp} \big\|_{2}^{2} &\leq \frac{1}{\lambda m} \sum_{k \neq i}\big(y_{k} - \bx_{k}^{\top}\bcoefX_{\sharp} \cdot \bz_{k}^{\top}\bcoefZ_{\sharp} \big)^{2} \nonumber \\
% \hspace{-1cm}
	&\leq \frac{2}{\lambda m} \sum_{k \neq i} \epsilon_{k}^{2} + \big( \bx_{k}^{\top}\bcoefX_{\star} \cdot \bz_{k}^{\top}\bcoefZ_{\star} - \bx_{k}^{\top}\bcoefX_{\sharp} \cdot \bz_{k}^{\top}\bcoefZ_{\sharp} \big)^{2}.
\end{align*}
Then, applying the inequality~\eqref{ineq1-auxiliary-lemma1-proof-u-component} in conjunction with the bound $\epsilon_{k}\lesssim \sigma \log(d)$, we further obtain the bound
\begin{align}\label{ineq1.5-auxiliary-lemma1-proof-u-component}
	\big\|\bcoefX^{-(i)} - \bcoefX_{\sharp} \big\|_{2}^{2} + \big\|\bcoefZ^{-(i)} - \bcoefZ_{\sharp} \big\|_{2}^{2} \lesssim  \frac{m\log^{2}(d)}{\lambda m} \big(\sigma^{2} + (\parcompX_{\sharp}\parcompZ_{\sharp}-1)^{2} + \perpcompX_{\sharp}^{2} + \perpcompZ_{\sharp}^{2}.
\end{align}
Putting the two pieces together and invoking the assumption $\lambda m\geq d\geq m$ yields the bound
\begin{align}\label{ineq2-auxiliary-lemma1-proof-u-component}
	\big \lvert \bx_{i}^{\top}(\bcoefX_{\sharp} - \bcoefX^{-(i)}) \cdot \bz_{i}^{\top}\bcoefZ_{\sharp}  \big\rvert + \big \lvert \bx_{i}^{\top} \bcoefX_{\sharp} \cdot \bz_{i}^{\top}(\bcoefZ_{\sharp} - \bcoefZ^{-(i)} ) \big \rvert \lesssim \log^{2}(d) \big( \sigma + \big \lvert \parcompX_{\sharp}\parcompZ_{\sharp}-1\big \rvert + \perpcompX_{\sharp} + \perpcompZ_{\sharp} \big).
\end{align}
Substituting inequalities~\eqref{ineq1-auxiliary-lemma1-proof-u-component} and~\eqref{ineq2-auxiliary-lemma1-proof-u-component} into inequality~\eqref{ineq-auxiliary-lemma1-proof-u-component} yields 
\begin{align}\label{ineq2.5-auxiliary-lemma1-proof-u-component}
	\big \lvert y_{i}+G_{i}\widetilde{G}_{i}-\widetilde{G}_{i}\bx_{i}^{\top}\bcoefX^{-(i)} - G_{i}\bz_{i}^{\top}\bcoefZ^{-(i)} \big \rvert \lesssim \log^{2}(d) \big(\sigma + \big\lvert \parcompX_{\sharp}\parcompZ_{\sharp}-1 \big \rvert+ \perpcompX_{\sharp} + \perpcompZ_{\sharp} \big).
\end{align}
Following identical steps, we obtain the inequality
\[
	\big\lvert y_{j}+G_{j}\widetilde{G}_{j}-\widetilde{G}_{j}\bx_{j}^{\top}\bcoefX^{-(i,j)} - G_{j}\bz_{j}^{\top}\bcoefZ^{-(i,j)} \big \rvert \lesssim \log^{2}(d) \big(\sigma + \big\lvert \parcompX_{\sharp}\parcompZ_{\sharp}-1 \big \rvert + \perpcompX_{\sharp} + \perpcompZ_{\sharp} \big).
\]
Further, note that since $\{\bx_{k},\bz_{k},\epsilon_{k}\}_{k=1}^{n} \in \mathcal{S}$, we have the bounds
$|\bu^{\top}\bSig_{i}^{-1}\ba_{i}|,|\bu^{\top}\bSig_{i,j}^{-1}\ba_{j}| \lesssim \log(d)/(\lambda m)$.
Putting the pieces together yields
\begin{align}\label{ineq3-auxiliary-lemma1-proof-u-component}
 	\bigg| f_{\bu}\big( \{\bx_{k},\bz_{k},\epsilon_{k}\}_{k=1}^{n} \big) - \Big \langle \bu,\begin{bmatrix} \bcoefX^{-(i,j)} \\ \bcoefZ^{-(i,j)}  \end{bmatrix} \Big \rangle \bigg| \lesssim \frac{\log^{3}(d)}{\lambda m} \cdot \big(\sigma + \big \lvert \parcompX_{\sharp}\parcompZ_{\sharp}-1 \big \rvert + \perpcompX_{\sharp} + \perpcompZ_{\sharp} \big).
\end{align}
Proceding similarly, since $\{\bx'_{k},\bz'_{k},\epsilon'_{k}\}_{k=1}^{n} \in \mathcal{S}$, we obtain that
\begin{align}\label{ineq4-auxiliary-lemma1-proof-u-component}
	\bigg| f_{\bu}\big( \{\bx'_{k},\bz'_{k},\epsilon'_{k}\}_{k=1}^{n} \big) - \Big \langle \bu,\begin{bmatrix} \bcoefX^{-(i,j)} \\ \bcoefZ^{-(i,j)} \end{bmatrix} \Big \rangle \bigg| \lesssim \frac{\log^{3}(d)}{\lambda n} \cdot \big(\sigma + \big| \parcompX_{\sharp}\parcompZ_{\sharp}-1 \big| + \perpcompX_{\sharp} + \perpcompZ_{\sharp} \big).
\end{align}
Putting inequalities~\eqref{ineq3-auxiliary-lemma1-proof-u-component} and~\eqref{ineq4-auxiliary-lemma1-proof-u-component} together and noting $|\parcompX_{\sharp}\parcompZ_{\sharp}-1| + \perpcompX_{\sharp} + \perpcompZ_{\sharp} \lesssim \sqrt{\Err_{\sharp}}$ yields the desired result.
It remains to establish the claim~\eqref{claim1-auxiliary-lemma-1}.\\

\medskip
\noindent \underline{Proof of the structural relations~\eqref{claim1-auxiliary-lemma-1}:} By the definition of $[\bcoefX_{+}\;\vert\;\bcoefZ_{+}]^{\top} = \prox([\bcoefX_{\sharp} \;\vert \; \bcoefZ_{\sharp}])$~\eqref{eq:closed-form-update}, we obtain that
\begin{align*}
	\begin{bmatrix} \bcoefX_{+} \\ \bcoefZ_{+} \\ \end{bmatrix} = \bSig^{-1} \bigg( \sum_{k \neq i}(y_{k}+G_{k}\widetilde{G}_{k})\cdot \ba_{k} + \lambda m \begin{bmatrix} \bcoefX_{\sharp} \\ \bcoefZ_{\sharp} \\ \end{bmatrix} \bigg) + 
	\big(y_{i}+G_{i}\widetilde{G}_{i} \big) \cdot \bSig^{-1} \ba_{i}.
\end{align*}
Applying the Sherman--Morrison formula yields $\bSig^{-1} = \bSig_{i}^{-1} - \bSig_{i}^{-1} \ba_{i} \ba_{i}^{\top} \bSig_{i}^{-1}/(1+\ba_{i}^{\top} \bSig_{i}^{-1} \ba_{i})$.  We substitute this expression into the preceding display to obtain
\begin{align}\label{eq1-proof-claim1-auxiliary-lemma-1}
\begin{bmatrix} \bcoefX_{+} \\ \bcoefZ_{+} \end{bmatrix} = \bigg( \bSig_{i}^{-1} - \frac{\bSig_{i}^{-1} \ba_{i} \ba_{i}^{\top} \bSig_{i}^{-1}}{1+\ba_{i}^{\top} \bSig_{i}^{-1} \ba_{i}} \bigg) \bigg( \sum_{k \neq i}(y_{k}+G_{k}\widetilde{G}_{k})\cdot \ba_{k} + \lambda n \begin{bmatrix}\bcoefX_{\sharp} \\ \bcoefZ_{\sharp}\end{bmatrix} \bigg)
  + \frac{\big(y_{i} + G_{i} \widetilde{G}_{i} \big) \bSig_{i}^{-1}\ba_{i} }{1+\ba_{i}^{\top} \bSig_{i}^{-1} \ba_{i}}.
\end{align}
By the definition of $(\bcoefX^{-(i)},\bcoefZ^{-(i)})$~\eqref{estimators-leave-i-sample-out}, we obtain that
\[
	\begin{bmatrix} \bcoefX^{-(i)} \\ \bcoefZ^{-(i)} \\ \end{bmatrix} = \bSig_{i}^{-1} \bigg( \sum_{k \neq i}(y_{k}+G_{k}\widetilde{G}_{k})\cdot \ba_{k} + \lambda n \begin{bmatrix} \bcoefX_{\sharp} \\ \bcoefZ_{\sharp} \\ \end{bmatrix} \bigg).
\]
Substituting this into equation~\eqref{eq1-proof-claim1-auxiliary-lemma-1} yields
\begin{align*}
	\begin{bmatrix} \bcoefX_{t+1} \\ \bcoefZ_{t+1} \\ \end{bmatrix} &= \begin{bmatrix} \bcoefX^{-(i)} \\ \bcoefZ^{-(i)} \\ \end{bmatrix} - \frac{\bSig_{i}^{-1} \ba_{i} \ba_{i}^{\top} \begin{bmatrix} \bcoefX^{-(i)} \\ \bcoefZ^{-(i)} \\ \end{bmatrix} }{1+\ba_{i}^{\top} \bSig_{i}^{-1} \ba_{i}} + \frac{(y_{i} + G_{i} \widetilde{G}_{i}) \bSig_{i}^{-1}\ba_{i} }{1+\ba_{i}^{\top} \bSig_{i}^{-1} \ba_{i}}  \\&= 
	\begin{bmatrix} \bcoefX^{-(i)} \\ \bcoefZ^{-(i)} \\ \end{bmatrix} + \frac{\Big(y_{i} + G_{i} \widetilde{G}_{i} - \widetilde{G}_{i}\bx_{i}^{\top} \bcoefX^{-(i)} - G_{i}\bz_{i}^{\top} \bcoefZ^{-(i)}  \Big) \bSig_{i}^{-1}\ba_{i} }{1+\ba_{i}^{\top} \bSig_{i}^{-1} \ba_{i}},
\end{align*}
where in the last step we re-arranged terms and used the notation $\ba_i^{\top} = \big[\widetilde{G}_i \bx_i^{\top}\; G_i\bz_i^{\top} \big]$. This proves the relation~\eqref{claim1-auxiliary-lemma-1a}. Identical steps yield the claim~\eqref{claim1-auxiliary-lemma-1b}. \qed

\subsection{Proof of Lemma~\ref{auxiliary-lemma2-component-expectation-concentration}}\label{proof-auxiliary-lemma2-component-expectation-concentration} 
We prove each part in turn, beginning with part (a).

\paragraph{Proof of Lemma~\ref{auxiliary-lemma2-component-expectation-concentration}(a):}
Note that several of the defining constraints of the regularity set $\mathcal{S}$ in~\eqref{definition-regularity-set-component-expectation} are standard (e.g. the high probability bound $\| \bx_i \|_2 \lesssim \sqrt{d}$).  We thus only provide probability bounds for the non-standard constraints.

Let $\bw_{1},\bw_{2} \in \mathbb{R}^{d}$ such that $\bSig_{i}^{-1}\bu = [\bw_1^{\top} \; \vert \;\bw_2^{\top}]^\top$. Note that $(\bx_{i},\bz_{i})$ are independent of $(\bw_1,\bw_2)$ and $\bu^{\top} \bSig_{i}^{-1} \ba_{i} = \widetilde{G_{i}} \bx_{i}^{\top} \bw_{1} + G_{i} \bz_{i}^{\top} \bw_{2}.$ Applying Hoeffding’s inequality~\citep[Theorem 2.2.6]{vershynin2018high} yields that with probability at least $1-d^{-200}$, all three of the following inequalities hold
\[
	\big|G_{i} \big|,\big|\widetilde{G_{i}}\big| \leq C\sqrt{\log(d)},\qquad \big|\bx_{i}^{\top} \bw_{1} \big| \leq C\sqrt{\log(d)}\|\bw_{1}\|_{2}, \qquad \text{ and } \qquad 
	\big|\bz_{i}^{\top} \bw_{2} \big| \leq C\sqrt{\log(d)}\|\bw_{2}\|_{2}.
\]
Consequently, with probability at least $1-d^{-200}$, 
\[
	\big|\bu^{\top} \bSig_{i}^{-1} \ba_{i} \big| \lesssim \log(d)\cdot (\|\bw_{1}\|_{2} + \|\bw_{2}\|_{2}) \lesssim \log(d) \cdot \big\|\bSig_{i}^{-1}\bu \big\|_{2} \leq \frac{\log(d)}{\lambda m},
\]
where the final inequality follows from the pair of bounds $\big\|\bSig_{i}^{-1} \big\|_{2} \leq (\lambda m)^{-1}$ and $\|\bu\|_{2} = 1$. Following the identical step yields an identical bound for the probability of the event $\big\{ \big|\bu^{\top} \bSig_{i,j}^{-1} \ba_{j} \big| \geq C\log(d)/(\lambda m) \big\}$.
Next, noting that $\bx_{i}$ is independent of $\bcoefX_{\sharp} - \bcoefX^{-(i)}$ and applying Hoeffding’s inequality yields the probabilistic inequality
\[
	\Pr\Big\{ \big|\bx_{i}^{\top} (\bcoefX_{\sharp} - \bcoefX^{-(i)})\big| \geq C \sqrt{\log(d)} \cdot \big\|\bcoefX_{\sharp} - \bcoefX^{-(i)} \big\|_{2} \Big\} \leq 2d^{-200}.
\]
Identical steps yield a similar bound on the probability of the event $\big\{\big|\bx_{i}^{\top} (\bcoefX_{\sharp} - \bcoefX^{-(i,j)}) \big| \geq C \sqrt{\log(d)} \cdot \big\|\bcoefX_{\sharp} - \bcoefX^{-(i,j)} \big\|_{2} \big\}$. Putting the pieces together and applying the union bound yields $\Pr\big\{ \{\bx_i,\bz_i,\epsilon_i\}_{i=1}^{n} \notin \mathcal{S} \big\} \leq d^{-180}$, as desired.

\paragraph{Proof of Lemmas~\ref{auxiliary-lemma2-component-expectation-concentration}(b) and (c):}
Next we reduce a tail bound on $f_{\bu}$ to a tail bound on the truncated function $f_{\bu}^{\downarrow}$. Recall the shorthand $\bM = \{\bx_i,\bz_i,\epsilon_i\}_{i=1}^{m}$ and $\bM' = \{\bx'_i,\bz'_i,\epsilon'_i\}_{i=1}^{m}$. Note that
\begin{align}\label{proof-aucillary-lemma2-u-component}
	\Pr\Big\{ \big|f_{\bu} - \EE\{f_{\bu}\} \big| \geq t \Big\} &\leq \Pr\Big\{ \big|f_{\bu}(\bM) - \EE\{f_{\bu}(\bM)\} \big| \geq t, \bM \in \mathcal{S} \Big\} + \Pr\Big\{ \bM \notin \mathcal{S} \Big\} \nonumber
	\\& \overset{\1}{\leq} \Pr\Big\{ \big|f_{\bu}^{\downarrow}(\bM) - \EE\{f_{\bu}(\bM)\} \big| \geq t, \bM \in \mathcal{S} \Big\} + d^{-180} \nonumber\\
	& \leq \Pr\Big\{ \big|f_{\bu}^{\downarrow}(\bM) - \EE\{f_{\bu}(\bM)\} \big| \geq t\Big\} + d^{-180} 
\end{align}
where step $\1$ follows by combining Eq.~\eqref{properties-truncate-f-u-component}---which ensures that $f_{\bu}(\bM)=f_{\bu}^{\downarrow}(\bM)$ for $\bM\in\mathcal{S}$---and Lemma~\ref{auxiliary-lemma2-component-expectation-concentration}(a). We next turn to bound $\big|\EE\{f_{\bu}\} - \EE\{f_{\bu}^{\downarrow}\} \big|$. Note that since $f_{\bu}$ and $f_{\bu}^{\downarrow}$ agree on $\mathcal{S}$, 
\begin{align}
  \big|\EE\{f_{\bu}\} - \EE\{f_{\bu}^{\downarrow}\} \big| &= \big| \EE\big\{f_{\bu}(\bM) \cdot \mathbbm{1}\{\bM \notin \mathcal{S}\} \big\} - \EE\big\{f_{\bu}^{\downarrow}(\bM) \cdot \mathbbm{1}\{\bM \notin \mathcal{S}\} \big\} \big| \nonumber \\ &\leq  \big|\EE\big\{f_{\bu}(\bM) \cdot \mathbbm{1}\{\bM \notin \mathcal{S}\} \big\} \big| +\big|\EE\big\{f_{\bu}^{\downarrow}(\bM) \cdot \mathbbm{1}\{\bM \notin \mathcal{S}\}\big\}\big|. \label{ineq-EEf-EEf-down}
\end{align}
We next bound the two terms in the RHS of the inequality in the display above. By definition of $f_{\bu}^{\downarrow}$~\eqref{definition-f-down-u-component}, we obtain that for any $\bM' \in \mathbb{R}^{(2d+1) \times m}$ and $\bM \in \mathcal{S} $
\begin{align*}
	f_{\bu}^{\downarrow}(\bM') \leq f_{\bu}(\bM) + \Delta \cdot \rho(\bM,\bM') + 2D \leq f_{\bu}(\bM) + m\Delta + 2d \overset{\1}{\lesssim} d,
\end{align*}
where step $\1$ follows from Lemma~\ref{auxiliary-lemma1-component-expectation-concentration}(a).  Applying Lemma~\ref{auxiliary-lemma1-component-expectation-concentration}(a) once more additionally yields the lower bound $	f_{\bu}^{\downarrow}(\bM') \geq \inf_{\bM \in \mathcal{S}} f_{\bu}(\bM) \geq -d$.
Consequently, we obtain that $|f_{\bu}^{\downarrow}| \lesssim d$, which completes the proof of part (b).  This additionally implies that
\begin{align}\label{ineq1-EEf-EEf-down}
	\big| \EE\big\{f_{\bu}^{\downarrow}(\bM) \cdot \mathbbm{1}\{ \bM \notin \mathcal{S} \} \big\} \big| \lesssim d \cdot \Pr\big\{ \bM \notin \mathcal{S} \big\}
	\leq d^{-160}.
\end{align}
Next we turn to bound $\big|\EE \big\{f_{\bu}(\bM) \cdot \mathbbm{1}\{\bM \notin \mathcal{S}\} \big\} \big|$. By definition of $f_{\bu}$, we note that
$|f_{\bu}| \leq \big\|\bcoefX_{+}\big\|_{2} + \big\|\bcoefZ_{+}\big\|_{2}$. Using the expression of $(\bcoefX_{+},\bcoefZ_{+})$~\eqref{eq:prox-linear-updates} and the bound $\|\bSig^{-1}\|_{2} \leq (\lambda m)^{-1}$, we obtain the inequality
\[
	\big\|\bcoefX_{+}\big\|_{2} + \big\|\bcoefZ_{+}\big\|_{2} \lesssim \frac{1}{\lambda m} \bigg\| \sum_{k = 1}^{m}\big(y_{k}+G_{k}\widetilde{G}_{k}\big)\cdot \ba_{k} \bigg\|_{2} + L_{\sharp} + \LZ_{\sharp}.
\]
Putting the pieces together yields the bound
\[
	\big|\EE\big\{ f_{\bu}(\bM) \cdot \mathbbm{1}\{\bM \notin \mathcal{S}\} \big\} \big| \leq \underbrace{\frac{1}{\lambda m} \EE\bigg\{ \bigg\| \sum_{k = 1}^{m}(y_{k}+G_{k}\widetilde{G}_{k})\cdot \ba_{k} \bigg\|_{2} \cdot \mathbbm{1}\big\{\bM \notin \mathcal{S}\big\} \bigg\}}_{T} + C\Pr\big\{\bM \notin \mathcal{S}\big\}.
\]
Then, we apply the Cauchy--Schwartz inequality to obtain the bound
\begin{align*}
 T \leq \EE\bigg\{ \bigg\| \sum_{k = 1}^{m}\big(y_{k}+G_{k}\widetilde{G}_{k}\big)\cdot \ba_{k} \bigg\|_{2}^{2} \bigg\}^{1/2} \cdot \sqrt{\Pr\big\{\bM \notin \mathcal{S} \big\}} \lesssim md(1+\sigma) \cdot d^{-90}.
\end{align*}
Putting the pieces together yields
\begin{align}\label{ineq2-EEf-EEf-down}
	\big|\EE\big\{f_{\bu}(\bM) \cdot \mathbbm{1}\{\bM \notin \mathcal{S}\} \big\} \big| \lesssim \frac{md(1+\sigma)}{\lambda m} d^{-90} + d^{-180} \leq 2d^{-85},
\end{align}
where the last step follows from the assumptions $\lambda m \geq Cd(1+\sigma)$ and $m\leq d$. Substituing inequalities~\eqref{ineq1-EEf-EEf-down} and~\eqref{ineq2-EEf-EEf-down} into inequality~\eqref{ineq-EEf-EEf-down} yields $\big|\EE\{f_{\bu}\} - \EE\{f_{\bu}^{\downarrow}\} \big| \lesssim d^{-81}$. Consequently, 
\begin{align*}
	\Pr\Big\{ \big|f_{\bu}^{\downarrow} - \EE\{f_{\bu}\} \big| \geq t \Big\}
	&\leq \Pr\Big\{ \big|f_{\bu}^{\downarrow} - \EE\{f_{\bu}^{\downarrow}\} \big| \geq t/2\Big\} + \Pr\Big\{ \big|\EE\{f_{\bu}^{\downarrow}\} - \EE\{f_{\bu}\} \big| \geq  t/2\Big\}
	\\&\leq \Pr\Big\{ \big|f_{\bu}^{\downarrow} - \EE\{f_{\bu}^{\downarrow}\} \big| \geq t/2\Big\},
\end{align*}
where the last step follows by setting $t>d^{-80}$. Combining the inequality in the display above with inequality~\eqref{proof-aucillary-lemma2-u-component} yields the desired result. \qed

\subsection{Proof of Lemma~\ref{lemma-leave-one-out}}\label{proof-lemma-leave-one-direction-out}
We will use the shorthand $G_{i} = \bx_{i}^{\top} \bcoefX_{\sharp}$ and $\GZ_{i} = \bz_{i}^{\top} \bcoefZ_{\sharp}$ for $i \in [m]$. Note that Eq.~\eqref{eq:closed-form-update} is equivalent to 
\[
	\big(\bcoefX_{+},\;\bcoefZ_{+} \big) = \argmin_{\bcoefX,\bcoefZ \in \mathbb{R}^{d}} \frac{1}{m} \sum_{i=1}^{m}\big(y_{i} + G_i\GZ_i - \GZ_i \bx_{i}^{\top} \bcoefX - G_i \bz_{i}^{\top} \bcoefZ \big)^{2} 
	+ \lambda \big\|\bcoefX - \bcoefX_t\big\|_2^2 + \lambda \big\|\bcoefZ - \bcoefZ_{t}\big\|_2^2.
\]
The optimization problem in the display admits the KKT condition 
\begin{subequations}
\begin{align}
	\label{KKT-1}
	\frac{1}{m} \sum_{i=1}^{m} 	\big(  \GZ_i \cdot \bx_{i}^{\top} \bcoefX_{+} + G_i \cdot \bz_{i}^{\top} \bcoefZ_{+}  - y_{i} - G_i\GZ_i \big)\GZ_i \bx_{i} + \lambda \big(\bcoefX_{+} - \bcoefX_\sharp \big) &= \boldsymbol{0},\\
	\label{KKT-2}
	\frac{1}{m} \sum_{i=1}^{m} 	\big(  \GZ_i \cdot \bx_{i}^{\top} \bcoefX_{+} + G_i \cdot \bz_{i}^{\top} \bcoefZ_{+}  - y_{i} - G_i\GZ_i \big)G_i \bz_{i} + \lambda\big(\bcoefZ_{+} - \bcoefZ_\sharp\big) &= \boldsymbol{0}.
\end{align}
\end{subequations}
By definition, we decompose $\bcoefX_{+} = \theta(\bu) \cdot \bu + \bO_{\bu} \bO_{\bu}^{\top} \bcoefX_{+},\; \bcoefZ_{+} = \thetatil(\bv) \cdot \bv + \bO_{\bv} \bO_{\bv}^{\top} \bcoefZ_{+}$. Consequently,
\begin{align}
	\bx_i^{\top} \bcoefX_{+} = \theta(\bu) \cdot \bx_i^{\top} \bu +  \bx_i^{\top} \bO_{\bu} \bO_{\bu}^{\top} \bcoefX_{+}, \qquad \text{ and } \qquad \bz_i^{\top} \bcoefZ_{+} = \thetatil(\bv) \cdot \bz_i^{\top} \bv +  \bz_i^{\top} \bO_{\bv} \bO_{\bv}^{\top} \bcoefZ_{+}. \label{eq-xtop-bcoefX-ztop-bcoefZ}
\end{align}
Taking the inner product between $\bu$ and both sides of Eq.~\eqref{KKT-1}, using the decomposition in the display above and rearranging the terms yields
\begin{align*}
	\Big(\frac{1}{m}\sum_{i=1}^{m} \GZ_{i}^2 (\bx_{i}^{\top} \bu)^{2} + \lambda \Big) \cdot \theta(\bu) &+ \Big( \frac{1}{m} \sum_{i=1}^{m}G_i\GZ_i(\bx_{i}^{\top} \bu)(\bz_{i}^{\top} \bv) \Big) \cdot \thetatil(\bv) + 
	\frac{1}{m}\sum_{i=1}^{m} \GZ_i^2 (\bx_i^{\top} \bu) (\bx_i^{\top} \bO_{\bu} \bO_{\bu}^{\top} \bcoefX_{+}) \\
	& + \frac{1}{m}\sum_{i=1}^{m}G_i \GZ_i (\bz_i^{\top} \bu) (\bz_i^{\top} \bO_{\bv} \bO_{\bv}^{\top} \bcoefZ_{+}) = \lambda \bu^{\top} \bcoefX_\sharp + \frac{1}{m} \sum_{i=1}^{m}(y_i + G_i\GZ_i)\GZ_i(\bx_i^{\top}\bu).
\end{align*}
Now writing the equation in the display above in matrix form yields
\begin{align}\label{eq1-separate-u-bOu}
	&\Big( \frac{1}{m} \| \bWtil \bX \bu \|_{2}^{2} + \lambda \Big) \cdot \theta(\bu) +  \frac{1}{m} \langle \bWtil \bX\bu, \bW\bZ\bv \rangle \cdot \thetatil(\bv) + \frac{1}{m} \langle \bWtil \bX\bu, \bWtil \bX\bO_{\bu} \bO_{\bu}^{\top} \bcoefX_{+} \rangle \nonumber \\
	& \qquad \qquad \qquad+ \frac{1}{m} \langle \bWtil \bX \bu, \bW \bZ \bO_{\bv} \bO_{\bv}^{\top} \bcoefZ_{+} \rangle = \lambda \langle \bu, \bcoefX_\sharp \rangle + \frac{1}{m} \langle \by + \diag(\bW\bWtil) , \bWtil \bX \bu \rangle.
\end{align}
Similarly, taking the inner product between $\bv$ and both sides of Eq.~\eqref{KKT-2} yields
\begin{align}\label{eq1-separate-v-bOv}
	&\Big( \frac{1}{m} \| \bW\bZ \bv \|_{2}^{2} + \lambda \Big) \cdot \thetatil(\bv) +  \frac{1}{m} \langle \bWtil \bX\bu, \bW\bZ\bv \rangle \cdot \theta(\bu) + \frac{1}{m} \langle \bW \bZ\bv, \bWtil \bX\bO_{\bu} \bO_{\bu}^{\top} \bcoefX_{+} \rangle \nonumber \\
	& \qquad \qquad \qquad + \frac{1}{m} \langle \bW\bZ \bv, \bW \bZ \bO_{\bv} \bO_{\bv}^{\top} \bcoefZ_{+} \rangle = \lambda \langle \bv, \bcoefZ_\sharp \rangle + \frac{1}{m} \langle \by + \diag(\bW\bWtil) , \bW \bZ \bv \rangle.
\end{align}
Recall the notation $\bA_{\bu,\bv} = [\bWtil \bX \bO_{\bu} \; \bW \bZ \bO_{\bv}]$, we collect Eq.~\eqref{eq1-separate-u-bOu} and Eq.~\eqref{eq1-separate-v-bOv} into the following system
\begin{align}\label{eq-separate-uv-bOu-bOv}
	&\bigg( \frac{1}{m} \begin{bmatrix} (\bWtil \bX \bu)^{\top} \\  (\bW \bZ \bv)^{\top} \end{bmatrix}  \begin{bmatrix} \bWtil \bX \bu \;  \bW \bZ \bv \end{bmatrix} + \lambda \bI \bigg) 
	\begin{bmatrix} \theta(\bu) \\ \thetatil(\bv) \end{bmatrix}  + \frac{1}{m} \begin{bmatrix} (\bWtil \bX \bu)^{\top} \\  (\bW \bZ \bv)^{\top} \end{bmatrix} \bA_{\bu,\bv} 
	\left[ \begin{array}{c} \bO_{\bu}^{\top} \bcoefX_{+} \\ \bO_{\bv}^{\top} \bcoefZ_{+} \end{array} \right] \nonumber \\ 
	&=
	\lambda \cdot  \begin{bmatrix} \langle \bcoefX_{\sharp}, \bu \rangle \\ \langle \bcoefZ_{\sharp}, \bv \rangle \end{bmatrix} + \frac{1}{m} \begin{bmatrix} (\bWtil \bX \bu)^{\top} \\  (\bW \bZ \bv)^{\top} \end{bmatrix} \Big( \by + \diag(\bW\bWtil) \Big).
\end{align}
Likewise, we multiply $\bO_{\bu}^{\top}$ on both sides of~\eqref{KKT-1}.  As before,  using the decomposition~\eqref{eq-xtop-bcoefX-ztop-bcoefZ} and writing the result in matrix form yields
\begin{align*}
&\frac{\theta(\bu)}{m} (\bWtil \bX \bO_{\bu})^{\top} (\bWtil \bX \bu) + \bigg( \frac{1}{m} (\bWtil \bX \bO_{\bu})^{\top} (\bWtil \bX \bO_{\bu}) + \lambda \bI \bigg) \bO_{\bu}^{\top} \bcoefX_{+} 
+ \frac{\thetatil(\bv)}{m} (\bWtil \bX \bO_{\bu})^{\top} (\bW \bZ \bv) 
  \\ & + \frac{1}{m} (\bWtil \bX \bO_{\bu})^{\top} (\bW \bZ \bO_{\bv}) \bO_{\bv}^{\top} \bcoefZ_{+} = \frac{1}{m} (\bWtil \bX \bO_{\bu})^{\top} (\by + \diag(\bW \bWtil) )  + \lambda \bO_{\bu}^{\top} \bcoefX_{\sharp}.
\end{align*}
Similarly, multiplying $\bO_{\bv}^{\top}$ on both sides of~\eqref{KKT-2} and writing the result in matrix form yields
\begin{align*}
&\frac{\theta(\bu)}{m} (\bW \bZ \bO_{\bv})^{\top} (\bWtil \bX \bu) + \frac{1}{m} (\bW \bZ \bO_{\bv})^{\top} (\bWtil \bX \bO_{\bu})   \bO_{\bu}^{\top} \bcoefX_{+} 
 + \frac{\thetatil(\bv)}{m} (\bW \bZ \bO_{\bv})^{\top} (\bW \bZ \bv) \\ 
 & \qquad \qquad+ \bigg( \frac{1}{m} (\bW \bZ \bO_{\bv})^{\top} (\bW \bZ \bO_{\bv})  + \lambda \bI \bigg) \bO_{\bv}^{\top} \bcoefZ_{+} =  \frac{1}{m} (\bW \bZ \bO_{\bv})^{\top} (\by + \diag(\bW \bWtil) ) +  \lambda \bO_{\bv}^{\top} \bcoefZ_{\sharp}.
\end{align*}
Combining the previous two displays yields
\begin{align*}
	&\Big(\frac{1}{m}\bA_{\bu,\bv}^{\top} \bA_{\bu,\bv} + \lambda \bI \Big) \left[ \begin{array}{c} \bO_{\bu}^{\top} \bcoefX_{+} \\  \bO_{\bv}^{\top} \bcoefZ_{+} \end{array} \right]  
	+ \frac{1}{m} \bA_{\bu,\bv}^{\top} \left[ \begin{array}{c} \bWtil \bX\bu \;  \bW \bZ \bv \end{array} \right] \left[ \begin{array}{c} \theta(\bu) \\  \thetatil(\bv) \end{array} \right] \\
	&\qquad \qquad \qquad \qquad \qquad \qquad \qquad \qquad = \frac{1}{m} \bA_{\bu,\bv}^{\top}  (\by + \diag(\bW \bWtil)) + \lambda \left[ \begin{array}{c} \bO_{\bu}^{\top} \bcoefX_{\sharp} \\  \bO_{\bv}^{\top}\bcoefZ_{\sharp}  \end{array} \right],
\end{align*}
and re-arranging yields
\begin{align*}
	&\left[ \begin{array}{c} \bO_{\bu}^{\top} \bcoefX_{+} \\  \bO_{\bv}^{\top} \bcoefZ_{+} \end{array} \right] = \Big(\bA_{\bu,\bv}^{\top} \bA_{\bu,\bv} + \lambda m\bI \Big)^{-1} \\
	& \qquad \qquad \qquad \cdot \bigg(  \bA_{\bu,\bv}^{\top}  (\by + \diag(\bW \bWtil)) + \lambda m\left[ \begin{array}{c} \bO_{\bu}^{\top} \bcoefX_{\sharp} \\  \bO_{\bv}^{\top}\bcoefZ_{\sharp}  \end{array} \right] -   
	 \bA_{\bu,\bv}^{\top} \left[ \begin{array}{c} \bWtil \bX\bu \;  \bW \bZ \bv \end{array} \right] \left[ \begin{array}{c} \theta(\bu) \\  \thetatil(\bv) \end{array} \right] \bigg)
\end{align*}
Now, substituing the equation in the display above into Eq.~\eqref{eq-separate-uv-bOu-bOv}, re-arranging the terms and using $\bP_{\bu,\bv} = \bI - \bA_{\bu,\bv} \Big(\bA_{\bu,\bv}^{\top} \bA_{\bu,\bv} + \lambda m\bI \Big)^{-1}\bA_{\bu,\bv}^{\top}$ yields the desired result. \qed

\subsection{Proof of Lemma~\ref{concentration-of-M11-M22-M12-M1-M2}}\label{subsec:proof-concentration-M11-M22-M12-M1-M2}
In Section~\ref{sec:proof-lem-M11-a}, we provide the proof of Lemma~\ref{concentration-of-M11-M22-M12-M1-M2}(a), in Section~\ref{sec:proof-concentration-M11-b}, we provide the proof of Lemma~\ref{concentration-of-M11-M22-M12-M1-M2}(b), in Section~\ref{sec:proof-concentration-M11-c}, we provide the proof of Lemma~\ref{concentration-of-M11-M22-M12-M1-M2}(c), and finally, in Section~\ref{sec:proof-concentration-M11-d}, we provide the proof of Lemma~\ref{concentration-of-M11-M22-M12-M1-M2}(d).

\subsubsection{Proof of Lemma~\ref{concentration-of-M11-M22-M12-M1-M2}(a)}\label{sec:proof-lem-M11-a}
We bound $\EE\big\{\big|M_{11}-V_{1} \big| \big\}$, noting that identical steps yield the same bound on $\EE\big\{\big|M_{22} - V_{2}\big|\big\}$.

Consider the shorthand $G_{i} = \bx_{i}^{\top} \bcoefX_{\sharp}$ and $\widetilde{G}_{i} = \bz_{i}^{\top} \bcoefZ_{\sharp}$ for all $i \in [m]$.   We then apply the triangle inequality to decompose $| M_{11}  - V_1$ as
\[
\lvert M_{11} - V_1 \rvert \leq T_1 + T_2 + T_3,
\]
where 
\begin{align*}
	T_1 &= \bigg|M_{11} - \frac{ \trace( \bP_{\bu_2,\bv_2} \bWtil^{2})}{m} \bigg|, \qquad T_2 = \bigg| \frac{ \trace(\bP_{\bu_2,\bv_2} \bWtil^{2})}{m} - \frac{1}{m} \sum_{i=1}^{m} \frac{\widetilde{G}_{i}^{2}}{1 + \frac{\widetilde{G}_{i}^{2}}{r_1} + \frac{G_{i}^{2}}{r_2}} \bigg|\\
	& \qquad \qquad \text{ and } \qquad T_3 = \bigg| \frac{1}{m} \sum_{i=1}^{m} \frac{\widetilde{G}_{i}^{2}}{1 + \frac{ \widetilde{G}_{i}^{2}}{r_1} + \frac{G_{i}^{2}}{r_2}} - V_1\bigg|.
\end{align*}
We bound each of $T_1, T_2$, and $T_3$ in sequence. 

\medskip
\noindent \underline{Bounding $T_1$:}
Since $\bX \bu_2$ is gaussian and independent of $\bWtil$ and $\bP_{\bu_2,\bv_2}$, applying the Hanson--Wright inequality\citep[Theorem 6.2.1]{vershynin2018high} yields that for all $t\geq 0$
\begin{align*}
	\Pr \big \{ T_{1} \geq t \;\big \vert \; \bP_{\bu_2,\bv_{2}}, \bWtil \big \} \leq 2\exp\bigg\{ -c\min\bigg( \frac{m^{2}t^{2}}{\|\bWtil^{2}\|_{F}^{2}}, \frac{mt}{\|\bWtil^{2}\|_{2}} \bigg) \bigg\},
\end{align*}
where we used the pair of bounds $\big\|\bWtil \bP_{\bu_2,\bv_2} \bWtil\big\|_{F}^{2} \leq \big\|\bWtil^{2}\big\|_{F}^{2}$ and $\big\|\bWtil \bP_{\bu_2,\bv_2} \bWtil\big\|_{2} \leq \big\|\bWtil^{2}\big\|_{2}$, each of which holds since $\boldsymbol{0}\preceq\bP_{\bu_2,\bv_2} \preceq \bI$. Consequently, since $T_{1}\geq 0$, we integrate the tail to obtain
\begin{align*}
	 \EE\big\{ T_{1} \; \big \vert \; \bP_{\bu_2,\bv_{2}}, \bWtil \big\} &= \int_{0}^{\infty} \Pr \big\{ T_{1} \geq t \; \big \vert \; \bP_{\bu_2,\bv_{2}}, \bWtil \big\} \mathrm{d}t \lesssim \frac{\big\|\bWtil^{2}\big\|_{F}}{m} + \frac{\big\|\bWtil^{2}\big\|_{2}}{m} \lesssim \frac{\big\|\bWtil^{2}\big\|_{F}}{m}
\end{align*}
Consequently, 
\begin{align}\label{T1-expectation-bound-M11}
	\EE\big\{T_{1} \big\} = \EE \big\{ \EE\big\{ T_{1} \; \big \vert \; \bP_{\bu_2,\bv_{2}}, \bWtil \big\} \big\} \lesssim \frac{1}{m}\EE\big\{ \big\|\bWtil^{2} \big\|_{F} \big\} \leq \frac{1}{m} \EE \Big\{ \big\|\bWtil^{2}\big\|_{F}^{2} \Big\}^{1/2} \lesssim \frac{1}{\sqrt{m}}.
\end{align}

\medskip
\noindent \underline{Bounding $T_2$:} We use the shorthand $\bar{\bx}_{i} = \bO_{\bu_2}^{\top} \bx_{i}$ and $\bar{\bz}_{i} = \bO_{\bv_2}^{\top} \bz_{i}$. We let $\ba_{i}^{\top} = \big[\widetilde{G}_{i} \cdot \bar{\bx}_{i}^{\top} \;\vert\; G_{i} \cdot \bar{\bz}_{i}^{\top} \big]$. Note that $\ba_{i}^{\top}$ is the $i$-th row of $\bA_{\bu_2,\bv_2}$. Moreover, we let 
$\bSig =  \sum_{j=1}^{m} \ba_{j} \ba_{j}^{\top} + \lambda m \bI$ and $\bSig_{i} = \sum_{j \neq i} \ba_{j} \ba_{j}^{\top} + \lambda m \bI$.  Applying the Sherman--Morrison formula, we then re-write
\begin{align*}
	\trace\big(\bWtil \bP_{\bu_2,\bv_2} \bWtil\big) = \sum_{i=1}^{m} \widetilde{G}_{i}^{2} \bP_{\bu_2,\bv_2}(i,i) = \sum_{i=1}^{m} \widetilde{G}_{i}^{2} \big(1-\ba_{i}^{\top} \bSig^{-1} \ba_{i} \big) = \sum_{i=1}^{m} \frac{\widetilde{G}_{i}^{2}}{1 + \ba_{i}^{\top} \bSig_{i}^{-1} \ba_{i}}.
\end{align*} Consequently, we bound $\EE\{T_2\}$ as
\begin{align}\label{upper-bound-T2-concentration-M11}
	\EE\big\{T_2\big\} &= \EE \bigg\{\bigg| \frac{1}{m} \sum_{i=1}^{m} \bigg( \frac{\widetilde{G}_{i}^{2}}{1 + \ba_{i}^{\top} \bSig_{i}^{-1} \ba_{i}} - \frac{\widetilde{G}_{i}^{2}}{1 + \frac{\widetilde{G}_{i}^{2}}{r_1} + \frac{G_{i}^{2}}{r_2}} \bigg) \bigg| \bigg\} \nonumber
	\\&\leq \EE\Big\{\Big|\widetilde{G}_{i}^{2} \big( \ba_{i}^{\top} \bSig_{i}^{-1} \ba_{i} -\widetilde{G}_{i}^{2}/r_1 - G_{i}^{2}/r_2 \big)  \Big| \Big\}.
\end{align}
Using the definition of $\bB,\bC,\bD$ and the block matrix inverse formula in Eq.~\eqref{eq:block-matrix-formula},  we obtain that $\ba_{i}^{\top} \bSig_{i}^{-1} \ba_{i} = \widetilde{G}_{i}^{2} H_{1} + G_{i}^{2}H_{2} - 2G_{i}\widetilde{G}_{i} H_{3} $, where we set $ H_{1} =  \bar{\bx}_{i}^{\top} \big(\bB-\bC\bD^{-1}\bC^{\top} \big)^{-1} \bar{\bx}_{i}$, $H_{2} = \bar{\bz}_{i}^{\top} \big(\bD-\bC^{\top}\bB^{-1}\bC \big)^{-1} \bar{\bz}_{i}$ and $H_{3} =  \bar{\bx}_{i}^{\top} \big(\bB-\bC\bD^{-1}\bC^{\top} \big)^{-1} \bC\bD^{-1} \bar{\bz}_{i}$.
Applying the triangle inequality yields
\begin{align}\label{ineq-decomposition-T2}
	\Big| \ba_{i}^{\top} \bSig_{i}^{-1} \ba_{i} - (\widetilde{G}_{i}^{2}/r_1 + G_{i}^{2}/r_2) \Big| \leq \widetilde{G}_{i}^{2} \cdot \big|H_{1} - r_{1}^{-1} \big| + G_{i}^{2} \cdot \big|H_{2} - r_{2}^{-1} \big| + 2\big| G_{i}\widetilde{G}_{i} \big| \cdot \big|H_{3} \big|.
\end{align}
Combining inequalities~\eqref{upper-bound-T2-concentration-M11} and~\eqref{ineq-decomposition-T2} then yields
\begin{align}\label{decomposition-T2-M11}
	\EE\big\{T_{2}\big\} \leq \EE\big\{\widetilde{G}_{i}^{4} \cdot \big|H_{1} - r_{1}^{-1}\big|\big\} + \EE\big\{ G_{i}^{2}\widetilde{G}_{i}^{2} \cdot \big|H_{2} - r_{2}^{-1} \big|\big\} + 2\EE\big\{\big| G_{i}\widetilde{G}_{i}^{3}\big| \cdot \big|H_{3}\big| \big\}.
\end{align}
We then bound each term in the RHS of the inequality in the display above. We will make use of the shorthand
\[
	\bP_{11} = \big(\bB-\bC\bD^{-1}\bC^{\top} \big)^{-1},\; \bP_{12} = -\big(\bB-\bC\bD^{-1}\bC^{\top}\big)^{-1} \bC\bD^{-1} ,\; \bP_{22} = \big(\bD-\bC^{\top}\bB^{-1}\bC\big)^{-1}.
\]
Applying Hanson-Wright Inequality yields 
\begin{align*}
	\Pr\Big\{ \big|H_1 - \trace\big(\bP_{11} \big) \big| \geq t \Big\} \overset{\1}{\leq} 2\exp \Big\{ -c \min\Big\{ \frac{(\lambda m)^{2} t^{2}}{d}, \lambda m t\Big\} \Big\} \overset{\2}{\leq} 2\exp\big\{-cd\min\{t^{2},t\}\big\},
\end{align*}
where step $\1$ follows from the bound $\big\|\bP_{11}\big\|_{2} \leq \big\|\bSig_{i}^{-1}\big\|_{2} \leq (\lambda m)^{-1}$ since $\bP_{11}$ is a submatrix of $\bSig_{i}^{-1}$ and step $\2$ follows by invoking the assumption $\lambda m \geq d $. Applying Lemma~\ref{lemma:probability-bound-trace-inverse} yields that for $t \gtrsim \log^{1.5}(d)/\sqrt{d}$,
\[
	\Pr\Big\{ \big|\trace\big( \bP_{11} \big) - r_{1}^{-1}\big| \geq t \Big\} \leq 2\exp \bigg\{   \frac{-cd^{2}t^{2}}{m\log^{2}(d)}\bigg\} + d^{-15}.
\]
Combining the two pieces yields that for $t \gtrsim \log^{1.5}(d)/\sqrt{d}$,
\begin{align}\label{ineq-bound-term1-T2}
	\Pr\Big\{ \big|H_1 - r_{1}^{-1} \big| \geq t \Big\} \leq 4\exp \big\{- c d\log^{-2}(d) \min\{t^{2},t\} \big\} + d^{-15}.
\end{align}
Consequently, we obtain that
\begin{align*}
	\EE\Big\{ \big(H_1 - r_{1}^{-1} \big)^{2} \Big\} &\leq \frac{C\log^{3}(d)}{d} + \EE\Big\{\big(H_1 - r_{1}^{-1} \big)^{2} \cdot \mathbbm{1}\big\{\big|H_1 - r_{1}^{-1} \big| \geq C \log^{1.5}(d)/\sqrt{d} \big\} \Big\} \\&\leq
	\frac{C\log^{3}(d)}{d} + \EE\Big\{\big|H_1 - r_{1}^{-1} \big|^{4} \Big\}^{1/2} \cdot \Pr\Big\{ \big|H_1 - r_{1}^{-1} \big| \geq C\log^{1.5}(d)/\sqrt{d}  \Big\}^{1/2} 
	\\&\leq \frac{C\log^{3}(d)}{d} + \EE \Big\{\big|H_1 - r_{1}^{-1} \big|^{4} \Big\}^{1/2} \cdot d^{-7}
\end{align*}
where the second step follows from the Cauchy--Schwarz inequality. Continuing, we obtain that
\[
	\EE\Big\{ \big|H_1 - r_{1}^{-1} \big|^{4} \Big\} \lesssim \EE \big\{H_{1}^{4}\big\} + r_{1}^{-4} \lesssim 1, 
\]
where we use $r_{1} \geq \lambda m/d \geq 1$ by definition of $r_{1}$~\eqref{eq:fixed-point} and $\lambda m\geq d$, and
\[
	\EE\big\{ \big|H_1\big|^{4} \big\} = \EE\Big\{ \big(\bar{\bx}_{i}^{\top} \big(\bB-\bC\bD^{-1}\bC^{\top} \big)^{-1} \bar{\bx}_{i} \big)^{4} \Big\} \leq \frac{\EE\Big\{ \big\|\bar{\bx}_{i}\big\|_{2}^{8}\Big\}}{(\lambda m)^{4}} \lesssim 1.
\]
Consequently, we obtain that $\EE\big\{(H_1 - r_{1}^{-1} )^{2}\big\} \lesssim \log^{3}(d)/d$.  Thus, applying the Cauchy--Schwarz inequality once more yields
\begin{align}\label{H_1-r_1-bound}
	\EE\Big\{\widetilde{G}_{i}^{4} \cdot \big|H_{1} - r_{1}^{-1}\big| \Big\} \leq \EE\Big\{\widetilde{G}_{i}^{8}\Big\}^{1/2} \cdot \EE \Big\{\big(H_1 - r_{1}^{-1} \big)^{2}\Big\}^{1/2} \lesssim \log^{1.5}(d)/\sqrt{d}.
\end{align}
Following the identical steps, we obtain 
\begin{align}\label{H_2-r_2-bound}
	\EE\Big \{ G_{i}^{2}\widetilde{G}_{i}^{2} \cdot \big|H_{2} - r_{2}^{-1}\big| \Big\} \leq \EE\Big\{G_{i}^{4}\widetilde{G}_{i}^{4}\Big\}^{1/2} \cdot \EE \Big\{\big(H_2 - r_{2}^{-1} \big)^{2}\Big\}^{1/2} \lesssim \log^{1.5}(d)/\sqrt{d}.
\end{align}
We turn now to bounding $H_3$.  We first apply Hoeffding’s inequality to obtain, for all $t \geq 0$,
\[
	\Pr\big \{\big|H_{3}\big| \geq t \; \big \vert \; \bP_{12},\bar{\bz}_{i} \big\} \leq 2 \exp\bigg\{ \frac{-t^{2}}{2\|\bP_{12} \bar{\bz}_{i}\|_{2}^{2}}\bigg\} \leq 2 \exp\bigg\{ \frac{-t^{2}}{2\|\bP_{12}\|_{2}^{2}\| \bar{\bz}_{i}\|_{2}^{2}}\bigg\} \leq 2 \exp\bigg\{ \frac{-(\lambda m)^{2}t^{2}}{2\| \bar{\bz}_{i}\|_{2}^{2}}\bigg\},
\]
where the last step follows from the bound $\big\| \bP_{12} \big\|_{2} \leq \big\| \bSig_{i}^{-1} \big\|_{2} \leq (\lambda m)^{-1} $.  Integrating the tail yields the bound 
\[
	\EE\big\{ H_{3}^{2}\; \big \vert \; \bP_{12},\bar{\bz}_{i} \big\} \leq \int_{0}^{\infty} \Pr\Big\{ H_{3}^{2} \geq t \;|\; \bP_{12},\bar{\bz}_{i} \Big\}\mathrm{d}t
	\leq \int_{0}^{\infty} 2 \exp\bigg\{ \frac{-(\lambda m)^{2}t}{2 \big\| \bar{\bz}_{i} \big\|_{2}^{2}}\bigg\} \mathrm{d}t \lesssim \frac{\big\| \bar{\bz}_{i} \big\|_{2}^{2}}{(\lambda m)^{2}}.
\]
Thus,
\[
 \EE\big\{ H_{3}^{2} \big\} \leq \EE\big\{ \EE\big\{ H_{3}^{2}\; \big \vert \; \bP_{12},\bar{\bz}_{i} \big\} \big\} \lesssim \frac{1}{(\lambda m)^{2}}\EE\Big\{ \big \| \bar{\bz}_{i} \big \|_{2}^{2}\Big\} \leq \frac{d}{(\lambda m)^{2}} \leq 1/d,
\]
where the last step follows by invoking the assumption $\lambda m\geq d$. Consequently, applying the Cauchy--Schwarz inequality yields
\begin{align}\label{H_3-r_2-bound}
	\EE\big\{ \big| G_{i}\widetilde{G}_{i} \big| \cdot \big|H_{3} \big| \big\} \leq \EE\big\{ G_{i}^{2}\widetilde{G}_{i}^{2} \big\}^{1/2} \cdot \EE\big\{ H_{3}^{2} \big\}^{1/2} \lesssim 1/\sqrt{d}.
\end{align}
Substituting inequalities~\eqref{H_1-r_1-bound}, ~\eqref{H_2-r_2-bound} and~\eqref{H_3-r_2-bound} into the RHS of inequality~\eqref{decomposition-T2-M11} yields
\begin{align}\label{T2-expectation-bound-M11}
	\EE\big\{T_{2}\big\} \lesssim \log^{1.5}(d)/\sqrt{d},
\end{align}
which completes the bound on $T_2$.

\medskip
\noindent \underline{Bounding $T_3$:} Note that $\big\{ \widetilde{G}_{i}^{2}/ \big(1 + \widetilde{G}_{i}^{2}/r_1 + G_{i}^{2}/r_2 \big) \big\}_{i=1}^{m} $ form a collection of i.i.d. sub-exponential random variables whose expectations are $V_1$~\eqref{eq:V-V1-V2}. Integrating the tail and applying Bernstein’s inequality~\citep[Theorem 2.8.2]{vershynin2018high} thus yields the inequality
\begin{align}\label{T3-expectation-bound-M11}
	\EE\big\{T_{3} \big\} \leq \int_{0}^{\infty} \Pr\big\{ T_{3} \geq t\big\} \mathrm{d}t \leq \int_{0}^{\infty} 2\exp\big\{-cm\min\{t^{2},t\}\big\} \mathrm{d}t \lesssim \frac{1}{\sqrt{m}}.
\end{align}

\medskip
\noindent \underline{Putting the pieces together:} Combining inequalities~\eqref{T1-expectation-bound-M11}, ~\eqref{T2-expectation-bound-M11} and~\eqref{T3-expectation-bound-M11} yields
\begin{align}\label{ineq:final-bound-M11-V1}
	\EE\big\{ \big|M_{11} - V_1\big| \big\} \lesssim \frac{1}{m} + \frac{\log^{1.5}(d)}{\sqrt{d}} + \frac{1}{\sqrt{m}} \lesssim \frac{\log^{1.5}(d)}{\sqrt{d}} + \frac{1}{\sqrt{m}}.
\end{align}
where the last step follows from the assumption $m\leq d$.  This yields the desired bound on $\EE\big\{ \big|M_{11} - V_1\big| \big\}$.  Turning to $\vert M_{22} - V_{2} \vert$, we note that 
\begin{align*}
	\big|M_{22} - V_{2}\big| &\leq \bigg|M_{22} - \frac{ \trace\big( \bP_{\bu_2,\bv_2} \bW^{2}\big)}{m} \bigg| + \bigg| \frac{ \trace\big(\bP_{\bu_2,\bv_2} \bW^{2}\big)}{m} - \frac{1}{m} \sum_{i=1}^{m} \frac{G_{i}^{2}}{1 + \frac{\widetilde{G}_{i}^{2}}{r_1} + \frac{G_{i}^{2}}{r_2}} \bigg| \\
	&\qquad\qquad+ 
	\bigg| \frac{1}{m} \sum_{i=1}^{m} \frac{G_{i}^{2}}{1 + \frac{ \widetilde{G}_{i}^{2}}{r_1} + \frac{G_{i}^{2}}{r_2}} - V_2\bigg|.
\end{align*}
The rest of the proof is identical to the bound on $\EE\big\{\big|M_{11} - V_{1}\big|\big\}$ and is omitted for brevity. \qed

\subsubsection{Proof of Lemma~\ref{concentration-of-M11-M22-M12-M1-M2}(b.)}\label{sec:proof-concentration-M11-b} Note that, by construction, $\bX\bu_{2}$ is independent of the tuple of random variables $\big(\bWtil,\bW,\bP_{\bu_{2},\bv_{2}}, \bZ\bv_{2}\big)$ by definition. Consequently, applying Hoeffding’s inequality yields
\begin{align*}
	\EE\big\{M_{12}^{2}\big\} &= \EE\Big\{ \EE\Big\{M_{12}^{2} \;|\; \big(\bWtil,\bW,\bP_{\bu_{2},\bv_{2}},\bZ\bv_{2} \big)\Big\} \Big\}
	\lesssim \frac{1}{m^{2}} \EE\Big\{ \big\| \bWtil\bP_{\bu_{2},\bv_{2}} \bW \bZ\bv_{2} \big\|_{2}^{2} \Big\}
	\\ &= \frac{1}{m^2}\EE\Big\{ \big\| \bWtil\bP_{\bu_{2},\bv_{2}} \bW \big\|_{F}^{2} \Big\} = \frac{1}{m^2}\EE\Big\{ \trace\big(\bW \bP_{\bu_2,\bv_2} \bWtil^{2} \bP_{\bu_2,\bv_2} \bW \big) \Big\}\\
	& \leq \frac{1}{m^2}\EE\Big\{ \trace\big(\bW  \bWtil^{2} \bW \big) \Big\} \lesssim \frac{1}{m},
\end{align*}
as desired. \qed

\subsubsection{Proof of Lemma~\ref{concentration-of-M11-M22-M12-M1-M2}(c)}\label{sec:proof-concentration-M11-c}
Recall the shorthand $
	G_{i} = \bx_{i}^{\top}\bcoefX_{\sharp}$, $\widetilde{G}_{i} = \bz_{i}^{\top} \bcoefZ_{\sharp}$, $\bar{\bx}_{i} = \bO_{\bu_{1}}^{\top} \bx_{i}$, $\bar{\bz}_{i} = \bO_{\bv_{1}}^{\top} \bz_{i}$, $\ba_{i}^{\top} = \big[ \widetilde{G}_{i} \bar{\bx}_{i}^{\top}\; \vert\; G_{i}\bar{\bz}_{i}^{\top} \big]$ and $\bSig = \sum_{i=1}^{n}\ba_{i}\ba_{i}^{\top} +\lambda m \bI.$
Noting $\ba_{i}^{\top}$ is the $i$-th row of $\bA_{\bu_{1},\bv_{1}}$ and expanding $M_{1}$ yields
\begin{align*}
M_{1} &= \frac{1}{m}\sum_{i=1}^{m}G_{i}^{2}\widetilde{G}_{i}^{2} \big(1-\ba_{i}^{\top} \bSig^{-1} \ba_{i} \big)  - \frac{1}{m}\sum_{i\neq j} G_{i}G_{j}\widetilde{G}_{i}\widetilde{G}_{j}\ba_{i}^{\top}\bSig^{-1}\ba_{j} \\&= 
\frac{1}{m}\sum_{i=1}^{m} \frac{ G_{i}^{2}\widetilde{G}_{i}^{2} }{1+\ba_{i}^{\top} \bSig_{i}^{-1} \ba_{i}} - \frac{1}{m}\sum_{i\neq j} G_{i}G_{j}\widetilde{G}_{i}\widetilde{G}_{j}\ba_{i}^{\top}\bSig^{-1}\ba_{j},
\end{align*}
where the last step follows by noting $\bSig_{i} = \bSig - \ba_{i}\ba_{i}^{\top}$ and applying the Sherman--Morrison formula to $\bSig^{-1}$. Applying the triangle inequality, we obtain the decomposition
\begin{align}\label{ineq:decomposition-M1}
	\big|M_{1} - V\big| \leq T_1 + T_2 + T_3,
\end{align}
where
\begin{align*}
		T_1 &= \bigg| \frac{1}{m}\sum_{i=1}^{m} \frac{ G_{i}^{2}\widetilde{G}_{i}^{2} }{1+\ba_{i}^{\top}\bSig_{i}^{-1} \ba_{i} } - \frac{1}{m}\sum_{i=1}^{n}\frac{G_{i}^{2} \widetilde{G}_{i}^{2}}{1+\frac{G_{i}^{2}}{r_{1}} + \frac{\widetilde{G}_{i}^{2}}{r_{2}}}\bigg|, \qquad T_2 = \bigg| \frac{1}{m}\sum_{i=1}^{n}\frac{G_{i}^{2} \widetilde{G}_{i}^{2}}{1+\frac{G_{i}^{2}}{r_{1}} + \frac{\widetilde{G}_{i}^{2}}{r_{2}}} - V \bigg|\\
		 &\qquad \qquad \text{ and } \qquad T_3 =  \bigg| \frac{1}{m} \sum_{i\neq j} G_{i}G_{j}\widetilde{G}_{i}\widetilde{G}_{j}\ba_{i}^{\top}\bSig^{-1}\ba_{j} \bigg|.
\end{align*}
Note that proceeding similarly in the proof of $\big|M_{11} -V_1\big|$ yields the bound $\EE\big\{T_1 + T_2\big\} \lesssim \log^{1.5}(d)/\sqrt{d} + \log^{3}(m)/\sqrt{m}$ so we omit the details of bounding $T_1$ and $T_2$.  In order to bound $\EE\big\{T_3\big\} \lesssim \log^{6}(d)/\sqrt{d}$, we apply the method of typical bounded differences with regularity set $\mathcal{S} \subseteq \mathbb{R}^{2d \times m}$ defined as
\begin{align}\label{definition-S-T3-M1}
	\mathcal{S} = \bigg\{ & \{\bx_{k},\bz_{k}\}_{k=1}^{m}: |G_{i}|,|\widetilde{G}_{i}| \leq C \sqrt{\log(d)},\; \|\bx_{i}\|_{2}, \|\bz_{i}\|_{2} \leq C\sqrt{d},\; \forall\; i\in[m], \nonumber
	\\& \Big|\ba_{i}^{\top} \bSig_{i}^{-1}\ba_{j} \Big|, \Big|\ba_{k}^{\top} \bSig_{ij}^{-1}\ba_{j} \Big| \leq \frac{C\log^{2}(d)\sqrt{d}}{\lambda m},\;\forall\; i\neq j \neq k \in [m], 
	\\&\Big| \ba_{i}^{\top} \bSig_{i}^{-1} \sum_{k\neq i}G_{k}\widetilde{G}_{k}\ba_{k}\Big|, \Big| \ba_{j}^{\top} \bSig_{ij}^{-1} \sum_{k\neq i,j}G_{k}\widetilde{G}_{k}\ba_{k}\Big| \leq \frac{C\log^{2.5}(d) \sqrt{md}}{\lambda m},\;\forall \;i\neq j \in [m]. \nonumber
	\bigg\},
\end{align}
where $\bSig_{i} = \sum_{k\neq i} \ba_{k}\ba_{k}^{\top} + \lambda m \bI$ and $\bSig_{ij} = \sum_{k\neq i,j} \ba_{k}\ba_{k}^{\top} + \lambda m \bI$ for each $i,j \in [m]$.  As the typical bounded differences argument has been carried through at several points in the manuscript, we omit the details for brevity. 

\subsubsection{Proof of Lemma~\ref{concentration-of-M11-M22-M12-M1-M2}(d)}\label{sec:proof-concentration-M11-d} By definition of $M_{3}$~\eqref{claim-alpha-iota-concentration}, we expand and decompose $M_{3}$ as
\begin{align*}
	M_{3} = \frac{1}{m}\diag\big(\bW\bWtil\big)^{\top} \bP_{\bu_1,\bv_1} \bX\bu_{2}\odot \bZ\bv_{2} = T_1 - T_2 - T_3,
\end{align*}
where 
\begin{align*}
	&T_1 = \frac{1}{m}\diag\big(\bW\bWtil \big)^{\top}\bX\bu_{2} \odot \bZ\bv_{2}, \qquad T_2 = \frac{1}{m} \sum_{i=1}^{m} G_{i}\widetilde{G}_{i} \bx_{i}^{\top}\bu_{2} \bz_{i}^{\top}\bv_{2} \ba_{i}^{\top}\bSig^{-1} \ba_{i}\\
	& \qquad \text{ and } \qquad  T_3 = \frac{1}{m} \sum_{i\neq j} G_{i}\widetilde{G}_{i} \bx_{j}^{\top}\bu_{2} \bz_{j}^{\top}\bv_{2} \ba_{i}^{\top}\bSig^{-1} \ba_{j}.
\end{align*}
The proof consists of bounding each of the absolute first moments $\EE\{\vert T_1 \vert\}, \EE\{\vert T_2 \vert\}$, and $\EE\{\vert T_3 \vert\}$ in turn.

\paragraph{Bounding $\EE\{|T_{1}|\}$:} Note that, by construction, $\bX\bu_{2}$ is independent of the tuple of random variables $\big(\bW,\bWtil,\bZ\bv_{2}\big)$. We thus compute
\begin{align}\label{T1-bound-last-term}
% \hspace{-1cm}
	\EE\big\{\big|T_{1}\big|\big\} &= \EE\Big\{ \EE\Big\{ \big|T_{1}\big| \;|\; \bW,\bWtil,\bZ\bv_{2} \Big\} \Big\} \lesssim \frac{1}{m}\EE \Big\{\big\|\bW\bWtil \bZ\bv_{2}\big\|_{2} \Big\}  \nonumber
	\\&\leq \frac{1}{m} \EE \Big\{ \big\|\bW\bWtil \bZ\bv_{2}\big\|_{2}^{2} \Big\}^{1/2} \leq \frac{1}{\sqrt{m}}.
\end{align}

\paragraph{Bounding $\EE\{|T_{2}|\}$:} Applying the rank-one update formula to compute $\bSig^{-1}$, we re-write
\[
	T_2 = \frac{1}{m} \sum_{i=1}^{m} \frac{ G_{i}\widetilde{G}_{i} \bx_{i}^{\top}\bu_{2} \bz_{i}^{\top}\bv_{2} \cdot \ba_{i}^{\top}\bSig_{i}^{-1} \ba_{i}}{1+\ba_{i}^{\top} \bSig_{i}^{-1} \ba_{i}}.
\]
We then apply the triangle inequality to bound $\vert T_2 \vert$ as
\begin{align*}
	|T_{2}| &\leq \bigg| \frac{1}{m} \sum_{i=1}^{m}  \frac{ G_{i}\widetilde{G}_{i} \bx_{i}^{\top}\bu_{2} \bz_{i}^{\top}\bv_{2} \cdot \ba_{i}^{\top}\bSig_{i}^{-1} \ba_{i}}{1+\ba_{i}^{\top} \bSig_{i}^{-1} \ba_{i}} - \frac{ G_{i}\widetilde{G}_{i} \bx_{i}^{\top}\bu_{2} \bz_{i}^{\top}\bv_{2} \big( \frac{\widetilde{G}_{i}^{2}}{r_{1}} + \frac{G_{i}^{2}}{r_{2}} \big)}{1+ G_{i}^{2}/r_{2} + \widetilde{G}_{i}^{2}/r_{1}}  \bigg| \\
	& \qquad \qquad \qquad \qquad \qquad \qquad \qquad \qquad+ \bigg| \frac{1}{n} \sum_{i=1}^{m} \frac{ G_{i}\widetilde{G}_{i} \bx_{i}^{\top}\bu_{2} \bz_{i}^{\top}\bv_{2} \big(\frac{G_{i}^{2}}{r_{2}} + \frac{\widetilde{G}_{i}^{2}}{r_{1}} \big)}{1+ G_{i}^{2}/r_{2} + \widetilde{G}_{i}^{2}/r_{1}} \bigg| 
	\\&\leq\frac{1}{m} \sum_{i=1}^{m} \big|G_{i}\widetilde{G}_{i} \bx_{i}^{\top}\bu_{2} \bz_{i}^{\top}\bv_{2} \big| \cdot \Big|\ba_{i}^{\top} \bSig_{i}^{-1} \ba_{i}  - \widetilde{G}_{i}^{2}/r_{1} - G_{i}^{2}/r_{2}) \Big| + \Big| \frac{1}{m}\sum_{i=1}^{m}\bx_{i}^{\top}\bu_{2} \delta_{i} \Big|,
\end{align*}
where $\delta_{i} = G_{i}\widetilde{G}_{i}\bz_{i}^{\top}\bv_{2} \big(G_{i}^{2}/r_{2} + \widetilde{G}_{i}^{2}/r_{1} \big)/\big(1+G_{i}^{2}/r_{2} + \widetilde{G}_{i}^{2}/r_{1}\big)$. Taking expectation on both sides of the inequality in the display above yields
\[
	\EE\big\{\big|T_{2}\big|\big\} \leq \underbrace{\EE\Big\{ \big|G_{i}\widetilde{G}_{i} \bx_{i}^{\top}\bu_{2} \bz_{i}^{\top}\bv_{2} \big| \cdot \Big|\ba_{i}^{\top} \bSig_{i}^{-1} \ba_{i} - G_{i}^{2}/r_{2} - \widetilde{G}_{i}^{2}/r_{1}) \Big| \Big\}}_{T'_{2}} + \underbrace{\EE\Big\{\Big| \frac{1}{m}\sum_{i=1}^{m}\bx_{i}^{\top}\bu_{2} \delta_{i} \Big| \Big\}}_{T''_{2}}.
\]
Following the identical steps of proof of inequality~\eqref{T2-expectation-bound-M11} yields the bound $T'_{2} \lesssim \log^{1.5}(d)/\sqrt{d}$. Note that $\{\delta_{i}\}_{i=1}^{n}$ is independent of $\{\bx_{i}^{\top} \bu_{2}\}_{i=1}^{n}$. Consequently, applying Hoeffding’s inequality yields
\[
	T''_{2} = \EE\bigg\{ \EE\bigg\{ \Big| \frac{1}{m}\sum_{i=1}^{m}\bx_{i}^{\top}\bu_{2} \delta_{i} \; \Big| \; \{\delta_{i}\}_{i=1}^{m} \bigg\} \bigg\} \lesssim \EE\bigg\{ \frac{1}{m} \Big( \sum_{i=1}^{m}\delta_{i}^{2} \Big)^{1/2}\bigg\} 
	\leq \frac{1}{m} \EE\bigg\{ \sum_{i=1}^{m}\delta_{i}^{2} \bigg\}^{1/2} \lesssim \frac{1}{\sqrt{m}},
\]
where the last step follows from $\EE \big\{\delta_{i}^{2} \big\} \lesssim 1$. Putting the pieces together, we obtain the bound
\begin{align}\label{T2-bound-last-term}
	\EE\big\{\big|T_{2}\big|\big\}  \lesssim \frac{\log^{1.5}(d)}{\sqrt{d}} +  \frac{1}{\sqrt{m}}.
\end{align}

\paragraph{Bounding $\EE\{|T_{3}|\}$:} Proceeding in a parallel fashion to the proof of the step~\eqref{definition-S-T3-M1} yields the bound $\EE\{|T_{3}|\} \lesssim \log^{6}(d)/\sqrt{d}$.

\paragraph{Putting the pieces together:} Applying inequalities~\eqref{T1-bound-last-term},~\eqref{T2-bound-last-term} and the bound of $\EE\big\{\big|T_3\big|\big\}$ yields $\EE\big\{\big|M_3\big|\big\} \lesssim \log^{6}(d)/\sqrt{d} + 1/\sqrt{m}$. \qed

\section{Deferred proofs for the orthogonal component}
This section is organized as follows.  In Section~\ref{sec:proof-stochastic-error-orthogonal-component}, we provide the proof of Lemma~\ref{lem:stochastic-error-orthogonal-component}.  Then, in Section~\ref{sec:proof-ineq-step1-eta-component-expectation}, we provide the proof of Lemma~\ref{lem:ineq-step1-eta-component-expectation}.  In Section~\ref{sec:proof-ineq3-step2-eta-component}, we provide the proof of Lemma~\ref{lem:ineq3-step2-eta-component}.  Finally, in Section~\ref{sec:proof-ineq-step3-eta-component}, we provide the proof of Lemma~\ref{lem:ineq-step3-eta-component}.  

\subsection{Proof of Lemma~\ref{lem:stochastic-error-orthogonal-component}}\label{sec:proof-stochastic-error-orthogonal-component}
We only prove the concentration of $\etaX_{+}^{2}$ since the proof for $\etaZ_{+}^2$ is identical.  The proof employs the method of typical bounded differences with the function $f:\mathbb{R}^{(2d+1)\times m} \rightarrow \mathbb{R}$ as
\begin{align}\label{definition-f-rho-eta}
	f\big( \{\bx_{i},\bz_{i},\epsilon_{i}\}_{i=1}^{m} \big) = \etaX_{+}^{2} = \big\|\bP_{\mathsf{span}\{\bcoefX_{\star},\bcoefX_{\sharp}\}}^{\perp} \bcoefX_{+} \big\|_2^{2},
\end{align}
 and the regularity set $\mathcal{S} \subseteq \mathbb{R}^{(2d+1) \times m}$ defined as
 \begin{align}\label{definition-regularity-set-eta-component-expectation}
 	\mathcal{S} = \bigg\{ &\{\bx_{i},\bz_{i},\epsilon_{i}\}_{i=1}^{m} \in \mathbb{R}^{(2d+1) \times m}: \quad \forall \; i,j \in [m], \nonumber \\ &\|\bx_{i}\|_{2},\; \|\bz_{i}\|_{2} \leq C\sqrt{d}, \quad |\bx_{i}^{\top} \bcoefX_{\star}|,\; |\bx_{i}^{\top} \bcoefX^{\perp}|,\; |\bz_{i}^{\top} \bcoefZ_{\star}|,\; |\bz_{i}^{\top} \bcoefZ^{\perp}| \leq C\sqrt{\log d},\quad |\epsilon_{i}| \leq C \sigma \sqrt{\log d}, \nonumber
 	\\& |\bx_{i}^{\top} (\bcoefX_{\sharp} - \bcoefX^{-(i)})| \leq C \sqrt{\log d} \cdot  \| \bcoefX_{\sharp} - \bcoefX^{-(i)} \|_{2}, \quad |\bx_{i}^{\top} (\bcoefX_{\sharp} - \bcoefX^{-(i,j)})| \leq C \sqrt{\log d} \cdot \| \bcoefX_{\sharp} - \bcoefX^{-(i,j)} \|_{2},  \nonumber
 	\\& |\bz_{i}^{\top} (\bcoefZ_{\sharp} - \bcoefZ^{-(i)})| \leq C \sqrt{\log d} \cdot \| \bcoefZ_{\sharp} - \bcoefZ^{-(i)} \|_{2},\quad |\bz_{i}^{\top} (\bcoefZ_{\sharp} - \bcoefZ^{-(i,j)})| \leq C \sqrt{\log d} \cdot \| \bcoefZ_{\sharp} - \bcoefZ^{-(i,j)} \|_{2}, \nonumber
 	\\& \Big| \Big\langle \bSig_{i}^{-1} \ba_{i}, \bP_{V}^{\perp} \begin{bmatrix} \bcoefX^{-(i)} \\ \bcoefZ^{-(i)} \\ \end{bmatrix}   \Big\rangle \Big| \leq \frac{C\log(d)}{\lambda m} \Big\| \bP_{V}^{\perp} \begin{bmatrix} \bcoefX^{-(i)} \\ \bcoefZ^{-(i)} \\ \end{bmatrix} \Big\|_{2} \quad \text{and} 
 	\\& \Big| \Big\langle \bSig_{ij}^{-1} \ba_{j}, \bP_{V}^{\perp} \begin{bmatrix} \bcoefX^{-(i,j)} \\ \bcoefZ^{-(i,j)} \\ \end{bmatrix}   \Big\rangle \Big| \leq \frac{C\log(d)}{\lambda m} \Big\| \bP_{V}^{\perp} \begin{bmatrix} \bcoefX^{-(i,j)} \\ \bcoefZ^{-(i,j)} \\ \end{bmatrix} \Big\|_{2}  \quad \text{for}\quad V = \mathsf{span}(\bcoefX_{\star},\bcoefX_\sharp) \times \mathbb{R}^{d}  \bigg\}, \nonumber
 \end{align}
 where $C$ denotes a large enough positive universal constant.  The remainder of the proof follows similar steps to Section~\ref{sec:stochastic-error-parallel-component}, whence we omit the details. \qed

\subsection{Proof of Lemma~\ref{lem:ineq-step1-eta-component-expectation}}\label{sec:proof-ineq-step1-eta-component-expectation}
We first recall the definition of $V_{1}$ as in Eq~\eqref{eq:V-V1-V2} as well as the characterization of $\langle \bcoefX_{+}, \bu \rangle$ as in~\eqref{eq-closed-form-u3-component} and apply the triangle inequality to obtain the decomposition
\begin{align*}
	&\Big| \EE\{\langle \bcoefX_{+}, \bu \rangle^{2}\} - \frac{\EE\big\{ \|\widetilde{\bW} \bR_{\bu,\bv}\|_{2}^{2} \big\}}{m^{2}(\lambda + V_{1})^{2}} \Big| \leq T_{1} + T_{2} + T_{3}, \text{ where }
\end{align*}
\begin{align*}
	&T_{1} = \EE\Bigg\{ \frac{ 2\big| (\widetilde{\bW}\bX \bu )^{\top} \bR_{\bu,\bv} \cdot \frac{M_{12}}{M_{22}+\lambda} (\bW\bZ\bv)^{\top} \bR_{\bu,\bv} \big| }{ m^{2}\Big(M_{11}+\lambda - \frac{M_{12}^{2}}{M_{22}+\lambda} \Big)^{2} } \Bigg\}, \qquad T_{2} = \EE\Bigg\{ \frac{  \frac{M_{12}^{2}}{(M_{22}+\lambda)^{2}} \Big( (\bW\bZ\bv)^{\top} \bR_{\bu,\bv} \Big)^{2} }{ m^{2}\Big(M_{11}+\lambda - \frac{M_{12}^{2}}{M_{22}+\lambda} \Big)^{2} } \Bigg\}\\
	&\qquad \qquad \text{ and } \qquad
	T_{3} = \Bigg|\EE\Bigg\{ \frac{ \Big( (\widetilde{\bW}\bX \bu )^{\top} \bR_{\bu,\bv} \Big)^{2} }{ m^{2}\Big(M_{11}+\lambda - \frac{M_{12}^{2}}{M_{22}+\lambda} \Big)^{2} } \Bigg\} - \frac{\EE\big\{ \|\widetilde{\bW} \bR_{\bu,\bv}\|_{2}^{2} \big\}}{m^{2}(\lambda + V_{1})^{2}} \Bigg|.
\end{align*}
The main bulk of the proof consists of establishing each of the following three inequalities
\begin{subequations}\label{ineq:bound-T1-bias-orth-Ruv-term}
	\begin{align}
		T_1 &\leq \frac{\Big(\sigma^{2} + \Err_{\sharp} \Big) \log^{2}(m)}{(\lambda m)\cdot \lambda \sqrt{m}}, \label{ineq:bound-T1-bias-orth-Ruv-term-a}\\
		T_2 &\leq \frac{\big(\sigma^{2} + \Err_{\sharp} \big)\log^{3}(m)}{(\lambda m)^{2}}, \qquad \text{ and } \label{ineq:bound-T1-bias-orth-Ruv-term-b}\\
		T_3 &\leq \frac{\sigma^{2} + \Err_{\sharp} }{\lambda m} \cdot \Big( \frac{\log(m)}{\lambda} \max\Big(\frac{\log(m)}{\sqrt{m}},\frac{\log^{1.5}(d)}{\sqrt{d}} \Big)+ \frac{\log^{3}(m)}{\lambda m}\Big) \label{ineq:bound-T1-bias-orth-Ruv-term-c}
	\end{align}
\end{subequations}
Combining these three inequalities with the decomposition at the beginning of the section yields the bound
\begin{align}
	\Big| \EE\{\langle \bcoefX_{+}, \bu \rangle^{2}\} - \frac{\EE\big\{ \|\widetilde{\bW} \bR_{\bu,\bv}\|_{2}^{2} \big\}}{m^{2}(\lambda + V_{1})^{2}} \Big| \lesssim \frac{\sigma^{2} + \Err_{\sharp}}{\lambda m} \cdot  \frac{\log^{3}(d)}{\lambda \sqrt{m}},
\end{align} 
as desired.  We turn now to establishing each of the three inequalities~\eqref{ineq:bound-T1-bias-orth-Ruv-term}.

\paragraph{Proof of the inequality~\eqref{ineq:bound-T1-bias-orth-Ruv-term-a}:} Recall the quantities $M_{11}, M_{22},$ and $M_{12}$ as defined in~\eqref{def:M11M12M22} and apply the Cauchy--Schwarz inequality to obtain the bound $M_{11}\cdot M_{22} \geq M_{12}^{2}$. Consequently, since $\lambda \geq 0$, $M_{11} \geq M_{12}^{2}/(\lambda + M_{22})$.  Note further that since $\bP_{\bu,\bv} \succeq \boldsymbol{0}$, we have the lower bound $M_{22}\geq 0$.  We thus upper bound $T_1$ as
\begin{align}\label{ineq-T1-eta-expectation}
	T_{1} &\leq \frac{2}{(\lambda n)^{3}} \EE\Big\{ \Big| (\widetilde{\bW}\bX\bu)^{\top} \bP_{\bu,\bv} \bW\bZ\bv \cdot 
	(\widetilde{\bW}\bX\bu)^{\top} \bR_{\bu,\bv} (\bW\bZ\bv)^{\top} \bR_{\bu,\bv} \Big| \Big\} \nonumber
	\\&\leq \frac{2}{(\lambda n)^{3}} \EE\Big\{ \Big((\widetilde{\bW}\bX\bu)^{\top} \bP_{\bu,\bv} \bW\bZ\bv\Big)^{2} \Big\}^{1/2} \cdot \EE\Big\{ \Big( (\widetilde{\bW}\bX\bu)^{\top} \bR_{\bu,\bv} (\bW\bZ\bv)^{\top} \bR_{\bu,\bv}\Big)^{2} \Big\}^{1/2},
\end{align}
where the last step follows from the Cauchy--Schwarz inequality. Continuing, we obtain the bound
\begin{align}\label{ineq1-T1-eta-expectation}
	\EE\big\{ \big((\widetilde{\bW}\bX\bu)^{\top} \bP_{\bu,\bv} \bW\bZ\bv\big)^{2} \big\} \overset{\1}{=} \EE\big\{ \big\| \widetilde{\bW} \bP_{\bu,\bv} \bW \big\|_{F}^{2} \big\} &\lesssim m \EE\Big\{ \max_{i\in[m]}G_i^2 \cdot \max_{i\in[m]} \widetilde{G}_i^2 \Big\} \\
	&\leq m\log^2(m),
\end{align}
where step $\1$ follows since $(\bX\bu,\bZ\bv)$ is independent of $(\widetilde{\bW},\bP_{\bu,\bv},\bW)$ (as $\bu,\bv$ are orthogonal to $\mathsf{span}(\bcoefX_{\star},\bcoefX_{t}),\mathsf{span}(\bcoefZ_{\star},\bcoefZ_{t})$, respectively), and in the last step follows from the i.i.d. Gaussian nature of the random variables $\{G_i,\widetilde{G}_i\}_{i\in[m]}$. 
Next, exploiting the independence between $(\bX\bu,\bZ\bv)$ and $(\bW,\widetilde{\bW},\bR_{\bu,\bv})$, we deduce
\begin{align*}
	\EE\big\{ \big( (\widetilde{\bW}\bX\bu)^{\top} \bR_{\bu,\bv} (\bW\bZ\bv)^{\top} \bR_{\bu,\bv}\big)^{2} \big\} &= 
	\EE\big\{ \big\| \widetilde{\bW} \bR_{\bu,\bv} \big\|_{2}^{2} \cdot  \big\| \bW \bR_{\bu,\bv} \big\|_{2}^{2} \big\}\\
	 &\leq \EE\Big\{ \max_{i \in [m]} G_i^2 \cdot \max_{i \in [m]} \widetilde{G}_i^2 \cdot \big\|\bR_{\bu,\bv} \big\|_{2}^{4} \Big\}.
\end{align*}
We then claim the following inequality, 
\begin{align}\label{claim1-Ruv-leq-Residue}
	\big\|\bR_{\bu,\bv} \big\|_{2}^{2} \leq \big\| \by - \bX\bcoefX_{t} \odot \bZ \bcoefZ_{t} \big\|_{2}^{2} = \sum_{i=1}^{m}(y_{i} - G_{i} \widetilde{G}_{i})^{2}.
\end{align}
Taking this as given---and providing its proof at the end---we put the pieces together and use the shorthand $a = \max_{i} G_i^2$ and $b =  \max_{i} \widetilde{G}_i^2$ to obtain the bound
\begin{align*}
	\EE\Big\{ \max_{i \in [m]} G_i^2 \cdot \max_{i \in [m]} \widetilde{G}_i^2 \cdot \big\|\bR_{\bu,\bv} \big\|_{2}^{4} \Big\} &\leq \EE\Big\{ ab \Big( \sum_{i=1}^{m} (y_{i} - G_{i}\widetilde{G}_{i} )^{2} \Big)^{2} \Big\} \lesssim m^{2} \EE\big\{ab (y_{i} - G_{i} \widetilde{G}_{i})^{4} \big\}.
\end{align*}
We further bound the RHS by applying the Cauchy--Schwarz inequality to obtain
\[
\EE\big\{ab (y_{i} - G_{i} \widetilde{G}_{i})^{4} \big\} \leq \EE\{a^2b^2\}^{1/2} \cdot \EE\{(y_{i} - G_{i} \widetilde{G}_{i})^{8} \}^{1/2} \lesssim \log^{2}(m) \cdot \big( \Err_{\sharp}^{2} + \sigma^4 \big),
\]
where in the last step we use the same decomposition as in Eq.~\eqref{ineq1-auxiliary-lemma1-proof-u-component}.  Assembling the various bounds yields 
\begin{align}\label{ineq3-T1-eta-expectation}
	\EE\big\{ \big( (\widetilde{\bW}\bX\bu)^{\top} \bR_{\bu,\bv} (\bW\bZ\bv)^{\top} \bR_{\bu,\bv}\big)^{2} \big\} \lesssim 
	\big(\sigma^{4} + \Err_{\sharp}^{2} \big) m^{2}\log^{2}(m).
\end{align}
Finally, we combine inequalities~\eqref{ineq-T1-eta-expectation},~\eqref{ineq1-T1-eta-expectation} and~\eqref{ineq3-T1-eta-expectation} to obtain
\begin{align}\label{ineq4-T1-eta-expectation}
	\hspace{-1cm}
	T_{1} \lesssim \frac{\sigma^{2} + \Err_{\sharp} }{(\lambda m)^{3}} \cdot \log(m) \sqrt{m} \cdot m\log(m) \leq  \frac{\Big(\sigma^{2} + \Err_{\sharp} \Big) \log^{2}(m)}{(\lambda m)\cdot \lambda \sqrt{m}},
\end{align}
as desired.  Before turning to bounding $T_2$, we provide the proof of the remaining inequality~\eqref{claim1-Ruv-leq-Residue}. 

\medskip
\noindent \underline{Proof of the inequality~\eqref{claim1-Ruv-leq-Residue}:} By the definition of $\bR_{\bu,\bv}$~\eqref{definition-Ruv-M11-etc}, we re-write
\[
\|\bR_{\bu,\bv}\|_{2}^{2} = \sum_{i=1}^{m}\big(y_{i} + G_{i}\widetilde{G}_{i} - \widetilde{G}_{i}\bx_{i}^{\top}\bO_{\bu}\bcoefX_{\bu,\bv} - G_{i}\bz_{i}^{\top} \bcoefZ_{\bu,\bv} \big)^{2}.
\]
We then substitute the definitions of $(\bcoefX_{\bu,\bv},\bcoefZ_{\bu,\bv})$ as in ~\eqref{defintion-estimator-u3-v3} to obtain the bound
\begin{align*}
	\|\bR_{\bu,\bv}\|_{2}^{2} \leq
	\sum_{i=1}^{m}\big(y_{i} + G_{i}\widetilde{G}_{i} - \widetilde{G}_{i}\bx_{i}^{\top}\bO_{\bu} \bO_{\bu}^{\top}\bcoefX_{\sharp} - G_{i}\bz_{i}^{\top} \bO_{\bv} \bO_{\bv}^{\top}\bcoefZ_{\sharp} \big)^{2} = \sum_{i=1}^{m}\big(y_{i} - G_{i}\widetilde{G}_{i} \big)^{2},
\end{align*}
where in the last step we used the facts that $\bO_{\bu}\bO_{\bu}^{\top} \bcoefX_{\sharp} = \bcoefX_{\sharp}$, $\bO_{\bv}\bO_{\bv}^{\top} \bcoefZ_{\sharp} = \bcoefZ_{\sharp}$, $G_{i} = \bx_{i}^{\top}\bcoefX_{\sharp}$, and $\widetilde{G}_{i} = \bz_{i}^{\top}\bcoefZ_{\sharp}$.

\paragraph{Proof of the inequality~\eqref{ineq:bound-T1-bias-orth-Ruv-term-b}:} Applying the pair of bounds $M_{11}\geq M_{12}^{2}/(\lambda + M_{22})$ and $M_{22}\geq 0$. we obtain the bound
\begin{align*}
	T_{2} &\leq \frac{1}{(\lambda m)^{4}} \EE\Big\{ \Big( (\widetilde{\bW}\bX\bu)^{\top} \bP_{\bu,\bv} \bW\bZ\bv \Big)^{2} \Big(
	(\bW\bZ\bv)^{\top} \bR_{\bu,\bv} \Big)^{2} \Big\} \nonumber
	\\& \leq \frac{1}{(\lambda m)^{4}}\EE\Big\{ \big\| \widetilde{\bW} \bP_{\bu,\bv} \bW\bZ\bv \big\|_{2}^{4} \Big\}^{1/2} \cdot \EE\Big\{ \Big((\bW\bZ\bv)^{\top} \bR_{\bu,\bv} \Big)^{4}\Big\}^{1/2},
\end{align*}
where in the final inequality we have first taken the expectation over $\bX \bu$ and subsequently applied the Cauchy--Schwarz inequality.  We then apply the following pair of inequalities 
\begin{align*}
	&\EE\Big\{ \big\| \widetilde{\bW} \bP_{\bu,\bv} \bW\bZ\bv \big\|_{2}^{4} \Big\}  \lesssim m^{2} \log^{4}(m) \qquad \text{ and }\\
	& \EE\Big\{ \Big((\bW\bZ\bv)^{\top} \bR_{\bu,\bv} \Big)^{4}\Big\} \lesssim \EE\Big\{ \big\| \bW \bR_{\bu,\bv}\big\|_{2}^{4}\Big\} \lesssim \big(\sigma^{4} + \Err_{\sharp}^{2} \big)m^{2}\log^{2}(m),
\end{align*}
to deduce the ultimate bound
\begin{align}\label{ineq-T2-eta-component}
	T_{2} \lesssim \frac{\big(\sigma^{2} + \Err_{\sharp} \big)\log^{3}(m)}{(\lambda m)^{2}},
\end{align}
as desired.

\paragraph{Proof of the inequality~\eqref{ineq:bound-T1-bias-orth-Ruv-term-c}:} Note that $\EE\big\{ \big( (\widetilde{\bW}\bX\bu)^{\top}\bR_{\bu,\bv} \big)^{2} \big\} = \EE\big\{ \big\| \widetilde{\bW}\bR_{\bu,\bv} \big\|_{2}^{2}\big\}$ and
\[
\Big| \Big(\lambda + M_{11} - M_{12}^{2}/(M_{22}+\lambda)\Big)^{-2} - (\lambda+V_{1})^{-2} \Big| \leq \lambda^{-2}\Big(|V_{1}-M_{11}|+ M_{12}^{2}/\lambda \Big).
\]
Consequently, we upper bound $T_3$ as
\begin{align*}
	T_{3} \leq \frac{\EE\Big\{ \big( (\widetilde{\bW}\bX\bu)^{\top}\bR_{\bu,\bv} \big)^{2} \cdot |V_{1}-M_{11}| \Big\}}{(\lambda m)^{2}} + \frac{\EE\Big\{ \big( (\widetilde{\bW}\bX\bu)^{\top}\bR_{\bu,\bv} \big)^{2} \cdot \big((\widetilde{\bW}\bX\bu)^{\top}\bP_{\bu,\bv} \bW\bZ\bv \big)^{2} \Big\}}{(\lambda m)^{2}\lambda m^{2}}.
\end{align*}
Applying the Cauchy--Schwarz inequality then yields
\begin{align*}
	\EE\Big\{ \big( (\widetilde{\bW}\bX\bu)^{\top}\bR_{\bu,\bv} \big)^{2} \cdot |V_{1}-M_{11}| \Big\} &\leq
	\EE\Big\{ \big( (\widetilde{\bW}\bX\bu)^{\top}\bR_{\bu,\bv} \big)^{4} \Big\}^{1/2} \cdot \EE\Big\{ |V_{1}-M_{11}|^{2} \Big\}^{1/2}.
	\\&\lesssim \big(\sigma^{2} + \Err_{\sharp}^{2}\big)m\log(m)\Big(\frac{\log(m)}{\sqrt{m}} + \frac{\log^{1.5}(d)}{\sqrt{d}} \Big),
\end{align*}
Next, we proceed similarly to the proofs of~\eqref{ineq:bound-T1-bias-orth-Ruv-term-a} and~\eqref{ineq:final-bound-M11-V1} to obtain the pair of inequalities
\begin{align*}
\EE\Big\{ \big( (\widetilde{\bW}\bX\bu)^{\top}\bR_{\bu,\bv} \big)^{4} \Big\} &\lesssim  \big(\sigma^{4} +\Err_{\sharp}^{2} \big)m^{2}\log^{2}(m), \qquad \text{ and } \\
 \EE\{|M_{11}-V_{1}|^{2}\} &\lesssim \frac{\log^{2}(m)}{m} + \frac{\log^{3}(d)}{d},
\end{align*}
so that 
\begin{align*}
	\EE\Big\{ \big( (\widetilde{\bW}\bX\bu)^{\top}\bR_{\bu,\bv} \big)^{2} \cdot |V_{1}-M_{11}| \Big\} \lesssim \big(\sigma^{2} + \Err_{\sharp}^{2}\big)m\log(m)\Big(\frac{\log(m)}{\sqrt{m}} + \frac{\log^{1.5}(d)}{\sqrt{d}} \Big).
\end{align*}
We then mimic the steps to prove the inequality~\eqref{ineq:bound-T1-bias-orth-Ruv-term-b} to deduce the bound
\[
\EE\big\{ \big( (\widetilde{\bW}\bX\bu)^{\top}\bR_{\bu,\bv} \big)^{2} \cdot \big((\widetilde{\bW}\bX\bu)^{\top}\bP_{\bu,\bv} \bW\bZ\bv \big)^{2} \big\} \lesssim \big(\sigma^{2} + \Err_{\sharp} \big)m^{2}\log^{3}(m).
\]
Putting the pieces together yields
\begin{align}\label{ineq-T3-eta-component}
	\EE\{T_{3}\} \lesssim \frac{\sigma^{2} + \Err_{\sharp} }{\lambda m} \cdot \Big( \frac{\log(m)}{\lambda} \max\Big(\frac{\log(m)}{\sqrt{m}},\frac{\log^{1.5}(d)}{\sqrt{d}} \Big)+ \frac{\log^{3}(m)}{\lambda m}\Big),
\end{align}
as desired.  \qed

\subsection{Proof of Lemma~\ref{lem:ineq3-step2-eta-component}}\label{sec:proof-ineq3-step2-eta-component}
Following the discussion preceding the statement of Lemma~\ref{lem:ineq3-step2-eta-component}, we deduce that
\begin{align}\label{ineq-step2-eta-component}
	\EE\big\{\|\widetilde{\bW} \bR_{\bu,\bv}\|_{2}^{2} \big\} = n\cdot \EE\bigg\{ \frac{\widetilde{G}_{i}^{2} \tau_{i}^{2}}{\big(1+\ba_{i}^{\top}\bSig_{i}^{-1} \ba_{i} \big)^{2}}\bigg\}.
\end{align}
We now replace the denominator $1+\ba_{i}^{\top}\bSig_{i}^{-1}\ba_{i}$ by $1+\widetilde{G}_{i}^{2}r_{1}^{-1} + G_{i}^{2}r_{2}^{-1}$.  To this end, we apply the triangle inequality and use the three bounds $\bSig_{i}^{-1} \succeq \boldsymbol{0}$ and $r_{1},r_{2} \geq 0$~\eqref{eq:fixed-point} to obtain the inequality
\begin{align}\label{ineq1-step2-eta-component}
	&\bigg| \EE\bigg\{ \frac{\widetilde{G}_{i}^{2} \tau_{i}^{2}}{\big(1+\ba_{i}^{\top}\bSig_{i}^{-1} \ba_{i} \big)^{2}} \bigg\} - \EE\bigg\{ \frac{\widetilde{G}_{i}^{2} \tau_{i}^{2}}{\big(1+\widetilde{G}_{i}^{2}r_{1}^{-1} + G_{i}^{2}r_{2}^{-1} \big)^{2}} \bigg\} \bigg|
	\leq \EE\Big\{ \widetilde{G}_{i}^{2} \tau_{i}^{2} \big|\ba_{i}^{\top} \bSig_{i}^{-1} \ba_{i} - \widetilde{G}_{i}^{2}r_{1}^{-1} - G_{i}^{2}r_{2}^{-1}\big| \Big\} \nonumber\\
	&\qquad \qquad \qquad \qquad \qquad \qquad \leq \EE\{\tau_{i}^{4}\}^{1/2} \cdot  \EE\Big\{\widetilde{G}_{i}^{4}  \big(\ba_{i}^{\top} \bSig_{i}^{-1} \ba_{i} - \widetilde{G}_{i}^{2}r_{1}^{-1} - G_{i}^{2}r_{2}^{-1}\big)^{2} \Big\}^{1/2}
\end{align}
Proceeding in an identical fashion to the argument between inequality~\eqref{ineq-decomposition-T2} to inequality~\eqref{T2-expectation-bound-M11} yields the bound
\[
\EE\Big\{\widetilde{G}_{i}^{4}  \big(\ba_{i}^{\top} \bSig_{i}^{-1} \ba_{i} - \widetilde{G}_{i}^{2}r_{1}^{-1} - G_{i}^{2}r_{2}^{-1}\big)^{2} \Big\} \lesssim \log^{3}(d)/d.
\]
Towards bounding $\EE\{\tau_{i}^{4}\}$, we note that by triangle inequality
\[
|\tau_{i}| \leq |y_{i} - G_{i}\widetilde{G}_{i}| + \big|G_{i} \cdot \bz_{i}^{\top} \big( \bcoefZ_{t} - \bO_{\bv} \bcoefZ_{\bu,\bv}^{-(i)} \big) \big| + \big| \widetilde{G}_{i} \cdot \bx_{i}^{\top} \big( \bcoefX_{t} - \bO_{\bu} \bcoefX_{\bu,\bv}^{-(i)} \big) \big|.
\]
Consequently, we obtain that
\begin{align}\label{ineq-tau-i-4th-moment}
	\EE\{\tau_{i}^{4}\} &\lesssim \EE\big\{ (y_{i} - G_{i}\widetilde{G}_{i})^{4} \big\} + \EE\big\{ \big|G_{i} \cdot \bz_{i}^{\top} \big( \bcoefZ_{t} - \bO_{\bv} \bcoefZ_{\bu,\bv}^{-(i)} \big) \big|^{4}\big\} + \EE\big\{ \big| \widetilde{G}_{i} \cdot \bx_{i}^{\top} \big( \bcoefX_{t} - \bO_{\bu} \bcoefX_{\bu,\bv}^{-(i)} \big) \big|^{4} \big\} \nonumber
	\\&\lesssim \EE\big\{ (y_{i} - G_{i}\widetilde{G}_{i})^{4} \big\} + \EE\big\{ \big\| \bcoefZ_{t} - \bO_{\bv} \bcoefZ_{\bu,\bv}^{-(i)}  \big\|_{2}^{4}\big\} + \EE\big\{ \big\| \bcoefX_{t} - \bO_{\bu} \bcoefX_{\bu,\bv}^{-(i)}  \big\|_{2}^{4}\big\},
\end{align}
where the last step follows since $(\bcoefX_{\bu,\bv}^{-(i)}, \bcoefZ_{\bu,\bv}^{-(i)})$ is independent of $(\bx_{i},\bz_{i})$. Continuing, since $(\bcoefX_{\bu,\bv}^{-(i)}, \bcoefZ_{\bu,\bv}^{-(i)})$ is the minimizer of~\eqref{defintion-estimator-u3-v3-i-sample-out}, we obtain that
\begin{align*}
	\lambda m \cdot (\|\bcoefX_{\bu,\bv}^{-(i)} - \bO_{\bu}^{\top}\bcoefX_{\sharp}\|_{2}^{2} + \|\bcoefZ_{\bu,\bv}^{-(i)} - \bO_{\bv}^{\top}\bcoefZ_{\sharp}\|_{2}^{2}) 
	\leq \sum_{j\neq i} \big(y_{j} - G_{j} \widetilde{G}_{j} \big)^{2}.
\end{align*}
The inequality in the display above further implies
\begin{align}\label{ineq-l2-distance-looestimator-estimatort}
	\|\bcoefX_{\bu,\bv}^{-(i)} - \bO_{\bu}^{\top}\bcoefX_{\sharp}\|_{2}^{2} + \|\bcoefZ_{\bu,\bv}^{-(i)} - \bO_{\bv}^{\top}\bcoefZ_{\sharp}\|_{2}^{2} \leq \frac{1}{\lambda m} \sum_{j\neq i} \big(y_{j} - G_{j} \widetilde{G}_{j} \big)^{2}.
\end{align}
Taking the square and expectation in sequence of both sides of the inequality in the display above yields 
\[
\EE\big\{ \|\bcoefX_{\bu,\bv}^{-(i)} - \bO_{\bu}^{\top}\bcoefX_{\sharp}\|_{2}^{4} \big\} 
+ \EE\big\{ \|\bcoefZ_{\bu,\bv}^{-(i)} - \bO_{\bv}^{\top}\bcoefZ_{\sharp}\|_{2}^{4} \big\} \lesssim 
\frac{m^{2}}{(\lambda m)^{2}} \EE\big\{ \big(y_{j} - G_{j} \widetilde{G}_{j} \big)^{4} \big\}.
\]
Combining the inequality in the display above with inequality~\eqref{ineq-tau-i-4th-moment} and invoking the assumption $\lambda m \geq d\geq m$, we obtain the bound
\begin{align}\label{ineq-taui-4th-moment}
	\EE\{\tau_{i}^{4}\} \lesssim \EE\big\{ \big(y_{j} - G_{j} \widetilde{G}_{j} \big)^{4} \big\} \lesssim \sigma^{4} + \Err_{\sharp}^{2},
\end{align}
where in the final inequality, we used the decomposition~\eqref{ineq1-auxiliary-lemma1-proof-u-component}. 
Substituting these bounds into the inequality~\eqref{ineq1-step2-eta-component} yields
\[
\bigg| \EE\bigg\{ \frac{\widetilde{G}_{i}^{2} \tau_{i}^{2}}{\big(1+\ba_{i}^{\top}\bSig_{i}^{-1} \ba_{i} \big)^{2}} \bigg\} - \EE\bigg\{ \frac{\widetilde{G}_{i}^{2} \tau_{i}^{2}}{\big(1+\widetilde{G}_{i}^{2}r_{1}^{-1} + G_{i}^{2}r_{2}^{-1} \big)^{2}} \bigg\} \bigg| \lesssim \big(  \sigma^{2} + \Err_{\sharp} \big) \frac{\log^{1.5}(d)}{\sqrt{d}}.
\]
Putting together the inequality in the display above and inequality~\eqref{ineq-step2-eta-component} yields
\begin{align}
	\bigg| \frac{\EE\big\{ \big\| \widetilde{\bW}\bR_{\bu,\bv} \big\|_{2}^{2} \big\}}{m^{2}(\lambda + V_{1})^{2} } - \frac{\EE\bigg\{ \frac{\widetilde{G}_{i}^{2} \tau_{i}^{2}}{\big(1+\widetilde{G}_{i}^{2}r_{1}^{-1} + G_{i}^{2}r_{2}^{-1} \big)^{2}} \bigg\}}{m(\lambda+V_{1})^{2}} \bigg| \lesssim \big(  \sigma^{2} + \Err_{\sharp} \big) \frac{\log^{1.5}(d)}{\sqrt{d} m\lambda^{2}},
\end{align}
as desired. \qed

\subsection{Proof of Lemma~\ref{lem:ineq-step3-eta-component}}\label{sec:proof-ineq-step3-eta-component}
\noindent Recall the definition of $(\bu_{1},\bu_{2},\bv_{1},\bv_{2})$~\eqref{definition-u1-u2-v1-v2} and let $S_{\sharp} = \mathsf{span}(\bcoefX_{\star},\bcoefX_{\sharp})$ and $\widetilde{S}_{\sharp} = \mathsf{span}(\bcoefZ_{\star},\bcoefZ_{\sharp})$. Decomposing by Gram–-Schmidt yields
\begin{align*}
	\bO_{\bu} \bcoefX_{\bu,\bv}^{-(i)} &= \underbrace{\langle \bO_{\bu} \bcoefX_{\bu,\bv}^{-(i)}, \bu_{1} \rangle}_{\vartheta_{1}^{-i}} \bu_{1} 
	+ \underbrace{ \langle \bO_{\bu} \bcoefX_{\bu,\bv}^{-(i)}, \bu_{2} \rangle}_{\vartheta_{2}^{-i}} \bu_{2}+ \underbrace{\bP_{S_{\sharp}}^{\perp} \bO_{\bu} \bcoefX_{\bu,\bv}^{-(i)}}_{\boldsymbol{\omega}^{-i}}, \qquad \text{ and }\\
	\bO_{\bv} \bcoefZ_{\bu,\bv}^{-(i)} &= \underbrace{\langle \bO_{\bv} \bcoefZ_{\bu,\bv}^{-(i)}, \bv_{1} \rangle}_{\widetilde{\vartheta}_{1}^{-i}} \bv_{1} 
	+ \underbrace{ \langle \bO_{\bv} \bcoefZ_{\bu,\bv}^{-(i)}, \bv_{2} \rangle}_{\widetilde{\vartheta}_{2}^{-i}} \bv_{2}  + \underbrace{\bP_{\widetilde{S}_{\sharp}}^{\perp} \bO_{\bv} \bcoefZ_{\bu,\bv}^{-(i)}}_{\widetilde{\boldsymbol{\omega}}^{-i}}.
\end{align*}
We further define $g_{i} = \bx_{i}^{\top} \bu_{2},\;g_{i}^{\perp} = \bx_{i}^{\top} \boldsymbol{\omega}^{-i},\;\widetilde{g}_{i} = \bz_{i}^{\top} \bv_{2}$ and $\widetilde{g}_{i}^{\perp} = \bz_{i}^{\top} \widetilde{\boldsymbol{\omega}}^{-i}$. By definition, $G_{i} = L_{\sharp}\bx_{i}^{\top}\bu_{1}$ and $\widetilde{G}_{i} = \LZ_{\sharp} \bz_{i}^{\top} \bv_{1}$. Consequently, $G_{i},g_{i},g_{i}^{\perp},\widetilde{G}_{i}, \widetilde{g}_{i}, \widetilde{g}_{i}^{\perp}$ are independent since $\bu_{1},\bu_{2},\boldsymbol{\omega}^{-i}$ are orthogonal to each other (similarly for $\bv_{1},\bv_{2},\widetilde{\boldsymbol{\omega}}^{-i}$ ), and $(\boldsymbol{\omega}^{-i},\widetilde{\boldsymbol{\omega}}^{-i})$ are independent of $(\bx_{i},\bz_{i})$. Using the orthogonal decomposition above, we thus write 
\begin{align*}
	\bx_{i}^{\top} \bO_{\bu}\bcoefX_{\bu,\bv}^{-(i)} = \frac{\vartheta_{1}^{-i} G_{i}}{L_{\sharp}} + \vartheta_{2}^{-i} g_{i} + g_{i}^{\perp} \qquad \text{ and } \qquad
	\bz_{i}^{\top} \bO_{\bv} \bcoefZ_{\bu,\bv}^{-(i)} = \frac{\widetilde{\vartheta}_{1}^{-i} \widetilde{G}_{i}}{\widetilde{L}_{\sharp}} + \widetilde{\vartheta}_{2}^{-i} \widetilde{g}_{i} + \widetilde{g}_{i}^{\perp}.
\end{align*}
Combining this with the representation $y_i = \epsilon_{i} + \Big(\frac{\parcompX_{\sharp}}{L_{\sharp}^{2}} G_{i} + \frac{\perpcompX_{\sharp}}{L_{\sharp}}g_{i}  \Big)\Big(\frac{\parcompZ_{\sharp}}{\widetilde{L}_{\sharp}^{2}} \widetilde{G}_{i} + \frac{\perpcompZ_{\sharp}}{\widetilde{L}_{\sharp}}\widetilde{g}_{i}  \Big)$, we deduce that
\begin{align*}
	&\tau_{i} 
	= \epsilon_{i} + \Big(\frac{\parcompX_{\sharp}\parcompZ_{\sharp}}{L_{\sharp}^{2} \widetilde{L}_{\sharp}^{2}} + 1- \frac{\vartheta_{1}^{-i}}{L_{\sharp}} - \frac{\widetilde{\vartheta}_{1}^{-i}}{\widetilde{L}_{\sharp}} \Big)G_{i} \widetilde{G}_{i} 
	\\& \qquad \qquad \qquad \qquad+ 
	\Big( \frac{\parcompX_{\sharp} \perpcompZ_{\sharp}}{L_{\sharp}^{2} \widetilde{L}_{\sharp}} - \widetilde{\vartheta}_{2}^{-i}\Big) G_{i}\widetilde{g}_{i} + \Big( \frac{\parcompZ_{\sharp} \perpcompX_{\sharp}}{\widetilde{L}_{\sharp}^{2} L_{\sharp}} - \vartheta_{2}^{-i}\Big) \widetilde{G}_{i}g_{i} + \frac{\perpcompX_{\sharp} \perpcompZ_{\sharp}}{L_{\sharp} \widetilde{L}_{\sharp}}g_{i}\widetilde{g}_{i} - \widetilde{G}_{i}g_{i}^{\perp}
	- G_{i}\widetilde{g}_{i}^{\perp}.
\end{align*}
Then, exploiting the fact that $G_{i},g_{i},g_{i}^{\perp},\widetilde{G}_{i}, \widetilde{g}_{i}, \widetilde{g}_{i}^{\perp}$ are independent and zero-mean, as well the relations $\EE\{g_{i}^{2}\} = \EE\{\widetilde{g}_{i}^{2}\} = 1,\; \EE\{(g_{i}^{\perp})^{2}\} = \EE\{\|\boldsymbol{\omega}^{-i}\|_{2}^{2}\},\; \EE\{(\widetilde{g}_{i}^{\perp})^{2}\} = \EE\{\|\widetilde{\boldsymbol{\omega}}^{-i}\|_{2}^{2}\}$, we obtain the decomposition
\begin{align}\label{eq-step3-eta-component}
	\hspace{-1cm}
	&\EE\bigg\{ \frac{\widetilde{G}_{i}^{2} \tau_{i}^{2}}{\big(1+\widetilde{G}_{i}^{2}r_{1}^{-1} + G_{i}^{2}r_{2}^{-1} \big)^{2}} \bigg\} =  \EE\bigg\{ \frac{ \Big( \sigma^{2} + \frac{\perpcompX_{\sharp}^{2} \perpcompZ_{\sharp}^{2} }{L_{\sharp}^{2} \widetilde{L}_{\sharp}^{2} } \Big) \widetilde{G}_{i}^{2} }{\big(1+\widetilde{G}_{i}^{2}r_{1}^{-1} + G_{i}^{2}r_{2}^{-1} \big)^{2}} \bigg\} + 
	\EE\bigg\{ \frac{ \EE\{\|\boldsymbol{\omega}^{-i}\|_{2}^{2}\} \widetilde{G}_{i}^{4} + \EE\{\|\widetilde{\boldsymbol{\omega}}^{-i}\|_{2}^{2}\} G_{i}^{2}\widetilde{G}_{i}^{2} }{\big(1+\widetilde{G}_{i}^{2}r_{1}^{-1} + G_{i}^{2}r_{2}^{-1} \big)^{2}} \bigg\} \nonumber
	\\&+ \EE\bigg\{ \Big(\frac{\parcompX_{\sharp}\parcompZ_{\sharp}}{L_{\sharp}^{2} \widetilde{L}_{\sharp}^{2}} +1 - \frac{\vartheta_{1}^{-i}}{L_{\sharp}} - \frac{\widetilde{\vartheta}_{1}^{-i}}{\widetilde{L}_{\sharp}} \Big)^{2} \bigg\} \EE\bigg\{ \frac{ G_{i}^{2}\widetilde{G}_{i}^{4} }{\big(1+\widetilde{G}_{i}^{2}r_{1}^{-1} + G_{i}^{2}r_{2}^{-1} \big)^{2}} \bigg\} \\&+ \EE\bigg\{ \Big( \frac{\parcompX_{\sharp} \perpcompZ_{\sharp}}{L_{\sharp}^{2} \widetilde{L}_{\sharp}} - \widetilde{\vartheta}_{2}^{-i}\Big)^{2} \bigg\} \EE\bigg\{ \frac{\widetilde{G}_{i}^{2} G_{i}^{2}}{\big(1+\widetilde{G}_{i}^{2}r_{1}^{-1} + G_{i}^{2}r_{2}^{-1} \big)^{2}} \bigg\} +
	\EE\bigg\{ \Big( \frac{\parcompZ_{\sharp} \perpcompX_{\sharp}}{\widetilde{L}_{\sharp}^{2} L_{\sharp}} - \vartheta_{2}^{-i}\Big)^{2} \bigg\} \EE\bigg\{ \frac{\widetilde{G}_{i}^{4} }{\big(1+\widetilde{G}_{i}^{2}r_{1}^{-1} + G_{i}^{2}r_{2}^{-1} \big)^{2}} \bigg\} \nonumber
\end{align}
We will use the following auxiliary lemma to approximately compute some terms of the equation in the display above. 
\begin{lemma}\label{lemma:eta-onesample-onecolumn-out-expectation-bound}
	Recall the definition of $\thetaXdet_{1},\thetaXdet_{2}$~\eqref{definition-theta12-det} and $\thetaZdet_{1},\thetaZdet_{2}$~\eqref{definition-tilde-theta12-det}. There exists a universal constant $C>0$ such that for $\delta = \max\bigg\{ \frac{C( \sigma^{2} + \Err_{\sharp})}{\lambda} \frac{\log^{6}(d)}{\sqrt{d}}, d^{-50} \bigg\}$, each of the following hold.
	\begin{subequations}
		\begin{align}\label{expectation-eta-omega-norm-square}
			&(a.)\;\; \big| \EE\big\{\etaX_{+}^{2} \big\} - \EE\big\{ \|\boldsymbol{\omega}^{-i}\|_{2}^{2} \big\} \big| \vee \big| \EE\big\{\etaZ_{+}^{2} \big\} - \EE\big\{ \| \widetilde{\boldsymbol{\omega}}^{-i}\|_{2}^{2} \big\} \big| \leq \frac{C(\sigma^{2} +\Err_{\sharp})}{\lambda \sqrt{m}},\\
			\label{expectation-bound-coefficient1}
			& (b.) \;\;\bigg| \EE\bigg\{ \Big(\frac{\parcompX_{\sharp}\parcompZ_{\sharp}}{L_{\sharp}^{2} \widetilde{L}_{\sharp}^{2}} +1 - \frac{\vartheta_{1}^{-i}}{L_{\sharp}} - \frac{\widetilde{\vartheta}_{1}^{-i}}{\widetilde{L}_{\sharp}} \Big)^{2} \bigg\} -  \Big(\frac{\parcompX_{\sharp}\parcompZ_{\sharp}}{L_{\sharp}^{2} \widetilde{L}_{\sharp}^{2}} +1 - \frac{\thetaXdet_{1}}{L_{\sharp}} - \frac{\thetaZdet_{1}}{\widetilde{L}_{\sharp}} \Big)^{2}  \bigg| \leq \delta,\\\label{expectation-bound-coefficient2}
			&(c.) \;\; \bigg| \EE\bigg\{ \Big( \frac{\parcompX_{\sharp} \perpcompZ_{\sharp}}{L_{\sharp}^{2} \widetilde{L}_{\sharp}} - \widetilde{\vartheta}_{2}^{-i}\Big)^{2} \bigg\} - \Big( \frac{\parcompX_{\sharp} \perpcompZ_{\sharp}}{L_{\sharp}^{2} \widetilde{L}_{\sharp}} - \thetaZdet_{2}\Big)^{2}\bigg| \vee \bigg| \EE\bigg\{ \Big( \frac{\parcompZ_{\sharp} \perpcompX_{\sharp}}{\widetilde{L}_{\sharp}^{2} L_{\sharp}} - \vartheta_{2}^{-i}\Big)^{2} \bigg\} - \Big( \frac{\parcompZ_{\sharp} \perpcompX_{\sharp}}{\widetilde{L}_{\sharp}^{2} L_{\sharp}} - \thetaXdet_{2}\Big)^{2} \bigg| \leq \delta.
		\end{align}
	\end{subequations}
\end{lemma}
We provide the proof of Lemma~\ref{lemma:eta-onesample-onecolumn-out-expectation-bound} in Section~\ref{proof-lemma-eta-onesample-onecolumn-out-expectation-bound}.  Continuing, note that by definition of $\thetaXdet_{1},\thetaXdet_{2}$~\eqref{definition-theta12-det} and $\thetaZdet_{1},\thetaZdet_{2}$~\eqref{definition-tilde-theta12-det}, we have
\begin{align*}
	&\Big(\frac{\parcompX_{\sharp}\parcompZ_{\sharp}}{L_{\sharp}^{2} \widetilde{L}_{\sharp}^{2}} +1 - \frac{\thetaXdet_{1}}{L_{\sharp}} - \frac{\thetaZdet_{\sharp}}{\widetilde{L}_{\sharp}} \Big)^{2}  = \frac{\lambda^{2}(\parcompX_{\sharp}\parcompZ_{\sharp}/(L_{\sharp}^{2}\widetilde{L}_{\sharp}^{2}) -1 )^{2}}{\big(\lambda + V(L_{\sharp}^{-2} + \widetilde{L}_{\sharp}^{-2}) \big)^{2}},  \qquad \Big( \frac{\parcompX_{\sharp} \perpcompZ_{\sharp}}{L_{\sharp}^{2} \widetilde{L}_{\sharp}} - \thetaZdet_{2}\Big)^{2} =  \Big(\frac{\parcompX_{\sharp} \perpcompZ_{\sharp}}{L_{\sharp}^{2} \widetilde{L}_{\sharp}} \frac{\lambda}{\lambda + V_{2}}\Big)^{2}, \\
	 & \qquad \qquad \qquad \qquad  \text{ and } \qquad
	\Big( \frac{\parcompZ_{\sharp} \perpcompX_{\sharp}}{\widetilde{L}_{\sharp}^{2} L_{\sharp}} - \thetaXdet_{2}\Big)^{2} = \Big( \frac{\parcompZ_{\sharp} \perpcompX_{\sharp}}{\widetilde{L}_{\sharp}^{2} L_{\sharp}} \frac{\lambda}{\lambda + V_{1}}\Big)^{2}.
\end{align*}
Substituting the approximations in Lemma~\ref{lemma:eta-onesample-onecolumn-out-expectation-bound} into Eq.~\eqref{eq-step3-eta-component} and recalling the definition of $V_{3}$~\eqref{definition-V3} yields the bound
\begin{align}
	&\bigg| \EE\bigg\{ \frac{\widetilde{G}_{i}^{2} \tau_{i}^{2}}{\big(1+\widetilde{G}_{i}^{2}r_{1}^{-1} + G_{i}^{2}r_{2}^{-1} \big)^{2}} \bigg\} -\bigg( V_{3}r_{1}^{2} + \EE\{\etaX_{+}^{2}\} \cdot  \EE\bigg\{ \frac{ r_{1}^{2}r_{2}^{2} \widetilde{G}_{i}^{4} }{(r_{1}r_{2} + r_{1}G_{i}^{2} + r_{2}\widetilde{G}_{i}^{2})^{2}}\bigg\} \nonumber\\
	&\qquad \;\; + \EE\{\etaZ_{+}^{2}\} \cdot \EE\bigg\{ \frac{r_{1}^{2}r_{2}^{2} G_{i}^{2}\widetilde{G}_{i}^{2} }{(r_{1}r_{2} + r_{1}G_{i}^{2} + r_{2}\widetilde{G}_{i}^{2})^{2}}\bigg\} \bigg) \bigg| \lesssim \max\bigg\{ \frac{C( \sigma^{2} + \Err_{\sharp})}{\lambda} \frac{\log^{6}(d)}{\sqrt{m}}, d^{-50} \bigg\},
\end{align} 
as desired. \qed

\subsubsection{Proof of Lemma~\ref{lemma:eta-onesample-onecolumn-out-expectation-bound}} \label{proof-lemma-eta-onesample-onecolumn-out-expectation-bound} 
We prove each part in turn.

\paragraph{Proof of Lemma~\ref{lemma:eta-onesample-onecolumn-out-expectation-bound}(a.):}
Let $\ba_{j}^{\top} = [\widetilde{G}_{j} \; \vert \; \bx_{j}^{\top}\; G_{j}\bz_{j}^{\top}]$ and recall $\tau_{j}$ as defined in Eq.~\eqref{claim-step2-eta-component}. We claim the following hold
\begin{align}\label{claim-KKT-relation-estimator(t+1)-u-v-isample-out}
	\hspace{-1cm}
	\begin{bmatrix} \bcoefX_{+}\\ \bcoefZ_{+} \\ \end{bmatrix} - 
	\begin{bmatrix} \bO_{\bu}\bcoefX_{\bu,\bv}^{-(i)}\\ \bO_{\bv}\bcoefZ_{\bu,\bv}^{-(i)} \\ \end{bmatrix}
	= \bSig_{i}^{-1} \bigg( \sum_{j \neq i} \tau_{j} \begin{bmatrix} \widetilde{G}_{j} \bx_{j}^{\top} \bu \bu \\ G_{j}\bx_{j}^{\top} \bv \bv \\ \end{bmatrix}  + \Big( y_{i} + G_{i} \widetilde{G}_{i} - \ba_{i}^{\top}\begin{bmatrix} \bcoefX_{+}\\ \bcoefZ_{+} \\ \end{bmatrix}  \Big) \ba_{i} \bigg) =: \boldsymbol{b},
\end{align} 
where $\bSig_{i} = \lambda m \bI + \sum_{j \neq i} \ba_{j}\ba_{j}^{\top}$. Define the linear subspace $V = S_{\sharp} \times \mathbb{R}^{d}$. Multiplying both sides of Eq.~\eqref{claim-KKT-relation-estimator(t+1)-u-v-isample-out} by $\bP_{V}^{\perp}$ and taking the norm yields
\[
\| \bP_{S_{\sharp}}^{\perp} \bcoefX_{+} \|_{2}^{2} = \| \bP_{S_{\sharp}}^{\perp} \bO_{\bu} \bcoefX_{\bu,\bv}^{-(i)}\|_{2}^{2} 
+ \|\bP_{V}^{\perp} \boldsymbol{b}\|_{2}^{2} + 2\Big \langle \bP_{V}^{\perp} \begin{bmatrix} \bO_{\bu}\bcoefX_{\bu,\bv}^{-(i)}\\ \bO_{\bv}\bcoefZ_{\bu,\bv}^{-(i)} \\ \end{bmatrix}  , \bP_{V}^{\perp} \boldsymbol{b} \Big\rangle.
\]
Note that $\etaX_{+}^{2} = \| \bP_{S_{\sharp}}^{\perp} \bcoefX_{+} \|_{2}^{2}$ and $\boldsymbol{\omega}^{-i} = \bP_{S_{\sharp}}^{\perp} \bO_{\bu} \bcoefX_{\bu,\bv}^{-(i)}$ by definition. Consequently, we obtain 
\begin{align}\label{ineq1-eta-onesample-onecolumn-out-expectation-bound}
	\big| \EE\big\{\etaX_{+}^{2} \big\} - \EE\big\{ \|\boldsymbol{\omega}^{-i}\|_{2}^{2} \big\} \big| 
	\leq \EE\big\{ \|\bP_{V}^{\perp} \boldsymbol{b} \|_{2}^{2}\big\} + 2\EE\big\{  \| \bP_{S_{\sharp}}^{\perp} \bO_{\bu} \bcoefX_{\bu,\bv}^{-(i)} \|_{2}^{2} \big\}^{1/2} \cdot \EE\big\{ \|\bP_{V}^{\perp} \boldsymbol{b} \|_{2}^{2} \big\}^{1/2}.
\end{align}
Applying the triangle inequality and using $\bSig_{i}^{-1} \preceq (\lambda n)^{-1} \bI$, we obtain the bound
\[
\EE\big\{ \|\bP_{V}^{\perp}\boldsymbol{b}\|_{2}^{2}\big\} \lesssim \frac{\EE\Big\{ \Big( \sum_{j \neq i} \tau_{j} \widetilde{G}_{j} \bx_{j}^{\top}\bu  \Big)^{2} + \Big( \sum_{j \neq i} \tau_{j} \widetilde{G}_{j} \bx_{j}^{\top}\bu  \Big)^{2} \Big\}}{(\lambda m)^{2}} + \frac{\EE\Big\{ \Big(y_{i} + G_{i} \widetilde{G}_{i} - \ba_{i}^{\top} \begin{bmatrix} \bcoefX_{+} \\ \bcoefZ_{+} \end{bmatrix} \Big)^{2} \|\ba_{i}\|_{2}^{2} \Big\}}{(\lambda m)^{2}}.
\]
Note that $\bx_{j}^{\top}\bu$ is independent of $\tau_{j}$~\eqref{claim-step2-eta-component} as $\bu^{\top} \bcoefX_{\star} = \bu^{\top} \bcoefX_{t} = 0$. Consequently, we obtain that
\begin{align}\label{ineq1.5-eta-onesample-onecolumn-out-expectation-bound}
	\EE\Big\{ \Big( \sum_{j \neq i} \tau_{j} \widetilde{G}_{j} \bx_{j}^{\top}\bu  \Big)^{2} \Big\} = (m-1)\EE\big\{\widetilde{G}_{j}^{2}\tau_{j}^{2} \big\} \leq m \EE\{\widetilde{G}_{j}^{4}\}^{1/2} \cdot \EE\{\tau_{j}^{4}\}^{1/2} \lesssim 
	m\big( \sigma^{2} + \Err_{\sharp} \big),
\end{align}
where the last step follows from the inequality~\eqref{ineq-taui-4th-moment}. Applying the Cauchy--Schwarz inequality then yields
\begin{align}\label{ineq1.8-eta-onesample-onecolumn-out-expectation-bound}
	\EE\Big\{ \Big(y_{i} + G_{i} \widetilde{G}_{i} - \ba_i^{\top} \begin{bmatrix} \bcoefX_{+} \\ \bcoefZ_{+} \end{bmatrix}  \Big)^{2} \|\ba_{i}\|_{2}^{2}\Big\} &\lesssim 
	\EE\Big\{ \Big(y_{i} + G_{i} \widetilde{G}_{i} - \ba_i^{\top} \begin{bmatrix} \bcoefX_{+} \\ \bcoefZ_{+} \end{bmatrix}  \Big)^{4} \Big\}^{\frac{1}{2}} \EE\big\{ \|\ba_{i}\|_{2}^{4}\big\}^{\frac{1}{2}} \nonumber \\
	 &\lesssim d\big( \sigma^{2} + \Err_{\sharp} \big),
\end{align}
where to obtain the last inequality, we used the same bound as inequality~\eqref{ineq2.5-auxiliary-lemma1-proof-u-component}. Putting the pieces together and invoking the assumption $\lambda m\geq d\geq m$ yields
$\EE\big\{ \|\bP_{V}^{\perp}\boldsymbol{b}\|_{2}^{2}\big\} \lesssim (\sigma^{2} + \Err_{\sharp})/(\lambda m)$. Then, since $\bP_{S_{\sharp}}^{\perp} \bcoefX_{\sharp} = \boldsymbol{0}$, we deduce the inequality
\[
\|\bP_{S_{\sharp}}^{\perp} \bO_{\bu}\bcoefX_{\bu,\bv}^{-(i)} \|_{2}^{2} = \|\bP_{S_{\sharp}}^{\perp} \bO_{\bu} \bcoefX_{\bu,\bv}^{-(i)} - \bP_{S_{\sharp}}^{\perp} \bcoefX_{\sharp} \|_{2}^{2} \leq \| \bcoefX_{\bu,\bv}^{-(i)} - \bO_{\bu}^{\top} \bcoefX_{\sharp} \|_{2}^{2} \leq  \frac{1}{\lambda m} \sum_{j \neq i} \big(y_{j} - G_{j}\widetilde{G}_{j} \big)^{2}, 
\]
where the last step follows from the inequality~\eqref{ineq-l2-distance-looestimator-estimatort}. Consequently, we obtain the bound
\[
\EE\Big\{ \|\bP_{S_{\sharp}}^{\perp} \bO_{\bu}\bcoefX_{\bu,\bv}^{-(i)} \|_{2}^{2} \Big\} \leq \EE\{(y_{j} - G_{j}\widetilde{G}_{j})^{2}\}/\lambda \leq (\sigma^{2} + \Err_{\sharp}) /\lambda.
\]
Assembling the bounds and substituting them into the inequality~\eqref{ineq1-eta-onesample-onecolumn-out-expectation-bound} yields $\big| \EE\big\{\etaX_{+}^{2} \big\} - \EE\big\{ \|\boldsymbol{\omega}^{-i}\|_{2}^{2} \big\} \big| \lesssim (\sigma^{2} +\Err_{\sharp})/(\lambda \sqrt{m})$. Proceeding similarly yields the analogous bound $\big| \EE\big\{\etaZ_{+}^{2} \big\} - \EE\big\{ \| \widetilde{\boldsymbol{\omega}}^{-i}\|_{2}^{2} \big\} \big| \lesssim (\sigma^{2} + \Err_{\sharp})/(\lambda \sqrt{m})$. This concludes the proof of inequality~\eqref{expectation-eta-omega-norm-square}.  Before proceeding to the proofs of parts (b.) and (c.), we first provide the proof of the deferred claim~\eqref{claim-KKT-relation-estimator(t+1)-u-v-isample-out}.

\medskip
\noindent \underline{Proof of the claim~\eqref{claim-KKT-relation-estimator(t+1)-u-v-isample-out}:} The KKT conditions of the optimization problems defining $[\bcoefX_{+},\bcoefZ_{+}] = G([\bcoefX_{\sharp}\;\vert\;\bcoefZ_{\sharp} ])$~\eqref{eq:closed-form-update} and $(\bcoefX_{\bu,\bv}^{-(i)}, \bcoefZ_{\bu,\bv}^{-(i)})$~\eqref{defintion-estimator-u3-v3-i-sample-out} yield
\begin{subequations}
	\begin{align}\label{KKT-condition-estimator-t+1}
		&\lambda m \begin{bmatrix} \bcoefX_{+} - \bcoefX_{\sharp}\\ \bcoefZ_{+}-\bcoefZ_{\sharp} \\ \end{bmatrix} = \sum_{j\neq i}\Big( y_{j} + G_{j} \widetilde{G}_{j} -\ba_{j}^{\top} \begin{bmatrix} \bcoefX_{+}\\ \bcoefZ_{+} \\ \end{bmatrix} \Big) \ba_{j} + \Big( y_{i} + G_{i} \widetilde{G}_{i} - \ba_{i}^{\top}\begin{bmatrix} \bcoefX_{+}\\ \bcoefZ_{+} \\ \end{bmatrix}  \Big) \ba_{i} \\
		\label{KKT-condition-estimator-u3-v3-isample-out-proof}
		&\lambda m \begin{bmatrix} \bcoefX_{\bu,\bv}^{-(i)} - \bO_{\bu}^{\top}\bcoefX_{\sharp}\\ \bcoefZ_{\bu,\bv}^{-(i)}-\bO_{\bv}^{\top}\bcoefZ_{\sharp} \\ \end{bmatrix} = \sum_{j\neq i}\Big( y_{j} + G_{j} \widetilde{G}_{j} -\ba_{j}^{\top} \begin{bmatrix} \bO_{\bu}\bcoefX_{\bu,\bv}^{-(i)}\\ \bO_{\bv}\bcoefZ_{\bu,\bv}^{-(i)} \\ \end{bmatrix} \Big) \begin{bmatrix}  \widetilde{G}_{j}\bO_{\bu}^{\top} \bx_{j} \\ G_{j}\bO_{\bv}^{\top} \bz_{j} \\ \end{bmatrix}.
	\end{align}
\end{subequations}
Multiplying $\bO_{\bu}$ or $\bO_{\bv}$ on both sides of Eq.~\eqref{KKT-condition-estimator-u3-v3-isample-out-proof} and using the facts that $\bO_{\bu}\bO_{\bu}^{\top} \bcoefX_{\sharp} = \bcoefX_{\sharp}$ and $\bO_{\bv}\bO_{\bv}^{\top} \bcoefZ_{\sharp} = \bcoefZ_{\sharp}$ then yields
\begin{align*}
	\lambda m \begin{bmatrix} \bO_{\bu}\bcoefX_{\bu,\bv}^{-(i)} - \bcoefX_{\sharp}\\ \bO_{\bv}\bcoefZ_{\bu,\bv}^{-(i)}-\bcoefZ_{\sharp} \\ \end{bmatrix} &= \sum_{j\neq i}\Big( y_{j} + G_{j} \widetilde{G}_{j} -\ba_{j}^{\top} \begin{bmatrix} \bO_{\bu} \bcoefX_{\bu,\bv}^{-(i)}\\ \bO_{\bv} \bcoefZ_{\bu,\bv}^{-(i)} \\ \end{bmatrix} \Big) 
	\begin{bmatrix}  \widetilde{G}_{j}\bO_{\bu}\bO_{\bu}^{\top} \bx_{j} \\ G_{j}\bO_{\bv}\bO_{\bv}^{\top} \bz_{j} \\ \end{bmatrix} \\&= \sum_{j\neq i}\Big( y_{j} + G_{j} \widetilde{G}_{j} -\ba_{j}^{\top} \begin{bmatrix} \bO_{\bu} \bcoefX_{\bu,\bv}^{-(i)}\\ \bO_{\bv} \bcoefZ_{\bu,\bv}^{-(i)} \\ \end{bmatrix} \Big)  \Big(\ba_{j} - \begin{bmatrix}  \widetilde{G}_{j} \bx_{j}^{\top}\bu\bu \\ G_{j} \bz_{j}^{\top}\bv\bv \\ \end{bmatrix} \Big),
\end{align*}
where the last step follows by $\bO_{\bu}\bO_{\bu}^{\top} = \bI - \bu\bu^{\top}$ and $\bO_{\bv}\bO_{\bv}^{\top} = \bI - \bv\bv^{\top}$.  Subtracting the equation in the preceding display from Eq.~\eqref{KKT-condition-estimator-t+1} and re-arranging the terms yields the desired result.

\paragraph{Proof of Lemma~\ref{lemma:eta-onesample-onecolumn-out-expectation-bound}(b.),(c.):}
We first claim that with $\Delta := \max\big\{ \frac{ \sigma + \Err_{\sharp}}{\lambda} \frac{\log^{6}(d)}{\sqrt{d}} , d^{-50} \big\}$,
\begin{align}\label{claim-expectation-bound-component-u3v3sampleout-det}
	\EE\{( \vartheta_{1}^{-i} - \thetaXdet_{1} )^{2}\}, \EE\{(\vartheta_{2}^{-i} - \thetaXdet_{2})^{2}\}, 
	\EE\{( \widetilde{\vartheta}_{1}^{-i} - \thetaZdet_{1} )^{2}\}, \EE\{ (\widetilde{\vartheta}_{2}^{-i} - \thetaZdet_{2})^{2}\} \lesssim \Delta^{2}.
\end{align}
We take this inequality as given for now, deferring its proof to the end.  We then have
\begin{align*}
	|\EE\{e^{2}\} - f^{2}| \leq &\EE\{|(e+f)(e-f)|\} = \EE\{|(e-f + 2f)(e-f)|\} \leq \EE\{|e-f|^{2}\} + 2|f|\cdot \EE\{|e-f|\}, \\ 
	\text{ with }\qquad &e = \frac{\parcompX_{\sharp}\parcompZ_{\sharp}}{L_{\sharp}^{2} \widetilde{L}_{\sharp}^{2}} + 1 - \frac{\vartheta_{1}^{-i}}{L_{\sharp}} - \frac{\widetilde{\vartheta}_{1}^{-i}}{\widetilde{L}_{\sharp}},\qquad \text{ and } f = \frac{\parcompX_{\sharp}\parcompZ_{\sharp}}{L_{\sharp}^{2} \widetilde{L}_{\sharp}^{2}} + 1 - \frac{\thetaXdet_{1}}{L_{\sharp}} - \frac{\thetaZdet_{1}}{\widetilde{L}_{\sharp}}.
\end{align*}
We next apply the inequalities~\eqref{claim-expectation-bound-component-u3v3sampleout-det} to obtain the bound
\[	
\EE\{|e-f|^{2}\} \lesssim \EE\{ ( \vartheta_{1}^{-i} - \thetaXdet_{1} )^{2} \} + \EE\{ ( \widetilde{\vartheta}_{1}^{-i} - \thetaZdet_{1} )^{2} \} \lesssim \Delta^{2},
\]
which, upon applying Jensen's inequality, yields the further bound $\EE\{|e-f|\} \leq \EE\{|e-f|^{2}\}^{1/2} \lesssim \Delta$.
Note that by definition of $\thetaXdet_{1}$~\eqref{definition-theta12-det} and $\thetaZdet_{1}$~\eqref{definition-tilde-theta12-det}, we obtain 
\[
|f| = \frac{\lambda | \parcompX_{\sharp} \parcompZ_{\sharp} - L_{\sharp}^{2}\widetilde{L}_{\sharp}^{2}|}{\lambda L_{\sharp}^{2}\widetilde{L}_{\sharp}^{2} + V(L_{\sharp}^{2} + \widetilde{L}_{\sharp}^{2})} \lesssim |\parcompX_{\sharp}\parcompZ_{\sharp} - 1| + \perpcompX_{\sharp} + \perpcompZ_{\sharp} \lesssim \sqrt{\Err_{\sharp}}.
\]
Putting together the pieces yields the inequality
\[
|\EE\{e^{2}\} - f^{2}| \lesssim \Delta^{2} +\sqrt{\Err_{\sharp}}\Delta \lesssim \max\bigg\{ \frac{ (\sigma^{2} + \Err_{\sharp})}{\lambda} \frac{\log^{6}(d)}{\sqrt{d}} , d^{-50} \bigg\}.
\]
This proves inequality~\eqref{expectation-bound-coefficient1}. Proceeding similarly proves inequality~\eqref{expectation-bound-coefficient2}.  It remains to establish the deferred inequalities~\eqref{claim-expectation-bound-component-u3v3sampleout-det}.

\noindent \underline{Proof of Claim~\eqref{claim-expectation-bound-component-u3v3sampleout-det}:} Recall that $\theta(\bu_{1}) = \langle\bu_{1},\bcoefX_{+}\rangle$ and $\vartheta_{1}^{-i} = \langle\bu_{1},\bO_{\bu}\bcoefX_{\bu,\bv}^{-(i)}\rangle$. Applying the triangle inequality yields the bound
\begin{align}\label{ineq-proof-claim-component-u3v3sampleout-det}
	\EE\big\{ \big(\vartheta_{1}^{-i} - \thetaXdet_{1} \big)^{2} \big\} \leq 
	2\EE \big\{ \big(\vartheta_{1}^{-i} - \theta(\bu_{1}) \big)^{2} \big\} + 2\EE \big\{ \big(\theta(\bu_{1}) - \thetaXdet_{1} \big)^{2} \big\}.
\end{align}
We proceed to bound each term on the RHS of the preceding display. Invoking the relation~\eqref{claim-KKT-relation-estimator(t+1)-u-v-isample-out}, we obtain
\begin{align*}
	|\vartheta_{1}^{-i} - \theta(\bu_{1})|^{2} \lesssim \underbrace{ \Big \langle \begin{bmatrix} \bu_{1}\\ \boldsymbol{0}\\ \end{bmatrix} , \bSig_{i}^{-1} \sum_{j\neq i} \tau_{j}  \begin{bmatrix} \widetilde{G}_{j}\bx_{j}^{\top}\bu \bu \\ G_{j}\bz_{j}^{\top}\bv \bv \\ \end{bmatrix} \Big \rangle^{2} }_{T_1} + \underbrace{ \Big(y_{i} + G_{i}\widetilde{G}_{i} - \ba_i^{\top} \begin{bmatrix} \bcoefX_{+} \\ \bcoefZ_{+} \end{bmatrix} \Big)^{2} \Big\langle \ba_{i}, \bSig_{i}^{-1} \begin{bmatrix} \bu_{1}\\ \boldsymbol{0}\\ \end{bmatrix} \Big\rangle^{2} }_{T_2}.
\end{align*}
Then, applying the Cauchy--Schwarz inequality and the inequality $\bSig_{i}^{-1} \preceq (\lambda m)^{-1}\bI$, we upper bound the expectation of $T_1$ as
\begin{align*}
	\EE\{T_1\} &\leq \frac{1}{(\lambda m)^{2}} \bigg( \EE\Big\{ \Big( \sum_{j \neq i} \tau_{j} \widetilde{G}_{j}\bx_{j}^{\top}\bu \Big)^{2} \Big\} + \EE\Big\{ \Big( \sum_{j \neq i} \tau_{j} G_{j}\bz_{j}^{\top}\bv \Big)^{2} \Big\} \bigg) \lesssim \frac{\sigma^{2} + \Err_{\sharp} }{\lambda^{2}m},
\end{align*}
where the last step follows from the inequality~\eqref{ineq1.5-eta-onesample-onecolumn-out-expectation-bound}. Let $\boldsymbol{e},\boldsymbol{f} \in \mathbb{R}^{d}$ such that $\bSig_{i}^{-1} \begin{bmatrix} \bu_{1}\\ \boldsymbol{0}\\ \end{bmatrix} = \begin{bmatrix} \boldsymbol{e}\\ \boldsymbol{f}\\ \end{bmatrix}$. Since both $\boldsymbol{e},\boldsymbol{f}$ are independent of $\ba_{i}$, we compute
\begin{align*}
	\EE\bigg\{ \Big\langle \ba_{i}, \bSig_{i}^{-1} \begin{bmatrix} \bu_{1}\\ \boldsymbol{0}\\ \end{bmatrix} \Big\rangle^{4}  \bigg\} = \EE\Big\{ \big(\widetilde{G}_{i} \bx_{i}^{\top} \boldsymbol{e} + G_{i} \bz_{i}^{\top} \boldsymbol{f} \big)^{4} \Big\} \lesssim \EE\Big\{ \big(\|\boldsymbol{e}\|_{2} + \|\boldsymbol{f}\|_{2}\big)^{4} \Big\} 
	\leq \frac{1}{(\lambda m)^{4}},
\end{align*}
where in the last step we used the inequality $\|\boldsymbol{e}\|_{2} + \|\boldsymbol{f}\|_{2} \leq (\lambda m)^{-1} \|\bu_{1}\|_{2} = (\lambda m)^{-1}$. Consequently, applying the Cauchy--Schwarz inequality in conjunction with inequality~\eqref{ineq1.8-eta-onesample-onecolumn-out-expectation-bound}, we obtain the bound
\begin{align*}
	\EE\bigg\{ \Big(y_{i} + G_{i}\widetilde{G}_{i} - \ba_i^{\top} \begin{bmatrix} \bcoefX_{+} \\ \bcoefZ_{+} \end{bmatrix} \Big)^{2} \Big\langle \ba_{i}, \bSig_{i}^{-1} \begin{bmatrix} \bu_{1}\\ \boldsymbol{0}\\ \end{bmatrix} \Big\rangle^{2} \bigg\} \leq \frac{\sigma^{2} +\Err_{\sharp}}{(\lambda m)^{2}}.
\end{align*}
Putting the pieces together then yields
\begin{align}\label{ineq1-proof-claim-component-u3v3sampleout-det}
	\EE\big\{ |\vartheta_{1}^{-i} - \theta(\bu_{1})|^{2} \big\} \lesssim (\sigma^{2} + \Err_{\sharp})/(\lambda^{2} m).
\end{align}
We now turn to bound $\EE\{|\theta(\bu_{1}) - \thetaXdet_{1}|^{2}\}$. Applying inequalities~\eqref{bound-u-expectation-u} and~\eqref{upper-bound-expectation-u-det} yields
\[
\Pr\big\{ |\theta(\bu_{1}) - \thetaXdet_{1}| \geq \Delta \big\} \leq 2d^{-110} \qquad  \text{ and } \qquad
\Delta = \max\bigg\{ \frac{ C(\sigma + \sqrt{\Err_{\sharp}}) }{\lambda} \frac{\log^{6}(d)}{\sqrt{d}}, d^{-50} \bigg\}.
\]
Consequently, we obtain that
\begin{align*}
	\EE\{|\theta(\bu_{1}) - \thetaXdet_{1}|^{2}\} &\leq \Delta^{2} + \EE\big\{|\theta(\bu_{1}) - \thetaXdet_{1}|^{2} \cdot \mathbbm{1}\big\{ |\theta(\bu_{1}) - \thetaXdet_{1}| \geq \Delta \big\}  \big\}
	\\&\leq \Delta^{2} + \EE\big\{ |\theta(\bu_{1}) - \thetaXdet_{1}|^{4} \big\}^{1/2} \cdot  \Pr\big\{ |\theta(\bu_{1}) - \thetaXdet_{1}| \geq \Delta \big\}^{1/2}
	\\& \overset{\1}{\leq} \Delta^{2} + (\sigma^{2} + 1)d^{2} \cdot  \sqrt{2} d^{-55} \lesssim \Delta^{2},
\end{align*}
where step $\1$ follows since $\EE\big\{ |\theta(\bu_{1}) - \thetaXdet_{1}|^{4} \big\} \lesssim \EE\{\|\bcoefX_{+}\|_{2}^{4}\} + C\leq (\sigma^{4} + 1)d^{4}$.  Combining the inequality in the display above with inequalities~\eqref{ineq-proof-claim-component-u3v3sampleout-det} and~\eqref{ineq1-proof-claim-component-u3v3sampleout-det} yields the bound
$\EE\big\{ \big(\vartheta_{1}^{-i} - \thetaXdet_{1} \big)^{2} \big\} \leq \Delta^{2}$.  The bounds on the remaining terms follow identically, hence we omit them here. \qed
\section{Deterministic convergence: Proof of Lemma~\ref{lemma:one-step-deterministic-convergence}}\label{proof-lemma-one-step-deterministic-convergence}
Note that $\Errdet = \big(\parcompdetX\parcompdetZ-1 \big)^{2} + (\perpcompdetX)^{2} + (\perpcompdetZ)^{2}
$. We control each of these three terms in turn. Throughout, we use the shorthand $L = \sqrt{\parcompX^{2} + \perpcompX^{2}}$ and $\LZ = \sqrt{\parcompZ^{2} + \perpcompZ^{2}}$.

\paragraph{Controlling $(\parcompdetX \parcompdetZ-1)^{2}$:}
Recall the $\alphaXmap_{m,d,\sigma,\lambda}$ and $\alphaZmap_{m,d,\sigma,\lambda}$---which characterize $\alpha^{\det}$ and $\widetilde{\alpha}^{\det}$, respectively---as defined in Eq.~\eqref{et_updates_eq}. We use the shorthand $\mathsf{D} = V(L^{2} + \LZ^{2}) + \lambda L^{2} \LZ^{2}$, where $V$ as defined in Eq~\eqref{eq:V-V1-V2}.  Expanding, we decompose
\begin{align*}
	\parcompdetX &= \parcompX + \frac{V(\parcompX\parcompZ -L^{2}\LZ^{2})}{L^{2}\mathsf{D}} \parcompX + \frac{V_{1}}{\lambda + V_{1}} \frac{\parcompZ\perpcompX^{2} }{L^{2} \LZ^{2} } \\
	&= 
	\underbrace{ \parcompX + \frac{V\parcompX^{2}\parcompZ(1-\parcompX \parcompZ) }{L^{2}\mathsf{D}} }_{T_{1}} + \underbrace{- \frac{V\parcompX(\perpcompX^{2} \parcompZ^{2} + \perpcompZ^{2} \parcompX^{2} + \perpcompX^{2}\perpcompZ^{2} )}{L^{2} \mathsf{D}} + \frac{V_{1}}{\lambda + V_{1}} \frac{\parcompZ\perpcompX^{2} }{L^{2} \LZ^{2} }}_{T_{2}},
\end{align*}
where in the last step we used the relation $\parcompX\parcompZ - L^{2} \LZ^{2} = \parcompX\parcompZ(1-\parcompX\parcompZ) - (\parcompX^{2}\perpcompZ^{2} + \parcompZ^{2}\perpcompX^{2} + \perpcompX^{2}\perpcompZ^{2})$.  We similarly decompose
\begin{align*}
	\parcompdetZ = 
	\underbrace{ \parcompZ + \frac{V\parcompZ^{2}\parcompX(1-\parcompX \parcompZ) }{\LZ^{2}\mathsf{D}} }_{T_{3}} + \underbrace{- \frac{V\parcompZ(\perpcompX^{2} \parcompZ^{2} + \perpcompZ^{2} \parcompX^{2} + \perpcompX^{2}\perpcompZ^{2} )}{\LZ^{2} \mathsf{D}} + \frac{V_{2}}{\lambda + V_{2}} \frac{\parcompX\perpcompZ^{2} }{L^{2} \LZ^{2} }}_{T_{4}}.
\end{align*}
Before continuing, we claim the following bounds
\begin{align}\label{claim1-constant-of-V-alpha-D}
	&(i.) \;\; 0.9 \leq \parcompX \parcompZ \leq 1.1,\qquad (ii.) \qquad \sqrt{0.03} \leq |\parcompX|, |\parcompZ| \leq 3,\qquad (iii.) \;\; \perpcompX, \perpcompZ \leq 0.1,\nonumber\\
	& \qquad \qquad \qquad \qquad (iv.) \;\;  V,V_{1},V_{2} \asymp 1, \qquad  \text{ and } \qquad  (v.) \;\;\mathsf{D} \asymp \lambda,
\end{align}
deferring its proof to the end.   Combining the decompositions of $\alpha^{\det}$ and $\widetilde{\alpha}^{\det}$ with the bounds in~\eqref{claim1-constant-of-V-alpha-D}, we obtain
\[
	|\parcompdetX - \parcompX|,\;|\parcompdetZ - \parcompZ| \lesssim \sqrt{\Err}/\lambda,
\]
which proves inequality~\eqref{alpha-change-after-det-map}. Consequently, we obtain that
\[
	|\parcompdetX| \leq |\parcompX| + C_{1}\sqrt{\Err}/\lambda \leq 3 + C_{1}\sqrt{c}/C \leq 4,
\]
where we invoked the assumptions $|\parcompX|\leq 3$, $\Err \leq c$ and $\lambda \geq C$. Similarly, $|\parcompdetZ| \leq 4$. This proves the second part of inequality~\eqref{upper-bound-det-error-alpha-iota}. Continuing, we re-write
\begin{align}\label{eq1-alpha-det-error}
	&(\parcompdetX\parcompdetZ - 1)^{2} = \big( (T_{1} + T_{2})(T_{3} + T_{4}) - 1 \big)^{2} = (T_{1}T_{3}-1 + T_{1}T_{4} + T_{2} T_{3} + T_{2}T_{4} )^{2} \nonumber\\
	& \qquad\qquad = (T_{1}T_{3}-1)^{2} + 2(T_{1}T_{3}-1)(T_{1}T_{4} + T_{2} T_{3} + T_{2}T_{4}) + (T_{1}T_{4} + T_{2} T_{3} + T_{2}T_{4})^{2}.
\end{align}
We first bound $(T_{1}T_{3}-1)^{2}$. Straightforward calculation yields
\begin{align*}
	T_{1}T_{3}-1 = (\parcompX \parcompZ - 1) \bigg( 1 - \frac{V\parcompX^{2} \parcompZ^{2} }{\mathsf{D}} (L^{-2} + \LZ^{-2}) + \frac{V^{2}\parcompX^{3} \parcompZ^{3} (\parcompX \parcompZ - 1)}{L^{2}\LZ^{2} \mathsf{D}^{2} } \bigg)
\end{align*}
Then, by the claim~\eqref{claim1-constant-of-V-alpha-D}, there exist universal, positive constants $c_{1},c_{2}$ such that both
\[
	\frac{c_{1}}{\lambda} \leq \frac{V\parcompX^{2} \parcompZ^{2} }{\mathsf{D}} (L^{-2} + \LZ^{-2}) \leq \frac{c_{2}}{\lambda} \qquad \text{ and } \qquad \frac{V^{2}|\parcompX^{3} \parcompZ^{3} (\parcompX \parcompZ - 1)|}{L^{2}\LZ^{2} \mathsf{D}^{2}} \leq \frac{c_{2}}{\lambda^{2}}.
\]
By letting $\lambda \geq C(1+\sigma)d/m \geq C$ so that $\frac{c_{2}}{\lambda^{2}} \leq \frac{c_{1}}{2\lambda}$, we obtain 
\[
	(\parcompX\parcompZ-1)^{2} \cdot \Big(1-\frac{2c_{2}}{\lambda}\Big)^{2}  \leq (T_{1}T_{3}-1)^{2} \leq (\parcompX\parcompZ-1)^{2} \cdot \Big( 1-\frac{c_{1}}{2\lambda}\Big)^{2}.
\]
Since $\lambda \geq C$ for a large enough $C$, we note that $(1-2c_{2}/\lambda)^{2} \geq 1 - 4c_{2}/\lambda$ and $(1-c_{1}/(2\lambda))^{2} \leq 1-c_{1}/(4\lambda)$. Putting the pieces together yields
\begin{align}\label{ineq1-alpha-part1-det-error}
(\parcompX\parcompZ-1)^{2} \cdot \Big(1-\frac{4c_{2}}{\lambda}\Big)  \leq (T_{1}T_{3}-1)^{2} \leq (\parcompX\parcompZ-1)^{2} \cdot \Big( 1-\frac{c_{1}}{4\lambda}\Big).
\end{align}
Next, we apply the claim~\eqref{claim1-constant-of-V-alpha-D} to obtain the pair of bounds
\[
	|T_{1}|, |T_{3}| \lesssim 1 \qquad \text{ and } \qquad |T_{2}|\;\vee\;|T_{4}| \lesssim (\perpcompX^{2} + \perpcompZ^{2}) / \lambda.
\]
Combining with the inequality~\eqref{ineq1-alpha-part1-det-error}, we conclude that both
\begin{align*}
	2|T_{1}T_{3}-1| \cdot |T_{1}T_{4} + T_{2} T_{3} + T_{2}T_{4}| &\leq \frac{C'|\parcompX\parcompZ-1| \cdot (\perpcompX^{2} + \perpcompZ^{2})}{\lambda},\qquad \text{ and } \\
	 (T_{1}T_{4} + T_{2} T_{3} + T_{2}T_{4})^{2} &\leq \frac{C'(\perpcompX^{2} + \perpcompZ^{2})^{2}}{\lambda^{2}}.
\end{align*}
Combining the inequalities in the preceding display with the inequality~\eqref{ineq1-alpha-part1-det-error} and Eq.~\eqref{eq1-alpha-det-error}, we obtain
\begin{subequations}\label{one-step-det-error-alpha-part}
\begin{align}
\big(\parcompdetX\parcompdetZ - 1\big)^{2} &\leq (\parcompX\parcompZ-1)^{2} \cdot \Big( 1-\frac{c_{1}}{4\lambda}\Big) + \frac{C'|\parcompX\parcompZ-1| \cdot (\perpcompX^{2} + \perpcompZ^{2})}{\lambda} + 
\frac{C'(\perpcompX^{2} + \perpcompZ^{2})^{2}}{\lambda^{2}},\label{det-error-alpha-upper} \\
\big(\parcompdetX\parcompdetZ - 1\big)^{2} &\geq (\parcompX\parcompZ-1)^{2} \cdot \Big( 1-\frac{2c_{2}}{\lambda}\Big) - \frac{C'|\parcompX\parcompZ-1| \cdot (\perpcompX^{2} + \perpcompZ^{2})}{\lambda} - 
\frac{C'(\perpcompX^{2} + \perpcompZ^{2})^{2}}{\lambda^{2}}.\label{det-error-alpha-lower}
\end{align}
\end{subequations}
Before turning to control the orthogonal component, we establish the deferred claims~\eqref{claim1-constant-of-V-alpha-D}.

\medskip
\noindent \underline{Proof of the claims~\eqref{claim1-constant-of-V-alpha-D}:} Using $\Err \leq c$, we obtain that $(\parcompX \parcompZ-1)^{2} \leq c$. When $c\leq 0.01$, we obtain that $0.9 \leq \parcompX \parcompZ \leq 1.1$, which establishes part (i.). 

By $L,\LZ \in [0.2,3]$, we obtain $|\parcompX|\leq L \leq 3$ and $|\parcompZ|\leq \LZ \leq 3$. Continuing, since $\Err \leq c$, we obtain $\perpcompX^{2} + \perpcompZ^{2} \leq 0.01$ for $c\leq 0.01$. Then, $\parcompX^{2} = L^{2} - \perpcompX^{2} \geq 0.04 - 0.01 = 0.03$ and similarly $\parcompZ^{2} \geq 0.03$. So we conclude that $|\parcompX|,|\parcompZ| \asymp 1$ and $\perpcompX,\perpcompZ \leq 0.1$, which establishes parts (ii.) and (iii.). 

Continuing, by definition of $V,V_{1},V_{2}$ in Eq~\eqref{eq:V-V1-V2}, we find that 
\[
	V \leq \EE\{G_{1}^{2}G_{2}^{2}\} = L^{2}\LZ^{2} \lesssim 1,\qquad V_{1} \leq \EE\{G_{2}^{2}\} = \LZ^{2} \lesssim 1,\qquad \text{ and } \qquad
	V_{2} \leq \EE\{G_{1}^{2}\} = L^{2} \lesssim 1.
\]
Towards lower bounding each of the quantities $V,V_{1},V_{2}$, we note that, by definition,
\[
	V = \EE\bigg\{ \frac{G_{1}^{2}G_{2}^{2} }{ 1 + G_{1}^{2}/r_{2} + G_{2}^{2}/r_{1} }\bigg\},\; V_{1} = \EE \bigg\{ \frac{G_{2}^{2}}{1 + G_{1}^{2}/r_{2} + G_{2}^{2}/r_{1}}\bigg\} \text{ and } V_{2} = \EE \bigg\{ \frac{G_{1}^{2}}{1 + G_{1}^{2}/r_{2} + G_{2}^{2}/r_{1}}\bigg\}.
\]
Further, by Eq.~\eqref{eq:fixed-point}, we obtain $r_1,r_2 \geq \lambda \oversamp = \lambda m/d \geq C$, 
where in the last step we have applied the assumption $\lambda m \geq Cd$. Putting the two pieces together yields the lower bound
\[
	V \geq  \EE\bigg\{ \frac{G_{1}^{2}G_{2}^{2} }{ 1 + G_{1}^{2}/C + G_{2}^{2}/C }\bigg\} \gtrsim 1,
\]
where the last step follows since $G_1 \sim \mathsf{N}(0,L^{2})$, $G_2 \sim \mathsf{N}(0,\LZ^{2})$ and $C>0$ is a universal constant. Similarly, we obtain $V_1,V_2 \gtrsim 1$. Consequently, we obtain $V,V_1,V_2 \asymp 1$, which establishes part (iv.).

Towards bounding $\mathsf{D}$, by $0.2 \leq L,\LZ \leq 3$, we obtain $\mathsf{D} = V(L^{2} + \LZ^{2}) + \lambda L^{2} \LZ^{2} \geq  0.2^{4} \cdot \lambda$. By $V \leq L^{2}\LZ^{2}$, we obtain $
	\mathsf{D} \leq  (L^{2} + \LZ^{2} + \lambda ) L^{2} \LZ^{2} \leq (\lambda + 18) \cdot 3^{4} \leq 2\cdot 3^{4} \cdot \lambda$, where the last step holds for $\lambda \geq 18$. Consequently, $\mathsf{D} \asymp \lambda$, establishing part (v.). This concludes the proof of the claims in~\eqref{claim1-constant-of-V-alpha-D}.

\paragraph{Controlling $(\perpcompdetX)^{2} + (\perpcompdetZ)^{2}$.} Let the functions $\iotaXmap_{m,d,\sigma,\lambda},\iotaZmap_{m,d,\sigma,\lambda},\etaXmap_{m,d,\sigma,\lambda},\etaZmap_{m,d,\sigma,\lambda}$---which characterize the orthogonal components $\beta^{\det}$ and $\widetilde{\beta}^{\det}$---be as defined in equations~\eqref{et_updates_eq},~\eqref{eta_updates_eq} and further let
\begin{align*}
	\iotadetX &= \iotaXmap_{m,d,\sigma,\lambda}(\parcompX,\perpcompX,\parcompZ,\perpcompZ),\qquad &(\etadetX)^{2} = \etaXmap_{m,d,\sigma,\lambda}(\parcompX,\perpcompX,\parcompZ,\perpcompZ),\\
	 \iotadetZ &= \iotaZmap_{m,d,\sigma,\lambda}(\parcompX,\perpcompX,\parcompZ,\perpcompZ),\; &(\etadetZ)^{2} = \etaZmap_{m,d,\sigma,\lambda}(\parcompX,\perpcompX,\parcompZ,\perpcompZ).
\end{align*}
Note that by definition, we have the decompositions $(\perpcompdetX)^{2} = (\iotadetX)^{2} + (\etadetX)^{2}$ and $(\perpcompdetZ)^{2} = (\iotadetZ)^{2} + (\etadetZ)^{2}$. We first bound $(\iotadetX)^{2} + (\iotadetZ)^{2}$ and then turn to bound $(\etadetX)^{2} + (\etadetZ)^{2}$. Recall the shorthand $\mathsf{D} = V(L^{2} + \LZ^{2}) + \lambda L^{2} \LZ^{2}$.  Then, we expand $\iotadetX$ according its definition to obtain 
\[	
	\iotadetX = \bigg(1+\frac{V(\parcompX\parcompZ - L^{2} \LZ^{2})}{\mathsf{D}L^{2}} - \frac{\parcompX\parcompZ}{L^{2}\LZ^{2}} \frac{V_{1}}{\lambda + V_{1}}\bigg) \cdot \perpcompX.
\]
Recall that $|\parcompX\parcompZ| \leq 1.1$, $|\parcompX|,|\parcompZ| \leq 3$ by the claim~\eqref{claim1-constant-of-V-alpha-D}, and $|\parcompX\parcompZ-1| \leq \sqrt{\Err} \leq \sqrt{c}, \perpcompX^{2} + \perpcompZ^{2} \leq \Err \leq c$ by assumption. Thus, applying the triangle inequality, we obtain
\[
   |\parcompX\parcompZ - L^{2} \LZ^{2}| =  |\parcompX\parcompZ(1-\parcompX\parcompZ) - (\parcompX^{2}\perpcompZ^{2} + \parcompZ^{2}\perpcompX^{2} + \perpcompX^{2}\perpcompZ^{2})|  \leq 1.1\sqrt{c} + 10c.
\]
Combining this with the claim~\eqref{claim1-constant-of-V-alpha-D}, we find that there exists universal constants $c'_1,c'_{2}$ and $c'_3$ such that both
\[
	\bigg|\frac{V(\parcompX\parcompZ - L^{2} \LZ^{2})}{\mathsf{D}L^{2}}\bigg| \leq \frac{c'_1}{\lambda}(\sqrt{c} + c) \qquad \text{ and } \qquad \frac{c'_{2}}{\lambda} \leq \frac{\parcompX\parcompZ}{L^{2}\LZ^{2}} \frac{V_{1}}{\lambda + V_{1}} \leq \frac{c'_{3}}{\lambda}.
\]
Consequently, 
\[
	\iotadetX \leq \bigg( 1 + \frac{c'_1}{\lambda}(\sqrt{c} + c) - \frac{c'_{2}}{\lambda} \bigg) \perpcompX \leq \bigg( 1 - \frac{c'_{2}}{2\lambda}\bigg) \perpcompX,
\]
where the last step follows by letting $c$ be a small enough constant to ensure the bound $c'_1(\sqrt{c} + c) \leq c'_{2}/2$.  Similarly, as long as $c\leq 0.01$, we lower bound $\iotadetX$ as
\[
	\iotadetX \geq \bigg( 1 - \frac{c'_1}{\lambda}(\sqrt{c} + c) - \frac{c'_{3}}{\lambda} \bigg) \perpcompX \geq \bigg( 1-\frac{c'_1+c'_3}{\lambda} \bigg) \perpcompX.
\]
Putting the upper and lower bound together and proceeding similarly for $\iotadetZ$, we obtain---for $\lambda \geq C$---the pair of inequalities
\begin{align}\label{ineq2-iotaX-det-upper-lower-bound}
	\bigg(1-\frac{2(c'_{1}+c'_{3})}{\lambda}\bigg)\perpcompX^{2} \leq (\iotadetX)^{2} &\leq \bigg(1-\frac{c'_{2}}{2\lambda}\bigg) \perpcompX^{2}, \qquad \text{ and } \nonumber\\
	\bigg(1-\frac{2(c'_{1}+c'_{3})}{\lambda}\bigg)\perpcompZ^{2} \leq (\iotadetZ)^{2} &\leq \bigg(1-\frac{c'_{2}}{2\lambda}\bigg) \perpcompZ^{2}.
\end{align}
We next turn to bound $(\etadetX)^{2} + (\etadetZ)^{2}$. Recall the quantities $V_{3}$ and $V_{4}$ as defined in Eq.~\eqref{definition-V3} and Eq.~\eqref{definition-V4}. We make the following claim
\begin{align}\label{claim2-bound-V3-V4}
	 \frac{C_{2} \sigma^{2} d^{2}}{(\lambda m)^{2}}\leq V_{3}, V_{4} \leq \frac{C_{1}(\Err + \sigma^{2}) d^2}{(\lambda m)^{2}},
\end{align}
postponing its proof until the end.  Taking the inequalities above as given and further using Eqs.~\eqref{det_updates_eta1} and~\eqref{det_updates_eta2}, we obtain the lower bounds
\begin{align}\label{ineq3-lower-bound-eta-det}
	(\etadetX)^{2} \geq \frac{m}{2d}V_{3} \geq \frac{m}{2d} \frac{C_{2}\sigma^{2}d^{2}}{(\lambda m)^{2}} \geq \frac{C_{2}d\sigma^{2}}{2\lambda^{2}m} \qquad\text{ and similarly } \qquad (\etadetZ)^{2} \geq \frac{C_{2}d\sigma^{2}}{2\lambda^{2}m}.
\end{align}
Turning to upper bounds on $\etadetX$ and $\etadetZ$, we apply Eqs.~\eqref{det_updates_eta1} and~\eqref{det_updates_eta2} to obtain the inequality
\begin{align}\label{ineq4-auxilary-upper-bound-eta-det}
	(\etadetX)^{2} &\leq \frac{m}{d} \Big( \frac{(\etadetX)^{2} \EE\{G_{2}^{4}\}}{r_{1}^{2}} + \frac{(\etadetZ)^{2} \EE\{ G_{1}^{2}G_{2}^{2}\}}{r_{1}^{2}} + V_{3} \Big)\nonumber\\
	 &\leq \frac{C' d}{\lambda^{2} m} \Big( (\etadetX)^{2} + (\etadetZ)^{2} + \Err +\sigma^{2} \Big),
\end{align}
where the last step follows from the bound $r_{1} \geq \lambda \oversamp = \lambda m/d$ (see Lemma~\ref{fixed-point-equations-unique-solution}), the upper bound of $V_{3}$ in claim~\eqref{claim2-bound-V3-V4} and by taking $C'$ large enough to ensure $C_{1},\EE\{G_{2}^{4}\},\EE\{G_{1}^{2} G_{2}^{2}\} \leq C'$.  Note that $C'd/(\lambda^{2} m) \leq C'm/(C^{2} d) \leq 1/2$ since $\lambda m \geq Cd$ and $C$ is a large enough constant. Consequently, rearranging the terms of inequality~\eqref{ineq4-auxilary-upper-bound-eta-det}, we deduce
\begin{align}\label{ineq5-auxilary-upper-bound-eta-det}
	(\etadetX)^{2} \leq \frac{2C' d}{\lambda^{2} m} \Big( (\etadetZ)^{2} + \Err +\sigma^{2} \Big).
\end{align}
Identical steps yield the analogous bound
\begin{align}\label{ineq6-auxilary-upper-bound-eta-det}
	(\etadetZ)^{2} \leq \frac{2C' d}{\lambda^{2} m} \Big( (\etadetX)^{2} + \Err +\sigma^{2} \Big).
\end{align}
Combining inequalities~\eqref{ineq5-auxilary-upper-bound-eta-det} and~\eqref{ineq6-auxilary-upper-bound-eta-det} yields
\begin{align}\label{ineq7-ineq3-upper-bound-eta-det}
	(\etadetX)^{2}, (\etadetZ)^{2} \leq \frac{10C'd}{\lambda^{2}m}  (\Err + \sigma^{2}).
\end{align}
This concludes the proof for $\etadetX$ and $\etadetZ$.   Before concluding, we provide the proof of the deferred claim~\eqref{claim2-bound-V3-V4}.

\medskip
\noindent\underline{Proof of Claim~\eqref{claim2-bound-V3-V4}:} Starting with the lower bound for $V_{3}$, by Eq.~\eqref{definition-V3}, we obtain that
\[
	V_{3} \geq \frac{\sigma^{2}}{r_{1}^{2}} \EE\bigg\{ \frac{G_{2}^{2}}{(1+G_{1}^{2}/r_{2} + G_{2}^{2}/r_1)^{2}} \bigg\} \overset{\1}{\gtrsim} \frac{\sigma^{2}}{r_{1}^{2}} \gtrsim \frac{\sigma^{2} d^{2}}{(\lambda m)^{2}},
\]
where in step $\1$ we use $r_1,r_2 \geq \lambda m/d \geq C$ (see Lemma~\ref{fixed-point-equations-unique-solution}) so that 
\[
	\EE\bigg\{ \frac{G_{2}^{2}}{(1+G_{1}^{2}/r_{2} + G_{2}^{2}/r_1)^{2}} \bigg\} \geq \EE\bigg\{ \frac{G_{2}^{2}}{(1+G_{1}^{2}/C + G_{2}^{2}/C)^{2}} \bigg\} \gtrsim 1,
\]
and in the last step we use $r_1 \leq 2\lambda m/d$ (see Lemma~\ref{fixed-point-equations-unique-solution}).
Turning to the upper bound for $V_{3}$, by Eq.~\eqref{definition-V3} and $0.2\leq L,\LZ \leq 3$, we obtain that
\begin{align*}	
	V_{3} &\lesssim (\sigma^{2} + \perpcompX^{2}\perpcompZ^{2})r_{2}^{-2} + (\parcompX \parcompZ - L^{2}\LZ^{2})^{2}  r_{2}^{-2} + \perpcompZ^{2} r_{2}^{-2} + \perpcompZ^{2} r_{2}^{-2} \lesssim \frac{(\sigma^{2} + \Err)d^2}{(\lambda m)^{2}},
\end{align*}
where the last step follows since $r_{2}\geq \lambda m/d$ (see Lemma~\ref{fixed-point-equations-unique-solution}) and 
$\perpcompX^{2}\perpcompZ^{2},\perpcompX^{2},\perpcompZ^{2},(\parcompX \parcompZ - L^{2}\LZ^{2})^{2}  \lesssim \Err$. The proof for bounding $V_{4}$ is identical so we omit it here. 

\paragraph{Putting together the pieces:} Combining inequalities~\eqref{det-error-alpha-upper},~\eqref{ineq2-iotaX-det-upper-lower-bound} and~\eqref{ineq7-ineq3-upper-bound-eta-det}, we obtain that
\begin{align*}
	\Errdet &= (\parcompdetX\parcompdetZ-1)^{2} + (\iotadetX)^{2} + (\iotadetZ)^{2} + (\etadetX)^{2} + (\etadetZ)^{2} 
	\\&\leq \Big(1-\frac{c_{1}}{4\lambda}\Big)(\parcompX \parcompZ-1)^{2} + \frac{C'|\parcompX\parcompZ-1| \cdot (\perpcompX^{2} + \perpcompZ^{2})}{\lambda} + 
\frac{C'(\perpcompX^{2} + \perpcompZ^{2})^{2}}{\lambda^{2}} \\& \qquad \qquad\qquad \qquad \qquad \qquad\qquad + \Big(1-\frac{c'_{2}}{2\lambda} \Big)(\perpcompX^{2} + \perpcompZ^{2}) + \frac{10C'd(\Err + \sigma^{2})}{\lambda^{2}m}.
\end{align*}
Then, using $|\parcompX\parcompZ-1|\leq \sqrt{\Err} \leq \sqrt{c}$ and $\perpcompX^{2}+\perpcompZ^{2} \leq \Err \leq c$, we upper bound the deterministic error as
\begin{align*}
\Errdet &\leq \Big( 1 - \frac{\min(c_{1}/4,c'_{2}/2)}{\lambda}  + \frac{C'\sqrt{c}}{\lambda} + \frac{C'\cdot c}{\lambda^{2}} + \frac{10C'd}{\lambda^{2}m}\Big) \cdot \Err + \frac{10C'd\sigma^{2}}{\lambda^{2}m} \\
&\leq \Big( 1 - \frac{\min(c_{1}/4,c'_{2}/2)}{2\lambda}\Big) \Err + \frac{10C'd\sigma^{2}}{\lambda^{2}m},
\end{align*}
where the last step follows by taking $c$ small enough and $\lambda \geq Cd/m$ with $C$ large enough to ensure
\begin{align*}
	 \frac{C'\sqrt{c}}{\lambda} + \frac{C'\cdot c}{\lambda^{2}} + \frac{10C'd}{\lambda^{2}m} \leq \frac{C'\sqrt{c}}{\lambda} + \frac{C'\cdot c}{\lambda^{2}} + \frac{10C'}{\lambda C} \leq \frac{\min(c_{1}/4,c'_{2}/2)}{2\lambda}.
\end{align*}
Towards lower bounding the deterministic error, we combine the inequalities~\eqref{det-error-alpha-lower},~\eqref{ineq2-iotaX-det-upper-lower-bound} and~\eqref{ineq3-lower-bound-eta-det} to obtain
\begin{align*}
	\Errdet &\geq (\parcompX\parcompZ-1)^{2} \cdot \Big( 1-\frac{2c_{2}}{\lambda}\Big) - \frac{C'|\parcompX\parcompZ-1| \cdot (\perpcompX^{2} + \perpcompZ^{2})}{\lambda} - 
	\frac{C'(\perpcompX^{2} + \perpcompZ^{2})^{2}}{\lambda^{2}} 
	\\& \qquad \qquad \qquad \qquad \qquad \qquad\qquad\quad + \bigg(1-\frac{2(c'_{1}+c'_{3})}{\lambda}\bigg)(\perpcompX^{2}+\perpcompZ^{2}) + \frac{C_{2}d\sigma^{2}}{2\lambda^{2}n} 
	\\&\geq \Big( 1 - \frac{2(c'_1 + c'_{3} + c_{2}) + C'(\sqrt{c}+c)}{\lambda} \Big) \Err + \frac{C_{2}d\sigma^{2}}{2\lambda^{2}n}.
\end{align*}
The lower bound follows upon taking $c\leq 0.01$. This proves inequality~\eqref{sharp-bounds-det-error} by noting that $c_{1},c_{2},c'_{1},c'_{2},c'_{3},C',C_{2}$ are universal, positive constants. Putting together inequalities~\eqref{det-error-alpha-upper} and~\eqref{ineq2-iotaX-det-upper-lower-bound} proves the first part of inequality~\eqref{upper-bound-det-error-alpha-iota}. This concludes the proof of Lemma~\ref{lemma:one-step-deterministic-convergence}. \qed

\section{Concentration of the trace inverse}\label{concentration-trace-inverse}
This section is dedicated to the following lemma---which establishes the concentration properties alluded to by Eq.~\eqref{introduction-concentartion-tarce-inverse}---as well as its proof.

\begin{lemma}\label{lemma:probability-bound-trace-inverse}
Consider $\bu, \bv \in \mathbb{S}^{d-1}$ and let the columns of $\bO_{\bu}, \bO_{\bv} \in \mathbb{R}^{d \times (d-1)}$ form orthonormal bases of the subspaces orthogonal to $\bu, \bv$, respectively.  
Suppose $\{\bx_{i},\bz_{i}\}_{i=1}^{m} \overset{\mathsf{i.i.d.}}{\sim} \mathsf{N}(0,\bI_{d})$.  Further, let
\begin{align}\label{notation-B-C-D-Sigma}
	&G_{i} = \bx_{i}^{\top}\bcoefX_{\sharp}, \widetilde{G}_{i} = \bz_{i}^{\top} \bcoefZ_{\sharp}\;\forall\;i\in[m],  &&\bB = \sum_{i=1}^{m} \widetilde{G}_{i}^{2} \bO_{\bu}^{\top} \bx_i (\bO_{\bu}^{\top} \bx_i)^{\top}  + \lambda m\cdot \bI_{d-1}, \nonumber \\
	&\bC = \sum_{i=1}^{m} G_{i}\widetilde{G}_{i} \cdot \bO_{\bu}^{\top}\bx_i (\bO_{\bv}^{\top}\bz_{i})^{\top}, &&\bD = \sum_{i=1}^{m} G_{i}^{2}\bO_{\bv}^{\top}\bz_i (\bO_{\bv}^{\top}\bz_{i})^{\top} + \lambda m\cdot \bI_{d-1},
\end{align}
and let $r_1,r_2$ be the solution of the fixed point equation~\eqref{eq:fixed-point}. Suppose $L_{\sharp},\widetilde{L}_{\sharp} \leq C$ for some universal constant $K > 0$. Then, there exist universal constants $d_{0},c_{1},C_{1}>0$ depending only on $K$ such that for $1\leq m\leq d, d\geq d_{0},\lambda m \geq C_{1}d$ and $t\geq \frac{C_{1}\log^{1.5}(d)}{\sqrt{d}}$, the following hold
\begin{align*}
	&\Pr \big\{ \big | \trace\big( (\bB-\bC \bD^{-1} \bC^{\top})^{-1} \big) - r_1^{-1}   \big| \geq t \big\} \leq 2\exp\bigg\{ \frac{-c_{1}d^2t^{2}}{m\log^{2}(d)}\bigg\} + d^{-15}, \\
	&\Pr \big\{ \big | \trace\big( (\bD-\bC^{\top} \bB^{-1} \bC)^{-1} \big) - r_2^{-1}  \big| \geq t \big\} \leq 2\exp\bigg\{   \frac{-c_{1}d^2t^{2}}{m\log^{2}(d)}\bigg\} + d^{-15}.
\end{align*}
\end{lemma}
\begin{proof}
We need the following two auxiliary lemmas.  We provide the proof of Lemma~\ref{lemma:trace-expectation-concentration} in Section~\ref{sec:proof-lemma-trace-expectation-concentration} and the proof of Lemma~\ref{lemma:expectation-of-trace-fixed-dolution-distance} in Section~\ref{sec:proof-lemma-expectation-trace-fixed-solution-distance}.

\begin{lemma}\label{lemma:trace-expectation-concentration}
Under the setting of Lemma~\ref{lemma:probability-bound-trace-inverse}, there exist positive constants $d_{0},c_{1},C_{1}>0$, depending only on $K$, such that for any $1\leq m \leq d, d\geq d_{0},\lambda m \geq C_{1}d$ and $t\geq 2d^{-24}$ the following hold
\begin{align*}
\hspace{-1cm}
	&\Pr \big\{ \big | \trace\big( (\bB-\bC \bD^{-1} \bC^{\top})^{-1} \big) - \EE \big\{  \trace\big( (\bB-\bC \bD^{-1} \bC^{\top})^{-1} \big) \big\} \big| \geq t \big\}  \leq 2\exp\bigg\{ \frac{-c_{1}d^2t^{2}}{m\log^{2}(d)}\bigg\} + d^{-20}, \\
	&\Pr \big\{ \big | \trace\big( (\bD-\bC^{\top} \bB^{-1} \bC)^{-1} \big) - \EE \big\{   \trace\big( (\bD-\bC^{\top} \bB^{-1} \bC)^{-1} \big) \big\} \big| \geq t \big\} \leq 2\exp\bigg\{  \frac{-c_{1}d^2t^{2}}{m\log^{2}(d)}\bigg\} + d^{-20}.
\end{align*}
\end{lemma}

\begin{lemma}\label{lemma:expectation-of-trace-fixed-dolution-distance}
Under the setting of Lemma~\ref{lemma:probability-bound-trace-inverse}, there exist positive constants $d_{0},C_1,C_2>0$, depending only on $K$, such that for $1\leq m \leq d, d\geq d_{0},\lambda m\geq C_1d$ the following hold
\begin{align*}
	 &\Big| \EE \big\{ \trace \big( (\bB - \bC \bD^{-1} \bC^{\top})^{-1} \big) \big\} - r_1^{-1} \Big| \vee \Big| \EE \big\{ \trace \big( (\bD - \bC^{\top} \bB^{-1} \bC)^{-1} \big) \big\} - r_2^{-1} \Big| \leq \frac{C_2\log^{1.5}(d)}{\sqrt{d}}.
\end{align*}
\end{lemma}

\noindent The desired result follows immediately upon decomposing
\begin{align*}
	\big| \trace\big( (\bB-\bC \bD^{-1} \bC^{\top})^{-1} \big) - r_1^{-1} \big) \big\} \big| \leq & \; \big| \trace\big( (\bB-\bC \bD^{-1} \bC^{\top})^{-1} \big) - \EE \big\{  \trace\big( (\bB-\bC \bD^{-1} \bC^{\top})^{-1} \big) \big\} \big| \\&+ \big| \EE \big\{  \trace\big( (\bB-\bC \bD^{-1} \bC^{\top})^{-1} \big) \big\} - r_1^{-1} \big|,
\end{align*}
and subsequently applying Lemma~\ref{lemma:trace-expectation-concentration} to control the first term in the RHS and applying Lemma~\ref{lemma:expectation-of-trace-fixed-dolution-distance} to control the second term.
\end{proof}

\subsection{Proof of Lemma~\ref{lemma:trace-expectation-concentration}}\label{sec:proof-lemma-trace-expectation-concentration}
We follow the strategy considered in Section~\ref{sec:stochastic-error-parallel-component}.  To this end, 
define the function $f: \mathbb{R}^{2d \times m} \rightarrow \mathbb{R}$ and the Hamming metric $\rho$ as
\begin{align}\label{definition-f}
	f\big(\{\bx_i,\bz_i\}_{i=1}^{m}\big) = \trace\big( (\bB - \bC \bD^{-1} \bC^{\top})^{-1} \big),
\end{align}
where the matrices $\bB,\bC$ and $\bD$ are as in~\eqref{notation-B-C-D-Sigma}, and consider the Hamming metric $\rho: \mathbb{R}^{2d \times m} \rightarrow \mathbb{R}$ as $\rho\big( \{\bx_i,\bz_i\}_{i=1}^{m}, \{\bx'_i,\bz'_i\}_{i=1}^{m} \big) = \sum_{i=1}^{m} \mathbbm{1}\big\{ (\bx_{i},\bz_{i}) \neq (\bx'_{i},\bz'_{i}) \big\}$.  As in Section~\ref{sec:stochastic-error-parallel-component}, the proof consists of three steps: First, we define a regularity set $\mathcal{S}$ on which the bounded difference property holds, then we apply a truncation argument to restrict ourselves to the regularity set $\mathcal{S}$, and finally we conclude by applying the typical bounded differences inequality.  We now execute each of these steps in sequence. 

\paragraph{Step 1: Constructing the regularity set $\mathcal{S}$.} Define $\mathcal{S} \subseteq \mathbb{R}^{2d \times m}$ as
\begin{align}\label{definition-of-regularity-set}
	\mathcal{S} = \Big\{ \{\widetilde{\bx}_i,\widetilde{\bz}_i\}_{i=1}^{m} \in \mathbb{R}^{2d\times m}:\; &\max_{i\in[m]}\big(\|\widetilde{\bx}_i\|_{2}^{2} + \|\widetilde{\bz}_i\|_{2}^{2}\big) \leq Cd \nonumber \\
	&,\;\max_{i \in [m]} |\widetilde{\bx}_i^{\top} \bcoefX_{\sharp} |,\max_{i \in [m]}|\widetilde{\bz}_{i}^\top \bcoefZ_{\sharp}|  \leq C \sqrt{\log d} \Big\},
\end{align}
where $C$ is a large enough universal constant. Throughout, we will use the shorthand $\bM = \{\bx_i,\bz_i\}_{i=1}^{m}$ and $\bM' = \{\bx'_i,\bz'_i\}_{i=1}^{m}$.  The next lemma, whose proof we defer to Section~\ref{sec:proof-aucillary-lemma1-trace-expectation-concentration}, establishes the bounded differences property on the regularity set $\mathcal{S}$. 
\begin{lemma}\label{aucillary-lemma1-trace-expectation-concentration} Consider $f$~\eqref{definition-f} and $\mathcal{S}$~\eqref{definition-of-regularity-set}. The following holds.
\begin{subequations}
\begin{itemize}
	\item[(a)] $ 0 \leq f(\bM) \leq \frac{d}{\lambda m},\quad \forall \; \bM \in \mathbb{R}^{2d \times n}.$
    \item[(b)] There exists a universal, positive constant $C_{0}$ such that if $\lambda m \geq d$, then 
    \[
	\big| f(\bM) - f(\bM') \big| \leq C_0 \cdot \log(d)/(\lambda m) : = \Delta,\quad  \text{ for } \quad \bM,\bM' \in \mathcal{S} \text{ such that } \rho(\bM,\bM') \leq 2.
	\]
\end{itemize}
\end{subequations}
\end{lemma}

\paragraph{Step 2: Truncation.} Let $D = 1$ and note that by Lemma~\ref{aucillary-lemma1-trace-expectation-concentration}(a), this setting ensures $D \geq \sup_{\bM \in \mathcal{S}} |f(\bM)|$ as long as $\lambda m\geq d$. We then define the function $f^{\downarrow}:\mathbb{R}^{2d\times m} \rightarrow \mathbb{R}$ as 
\begin{align}\label{definition-f-down}
	f^{\downarrow}(\bM) = \inf_{\bM' \in \mathcal{S}} \Big\{ f(\bM') + \Delta \cdot \rho(\bM,\bM') + 2D \cdot \mathbbm{1}\{ \rho(\bM,\bM') > 1 \} \Big\},
\end{align}
where $\Delta$ is as defined in Lemma~\ref{aucillary-lemma1-trace-expectation-concentration}. Applying Lemma~\ref{aucillary-lemma1-trace-expectation-concentration} in conjunction with Lemma~\ref{lemma:bounded-difference-truncate}, we deduce that both 
\begin{align}\label{properties-truncate-f-trace-inverse}
f^{\downarrow}(\bM) = f(\bM), \qquad \text{ and } \qquad  |f(\bM) - f(\bM')| \leq 2 \Delta,
\end{align}
for all $\bM\in \mathcal{S}$, $\bM' \in \mathbb{R}^{2d \times m}$ which satisfy $\rho(\bM, \bM') \leq 1$.  We emphasize that $\bM'$ need not be contained in the regularity set $\mathcal{S}$ for this to hold.  The following lemma establishes that with high probability, $\bM$ is contained in the regularity set, and further, that to bound $| f(\bM) - \EE f(\bM) |$, it suffices to establish deviation bounds on $f^{\downarrow}$.  We provide the proof of Lemma~\ref{aucillary-lemma3-trace-expectation-concentration} in Section~\ref{sec:proof-aucillary-lemma3-trace-expectation-concentration}.

\begin{lemma}\label{aucillary-lemma3-trace-expectation-concentration}
	Consider the setting of Lemma~\ref{lemma:trace-expectation-concentration} and let $\bM = \{ \bx_i, \bz_i\}_{i=1}^{m}$.  Then, for all $t\geq 2d^{-24}$, the following hold.
	\begin{itemize}
		\item [(a.)] $\Pr\{ \bM \in \mathcal{S} \} \geq 1-d^{-25}$, and
		\item[(b.)] $\Pr\{ | f(\bM) - \EE f(\bM) | \geq t \} \leq \Pr\{ | f^{\downarrow}(\bM) - \EE f^{\downarrow}(\bM) | \geq t/2 \} +  d^{-25}$.
	\end{itemize}
\end{lemma}

\paragraph{Step 3: Putting together the pieces.} By definition of $f^{\downarrow}$, for any $\bM \in \mathbb{R}^{2d \times m}$ and $\bM' \in \mathcal{S} $, 
\[
f^{\downarrow}(\bM) \leq f(\bM') + \Delta \rho(\bM,\bM') + 2D \leq f(\bM') + m\Delta + 2\lesssim \log d,
\]
where the last step follows from part Lemma~\ref{aucillary-lemma1-trace-expectation-concentration}(a), which ensures that $f(\bP') \leq 1$ since, by assumption, $\lambda m \geq d$ and $(m+1)\Delta \lesssim m\log d/d \leq \log d$.  Then, for $\rho(\bM, \bM') \leq 1$, we invoke the inequality~\eqref{properties-truncate-f-trace-inverse}, to obtain
\[
\big| f^{\downarrow}(\bM) - f^{\downarrow}(\bM') \big| \leq 
\left\{ \begin{array}{c} 2\Delta, \text{ if } \bM \in \mathcal{S},\\ C\log(d), \text{otherwise}. \\ \end{array}\right.
\]
Next, we set $\gamma_{k} = d^{-1},c_{k} = 2\Delta, d_{k} = C\log(d)$ for each $k \in [m]$, and apply~\citet[Theorem 2]{warnke2016method} to obtain the bound
\[
\Pr\Big\{ \big| f^{\downarrow}(\bM) - \EE f^{\downarrow}(\bM) \big| \geq t/2 \Big\} \leq 2\exp\bigg\{ \frac{-cd^{2}t^{2}}{m\log^{2}(d)}\bigg\} + 2d^{-23}.
\]
Combining the inequality in the display with Lemma~\ref{aucillary-lemma3-trace-expectation-concentration} yields the first of the desired inequalities in Lemma~\ref{lemma:trace-expectation-concentration}. Proceeding in an identical fashion yields the second. \qed

\subsubsection{Proof of Lemma~\ref{aucillary-lemma1-trace-expectation-concentration}}\label{sec:proof-aucillary-lemma1-trace-expectation-concentration} We prove each part in turn.

\paragraph{Proof of part (a):}
By the block matrix inverse Eq.~\eqref{eq:block-matrix-formula}, we note that 
\[
	\big\| \big(\bB - \bC\bD^{-1}\bC^{\top} \big)^{-1} \big\|_{\mathrm{op}} \leq \Big\| \Big(\sum_{i=1}^{m} \ba_i \ba_i^{\top} + \lambda m \bI \Big)^{-1} \Big\|_{\mathrm{op}} \leq (\lambda m)^{-1}.
\]
The desired result immediately follows by the definition of $f$.

\paragraph{Proof of part (b):}
By the triangle inequality, it suffices to prove 
\begin{align}\label{claim-proof-aucillary-lemma1-trace}
 	| f(\bM) - f(\bM') | \leq \Delta/2 \qquad \text{ for all } \qquad  \bM,\bM' \in \mathcal{S} \text{ such that } \rho(\bM,\bM') \leq 1.
\end{align} 
Towards establishing the above inequality, since $\rho(\bM,\bM') \leq 1$, we assume, without loss of generality, that $(\bx_i,\bz_i) = (\bx'_i,\bz'_i)$ for $2\leq i \leq m$.  
Continuing, let $G'_{1} = \langle \bx'_{1}, \bcoefX_{\sharp} \rangle$, $\widetilde{G}'_{1}= \langle \bz'_{1} \bcoefZ_{\sharp} \rangle$ and $(\ba'_{1})^{\top} = [\widetilde{G}'_{1} (\bx'_{1})^{\top} \bO_{\bu} \;\vert \; G'_{1} (\bz'_{1})^{\top} \bO_{\bv} ] $. Similarly, let $(\ba_{i})^{\top} = [ \widetilde{G}_{i} (\bx_{i})^{\top} \bO_{\bu} \;\vert \; G_{i} (\bz_{1})^{\top} \bO_{\bv} ] $ for $i\in[m]$, and let $\bSig = \sum_{i=1}^{m}\ba_{i} \ba_{i}^{\top} + \lambda m \bI$ and $\bSig_{1} = \bSig - \ba_{1}\ba_{1}^{\top}$.  We then consider matrices $\bB_1, \bC_1$, and $\bD_1$ which are perturbations of $\bB, \bC$, and $\bD$, defined as 
\begin{align}\label{notation-B1-C1-D1-Sig1}
	\bB_{1} = \bB - \widetilde{G}_{1}^{2} \bO_{\bu}^{\top} \bx_1 &(\bO_{\bu}^{\top} \bx_1)^{\top}, \qquad \bC_{1} = \bC - G_{1}\widetilde{G}_{1} \cdot \bO_{\bu}^{\top}\bx_1 (\bO_{\bv}^{\top}\bz_{1})^{\top}, \qquad \text{ and }\nonumber \\
	& \bD_{1} = \bD - G_{1}^{2} \bO_{\bv}^{\top} \bz_1 (\bO_{\bv}^{\top} \bz_1)^{\top}.
\end{align}
We additionally recall that, by the relation~\eqref{eq:block-matrix-proof-sketch}, 
\begin{align*}
\bSig^{-1} = \begin{bmatrix} (\bB-\bC\bD^{-1}\bC^{\top})^{-1} & -(\bB-\bC\bD^{-1}\bC^{\top})^{-1} \bC\bD^{-1}\\  -(\bD-\bC^{\top}\bB^{-1}\bC)^{-1} \bC^{\top}\bB^{-1} & (\bD-\bC^{\top}\bB^{-1}\bC)^{-1} \\ \end{bmatrix}.
\end{align*}
Moreover, by the Sherman–Morrison formula, $\bSig^{-1} = \bSig_1^{-1} - \frac{\bSig_1^{-1} \ba_{1} \ba_{1}^{\top} \bSig_1^{-1}}{1+\ba_1^{\top} \bSig_1^{-1} \ba_{1}}$.  Combining this with the previous display and only using the $d\times d$ submatrix in the upper left corner yields
\[
	(\bB-\bC\bD^{-1}\bC^{\top})^{-1} = \underbrace{(\bB_1-\bC_1\bD_1^{-1}\bC_1^{\top})^{-1}}_{\bP_{11}} - \frac{ \big[ \bSig_1^{-1} \big]_{[d],[2d]} \ba_1 \ba_1^{\top} \big[ \bSig_1^{-1} \big]_{[d],[2d]}^{\top} }{1+\ba_1^{\top} \bSig_1^{-1} \ba_{1}},
\]
where $\big[ \bSig_1^{-1} \big]_{[d],[2d]} \in \mathbb{R}^{d \times 2d}$ is the submatrix of $\bSig_{1}^{-1}$ which consists of the first $d$ rows. Taking the trace of both sides of the equation in the display above yields
\[
	f(\bM) = \trace( \bP_{11} ) - \frac{\big\|\big[ \bSig_1^{-1} \big]_{[d],[2d]} \ba_1 \big \|_{2}^{2}}{1+\ba_1^{\top} \bSig_1^{-1} \ba_{1}}, \qquad \text{ and } \qquad  f(\bM') = \trace( \bP_{11} ) - \frac{\big\|\big[ \bSig_1^{-1} \big]_{[d],[2d]} \ba'_1 \big \|_{2}^{2}}{1+(\ba'_1)^{\top} \bSig_1^{-1} \ba'_{1}}.
\]
Thus, since $\bSig_{1} \succeq \boldsymbol{0}$, we find that
\[
	\big| f(\bM) - f(\bM') \big| \leq \Big\|\big[ \bSig_1^{-1} \big]_{[d],[2d]} \ba_1 \Big \|_{2}^{2} + \Big\|\big[ \bSig_1^{-1} \big]_{[d],[2d]} \ba'_1 \Big \|_{2}^{2}.
\]
Moreover, since $\|\big[ \bSig_1^{-1} \big]_{[d],[2d]}\|_{2} \leq \|\bSig_1^{-1}\|_{2} \leq (\lambda m)^{-1}$, we obtain the inequality
\begin{align*}
	\Big\|\big[ \bSig_1^{-1} \big]_{[d],[2d]} \ba_1 \Big \|_{2}^{2} &\leq  \frac{\widetilde{G}_{1}^{2} \|\bx_{1}\|_{2}^{2} + G_{1}^{2}\|\bz_{1}\|_{2}^{2} }{(\lambda m)^{2}} \leq \frac{C'\log(d)}{\lambda m},
\end{align*}
where the last step follows since $\bM \in \mathcal{S}$ and $\lambda m \geq d$, so that
$
	|G_1|,|\widetilde{G}_1| \lesssim \sqrt{\log d},\; \|\bx_1\|_2^2$, and $\|\bz_1\|_2^2 \lesssim d$.
Similarly, since $\bM' \in \mathcal{S}$, we obtain the bound $\big\|\big[ \bSig_1^{-1} \big]_{[d],[2d]} \ba'_1 \big \|_{2}^{2} \lesssim \log(d)/(\lambda m)$. Putting the pieces together yields $|f(\bM) - f(\bM')| \lesssim \log(d)/(\lambda m)$, which concludes the proof. \qed

\subsubsection{Proof of Lemma~\ref{aucillary-lemma3-trace-expectation-concentration}}\label{sec:proof-aucillary-lemma3-trace-expectation-concentration} 
We prove each part in turn, beginning with part (a).

\paragraph{Proof of Lemma~\ref{aucillary-lemma3-trace-expectation-concentration}(a):} We invoke~\citet[Theorem 3.1.1]{vershynin2018high} to obtain the following four concentration inequalities for each $i \in [m]$,
\begin{align}\label{probability-bound-1-lemma3-trace-expectation-concentration}
	&\Pr\big\{ \|\bx_i\|_2^{2} \geq Cd \big\} \leq 2d^{-30}, \qquad \Pr\big\{ \|\bz_i\|_2^{2} \geq Cd \big\} \leq 2d^{-30}, \qquad \Pr\big\{ |G_i| \geq C\sqrt{\log d} \big\} \leq 2d^{-30}\nonumber\\
	& \qquad \qquad  \qquad \qquad \text{ and } \qquad \Pr\big\{ |\widetilde{G}_i| \geq C\sqrt{\log d} \big\} \leq 2d^{-30}.
\end{align}
The conclusion follows upon applying the union bound.

\paragraph{Proof of Lemma~\ref{aucillary-lemma3-trace-expectation-concentration}(b):} The proof of this part is similar to Section~\ref{proof-auxiliary-lemma2-component-expectation-concentration}.  We thus omit the steps which are identical, and restrict ourselves to the differences. 

It suffices to bound $|\EE\{f\} - \EE\{f^{\downarrow}\}|$. To this end, by definition of the function $f^{\downarrow}$, we find that for any $\bM \in \mathbb{R}^{2d \times m}$ and $\bM' \in \mathcal{S} $, the following inequality holds
\[
	f^{\downarrow}(\bM) \leq f(\bM') + \Delta \rho(\bM,\bM') + 2D \leq f(\bM') + m\Delta + 2\lesssim \log d.
\]
Above, the final inequality follows since by Lemma~\ref{aucillary-lemma1-trace-expectation-concentration}, $f(\bM') \leq 1$,  as long as $\lambda m \geq d$ and $m\Delta \lesssim m\log d/d \leq \log d$.
Consequently, we bound its expectation as
\begin{align*}
	\EE\{f^{\downarrow}\} &= \EE\{f^{\downarrow}(\bM) \cdot \mathbbm{1}\{ \bM \in \mathcal{S} \}\} + \EE\{f^{\downarrow} \cdot \mathbbm{1}\{ \bM \notin \mathcal{S} \}\} 
	\\&\overset{\1}{\leq} \EE\{f(\bM) \cdot \mathbbm{1}\{ \bM \in \mathcal{S} \} \} + C\log(d) \cdot \EE\{ \mathbbm{1}\{ \bM \notin \mathcal{S} \} \} \leq \EE\{f(\bM)\} + d^{-24},
\end{align*}
where step $\1$ follows from the inequality~\eqref{properties-truncate-f-trace-inverse} as well as since $f$ and $f^{\downarrow}$ agree on $\mathcal{S}$. Similarly,
\begin{align*}
	\EE\{f\} &= \EE\{f(\bM) \cdot \mathbbm{1}\{ \bM \in \mathcal{S} \}\} + \EE\{f(\bM) \cdot \mathbbm{1}\{ \bM \notin \mathcal{S} \}\} 
	\\&\overset{\1}{\leq} \EE\{f^{\downarrow}(\bM)\} + \EE\{ \mathbbm{1}\{ \bM \notin \mathcal{S} \} \}
	\leq \EE\{f^{\downarrow}\} + d^{-25},
\end{align*}
where step $\1$ follows since $f$ and $f^{\downarrow}$ agree on $\mathcal{S}$, and as $f \leq 1$ by Lemma~\ref{aucillary-lemma1-trace-expectation-concentration}(a) for $\lambda m \geq d$. Putting the pieces together yields $\big| \EE\{f\} - \EE\{f^{\downarrow}\} \big| \leq d^{-24}$.
The desired result then follows by proceeding similarly as in Section~\ref{proof-auxiliary-lemma2-component-expectation-concentration}. \qed

\subsection{Proof of Lemma~\ref{lemma:expectation-of-trace-fixed-dolution-distance}} \label{sec:proof-lemma-expectation-trace-fixed-solution-distance}
Consider $\bB_1,\bC_1,\bD_1$ as defined in~\eqref{notation-B1-C1-D1-Sig1} and recall the shorthand $G_i = \bx_i^{\top}\bcoefX_{\sharp}$ and $\widetilde{G}_i = \bz_i^{\top} \bcoefZ_{\sharp}$.  Further define the leave one out quantities $\bar{\bx}_{i} = \bO_{\bu}^{\top} \bx_{i},\; \bar{\bz}_{i} = \bO_{\bv}^{\top} \bz_{i}$ and $\ba_{i}^{\top} =[ \widetilde{G}_{i}\bar{\bx}_{i}^{\top}  G_{i} \bar{\bz}_{i}^{\top}]$.
A straightforward calculation yields
\begin{align}\label{fixed-point-equation-matrix}
	\bI_{2d-2} = \EE\big\{ \bSig^{-1} \bSig\big\} &= \EE\Big\{ \bSig^{-1} \Big(\sum_{i=1}^{m} \ba_{i} \ba_{i}^{\top} + \lambda m \bI_{2d-2}\Big) \Big\} \nonumber = m \cdot \EE\big\{ \bSig^{-1}  \ba_{1} \ba_{1}^{\top} \big\} + \lambda m \cdot \EE\big\{ \bSig^{-1} \big \} \nonumber \\
	& = m \cdot \EE \bigg\{ \frac{\bSig_{1}^{-1} \ba_1 \ba_{1}^{\top} }{1 + \ba_{1}^{\top} \bSig_{1}^{-1} \ba_{1} }\bigg\} + \lambda m \cdot \EE\big\{ \bSig^{-1} \big \},
\end{align}
where the last step follows from the Sherman–Morrison formula.  Next, define the three matrices $\bP_{11} = (\bB_1-\bC_1\bD_1^{-1}\bC_1^{\top})^{-1}$, $\bP_{12} = -\bP_{11} \bC_1\bD_1^{-1}$, and $\bP_{22} = (\bD_1-\bC_1^{\top}\bB_1^{-1}\bC_1)^{-1}$ and apply the block matrix inversion formula to obtain
\begin{align}\label{block-matrix-inverse-1}
\hspace{-1cm}
	\bSig_{1}^{-1} =\begin{bmatrix} \bB_{1} & \bC_{1} \\ \bC_{1}^{\top} & \bD_{1} \end{bmatrix}^{-1} &= \left[\begin{array}{cc} \bP_{11} & \bP_{12}\\  \bP_{12}^{\top} & \bP_{22} \\ \end{array}\right]. 
\end{align}
Combining Eqs.~\eqref{fixed-point-equation-matrix} and~\eqref{block-matrix-inverse-1} yields the pair of relations
\begin{align*}
	\bI_{d-1} &= m \cdot \EE\bigg\{ \frac{\widetilde{G}_{1}^{2} \bP_{11} \bar{\bx}_{1}\bar{\bx}_{1}^{\top} + G_{1}\widetilde{G}_{1}\bP_{12}\bar{\bz}_{1}\bar{\bx}_{1}^{\top} }{1 + \ba_{1}^{\top} \bSig_{1}^{-1} \ba_{1}} \bigg\} + \lambda m \cdot \EE\big\{ (\bB-\bC\bD^{-1}\bC^{\top})^{-1} \big\} \\
	\bI_{d-1} &= m \cdot \EE\bigg\{ \frac{G_{1}^{2} \bP_{22} \bar{\bz}_{1} \bar{\bz}_{1}^{\top} + G_{1}\widetilde{G}_{1}\bP_{12}^{\top}\bar{\bx}_{1}\bar{\bz}_{1}^{\top} }{1 + \ba_{1}^{\top} \bSig_{1}^{-1} \ba_{1}} \bigg\} + \lambda m \cdot \EE\big\{ (\bD-\bC^{\top}\bB^{-1}\bC)^{-1} \big\}.
\end{align*}
Taking the trace of both sides of each relation above yields
\begin{subequations}\label{eq:fixed-point-empirical}
\begin{align}
 \frac{d-1}{m} &= \underbrace{ \EE\bigg\{ \frac{\widetilde{G}_{1}^{2} \bar{\bx}_{1}^{\top} \bP_{11} \bar{\bx}_1 + G_{1}\widetilde{G}_{1}\bar{\bx}_{1}^{\top}\bP_{12} \bar{\bz}_{1} }{1 + \ba_{1}^{\top} \bSig_{1}^{-1} \ba_{1} } \bigg\} }_{\Omega_1} + \lambda \cdot \EE\big\{ \trace \big( (\bB-\bC\bD^{-1}\bC^{\top})^{-1} \big) \big\},\\
\frac{d-1}{m} &= \underbrace{ \EE\bigg\{ \frac{G_{1}^{2} \bar{\bz}_{1}^{\top} \bP_{22} \bar{\bz}_1 + G_{1}\widetilde{G}_{1} \bar{\bz}_{1}^{\top}\bP_{12}^{\top}\bar{\bx}_{1} }{1 + \ba_{1}^{\top} \bSig_{1}^{-1} \ba_{1}} \bigg\} }_{\Omega_2} + \lambda \cdot \EE\big\{ \big( \trace(\bD-\bC^{\top}\bB^{-1}\bC)^{-1} \big) \big\}.
\end{align}
\end{subequations}
We then define $\tau$ and $\widetilde{\tau}$ as $\tau = \EE\big\{ \trace \big( (\bB-\bC\bD^{-1}\bC^{\top})^{-1} \big) \big\}$ and $\widetilde{\tau} = \EE\big\{ \big( \trace(\bD-\bC^{\top}\bB^{-1}\bC)^{-1} \big) \big\}$, respectively, and use these to define the scalars 
\begin{align*}
	\overline{\Omega}_1 = \EE\Big\{ \frac{\widetilde{G}_{1}^2 \tau}{1 + \widetilde{G}_{1}^2 \tau + G_1^{2} \widetilde{\tau}} \Big\} \qquad \text{ and } \qquad 
	\overline{\Omega}_2 = \EE\Big\{ \frac{G_1^2 \widetilde{\tau}}{1 + \widetilde{G}_{1}^2 \tau + G_1^{2} \widetilde{\tau}} \Big\}.
\end{align*}
We next make the following claim---deferring its proof to the end---which shows that $\Omega_1 \approx \widebar{\Omega}_1$ and $\Omega_2 \approx \widebar{\Omega}_2$
\begin{align}\label{claim-proof-fixed-point-equation}
	\big| \Omega_1 - \overline{\Omega}_1\big| \; \vee \;
	\big| \Omega_2 - \overline{\Omega}_2\big| \lesssim \frac{\log^{1.5}(d)}{\sqrt{d}}.
\end{align}
We then define the scalars $\omega_1$ and $\omega_2$ as
\[
\omega_1 = \EE\bigg\{ \frac{\widetilde{G}_{1}^2r_1^{-1}}{1+G_1^2r_2^{-1} + \widetilde{G}_{1}^2r_1^{-1}}\bigg\} \qquad \text{ and } \qquad \omega_2 = \EE\bigg\{ \frac{G_1^2r_2^{-1}}{1+G_1^2r_2^{-1} + \widetilde{G}_{1}^2r_1^{-1}}\bigg\},
\]
and compare the fixed point equation~\eqref{eq:fixed-point} with equation~\eqref{eq:fixed-point-empirical}, to obtain the pair
\begin{align*}
	\frac{1}{m} + \Omega_1 + \lambda \cdot \tau = \omega_1 + \lambda \cdot r_1^{-1} \qquad \text{ and } \qquad \frac{1}{m} + \Omega_2 + \lambda \cdot \widetilde{\tau} = \omega_2 + \lambda \cdot r_2^{-1}.
\end{align*}
Combining the pair of equations in the preceding display with the inequality~\eqref{claim-proof-fixed-point-equation} yields
\begin{align}\label{ineq:fixed-point-equation-difference}
	\big| \omega_1 + \lambda \cdot r_1^{-1} - \big( \overline{\Omega}_1 + \lambda \cdot \tau\big) \big| \; \vee \;
	\big| \omega_2 + \lambda \cdot r_2^{-1} - \big( \overline{\Omega}_2 + \lambda \cdot \widetilde{\tau}\big) \big| \lesssim \frac{\log^{1.5}(d)}{\sqrt{d}}+\frac{1}{m}.
\end{align}
Then, by the triangle inequality, 
\begin{align*}
\big| \omega_1 -  \overline{\Omega}_1 \big| \; \vee  \; \big| \omega_2 -  \overline{\Omega}_2 \big|  \leq  C \cdot |\tau -r_1^{-1}| + C \cdot |\widetilde{\tau} - r_2^{-1}|
\end{align*}
Applying the triangle inequality in conjunction with the inequality in the display above yields both
\begin{align*}
	&\big| \omega_1 + \lambda \cdot r_1^{-1} - \big( \overline{\Omega}_1 + \lambda \cdot \tau\big) \big| \geq (\lambda - C) \cdot |\tau - r_1^{-1}| - C \cdot |\widetilde{\tau} - r_2^{-1}|, \qquad \text{ and }\\
	&\big| \omega_2 + \lambda \cdot r_2^{-1} - \big( \overline{\Omega}_2 + \lambda \cdot \widetilde{\tau}\big) \big| \geq (\lambda - C) \cdot |\widetilde{\tau} - r_2^{-1}| - C \cdot |\tau - r_1^{-1}|.
\end{align*}
We next combine this pair of inequalities with the inequality~\eqref{ineq:fixed-point-equation-difference} to obtain
\begin{align*}
	(\lambda - C) \cdot |\tau - r_1^{-1}| -C \cdot |\widetilde{\tau} - r_2^{-1}|  \; \vee \;
	(\lambda - C) \cdot |\widetilde{\tau} - r_2^{-1}| - C \cdot |\tau - r_1^{-1}| \lesssim \frac{\log^{1.5}(d)}{\sqrt{d}} + \frac{1}{m}.
\end{align*}
Consequently, as long as $\lambda \geq 10Cd/m \geq 10C$, we conclude that
\[
	|\tau - r_1^{-1}|,\quad  |\widetilde{\tau} - r_2^{-1}| \lesssim \log^{1.5}(d)(\lambda\sqrt{d}) + (\lambda m)^{-1} \lesssim \frac{\log^{1.5}(d)}{\sqrt{d}},
\]
as desired.  It remains to prove the inequality~\eqref{claim-proof-fixed-point-equation}.
\paragraph{Proof of the inequality~\eqref{claim-proof-fixed-point-equation}:}
We first expand $\bSig_1^{-1}$ to obtain the expression
\[
\ba_{1}^{\top} \bSig_{1}^{-1} \ba_{1} = \widetilde{G}_{1}^{2} \bar{\bx}_{1}^{\top} \bP_{11} \bar{\bx}_1 + G_{1}^{2} \bar{\bz}_{1}^{\top} \bP_{22} \bar{\bz}_1 + 2G_{1}\widetilde{G}_{1} \bar{\bx}_{1}^{\top}\bP_{12} \bar{\bz}_{1}.
\]
 Then, applying the triangle yields the bound
\begin{align*}
	\bigg| \EE\bigg\{ \frac{\widetilde{G}_{1}^{2} \bar{\bx}_{1}^{\top} \bP_{11} \bar{\bx}_1 + G_{1}\widetilde{G}_{1}\bar{\bx}_{1}^{\top}\bP_{12} \bar{\bz}_{1} }{1 + \ba_{1}^{\top} \bSig_{1}^{-1} \ba_{1} } \bigg\} &- \EE\bigg\{ \frac{\widetilde{G}_{1}^{2} \bar{\bx_{1}}^{\top} \bP_{11} \bar{\bx}_1 }{1 + \widetilde{G}_{1}^{2} \bar{\bx}_{1}^{\top} \bP_{11} \bar{\bx}_1 + G_{1}^{2} \bar{\bz}_{1}^{\top} \bP_{22} \bar{\bz}_1 } \bigg\} \bigg| \leq 3\EE\big\{ \big|G_{1}\widetilde{G}_{1}\bar{\bx}_{1}^{\top}\bP_{12}\bar{\bz}_{1} \big|\big\},
\end{align*}
the right hand side of which we further bound by applying the Cauchy-Schwarz inequality to obtain
\[
	\EE\big\{ \big|G_{1}\widetilde{G}_{1}\bar{\bx}_{1}^{\top}\bP_{12}\bar{\bz}_{1} \big|\big\} \leq \sqrt{ \EE\{G_{1}^{2}\widetilde{G}_{1}^{2}\}} \cdot \sqrt{\EE\{(\bar{\bx}_{1}^{\top}\bP_{12}\bar{\bz}_{1})^{2}\}}.
\]
Explicit evaluation then yields the upper bound $\EE\{G_{1}^{2}\widetilde{G}_{1}^{2}\} \lesssim 1$.  Moreover,
\[
 \EE\{(\bar{\bx}_{1}^{\top}\bP_{12}\bar{\bz}_{1})^{2}\}= \EE\{\|\bP_{12}\bar{\bz}_{1}\|_{2}^{2}\} \overset{\1}{\leq} \EE\{\|\bP_{12}\|_{2}^{2}\} \cdot \EE\{\|\bar{\bz}_{1}\|_{2}^{2}\} \overset{\2}{\leq} d/(\lambda m)^{2}
\]
where step $\1$ follows by exploiting the independence between $\bP_{12}$ and $\bar{\bz}_{1}$ and step $\2$ follows from the pair of inequalities
$\|\bP_{12}\|_{2} \leq \|\bSig_{1}^{-1}\|_{2} \leq 1/(\lambda m)$ and $\EE\{\|\bar{\bz}_{1}\|_{2}^{2}\} \leq d$. Putting the pieces together and invoking the assumption $\lambda m \geq d$ yields
\begin{align}\label{ineq1-Term1-expectation-trace-inverse}
	\bigg| \EE\bigg\{ \frac{\widetilde{G}_{1}^{2} \bar{\bx}_{1}^{\top} \bP_{11} \bar{\bx}_1 + G_{1}\widetilde{G}_{1}\bar{\bx}_{1}^{\top}\bP_{12} \bar{\bz}_{1} }{1 + \ba_{1}^{\top} \bSig_{1}^{-1} \ba_{1} } \bigg\} - \EE\bigg\{ \frac{\widetilde{G}_{1}^{2} \bar{\bx_{1}}^{\top} \bP_{11} \bar{\bx}_1 }{1 + \widetilde{G}_{1}^{2} \bar{\bx}_{1}^{\top} \bP_{11} \bar{\bx}_1 + G_{1}^{2} \bar{\bz}_{1}^{\top} \bP_{22} \bar{\bz}_1 } \bigg\} \bigg| \lesssim \frac{1}{\sqrt{d}}.
\end{align}
We then claim the following pair of inequalities, postponing their proof to the end of the section
\begin{align}\label{ineq2-Term2-expectation-trace-inverse}
	\EE\{| \bar{\bx}_1^{\top}\bP_{11} \bar{\bx}_1 - \tau|\} \lesssim \frac{\log^{\frac{3}{2}}(d)}{\sqrt{d}} \qquad \text{ and } \qquad \EE\{| \bar{\bz}_1^{\top}\bP_{22} \bar{\bz}_1 - \widetilde{\tau}|\} \lesssim \frac{\log^{\frac{3}{2}}(d)}{\sqrt{d}}.
\end{align}
Combining the inequalities~\eqref{ineq1-Term1-expectation-trace-inverse},~\eqref{ineq2-Term2-expectation-trace-inverse} and the triangle inequality finally yields the bound
\begin{align*}
	\bigg| \EE\bigg\{ \frac{\widetilde{G}_{1}^{2} \bar{\bx}_{1}^{\top} \bP_{11} \bar{\bx}_1 + G_{1}\widetilde{G}_{1}\bar{\bx}_{1}^{\top}\bP_{12} \bar{\bz}_{1} }{1 + \ba_{1}^{\top} \bSig_{1}^{-1} \ba_{1} } \bigg\} - \EE\bigg\{ \frac{\widetilde{G}_{1}^{2} \tau }{1 + \widetilde{G}_{1}^{2} \tau + G_{1}^{2} \widetilde{\tau} } \bigg\} \bigg|  \lesssim \frac{\log^{\frac{3}{2}}(d)}{\sqrt{d}},
\end{align*}
which proves one of the pair of inequalities in~\eqref{claim-proof-fixed-point-equation}. The remaining inequality admits an identical proof.  We turn now to establishing the inequality~\eqref{ineq2-Term2-expectation-trace-inverse}\\ 

\medskip
\noindent \underline{Proof of inequality~\eqref{ineq2-Term2-expectation-trace-inverse}}:
We use the shorthand $\tau_1 = \EE\big\{ \trace \big( (\bB_1-\bC_1\bD_{1}^{-1}\bC_1^{\top})^{-1} \big) \big\}$ and $\widetilde{\tau}_1 = \EE\big\{ \trace \big( (\bD_1-\bC_{1}^{\top}\bB_{1}^{-1}\bC_{1})^{-1} \big) \big\}$.   We then apply the triangle inequality to obtain the decomposition
\[
	| \bx_1^{\top}\bP_{11} \bx_1 - \tau| \leq \underbrace{ \big| \bx_1^{\top}\bP_{11} \bx_1 - \trace(\bP_{11}) \big| }_{T_1} + \underbrace{| \trace(\bP_{11}) - \tau_1|}_{T_2} + \underbrace{ |\tau_1 - \tau| }_{T_3},
\]
and proceed to bound the expectation of each of $T_1, T_2$, and $T_3$ in turn, starting with $\EE\{T_1\}$. 
Towards bounding $T_1$, we note that $\|\bP_{11}\|_{F}^{2} \leq \frac{d}{(\lambda m)^{2}} \leq \frac{1}{\lambda m}$ and $\|\bP_{11}\|_{\mathrm{op}} \leq \frac{1}{\lambda m}$.  Then, we apply the Hanson--Wright inequality to obtain the inequality
$\Pr\{ T_{1} \geq t\} \leq 2\exp\{-c\lambda m \min(t^{2},t)\}$ for all $t\geq 0$. Consequently, we integrate the tail to obtain the expectation bound $\EE \{ T_1 \} \lesssim (\lambda m)^{-1/2} \leq d^{-1/2}$. 

Turning to bounding $\EE\{T_2\}$, we take $C$ to denote a large enough constant and set $\delta = C\log^{3/2}(d)d^{-1/2}$.  We then consider the decomposition $\EE\{T_2\} = \EE\{T_2 \cdot \mathbbm{1}\{T_{2} \leq \delta\}\} + \EE\{T_2 \cdot \mathbbm{1}\{T_{2} > \delta\}\}$, noting 
 that $\EE\{T_2 \cdot \mathbbm{1}\{T_{2} \leq \delta \}\} \leq \delta$. On the other hand, an application of the Cauchy--Schwarz inequality yields $
	\EE\big\{T_2 \cdot \mathbbm{1}\{T_{2} > \delta\}\big\} \leq \sqrt{ \EE\{T_2^{2}\}} \cdot \sqrt{\EE\{\mathbbm{1}\{T_{2} > \delta\}\}}$.  Then, using the inequality $\| \bP_{11} \|_{\mathrm{op}} \leq (\lambda m)^{-1}$, we obtain the bound
	\[
	\EE\{T_2^{2}\} \leq \EE\{\trace(\bP_{11})^{2}\}\leq  \Big( \frac{d}{\lambda m} \Big)^{2} \leq 1,
	\]
	where the final inequality follows by assumption.  Moreover, by Lemma~\ref{lemma:trace-expectation-concentration},
Note that 
\begin{align*}
	 \EE\{\mathbbm{1}\{T_{2} > \delta\}\} \leq \Pr\{|\trace{\bP_{11}} - \EE\{ \trace{\bP_{11}} \}| \geq \delta \} \leq d^{-18},
\end{align*}
Assembling the pieces yields the ultimate bound $\EE\{T_{2}\} \lesssim \log^{3/2}(d)/\sqrt{d}$. 

Towards bounding $T_3$, we apply the Sherman–Morrison formula to obtain
\[
	\bSig^{-1} = \big(\ba_{1}\ba_{1}^{\top} + \bSig_{1}^{-1} \big)^{-1} = \bSig_{1}^{-1} - \frac{\bSig_{1}^{-1} \ba_{1}\ba_{1}^{\top} \bSig_{1}^{-1} }{1 + \ba_{1}^{\top} \bSig_{1}^{-1} \ba_{1} }.
\]
Focusing on the $d\times d$ submatrices in the upper left of both sides of the preceding equation yields the relation
\[
	\big( \bB - \bC \bD^{-1} \bC^{\top} \big)^{-1} = \big( \bB_1 - \bC_1 \bD_{1}^{-1} \bC_1^{\top} \big)^{-1} - \frac{ \big(\widetilde{G}_{1} \bP_{11} \bar{\bx}_1 + G_{1}\bP_{12}\bar{\bz}_{1} \big)\big( \widetilde{G}_{1} \bP_{11} \bar{\bx}_1 + G_1\bP_{12}\bar{\bz}_{1} \big)^{\top} }{1+\ba_{1}^{\top} \bSig_{1}^{-1} \ba_{1}}
\]
We then take the trace of both sides and re-arrange terms to obtain the bound
\[
	\Big| \trace\big( \big( \bB - \bC \bD^{-1} \bC^{\top} \big)^{-1} \big) - \trace\big( \big( \bB_1 - \bC_1 \bD_{1}^{-1} \bC_1^{\top} \big)^{-1} \big) \Big| \leq \big\| \widetilde{G}_{1} \bP_{11} \bar{\bx}_1 + G_1\bP_{12}\bar{\bz}_{1} \big\|_{2}^{2},
\]
where we have additionally used the fact that $\bSig_{1} \succeq 0$. 
Consequently, we deduce the bound
\begin{align*}
	\EE\{ T_{3} \} &\leq \EE\big\{ \big\| \widetilde{G}_{1} \bP_{11} \bar{\bx}_1 + G_1\bP_{12}\bar{\bz}_{1} \big\|_{2}^{2} \big\} \lesssim
	\frac{d}{(\lambda m)^{2}} \leq d^{-1/2},
\end{align*}
where in the last step we have used the bounds $\|\bP_{11}\|_{2},\|\bP_{12}\|_{2} \leq (\lambda m)^{-1}$. Collecting the bounds on $\EE T_{1}, \EE T_{2}$, and $\EE T_{3}$ yields
\[ 
	\EE\big\{ | \bar{\bx}_{1}^{\top} \bP_{11}  \bar{\bx} - \tau | \big\} \lesssim \frac{\log^{\frac{3}{2}}(d)}{\sqrt{d}}.
\]
The proof of the second inequality is identical and we thus omit it. \qed
\section{Auxiliary lemmas}\label{aucillary-lemmas}
\begin{lemma}\label{lemma:bounded-difference-truncate}
Let $f:\mathbb{R}^{d \times m} \rightarrow \mathbb{R}$. Define the Hamming metric $\rho: \mathbb{R}^{d \times m} \times \mathbb{R}^{d \times m} \rightarrow [m]$ such that if $\bM,\bM'$ differ in exactly $k$ columns then $\rho(\bM,\bM') = k$. Let $\mathcal{S} \subseteq \mathbb{R}^{d \times m}$ and suppose there exists some $\Delta \geq 0$ such that for all $\bM, \bM' \in \mathcal{S}$ which satisfy $\rho(\bM, \bM') \leq 2$
\begin{align}\label{assumption-bd-truncate}
	| f(\bM) - f(\bM') | \leq \Delta.
\end{align} 
Let $D \in \mathbb{R}$ such that $D \geq \sup_{\bM \in \mathcal{S} } |f(\bM)|$ and define $f^{\downarrow}:\mathbb{R}^{d \times m} \rightarrow \mathbb{R}$ as
\begin{align}\label{definition-f-downarrow}
	f^{\downarrow}(\bM) = \inf_{\bM' \in \mathcal{S}} \Big\{ f(\bM') + \Delta \cdot \rho(\bM,\bM') + 2D \cdot \mathbbm{1}\{ \rho(\bM,\bM') > 1 \} \Big\}.
\end{align}
Then, for any $\bM\in \mathcal{S}$ and $\bM'\in \mathbb{R}^{d \times m}$ such that $\rho(\bM,\bM') \leq 1$,  both
\begin{align*}
	f^{\downarrow}(\bM) = f(\bM) \qquad \text{ and } \qquad |f^{\downarrow}(\bM) - f^{\downarrow}(\bM')| \leq \Delta.
\end{align*}
\end{lemma}

\begin{proof} First note that, by definition, $f^{\downarrow}(\bM) \leq f(\bM)$ for $\bM \in \mathcal{S}$. For any $\bM'' \in \mathcal{S}$, by the assumption~\eqref{assumption-bd-truncate} and the definition of $D$, we note the inequality
\[
	f(\bM'') + \Delta\cdot \rho(\bM,\bM'') + 2D \cdot \mathbbm{1}\{ \rho(\bM,\bM'') > 1 \geq f(\bM).
\]
Consequently, we obtain that $f^{\downarrow}(\bM) \geq f(\bM)$ for $\bM \in \mathcal{S}$ by the definition of $f^{\downarrow}$. We thus conclude that $f(\bM) = f^{\downarrow}(\bM)$ for $\bM \in \mathcal{S}$. 

Next, let $\bM' \in \mathbb{R}^{d \times m}, \bM \in \mathcal{S}$ and $\rho(\bM,\bM') \leq 1$. If $\bM' \in \mathcal{S}$, then we invoke the assumption~\eqref{assumption-bd-truncate} to obtain the bound
\[
	|f^{\downarrow}(\bM) - f^{\downarrow}(\bM')| = |f(\bM) - f(\bM')| \leq \Delta.
\]
 Now, consider $\bM' \notin \mathcal{S}$. Note that by the definition of $f^{\downarrow}$,
\begin{align}\label{ineq1-lemma-bounded-difference-truncate}
	f^{\downarrow}(\bM') \leq f(\bM) + \Delta = f^{\downarrow}(\bM) + \Delta.
\end{align}
We claim that $f^{\downarrow}(\bM') \geq f^{\downarrow}(\bM)$.  To see this, note that for any $\bM'' \in \mathcal{S}$, $\rho(\bM',\bM'') \geq 1$ since $\bM' \notin \mathcal{S}$. If $\rho(\bM',\bM'') > 1$ then
\begin{align*}
	f(\bM'') + \Delta \cdot \rho(\bM',\bM'') + 2D \cdot \mathbbm{1}\{ \rho(\bM',\bM'') > 1 \} \geq f(\bM'') + 2D \geq f(\bM) = f^{\downarrow}(\bM).
\end{align*}
Otherwise $\rho(\bM',\bM'') = 1$. Consequently,
\begin{align*}
	f(\bM'') + \Delta \cdot \rho(\bM',\bM'') + 2D \cdot \mathbbm{1}\{ \rho(\bM',\bM'') > 1 \} = f(\bM'') + \Delta \overset{\1}{\geq} f(\bM) = f^{\downarrow}(\bM),
\end{align*}
where step $\1$ follows from assumption~\eqref{assumption-bd-truncate} and the fact that  $\rho(\bM,\bM'') \leq \rho(\bM,\bM') + \rho(\bM',\bM'') \leq 2$. Consequently, putting the two cases together yields 
\begin{align*}
	f^{\downarrow}(\bM') = \inf_{\bM'' \in \mathcal{S}} \Big\{ f(\bM'') + \Delta \cdot \rho(\bM',\bM'') + 2D \cdot \mathbbm{1}\{ \rho(\bM',\bM'') > 1 \} \Big\} \geq f^{\downarrow}(\bM),
\end{align*}
as desired.  Combining the bound $f^{\downarrow}(\bM') \geq f^{\downarrow}(\bM)$ with the inequality~\eqref{ineq1-lemma-bounded-difference-truncate} yields
\[
	|f^{\downarrow}(\bM) - f^{\downarrow}(\bM')| \leq \Delta \qquad \text{ for } \qquad \rho(\bM,\bM') \leq 1 \text{ and } \bM\in \mathcal{S},
\]
as desired.
\end{proof}

\begin{lemma}\label{fixed-point-equations-unique-solution}
Recall the definitions of fixed point equations~\eqref{eq:fixed-point},~\eqref{det_updates_eta1} and~\eqref{det_updates_eta2}, and $G_1 \sim \mathsf{N}(0,L^2)$, $G_2 \sim \mathsf{N}(0,\LZ^2)$ are independent. There exists a universal positive constant $C$ such that if $\lambda \geq C\max\{1, L^{4},\LZ^{4}\}d/m$ then the following holds.
\begin{enumerate}
	\item[(a.)] \label{item1-lemma} Fixed point equation~\eqref{eq:fixed-point} has a unique positive solution $(r_{1},r_{2})$, and $r_{1},r_{2} \geq \lambda m/d$. Moreover, if $\lambda \geq \max\{1,L^{2},\LZ^{2}\}$ then $r_{1},r_{2} \leq 2\lambda m/d$.
	\item[(b.)] \label{item2-lemma} Fixed point equations~\eqref{det_updates_eta1} and~\eqref{det_updates_eta2} have a unique, nonnegative solution.
	\item[(c.)] \label{item3-lemma} Recall that $\Err_{\sharp} = (\parcompX_{\sharp} \parcompZ_{\sharp}-1)^{2}+\perpcompX_{\sharp}^{2} + \perpcompZ_{\sharp}^{2}$. If $\perpcompX_{\sharp}, \perpcompZ_{\sharp} \leq 0.1$ and $0.3 \leq \|\bcoefX_{\sharp}\|_{2}, \|\bcoefZ_{\sharp}\|_{2} \leq 1.7$, then we have
	\[
		\frac{1}{5} \cdot \| \bcoefX_{\sharp} \bcoefZ_{\sharp}^{\top} - \bcoefX_{\star} \bcoefZ_{\star}^{\top}\|_{F}^{2} \leq \Err_{\sharp} \leq 12.5 \cdot \| \bcoefX_{\sharp} \bcoefZ_{\sharp}^{\top} - \bcoefX_{\star} \bcoefZ_{\star}^{\top}\|_{F}^{2}. 
	\]
\end{enumerate}
\end{lemma}

\begin{proof} We prove each item in turn.

\paragraph{Proof of Lemma~\ref{fixed-point-equations-unique-solution}(a.):} We first show that fixed point equation~\eqref{eq:fixed-point} has a unique solution in the set $\mathcal{S} = \{(r_1,r_2): r_1, r_2 \geq \lambda m/d\}$. To this end, we reformulate equation~\eqref{eq:fixed-point} as 
\begin{align}\label{fixed-point-eq-reformulate}
	\left[ \begin{array}{c} r_{1} \\ r_{2}\end{array} \right] = g(r_{1},r_{2}), \qquad \text{ where } \qquad g(r_1,r_2) = \oversamp \cdot \left[ \begin{array}{c} \EE\big\{ \frac{r_1r_2G_2^2}{r_1r_2 + r_1G_1^2 + r_2G_2^2}\big\} + \lambda \\ 
	\EE\big\{ \frac{r_1r_2G_1^2}{r_1r_2 + r_1G_1^2 + r_2G_2^2}\big\} + \lambda \end{array} \right]
\end{align}
By the contraction mapping theorem, it suffices to show: 
\begin{enumerate}
	\item[(a)] Set $\mathcal{S} = \{(r_1,r_2): r_1, r_2 \geq \lambda m/d\}$ is a closed set;
	\item[(b)] For all $(r_{1},r_{2}) \in \mathcal{S}$, we have $g(r_1,r_2) \in \mathcal{S}$;
	\item[(c)] There exists $q<1$ such that the Jacobian satisfies $\| \nabla g(r_1,r_2)\|_{1} \leq q$ for all $(r_1,r_2) \in \mathcal{S}$.
\end{enumerate}
Note that the first point is true. For the second point, if $r_{1},r_{2}\geq 0$, then clearly $g(r_1,r_2) \in \mathcal{S}$ since the expectations in Eq.~\eqref{fixed-point-eq-reformulate} are non-negative. We next verify the third point. Computing the Jacobian of $g$ yields
\begin{align*}
	\nabla g(r_1,r_2) =  \oversamp \cdot \left[  \begin{array}{cc}  \EE\big\{ \frac{r_2^2 G_2^4}{(r_1r_2 + r_1G_1^2 + r_2G_2^2)^{2}}\big\} & 
	\EE\big\{ \frac{r_1r_2 G_1^2G_2^2}{(r_1r_2 + r_1G_1^2 + r_2G_2^2)^{2}}\big\} \\
	\EE\big\{ \frac{r_1 r_2 G_1^2G_2^2}{(r_1r_2 + r_1G_1^2 + r_2G_2^2)^{2}}\big\} & 
	\EE\big\{ \frac{r_1^2 G_1^4}{(r_1r_2 + r_1G_1^2 + r_2G_2^2)^{2}}\big\} 
	\end{array} \right].
\end{align*}
We thus compute
\begin{align*}
	\| \nabla g(r_1,r_2) \|_{1} &= \oversamp \cdot \EE \bigg\{ \frac{r_2^2 G_2^4 + 2r_1r_2 G_1^2G_2^2 + r_1^2 G_1^4 }{(r_1r_2 + r_1G_1^2 + r_2G_2^2)^{2}} \bigg\} \\ &\leq \oversamp \cdot \frac{r_2^2 \EE\{G_2^4\} + 2r_1r_2 \EE\{G_1^2\} \EE\{G_2^2\} + r_1^2 \EE\{G_1^4\} }{(r_1r_2)^2}  \overset{\1}{\leq}
	\frac{3\LZ^{4} + 2L^2 \LZ^2 + 3L^4}{\lambda^2 \oversamp} < 0.5,
\end{align*}
where step $\1$ follows from the bounds $r_1,r_2\geq \lambda m/d = \lambda \oversamp$,  and the facts that $G_1 \sim \mathsf{N}(0,L^2)$ and $G_2 \sim \mathsf{N}(0,\LZ^2)$; and last step follows from the bound $\lambda \geq C \max\{L^4,\LZ^4\} \oversamp$ with $C$ a large enough constant. Taking stock, we have showed that fixed point equation~\eqref{eq:fixed-point} has a unique solution in the set $\mathcal{S} = \{(r_1,r_2): r_1, r_2 \geq \lambda m/d\}$.

Next, if fixed point equation~\eqref{eq:fixed-point} admits a non-negative solution $r_1,r_2 \geq 0$, then since the expectations in Eq.~\eqref{eq:fixed-point} are non-negative, we obtain that $r_{1},r_{2} \geq \lambda \oversamp = \lambda m/d$, so that $r_1,r_2$ must be in $\mathcal{S}$. Therefore, we conclude that fixed point equation~\eqref{eq:fixed-point} has a unique, non-negative solution $r_1,r_2 \geq 0$.

Finally, note that 
\[
	\EE\Big\{ \frac{r_1r_2G_2^2}{r_1r_2 + r_1G_1^2 + r_2G_2^2} \Big\} \leq \EE\{G_{2}^{2}\} = \LZ^{2}\qquad \text{ and } \qquad
	\EE\Big\{ \frac{r_1r_2G_1^2}{r_1r_2 + r_1G_1^2 + r_2G_2^2} \Big\} \leq L^{2}.
\]
Consequently, if $\lambda \geq \max\{ L^2,\LZ^2\}$, then from Eq.~\eqref{fixed-point-eq-reformulate}, we obtain that $r_1 \leq \oversamp(\lambda + \LZ^2) \leq 2\lambda \oversamp = 2\lambda m/d$ and similarly $r_{2} \leq 2\lambda m/d$, which concludes the proof of part (a). 

\paragraph{Proof of Lemma~\ref{fixed-point-equations-unique-solution}(b.):} To reduce the notational burden, we use the shorthand
\begin{align*}
	a &= 1- \EE\Big\{ \frac{ \frac{(d-2)m}{d^{2}} r_2^2 G_2^4}{(r_1r_2 + r_1G_1^2 + r_2G_2^2)^2}\Big\},\qquad& a_2 = \EE\Big\{ \frac{ -\frac{(d-2)m}{d^{2}} r_2^2 G_1^2G_2^2}{(r_1r_2 + r_1G_1^2 + r_2G_2^2)^2}\Big\},\\
	a_3 &= \EE\Big\{ \frac{ -\frac{(d-2)m}{d^{2}} r_1^2 G_1^2G_2^2}{(r_1r_2 + r_1G_1^2 + r_2G_2^2)^2}\Big\}, \qquad \text{ and } 
	&a_4 = 1 - \EE\Big\{ \frac{ \frac{(d-2)m}{d^{2}} r_1^2 G_1^4}{(r_1r_2 + r_1G_1^2 + r_2G_2^2)^2}\Big\}.
\end{align*}
Equipped with this notation, we recall $V_3$~\eqref{definition-V3} and $V_4$~\eqref{definition-V4}, and re-write the fixed point equations~\eqref{det_updates_eta1} and~\eqref{det_updates_eta2} as  
\[
	\left[ \begin{array}{cc} a_1 & a_2 \\ a_3 & a_4 \end{array} \right] \left[ \begin{array}{c} \etaX^2 \\ \etaZ^2 \end{array} \right] 
	= \frac{(d-2)m}{d^{2}} \left[  \begin{array}{c} V_3 \\ V_4 \end{array} \right].
\]
The linear equation in the preceding display admits a unique solution if and only if the determinant $a_1a_4-a_2a_3 \neq 0$. Note that
\[
 a_1 \geq 1 - \frac{m}{d} \frac{ \EE\{G_2^4\} }{r_1^{2}} \geq 1 - \frac{3\LZ^{4}}{\lambda^{2} m/d} \geq 0.9 \qquad \text{ and } \qquad
 |a_2| \leq  \frac{m}{d}\frac{L^{2}\LZ^{2}}{r_1^2} \leq \frac{L^{2}\LZ^{2}}{\lambda^{2} m/d} \leq 0.1,
\]
since $r_1,r_2 \geq \lambda m/d$ and $\lambda \geq C\max\{1, L^{4},\LZ^{4}\}d/m$. Similarly, we find that $a_4\geq 0.9$ and $|a_3| \leq 0.1$. Consequently, $a_1a_4 - a_2a_3\geq 0.9^{2} - 0.1^{2} >0$, so that the solution of equations~\eqref{det_updates_eta1} and~\eqref{det_updates_eta2} is unique. Moreover, solving the linear equation yields
\[
	\etaX^{2} = \frac{(d-2)m}{d^{2}}\frac{a_4V_{3} - a_2 V_4}{a_1a_4-a_2a_3} \qquad \text{ and } \qquad 
	\etaZ^{2} = \frac{(d-2)m}{d^{2}}\frac{a_1V_{3} - a_3 V_4}{a_1a_4-a_2a_3}.
\]
Since $a_1,a_4 \geq 0.9$, $a_2,a_3 \leq 0$ and $V_3,V_4 \geq 0$, we obtain that the solution is non-negative.

\paragraph{Proof of Lemma~\ref{fixed-point-equations-unique-solution}(c.):} By definition we have 
\[
	\| \bcoefX_{\sharp} \bcoefZ_{\sharp}^{\top} -  \bcoefX_{\star} \bcoefZ_{\star}^{\top} \|_{F}^{2} = \|\bcoefX_{\sharp} \bcoefZ_{\sharp}^{\top}\|_{F}^{2} - 2\langle \bcoefX_{\sharp} \bcoefZ_{\sharp}^{\top}, \bcoefX_{\star} \bcoefZ_{\star}^{\top} \rangle + 
	\|\bcoefX_{\star} \bcoefZ_{\star}^{\top}\|_{F}^{2} = (\parcompX_\sharp \parcompZ_\sharp - 1)^{2} + \parcompX_{\sharp}^{2}\perpcompZ_{\sharp}^{2}
	+ \parcompZ_{\sharp}^{2}\perpcompX_{\sharp}^{2} + \perpcompX_{\sharp}^{2}\perpcompZ_{\sharp}^{2}.
\]
Using $\perpcompX_{\sharp} \leq 0.1$, $|\parcompX_{\sharp}| \leq \|\bcoefX_{\sharp}\|_{2} \leq 1.7$ and $|\parcompZ_{\sharp}| \leq \|\bcoefZ_{\sharp}\|_{2} \leq 1.7$, we obtain that
\[
	\| \bcoefX_{\sharp} \bcoefZ_{\sharp}^{\top} -  \bcoefX_{\star} \bcoefZ_{\star}^{\top} \|_{F}^{2} \leq (1.7^{2} + 0.1) \cdot 
	\big[(\parcompX_\sharp \parcompZ_\sharp - 1)^{2} + \perpcompX_{\sharp}^{2} + \perpcompZ_{\sharp}^{2}
	\big]\leq 5\cdot \Err_{\sharp}.
\]
Continuing, note that $\parcompX_{\sharp}^{2} = \|\bcoefX_{\sharp}\|_{2}^{2} - \perpcompX_{\sharp}^{2} \geq 0.3^{2} - 0.1^{2} = 0.08$ and similarly $\parcompZ_{\sharp}^{2} \geq 0.08$. Thus,
\[
	\| \bcoefX_{\sharp} \bcoefZ_{\sharp}^{\top} -  \bcoefX_{\star} \bcoefZ_{\star}^{\top} \|_{F}^{2} \geq 0.08 \big[ (\parcompX_\sharp \parcompZ_\sharp - 1)^{2} + \perpcompX_{\sharp}^{2} + \perpcompZ_{\sharp}^{2} \big] \geq 0.08 \Err_{\sharp}.
\]
Putting the pieces together yields the desired result.
\end{proof}

\end{document}